\documentclass[12pt, twoside, openany]{book}
\usepackage[utf8]{inputenc}
\usepackage[title]{appendix}
\usepackage{geometry}               		% See geometry.pdf to learn the layout options. There are lots.
\usepackage{graphicx}				% Use pdf, png, jpg, or eps‚Äö√†√∂‚àö¬∫ with pdflatex; use eps in DVI mode
\usepackage{xcolor}								
\usepackage{amssymb}

\newcommand{\norm}[1]{\left\lVert#1\right\rVert}

\newcommand{\R}{\mathbb R}
\newcommand{\wh}{\widehat}
\usepackage[T1]{fontenc}
\usepackage{mathtools}
\usepackage{mathrsfs}
\usepackage{amsmath, amssymb, amsthm, enumerate, textcomp, fullpage, cite}
\usepackage{listings}
\usepackage{enumitem}
\usepackage{lmodern}
\definecolor{mygreen}{RGB}{28,172,0} % color values Red, Green, Blue
\definecolor{mylilas}{RGB}{170,55,241}

\usepackage{tikz}
\usepackage{esvect}
\usepackage{hyperref}
\definecolor{red}{rgb}{1,0,0}

\numberwithin{equation}{chapter}

\newtheorem{theorem}{Theorem}[chapter]
\newtheorem{lemma}[theorem]{Lemma}
\newtheorem{proposition}[theorem]{Proposition}
\newtheorem{corollary}[theorem]{Corollary}
\newtheorem{definition}[theorem]{Definition}
\newtheorem{remark}[theorem]{Remark}
\newtheorem{conj}[theorem]{Conjecture}

\title{Blow-Up Dynamics for the $L^2$-critical case of the Two-Dimensional Zakharov-Kuznetsov equation}
\author{Francisc Bozgan \thanks{NYUAD Research Institute, New York University Abu Dhabi, Saadiyat Island, PO Box 129188, Abu Dhabi, United Arab Emirates, {\sf ftb208@nyu.edu}}\quad Tej-Eddine Ghoul \thanks{Department of Mathematics, New York University Abu Dhabi, Saadiyat Island, PO Box 129188, Abu Dhabi, United Arab Emirates, {\sf teg6@nyu.edu}} \quad Nader Masmoudi \thanks{Department of Mathematics, New York University in Abu Dhabi, Saadiyat Island, P.O. Box 129188, Abu Dhabi, United Arab Emirates-- Courant Institute of Mathematical Sciences, New York University, 251 Mercer Street, New York, NY 10012, USA, {\sf nm30@nyu.edu}} \quad Kai Yang \thanks{College of Mathematics and Statistics, Chongqing University, Chongqing City, {\sf yangkai10@hotmail.com}}}
\date{}
\providecommand{\DIFaddtex}[1]{{\color{red}#1}} %DIF PREAMBLE
\providecommand{\DIFdeltex}[1]{} %DIF PREAMBLE
\providecommand{\DIFaddbegin}{} %DIF PREAMBLE
\providecommand{\DIFaddend}{} %DIF PREAMBLE
\providecommand{\DIFdelFL}[1]{} %DIF PREAMBLE
\makeatletter %DIF PREAMBLE
\let\RevisionOriginalBibitem\bibitem %DIF PREAMBLE
\renewcommand{\bibitem}[2][]{\ifcsname RevisionBib@#2\endcsname\color{red}\else\color{black}\fi\RevisionOriginalBibitem[#1]{#2}} %DIF PREAMBLE

\makeatother %DIF PREAMBLE
\providecommand{\DIFadd}[1]{\texorpdfstring{\DIFaddtex{#1}}{#1}} %DIF PREAMBLE
\newenvironment{abstract}{%
  \chapter*{Abstract}%
  \addcontentsline{toc}{chapter}{Abstract}%
  \small\begin{quotation}%
}{%
  \end{quotation}\clearpage
}
\let\BookOriginalThebibliography\thebibliography
\renewcommand{\thebibliography}[1]{%
  \BookOriginalThebibliography{#1}%
  \addcontentsline{toc}{chapter}{\bibname}%
}
\begin{document}
\frontmatter
\maketitle
\begin{abstract}
We study blow-up dynamics for the $L^2$-critical cubic
Zakharov--Kuznetsov equation in two dimensions,
\[
    \partial_tu+\partial_{x_1}(\Delta u+u^3)=0
    \qquad\text{on }\mathbb R^2.
\]
For a class of localized $H^1$ perturbations of the ground state $Q$, we
establish a trichotomy near the soliton manifold: exit from a small
$L^2$-tube, global asymptotic stability, or finite-time blow-up. In the
stable blow-up regime, the solution concentrates a single bubble and
\[
    \lambda(t)\sim \ell_0(T-t)^{1/(3-c)},
\]
where $\ell_0>0$ depends on the initial datum and $c\in(1,2)$ is an
explicit constant determined by the transverse tail of the first-order
approximate profile. Consequently,
\[
    \|\nabla u(t)\|_{L^2}
    \sim
    \frac{\|\nabla Q\|_{L^2}}
         {\ell_0(T-t)^{1/(3-c)}}.
\]
After subtraction of the concentrating soliton, the radiation converges
strongly in $L^p(\mathbb R^2)$ for every $2\leq p<\infty$ to a common
nonzero profile $u^*$, while
\[
    u^*\notin H^s(\mathbb R^2)
    \qquad\text{for every }s\geq\frac c2.
\]
The stable blow-up branch is open in the relative $H^1$ topology of the
localized class. Finally, every non-soliton datum in this class with
non-positive energy blows up in finite time. Interval-arithmetic computer-assisted proofs certify the numerical
inputs to the virial coercivity argument. They also yield a rigorous
enclosure of $c$, justifying the polynomial moment of order $21$
imposed on the initial data.
\end{abstract}

% BEGIN TABLE OF CONTENTS
\clearpage
\tableofcontents
\clearpage
% END TABLE OF CONTENTS
\mainmatter
\chapter{Introduction}
The generalized Zakharov-Kuznetsov (gZK) equation with initial data $u_0$ reads as

\begin{equation*}
\begin{cases} \partial_t u+\partial_{x_1} (\Delta u+u^p)=0, \mbox{ } x=(x_1,\ldots, x_N)\in \mathbb{R}^N, \mbox{ }p\in \mathbb{N}, p\geq 2, \mbox{ } t\geq 0,\\
u(0,x_1,\ldots, x_N)=u_0(x_1, \ldots, x_N),
\end{cases}
\end{equation*} 
(where $\Delta$ is the Laplacian on $\R^N$) which is a higher-dimensional extension of the generalized Korteweg-de Vries (KdV) model for shallow water waves 
\begin{equation*}
\partial_t u+\partial_x(\partial_x^2u+u^p)=0, \mbox{ } x\in \mathbb{R}, \mbox{ }t\geq 0.
\end{equation*}
When $p=2$, the above equation is called the KdV equation and if $p=3$ it is called the modified KdV equation. 
In this book, we are interested in the gZK equation, which was originally derived by Zakharov and Kuznetsov to describe weakly magnetized ion-acoustic waves in a plasma comprising cold ions and hot isothermal electrons in the presence of a uniform magnetic field in $3$D \cite{ZakharovKuznetsov}. 
During the lifespan of the solution $u(t),$ we have the conserved mass and energy: for $x \in \mathbb{R}^N,$
$$M[u(t)]=M[u(0)]=\int_{\mathbb{R}^N}u(t,x)^2dx$$
and 

$$E[u(t)]=E[u(0)]=\frac{1}{2}\int_{\mathbb{R}^N}|\nabla u(t,x)|^2dx-\frac{1}{p+1}\int_{\mathbb{R}^N}[u(t,x)]^{p+1}dx.$$
 The scaling symmetry is $u_{\lambda}=\lambda^{\frac{2}{p-1}}u(\lambda^3t,\lambda x)$ and 
 $$\|u_{\lambda}(0,\cdot)\|_{\dot{H}^s(\mathbb{R}^N)}=\lambda^{\frac{2}{p-1}+s-\frac{N}{2}}\|u_0\|_{\dot{H}^s(\mathbb{R}^N)},$$ 
so the two-dimensional cubic ZK equation is $L^2$-critical when $s=0$. 
 From now on, we will discuss only the two-dimensional case, namely $N=2$ and we call the spatial variables $(x_1,x_2).$ 
 The generalized ZK equation has a family of traveling waves  (or solitary waves, solitons) and they travel in the $x_1$-direction 
 $$u(t,x_1,x_2)=Q_{\kappa}(x_1-\kappa t,x_2)$$
 with $Q_{\kappa}(x_1,x_2)\rightarrow 0$ as $|(x_1,x_2)| \rightarrow +\infty.$ Here, for $\kappa>0$, $Q_{\kappa}$ is the dilation of the ground state: 

$$Q_{\kappa}(x)={\kappa}^{\frac{1}{p-1}}Q({\kappa}^{\frac{1}{2}}x), \mbox{ } x=(x_1,x_2),$$
with $Q$ being the unique radial positive solution in $H^1(\mathbb{R}^2)$ of the nonlinear elliptic equation $-\Delta Q+Q-Q^p=0$ (for existence, see \cite{WalterStrauss}, \cite{BerestyckiLionsPeletier}, for uniqueness \cite{Kwong}). From \cite{GidasNiNirenberg}, we note that $Q \in C^{\infty}(\mathbb{R}^2),$ $\partial_rQ(r)<0$ for any $r=|x|>0,$ and for any multi-index $\alpha,$ 
$$|\partial^{\alpha}Q(x)|\leq c(\alpha)e^{-|x|} \mbox{ for any } x\in \mathbb{R}^2.$$

In terms of the stability of the traveling wave, we start with the following definition. 
\begin{definition} 
Denote 
$$U_{\alpha}=\Big\{ u \in H^1(\mathbb{R}^2) : \inf_{\vec{y} \in \mathbb{R}^2}\|u(\cdot)-Q(\cdot+\vec{y})\|_{H^1}\leq \alpha\Big\}.$$
We say that $Q$, the radially symmetric solution of $-\Delta Q+Q-Q^p=0$, is stable if for all $\alpha> 0,$ there exists $\delta>0$ such that if $u_0 \in U_{\delta},$ then the corresponding solution $u(t)$ is defined for all $t\geq 0$ and $u(t)\in U_{\alpha}$ for all $t\geq 0.$ Otherwise, we say $Q$ is unstable. 
\end{definition}
In her pivotal work on dispersive solitary waves across higher dimensions, de Bouard \cite{deBouard} demonstrated that in dimension two, the stability of traveling waves of the form \(Q(x_1-t, x_2)\) depends critically on the nonlinearity exponent \(p\). Specifically, these waves are stable for \(p < 1 + 4/N=3\) and become unstable for \(p > 1 + 4/N=3\). This analysis draws upon foundational concepts developed by Bona, Souganidis, and Strauss \cite{BonaSouganidisStrauss} for instability, as well as stability frameworks by Grillakis, Shatah, and Strauss \cite{GrillakisShatahStrauss}. For more on the stability of the solutions, see \cite{Kuznetsov}. 

For the \(L^2\)-critical case in two dimensions, where \(p=3\), Farah, Holmer, and Roudenko \cite{FarahHolmerRoudenko} applied methods initially established by Merle and Martel \cite{MerleMartel} to demonstrate the instability of solitons within this regime. Further explorations by the same group \cite{FarahHolmerRoudenko2} offered an alternative proof of instability in the supercritical case for \(p>3\), utilizing techniques adapted from Combet's work \cite{Combet} on the generalized KdV equation. Moreover, C\^{o}te, Mu\~{n}oz, Pilod, and Simpson \cite{CoteMunozPilodSimpson} proved the asymptotic stability for \( p=2 \).

From now on, we will consider only $p=3.$ 
It follows from the conservation of the mass and energy, together with the sharp Gagliardo--Nirenberg inequality \cite{Weinstein1983}, that if $\|u\|_{L^2}<\|Q\|_{L^2}$ we have 
$$\|\nabla u\|^{2}_{L^2}\leq \frac{2E(u)}{1-\frac{\|u\|_{L^2}^2}{\|Q\|_{L^2}^2}}.$$

In this case, we will have global well-posedness for the equation, thus blow-up can only occur in the case $\|u\|_{L^2}\geq \|Q\|_{L^2}.$ 

The local well-posedness theory is well-known and we state the following known results: 

\begin{theorem} 
The cubic 2D ZK equation is locally well-posed in $H^{s}(\mathbb{R}^2)$, with $s\geq \frac{1}{4}.$ (\cite{RibaudVento}, \cite{Kinoshita}). Moreover, if $\|u_0\|_{L^2}<\|Q\|_{L^2},$ we have global well-posedness in $H^1(\mathbb{R}^2)$ in \cite{LinaresPastor}, which was further improved to $H^{s}(\mathbb{R}^2)$ with $s>\frac{53}{63}$ in \cite{LinaresPastor2} and $H^{s}(\mathbb{R}^2)$ with $s>\frac{3}{4}$ in \cite{BhattacharyaFarahRoudenko}.
\end{theorem}

We define $T>0$ to be the maximal time of existence for the cubic 2D ZK equation. For the blow-up question of the gZK, we saw previously that we need $\|Q\|_{L^2}\leq\|u\|_{L^2}.$ From the local well-posedness theory (\cite{Faminski}, \cite{LinaresPastor2}, \cite{RibaudVento}), if $T<+\infty,$ then 
$$\lim_{t \nearrow T}\|\nabla u(t)\|_{L^{2}_{x_1x_2}}=+\infty.$$
If $T=+\infty,$ then either 
$$\lim_{t \to +\infty}\|\nabla u(t)\|_{L^{2}_{x_1x_2}}=+\infty \mbox{ or } \liminf_{t \to +\infty}\|\nabla u(t)\|_{L^{2}_{x_1x_2}}<+\infty$$
are a priori possible. In either case $T<+\infty$ or $T=+\infty,$ we say that $u(t)$ blows up at forward time $T$ if 
$$\liminf_{t \nearrow T}\|\nabla u(t)\|_{L^{2}_{x_1x_2}}=+\infty.$$

In terms of blow-up results for the $L^2$ critical ZK equation, we have the following result: 

\begin{theorem} \cite{FarahHolmerRoudenkoYang}.
There exists $\alpha_0>0$ such that the following holds. Suppose $u(t)$ is a solution in $H^1$ of the 2D cubic ZK equation with $E_0<0$ and 
$$0<\|u\|_{L^2}^2-\|Q\|_{L^2}^2\leq \alpha_0.$$
Then $u(t)$ blows up in finite or infinite forward time. 
\end{theorem}

Their proof is based on the blow-up analysis for the $L^2$ critical gKdV equation in \cite{Merle01} and \cite{MartelMerle00}. 

\chapter{Main Result} 

\section{Preliminaries and Motivation}

In this section, we introduce the notation and preliminary concepts necessary for our main result. Let \( Q \) denote the unique, positive, radial solution in \( \mathbb{R}^2 \) to the nonlinear elliptic equation
\[
-\Delta Q + Q - Q^3 = 0.
\]
This solution is fundamental in our analysis of the generalized Zakharov-Kuznetsov equation. Define \( g_Q(x_2) \) as 
\begin{equation}\label{definitionofg}
    g_Q(x_2) = \int_{-\infty}^{\infty} x_2 Q_{x_2}(x_1, x_2) \, dx_1.
\end{equation}
The Fourier transform of \( g_Q \), denoted \( \widehat{g_Q}(\xi) \), is taken over the real line \( \mathbb{R} \).

The crucial constant \( c \) is then defined by the formula:
\begin{equation}
c = \frac{\int_{\mathbb{R}} \frac{2}{\xi^2 + 1} |\widehat{g_Q}(\xi)|^2 \, d\xi}{\int_{\mathbb{R}} |\widehat{g_Q}(\xi)|^2 \, d\xi}.
\label{eq:definitionofc}
\end{equation}
It is easy to see that \( 0 < c < 2 \). We can say something even stronger. 

\begin{proposition}
    The constant $c$ is strictly greater than $1$.
\end{proposition}
The proof is in Appendix \ref{appendixD}. For a sharper estimate, Theorem \ref{cap:verified-main} provides the rigorous interval certificate 
\[
 c\in[1.66126332085379, 1.66170874135113]
       \subset\left(\frac{21}{13},\frac53\right).
\]
 See Appendix \ref{appendixE}. Importantly, this constant \( c \) is related to the exponent of the blow-up rate in our main theorem concerning the blow-up behavior of solutions to the $L^2$-critical Zakharov-Kuznetsov equation in two dimensions, 
\begin{equation} \label{eq:ZK}
\begin{cases}
u_t+\partial_{x_1}(\Delta u+u^3)=0,\quad (x_1,x_2)\in \mathbb{R}^2, t\geq 0,\\
u(0,x_1,x_2)=u_0(x_1,x_2) \in H^1(\mathbb{R}^2).
\end{cases}
\end{equation} 
The constant \(c\) determines the exponent \(1/(3-c)\) appearing in
the blow-up law stated below. Its analytic origin is tied to the
transverse structure of the first-order approximate profile and will
be derived in the subsequent modulation analysis. 

We study the dynamics near the soliton and compare the resulting
blow-up rate with the numerical conjecture of Klein, Roudenko, and
Stoilov \cite{KleinRoudenkoStoilov}. Their conjecture proposes the
exponent $1/2$ for the $L^2$-critical generalized  Zakharov--Kuznetsov
equation and is stated below.

\begin{conj} 
Consider the critical 2D ZK equation \eqref{eq:ZK}. If $u_0\in \mathcal{S}(\mathbb{R}^2)$ is sufficiently localized and $\|u_0\|_{L^2}>\|Q\|_{L^2},$ then the solution blows up in finite time $T$ such that as $t\rightarrow T$
$$u(x_1,x_2,t)-\frac{1}{\lambda(t)}Q\Big(\frac{x_1-x_1(t)}{\lambda(t)},\frac{x_2-x_2(t)}{\lambda(t)}\Big)\rightarrow \widetilde{u} \in L^2,$$
 with 
 $$\|\partial_{x_1} u(t)\|_{L^2}\sim \frac{1}{(T-t)^{\frac{1}{2}}}, \quad \lambda(t)\sim (T-t)^{\frac{1}{2}},\quad  x_1(t)\sim \frac{1}{T-t}, \quad x_2(t)\rightarrow y^*\in \mathbb{R}.$$
\end{conj}

Our investigation targets specific initial data close to the ground state \( Q \), the unique positive radial solution of the equation \( -\Delta Q + Q - Q^3 = 0 \) in \( \mathbb{R}^2 \). For $i\in \mathbb{N}$, we define the set of initial conditions by:
\[
\mathcal{A}_{\alpha_0}^{(i)} = \left\{ u_0 = Q + \varepsilon_0 : \|\varepsilon_0\|_{H^1(\mathbb{R}^2)} < \alpha_0 \text{ and } \int_{x_1 > 0} x_1^i \varepsilon_0^2(x_1, x_2) \, dx_1 \, dx_2 < 1 \right\}.
\]

To analyze the behavior of solutions, we consider the \( L^2 \)-modulated tube surrounding the soliton manifold:
\[
\mathcal{T}_{\alpha^*} = \left\{ u \in H^1(\mathbb{R}^2) : \inf_{\lambda_0 > 0, (\tilde{x}_1, \tilde{x}_2) \in \mathbb{R}^2} \left\| u - \frac{1}{\lambda_0} Q\left(\frac{\cdot - \tilde{x}_1}{\lambda_0}, \frac{\cdot - \tilde{x}_2}{\lambda_0}\right) \right\|_{L^2_{x_1x_2}} < \alpha^* \right\}.
\]
The constants \( \alpha_0 \) and \( \alpha^* \) are chosen such that \( 0 < \alpha_0 \ll \alpha^* \ll 1 \).

\section{Main Theorem}\label{mainthm}

\begin{theorem}\label{maintheorem} 
For universal constants \(0<\alpha_0\ll \alpha^*\ll1\), the following
conclusions hold for solutions $u(t)$ of equation \eqref{eq:ZK} with initial data $u_0\in \mathcal{A}_{\alpha_0}^{(21)}$:
\begin{itemize}
    \item[(a)] If the energy \( E(u_0) \leq 0 \) and \( u_0 \) is not the soliton up to symmetries, then \( u(t) \) blows up in finite time \( T \), and for all \( t \in [0, T) \), \( u(t) \) remains within \( \mathcal{T}_{\alpha^*} \).
    \item[(b)] Assume that $u(t)$ blows up in finite time $T$ and remains
in $\mathcal T_{\alpha^*}$ for every $t\in[0,T)$. Then there exists
$\ell_0=\ell_0(u_0)>0$ such that, with $c$ defined in
\eqref{eq:definitionofc},
\[
    \|\nabla u(t)\|_{L^2(\mathbb R^2)}
    \sim
    \frac{\|\nabla Q\|_{L^2(\mathbb R^2)}}
         {\ell_0(T-t)^{1/(3-c)}}
    \qquad\text{as }t\uparrow T.
\]
Moreover, there exist modulation parameters
$\lambda(t)$, $x_1(t)$, $x_2(t)$ and a nonzero function
\[
    u^*\in\bigcap_{2\leq p<\infty}L^p(\mathbb R^2)
\]
such that, for every $2\leq p<\infty$,
\[
\left\|
 u(t,x_1,x_2)
 -\frac1{\lambda(t)}
 Q\left(
     \frac{x_1-x_1(t)}{\lambda(t)},
     \frac{x_2-x_2(t)}{\lambda(t)}
 \right)
 -u^*
\right\|_{L^p(\mathbb R^2)}
\longrightarrow0
\qquad\text{as }t\uparrow T.
\]
The limiting radiation satisfies
\[
    \|u^*\|_{L^2(\mathbb R^2)}^2
    =
    \|u_0\|_{L^2(\mathbb R^2)}^2
    -
    \|Q\|_{L^2(\mathbb R^2)}^2,
\]
and
\[
    u^*\notin H^s(\mathbb R^2)
    \qquad\text{for every }s\geq\frac c2.
\]
Finally,
\[
\lambda(t)\sim\ell_0(T-t)^{1/(3-c)},
\qquad
x_1(t)\sim
\frac{3-c}{c-1}\,
\frac1{\ell_0^2(T-t)^{(c-1)/(3-c)}},
\qquad
x_2(t)\longrightarrow x_{2,\infty}\in\mathbb R.
\]

    \item[(c)] Openness of the stable blow-up branch. Let \(T_{u_0}\) denote the maximal
time of existence of the solution with initial datum \(u_0\), and define
\[
    \mathcal O
    =
    \left\{
    u_0\in H^1:
    u(t)\in\mathcal T_{\alpha^*}
    \text{ for all }t\in[0,T_{u_0})
    \right\},
\]
and denote the subset of solutions that blow up in finite time by
\[
    \mathcal O_b
    =
    \left\{
    u_0\in\mathcal O:
    T_{u_0}<\infty
    \text{ and }
    \|\nabla u(t)\|_{L^2}\to\infty
    \right\}.
\]
Then
$\mathcal O_b\cap\mathcal A_{\alpha_0}^{(21)}$
is open in $\mathcal A_{\alpha_0}^{(21)}$ for the relative $H^1$
topology.
\end{itemize}
\end{theorem}

\section{Detailed Comments}

\textit{Comments:}
\begin{itemize}
    \item[(1).]
There has been substantial progress concerning finite-time blow-up dynamics for the 
$L^2$-critical generalized Korteweg--de Vries equation
\[
\begin{cases}
u_t + \partial_x(u_{xx} + u^5) = 0, \quad x \in \mathbb{R}, \ t\geq 0, \\
u(0,x) = u_0(x), \quad x \in \mathbb{R},
\end{cases}
\tag{gKdV}
\label{gKdV}
\]
see \cite{MartelMerle00,Merle01,MartelMerle02,MartelMerle02bis,
MartelMerleRaphael1,MartelMerleRaphael15,MartelMerleRaphael15bis,
CombetMartel,MartelPilod24}. The analysis in \cite{MartelMerleRaphael1} yields stable finite-time blow-up  near the ground state, asymptotic laws for the modulation
parameters, and stability of this regime under perturbations of the
initial data. These conclusions have counterparts in
Theorem~\ref{maintheorem}. The blow-up exponent and the regularity of
the limiting radiation differ, as described below. 

    \item[(2).] The blow-up rate \(\|\nabla u(t)\|_{L^2}\) is significantly faster, behaving as \(\frac{1}{(T-t)^{\frac{1}{3-c}}}\), compared to the self-similar blow-up rate of \(\frac{1}{(T-t)^{\frac{1}{3}}}\). In particular, for \(c > 1\), the blow-up locus in the \(x_1\) direction recedes to infinity, while in the \(x_2\) direction it converges to a fixed point. This distinction underscores the critical role of the decay behavior in \(x_1\) at infinity in the initial data, a feature highlighted by the specific weighting in the variable \(x_1\) for \(u_0 \in \mathcal{A}_{\alpha_0}^{(21)}\). In contrast, the \(x_2\) variable does not require similar decay conditions. See Theorem \ref{rigiditytheorem} for more details. The weight $x_{1}^{21}$ in the definition of $\mathcal{A}_{\alpha_0}^{(21)}$ is not optimal. For example, observe that a smaller weight, i.e. $x_{1}^{18},$ would be sufficient to prove the blow-up dynamics for the negative energy case. The moment of order $21$ is required by the proof of blow-up for the zero energy case (see Remark \ref{21vs18}). 
    \item[(3).] Employing the sharp Gagliardo--Nirenberg inequality \cite{Weinstein1983}, we find:
    \[
    \frac{1}{2}\|\nabla u\|^{2}_{L^2} \left(1-\frac{\|u_0\|^{2}_{L^2}}{\|Q\|^{2}_{L^2}}\right) \leq E_0,
    \]
    indicating that \(E(u_0) \leq 0\) necessarily implies \(\|u_0\|_{L^2} > \|Q\|_{L^2}\), unless \(u_0\) is equivalent to \(\pm Q\) up to scaling and translations. For data in $\mathcal A_{\alpha_0}^{(21)}$, Theorem~\ref{maintheorem} proves blow-up in the non-soliton case.

    \item[(4).]
The behavior of the limiting radiation differs sharply from the
$L^2$-critical gKdV case. In \cite{MartelMerleRaphael1}, the asymptotic
radiation belongs to $H^1$, and strong convergence in $H^1$ is
conjectured. Here, the radiation converges strongly in
$L^p(\mathbb R^2)$ for every $2\leq p<\infty$, but the limiting profile
satisfies
\[
    u^*\notin H^\varsigma(\mathbb R^2)
    \qquad\text{for every }\varsigma\geq\frac c2.
\]
Thus the failure of strong $H^1$ convergence is intrinsic: the terminal
profile itself does not belong to the energy space. The use of a
Duhamel--dispersive argument is closer in spirit to the NLS blow-up
analysis in \cite{MerleRaphaelNLS}, but the regularity obstruction here
is genuinely two-dimensional. It is generated by the transverse tail of
the first-order profile and is detected in the anisotropic high-frequency
region
\[
    |\xi_1|\asymp|\xi_2|^{1-c},
    \qquad |\xi_2|\to\infty.
\]
See Theorem~\ref{thm:radiation-not-Hs}.

    \item[(5).] Conjecture $2$ in \cite{KleinRoudenkoStoilov} predicts, for the $L^2$-critical two-dimensional ZK equation, the blow-up rate $\|\partial_{x_1}u(t)\|_{L^2}\sim (T-t)^{-1/2}$.

 In the stable near-soliton  regime, the modulated decomposition and
\eqref{convergencepartial} -- \eqref{convergencenabla} give
$\lambda\|\partial_{x_1}u\|_2\to\|Q_{x_1}\|_2>0$ and
$\lambda\|\nabla u\|_2\to\|\nabla Q\|_2>0$. Thus the two derivative
norms have the same blow-up exponent. Consequently, the rate conjectured in \cite{KleinRoudenkoStoilov} would correspond, in our notation, to the formal value $c=1.$
For initial data in the perturbative class $\mathcal A_{\alpha_0}^{(21)}$, our analysis gives the constant $c$ defined in \eqref{eq:definitionofc}, and the numerical computation yields $c\in[1.66126332085379, 1.66170874135113],$ hence $\frac{1}{3-c}\in [0.7469728853905, 0.7472214987115]$.
 In Appendix \ref{appendixE} we include the rigorous proof of the enclosure of $c.$

 The experiments in \cite{KleinRoudenkoStoilov}  include data such as
$1.1Q$ and  $3e^{-x_1^2-x_2^2}$; their membership in the sufficiently
small class $\mathcal A_{\alpha_0}^{(21)}$  is not established here.
The conjecture is based on numerical observations, whereas our result
concerns the stable perturbative regime near the ground state. In this
regime, the exponent is $1/(3-c)$ and differs from $1/2$. This comparison
does not classify possible blow-up regimes outside that class.

    \item[(6).] This extends the result of \cite{FarahHolmerRoudenkoYang} to the zero-energy case and provides the corresponding dynamical classification. 

    \item[(7).] After the first version of this manuscript was posted, Chen et al. posted an article \cite{ChenLanYuan} deriving a similar result to ours. However, they use a different approach by considering the adjoint problem which leads to a different energy functional. 
\end{itemize}

The following trichotomy describes the dynamics near the soliton, in the same three regimes as in the critical gKdV analysis of \cite{MartelMerleRaphael1}.

\begin{theorem} \label{trichotomytheorem}
There exist universal constants  $0<\alpha_0\ll \alpha^*\ll 1$ such that the following holds. Let $u_0 \in \mathcal{A}_{\alpha_0}^{(21)}.$ Then, we have a trichotomy for the dynamics up to the first exit from the soliton tube, namely exactly one of the following three cases occurs: 
\begin{itemize} 
\item[(1)] (Exit of Tube) There exists $t^*\in(0,T)$ such that $u(t^*)\notin \mathcal{T}_{\alpha^*}.$
\item[(2)] (Stable Blow-up) For all $t \in [0,T), u(t)\in \mathcal{T}_{\alpha^*}$ and the solution blows up in finite time $T<+\infty$ in the way described by Theorem \ref{maintheorem}.
\item[(3)]  (Asymptotic Stability) The solution is global, with $u(t)\in \mathcal{T}_{\alpha^*},$ for all $t\geq 0$ and there exists $\lambda_{\infty}>0, x_1(t)\in C^1, x_{2,\infty}\in \mathbb{R}$ such that 
$$\lambda_{\infty}u(t,\lambda_{\infty}\cdot+x_1(t), \lambda_{\infty}\cdot+x_{2,\infty})\rightarrow Q \mbox{ in } H^1_{loc} \mbox{ as }t \rightarrow +\infty,$$
with $|\lambda_{\infty}-1|\leq o_{\alpha_0\rightarrow 0}(1)$ and  $x_1(t)\sim \frac{t}{\lambda_{\infty}^{2}}$ as $t\rightarrow +\infty.$
\end{itemize}
\end{theorem}

\textbf{Notation.} Throughout this book, we denote the spatial variable by $x=(x_1, x_2) \in \mathbb{R}^2$ and the time variable by $t$. We reserve $y=(y_1, y_2)$ and $s$ for the rescaled spatial and time variables, respectively.

We define the linearized operator $L$ around the ground state $Q$ by
\[
L = -\Delta + 1 - 3Q^2.
\]
We also introduce the scaling operator $\Lambda$, defined as
\[
\Lambda f = f + x \cdot \nabla f = f + x_1 \partial_{x_1} f + x_2 \partial_{x_2} f.
\]
Regarding asymptotic parameters, for any $0 < \alpha \ll 1$, the notation $\delta(\alpha)$ represents a generic small positive constant that satisfies
\[
\lim_{\alpha \to 0} \delta(\alpha) = 0.
\]
We write $H^\infty(\mathbb R^2)=\bigcap_{m\geq0}H^m(\mathbb R^2)$.
Finally, unless otherwise specified, $(\cdot, \cdot)$ denotes the standard scalar product of real-valued functions in $L^2(\mathbb{R}^2)$:
\[
(f, g) = \int_{\mathbb{R}^2} f(x) g(x) \, dx.
\]

\section{Heuristics and Construction of the Approximate Profile}

The proof relies on the construction of an approximate blow-up profile that captures the self-similar nature of the collapse. We seek an approximate solution to the Zakharov-Kuznetsov (ZK) equation of the form:
\[
u(t,x_1,x_2) \approx \frac{1}{\lambda(t)} Q_{b(t)}\left(\frac{x_1 - x_1(t)}{\lambda(t)}, \frac{x_2 - x_2(t)}{\lambda(t)}\right).
\]
Here, the dynamics of the scaling parameter $\lambda(t)$, the translation parameters $(x_1(t), x_2(t))$, and the auxiliary scaling parameter $b(t)$ are governed by the system:
\begin{equation*}
\frac{ds}{dt} = \frac{1}{\lambda^3(t)}, \quad b = -\frac{\lambda_s}{\lambda}, \quad \frac{(x_1)_s}{\lambda} = 1, \quad \frac{(x_2)_s}{\lambda} = 0.
\end{equation*}
Substituting this ansatz into \eqref{eq:ZK} yields an approximate self-similar equation for the profile $Q_b$:
\[
b_s \frac{\partial Q_b}{\partial b} + b \Lambda Q_b + (\Delta Q_b - Q_b + Q_b^3)_{x_1} = 0,
\]
where $\Lambda$ is the scaling operator. To ensure solvability, we expand $Q_b$ and the evolution law for $b$ into powers of $b$:
\[
Q_b = Q + b P_1 + b^2 P_2 + \ldots, \quad b_s = -a_0 - a_1 b - a_2 b^2 - \ldots
\]
By matching the powers of $b$, we determine the coefficients:
\begin{itemize}
    \item \textbf{At $O(1)$:} The equation reduces to $(\Delta Q - Q + Q^3)_{x_1} = 0$. Since $Q$ is the ground state solitary wave, this is satisfied, implying $a_0 = 0$.
    \item \textbf{At $O(b)$:} The linearized equation is $a_1 P_1 + (L P_1)_{x_1} - \Lambda Q = 0$. Solvability requires orthogonality conditions. 
    After integrating in $y_1$, solvability of the equation for $P$ requires the resulting right-hand side to be orthogonal to $\operatorname{ker}L=\operatorname{span}\{Q_{y_1},Q_{y_2}\}$. The $Q_{y_1}$-condition follows from the critical identity $(\Lambda Q,Q)=0$, while the $Q_{y_2}$-condition follows by parity. This simplifies to $a_1 = 0$ and defines $P_1$ via $(L P_1)_{x_1} = \Lambda Q$. Note that $P_1$ is not in $L^2$, necessitating the use of cut-off functions later.
    \item \textbf{At $O(b^2)$:} The equation is
    \[
    -a_2P_1+\Lambda P_1+(3QP_1^2)_{x_1}-(LP_2)_{x_1}=0.
    \]
    Pairing with $Q$ gives the solvability condition
    \[
    a_2(P_1,Q)=\bigl(\Lambda P_1+(3QP_1^2)_{x_1},Q\bigr),
    \]
    which fixes $a_2=c>0$ according to \eqref{eq:definitionofc}.
\end{itemize}

Consequently, the formal dynamics are encapsulated by the reduced system:
\begin{equation*}
b_s + cb^2 = 0, \quad b = -\lambda_s \lambda^{-1}.
\end{equation*}
Integrating this system elucidates the phase portrait. If we set $\lambda(0)=1$, we observe a trichotomy based on the initial scaling velocity $b_0$:
\begin{itemize}
    \item \textbf{$b_0 < 0$:} $\lambda(t)$ grows indefinitely (Exit Case).
    \item \textbf{$b_0 = 0$:} $\lambda(t) \equiv 1$ (Threshold/Soliton Case).
    \item \textbf{$b_0 > 0$:} Finite time blow-up occurs at $T = [(3-c)b_0]^{-1}$, with the blow-up rate $\lambda(t) \sim (T-t)^{\frac{1}{3-c}}$.
\end{itemize}

\section{Outline of the Proof}

The proof follows the modulation--monotonicity strategy developed for the
mass-critical generalized Korteweg--de Vries equation, with several genuinely
two-dimensional modifications. In the gKdV problem, the nonlocalized part of
the first-order profile has a constant tail and leads to
\[
    b_s+2b^2=0.
\]
Here the left-hand tail depends on the transverse variable \(y_2\); the
associated transverse elliptic equation produces the constant \(c\in(1,2)\)
and the formal law
\[
    b_s+cb^2=0.
\]
This changes the physical blow-up rate from \(\lambda(t)\sim T-t\) to
\[
    \lambda(t)\sim (T-t)^{\frac1{3-c}}.
\]
Throughout this outline, \(\|\varepsilon\|_{H^1_w}\) and
\(\|\varepsilon\|_{L^2_w}\) denote the weighted norms introduced below; they
measure the radiation near the soliton and propagate its polynomial
localization in the positive \(y_1\)-direction.

\subsection*{Decomposition and Orthogonality Conditions}

As long as the solution remains in a sufficiently small \(L^2\)-neighborhood
of the soliton manifold, we write
\[
 u(t,x_1,x_2)
 =
 \frac1{\lambda(t)}
 \left(Q_{b(t)}+\varepsilon\right)
 \left(
 s,
 \frac{x_1-x_1(t)}{\lambda(t)},
 \frac{x_2-x_2(t)}{\lambda(t)}
 \right),
 \qquad
 \frac{ds}{dt}=\frac1{\lambda^3(t)}.
\]
The approximate profile is
\[
 Q_b(y)=Q(y)+b\chi\left(|b|^\gamma y_1\right)P(y),
 \qquad 0<\gamma<1,
\]
where
\[
    (LP)_{y_1}=\Lambda Q,
    \qquad
    L=-\Delta+1-3Q^2.
\]
The cut-off is needed because \(P\) has a nonzero left-hand tail and is not in
\(L^2(\mathbb R^2)\).

The four parameters \(\lambda,b,x_1,x_2\) are fixed by
\begin{equation}\label{weightortho}
    (\varepsilon,Q)
    =
    (\varepsilon,\varphi Q_{y_2})
    =
    (\varepsilon,\varphi\Lambda Q)
    =
    (\varepsilon,\varphi Q_{y_1})
    =0,
    \qquad
    \varphi(y_1)=1+e^{y_1/\alpha_1}.
\end{equation}
This choice serves simultaneously in the modulation decomposition and in the
coercivity estimates for both \(L\) and the virial operator introduced below.
The direction \(Q\) determines \(b\), since \((P,Q)>0\), and removes the
negative direction of \(L\); the odd direction \(\varphi Q_{y_2}\) determines
\(x_2\); and \(\varphi\Lambda Q\), \(\varphi Q_{y_1}\) determine \(\lambda\)
and \(x_1\). Their nondegeneracy is encoded by
\[
M^*
=
\begin{pmatrix}
(\Lambda Q,\varphi\Lambda Q)
&
(Q_{y_1},\varphi\Lambda Q)
\\
(\Lambda Q,\varphi Q_{y_1})
&
(Q_{y_1},\varphi Q_{y_1})
\end{pmatrix}.
\]
Indeed, the Jacobian determinant factors, up to sign, as
\[
    (P,Q)(Q_{y_2},\varphi Q_{y_2})\det M^*,
\]
and all three factors are nonzero.

Differentiating \eqref{weightortho} gives the rough modulation bounds
\[
 \left|\frac{\lambda_s}{\lambda}+b\right|
 +\left|\frac{(x_1)_s}{\lambda}-1\right|
 +\left|\frac{(x_2)_s}{\lambda}\right|
 \lesssim b^2+\|\varepsilon\|_{H^1_w}.
\]
These equations contain terms linear in the localized radiation. Projecting
instead against suitable primitive and localized functions gives the sharper
law
\[
    |b_s+cb^2|
    \lesssim
    |b|^3+|b|\|\varepsilon\|_{L^2_w}
    +\|\varepsilon\|_{H^1_w}^2.
\]
Thus the geometric modulation equations retain linear radiative errors, while
the equation governing the leading scaling dynamics has only quadratic or
higher-order forcing after correction.

\subsection*{Mixed Energy Functional and Dispersive Analysis}

The radiation is controlled by mixed energy--virial functionals
\[
\mathcal F_i(s)
=
\iint
\left\{
|\nabla\varepsilon|^2\psi(y_1)
+\varepsilon^2\phi_i(y_1)
-\frac12\left[(Q_b+\varepsilon)^4-Q_b^4-4Q_b^3\varepsilon\right]\psi(y_1)
\right\}.
\]
The gradient weight \(\psi\) and mass weights \(\phi_i\) decay exponentially
on the left, while near the soliton \(\phi_i'\) is proportional to
\(\varphi'\), reproducing the virial form associated with the same weight
\(\varphi\) as in \eqref{weightortho}. On the right, \(\phi_i\) grows
polynomially to propagate the localization of the initial data. They are also chosen
to control the scaling drift \((\lambda_s/\lambda)\Lambda\varepsilon\) in the
derivative of \(\mathcal F_i/\lambda^j\). The improved modulation equations
are obtained separately by projection; the role of the mixed energy estimate
is to make their radiative errors integrable.

For the relevant \(j\geq0\), we prove
\[
 \frac{d}{ds}\left(\frac{\mathcal F_i}{\lambda^j}\right)
 +\frac{\|\varepsilon\|_{H^1_w}^2}{\lambda^j}
 \lesssim
 \frac{b^4}{\lambda^j}.
\]
The source \(b^4\) comes from the localized profile error. The case \(j=0\)
gives scale-invariant spacetime control, while positive \(j\) yield the
stronger estimates required as \(\lambda\) becomes small.

This functional differs from that of \cite{MartelMerleRaphael1}. In the gKdV
argument, the weighted mass term carries an auxiliary time-dependent factor
that cancels an interaction between the scaling dynamics and slowly localized
modulation quantities. Here the weights and the direct treatment of the
scaling term in \(d(\mathcal F_i/\lambda^j)/ds\) allow us to work without such
an internal correction. Two separate spectral inputs are used: coercivity of
\(L\) makes \(\mathcal F_i\) equivalent to the weighted \(H^1\)-norm, whereas
coercivity of the anisotropic virial operator produces the negative spacetime
term.

\subsection*{Virial Estimate and Coercivity}

On the soliton core, the principal quadratic form is
\[
\begin{aligned}
\mathcal B_\varphi(\varepsilon)
=
\iint
\Big[&
(3\varepsilon_{y_1}^2+\varepsilon_{y_2}^2)\varphi'
+
\left(
\varphi'-\varphi'''
-3Q^2\varphi'+6QQ_{y_1}\varphi
\right)\varepsilon^2
\Big].
\end{aligned}
\]
The coefficients \(3\) and \(1\) reflect the anisotropy of the ZK dispersion.
For a general increasing \(\varphi\), the conjugation
\[
    v=\sqrt{\varphi'}\,\varepsilon
\]
identifies this form with the self-adjoint operator
\[
\begin{aligned}
\mathcal L_\varphi
={}&
-3\partial_{y_1}^2-\partial_{y_2}^2+1
+\frac12\frac{\varphi'''}{\varphi'}
-\frac34\left(\frac{\varphi''}{\varphi'}\right)^2
-3Q^2+6QQ_{y_1}\frac{\varphi}{\varphi'}.
\end{aligned}
\]
For \(\varphi(y_1)=1+e^{y_1/\alpha_1}\), this becomes
\[
\mathcal L
=
-3\partial_{y_1}^2-\partial_{y_2}^2
+\left(1-\frac1{4\alpha_1^2}\right)
-3Q^2
+6\alpha_1\left(1+e^{-y_1/\alpha_1}\right)QQ_{y_1}.
\]
Since \(\alpha_1>1\),
\[
    \varphi'-\varphi'''
    =\left(1-\frac1{\alpha_1^2}\right)\varphi'>0,
    \qquad
    1-\frac1{4\alpha_1^2}>0.
\]
Choosing \(\alpha_1\) sufficiently close to one also places all weighted
modulation directions in \(H^1(\mathbb R^2)\), and
\[
    \frac{\varphi}{\varphi'}
    =\alpha_1\left(1+e^{-y_1/\alpha_1}\right).
\]
This ratio is bounded on the right, where \(QQ_{y_1}<0\), and large on the
left, where \(QQ_{y_1}>0\); hence the negative part of the potential is
controlled while the favorable part is reinforced. Its precise value is also
chosen so that the finite-dimensional inverse-Gram condition below has the
required sign.

The comparison with the gKdV virial argument is useful. For the linear weight
\(\varphi(y)=y\), the one-dimensional computation gives
\[
    \mathcal L_{\mathrm{gKdV}}
    =-3\partial_y^2+1-5Q^4+20yQ^3Q_y,
\]
the operator used in \cite{MartelMerleRaphael1}. That weight cannot be used in
the ZK problem: by parity,
\[
    (Q_{y_2},y_1Q_{y_2})=0,
\]
so the transverse translation would not be determined, and
\(6y_1QQ_{y_1}<0\) on both sides of the soliton. The positive exponential
weight avoids both defects.

The classical one-constraint Weinstein criterion in \cite{Weinstein} treats an operator with one
negative direction. Since \(\mathcal L\) has two negative directions while the
modulation decomposition provides four constraints, we use the following
generalized form, stated precisely in
Proposition~\ref{Generalized Weinstein Lemma}. Let $T$ be self-adjoint
with an $n$-dimensional negative spectral subspace. For constraint
vectors $f_1,\ldots,f_m\in\operatorname{Ran}T$, choose
$g_k\in D(T)$ with $Tg_k=f_k$, and form
\[
 G_{k\ell}=(g_k,f_\ell).
\]
This form is independent of the preimages: their differences belong to
$\ker T$, which is orthogonal to $\operatorname{Ran}T$. If $G$ has an
$n$-dimensional  strictly negative subspace, the constraints
$(u,f_k)=0$ remove the negative spectral directions. Kernel directions
are treated separately as in the cited proposition. In the virial
application, $0\in\rho(T)$ and the preimages are $T^{-1}f_k$; the
spectral gap then yields strict coercivity. This formulation also
applies to the linearized operator $L$, whose translation kernel is
removed by the additional orthogonality conditions.

After conjugation, the four conditions become
\[
    (v,f_1)=(v,f_2)=(v,f_3)=(v,f_o)=0,
\]
where, up to harmless normalizations,
\begin{equation}\label{normalizedortho}
\begin{aligned}
 f_1&=\frac{Q}{\sqrt{\varphi'}},
 &
 f_2&=\frac{\varphi Q_{y_1}}{\sqrt{\varphi'}},
 &
 f_3&=\frac{\varphi\Lambda Q}{\sqrt{\varphi'}},
 &
 f_o&=\frac{\varphi Q_{y_2}}{\sqrt{\varphi'}}.
\end{aligned}
\end{equation}
The first three functions are even in \(y_2\), whereas \(f_o\) is odd, and
\(\mathcal L\) preserves this decomposition. The spectral analysis shows that
\(\mathcal L\) has exactly two negative eigenvalues, one in each parity
sector, and no zero eigenvalue. In the odd sector,
\[
    (\mathcal L^{-1}f_o,f_o)<0.
\]
In the even sector, the matrix
\[
    \left((\mathcal L^{-1}f_k,f_\ell)\right)_{1\leq k,\ell\leq3}
\]
has a negative direction, so there exists
\(0\neq f_e\in\operatorname{span}\{f_1,f_2,f_3\}\) with
\[
    (\mathcal L^{-1}f_e,f_e)<0.
\]
Because \(v\) is orthogonal to \(f_1,f_2,f_3\), it is also orthogonal to
\(f_e\). Parity gives \((\mathcal L^{-1}f_e,f_o)=0\); hence the inverse Gram matrix on
\(\operatorname{span}\{f_e,f_o\}\) is negative definite. The generalized
criterion yields
\[
    (\mathcal Lv,v)\gtrsim\|v\|_{H^1}^2,
\]
and undoing the conjugation gives coercivity of $\mathcal B_\varphi$.
The spectral and inverse-pairing certificates are given in
Appendix~\ref{appendixE}; the abstract criterion and its application to
$L$ and $\mathcal L$ are proved in Appendix~\ref{AppendixB}.

\subsection*{Rigidity and Asymptotic Dynamics}

The refined modulation equations produce a corrected quantity \(\mathcal G(s)\),
obtained by multiplying \(b/\lambda^c\) by a positive factor of the form
\[
    1+O\bigl(\|\varepsilon(s)\|_{L^2_w}\bigr),
\]
such that
\[
    \int_0^{s^*}|\mathcal G'(s)|\,ds<\infty,
\]
where \(s^*\) is the endpoint of the modulation interval. Hence, for some \(c_0\in\mathbb R\),
\[
    \mathcal G(s)\longrightarrow c_0
    \qquad\text{as }s\to s^*.
\]
Since the correction factor is positive and close to one, the sign of \(c_0\)
selects the three regimes.

If \(c_0>0\), the refined modulation estimates make the error in the scaling equation $o(b/\lambda^2)$. Then  eventually \(b>0\) and
\[
    b(t)\sim c_0\lambda^c(t),
    \qquad
    \lambda_t\sim-\frac{b}{\lambda^2}
    \sim-c_0\lambda^{c-2}.
\]
Because \(c<3\), \(\lambda\) reaches zero at a finite time \(T\), with
\[
    \lambda(t)\sim\ell_0(T-t)^{\frac1{3-c}}.
\]
Integration of the remaining modulation equations gives the asymptotics of
\(b(t)\), \(x_1(t)\), and \(x_2(t)\). This is the stable blow-up regime.

If \(c_0=0\), the weighted estimates imply
\[
    b(t)\to0,
    \qquad
    \|\varepsilon(t)\|_{H^1_w}\to0,
    \qquad
    \lambda(t)\to\lambda_\infty>0.
\]
Moreover,
\[
    x_1(t)\sim\frac{t}{\lambda_\infty^2},
\]
and \(x_2(t)\) converges. This is the asymptotic-stability regime. If
\(c_0<0\), then eventually \(b<0\), the scale increases, and the solution is
driven to the boundary of the prescribed tube: this is the Exit regime.

The organization differs from \cite{MartelMerleRaphael1}. There, a separation
or domination time is introduced: after it,
\[
    |b|\gg\|\varepsilon\|_{H^1_w}^2
\]
is trapped, and the sign of \(b\) distinguishes blow-up from exit; if no such
time exists, the solution belongs to the soliton regime. Here the primary
organizer is instead the limit of the corrected quantity comparable to
\(b/\lambda^c\). Domination estimates are still used locally, notably in the
case \(c_0=0\) and in the proof of openness, but they do not define the
classification.

For openness, choose a late time along a reference blow-up solution at which
\(b/\lambda^c\) has a fixed positive margin and the weighted radiation is much
smaller. Continuous dependence of the flow and modulation parameters, together
with propagation of the high spatial moment, transfers these inequalities to
nearby data; a trapping argument then preserves the positive sign and forces
the same finite-time blow-up regime.

\subsection*{Asymptotic Radiation: Strong $L^p$ Convergence and Sobolev Obstruction}

In the blow-up regime, write
\[
    u=\mathcal Q+\widetilde u,
\]
where
\[
\mathcal Q(t,x)
=
\frac1{\lambda(t)}
Q_{b(t)}
\left(
\frac{x_1-x_1(t)}{\lambda(t)},
\frac{x_2-x_2(t)}{\lambda(t)}
\right)
\]
is the concentrating approximate profile and
\[
\widetilde u(t,x)
=
\frac1{\lambda(t)}
\varepsilon
\left(
t,
\frac{x_1-x_1(t)}{\lambda(t)},
\frac{x_2-x_2(t)}{\lambda(t)}
\right)
\]
is the physical radiation. It satisfies a forced ZK equation whose
source contains the approximate-profile error and the modulation errors.

Rather than propagating localized $H^1$ monotonicity estimates
progressively farther from the soliton, as in
\cite{MartelMerleRaphael1}, we compare two time translates of the
radiation through Duhamel's formula, following ideas from the
$L^2$-supercritical NLS analysis of \cite{MerleRaphaelSzeftel} and the
$L^2$-supercritical gKdV analysis of \cite{Lan}. The refined
inhomogeneous Strichartz estimate of \cite{Foschi}, together with the
sharp modulation asymptotics and the weighted spacetime bounds, shows
that $\widetilde u(t)$ is Cauchy in $L^2(\mathbb R^2)$ as $t\uparrow T$.
The interpolated dispersive estimate for the linear ZK group, applied on
a shrinking terminal time interval, extends the same argument to every
$2<p<\infty$. Hence there exists a common profile
\[
    u^*\in\bigcap_{2\leq p<\infty}L^p(\mathbb R^2)
\]
such that $\widetilde u(t)\to u^*$ strongly in every finite $L^p$.
The profile bounds and $b\sim c_0\lambda^c$ also give, in the physical
variables, for $2\leq p<\infty$,
\[
 \left\|\mathcal Q(t,\cdot)-\frac1{\lambda(t)}
 Q\!\left(\frac{\cdot-x(t)}{\lambda(t)}\right)\right\|_{L^p}
 \longrightarrow0.
\]
Thus subtracting the exact rescaled ground state gives the same limit.

Mass decoupling gives
\[
    \|u^*\|_{L^2}^2
    =
    \|u_0\|_{L^2}^2-\|Q\|_{L^2}^2.
\]
In the gKdV problem, the possible minimal-mass case is excluded using
the classification in \cite{MM02}; here minimal mass is excluded within
the stable blow-up branch, see Remark~\ref{nominimalmassblowup}. Consequently,
$u^*\not\equiv0$.

We then determine the Sobolev regularity of the radiation. A
quantitative expansion of the left-hand tail of $P$ isolates the leading
cutoff-transition contribution to the profile error. After passing the
Duhamel formula to the terminal time in $L^2$ on bounded frequency
regions, a uniform phase analysis identifies a nonvanishing Fourier
component in the anisotropic region
\[
    |\xi_2|\sim R,
    \qquad
    |\xi_1|\sim R^{1-c}.
\]
The leading contribution has $L^2_\xi$ size comparable to $R^{-c/2}$,
whereas the initial radiation, the remaining profile and modulation
forcing, and the nonlinear Duhamel term are all of smaller order. A
dyadic summation therefore gives
\[
    u^*\notin H^\varsigma(\mathbb R^2)
    \qquad\text{for every }\varsigma\geq\frac c2;
\]
see Theorem~\ref{thm:radiation-not-Hs}. In particular, $u^*\notin H^1$
because $c<2$. This is a genuinely two-dimensional difference from the
critical gKdV result, where the limiting radiation belongs to $H^1$.
\subsection*{Non-positive Energy}

A variational compactness argument first excludes the Exit regime when
\(E_0\leq0\), leaving only blow-up and asymptotic stability. If \(E_0<0\), the
latter would imply, by conservation of energy and decay of the radiation,
\[
    \lambda^2(t)|E_0|
    +\|\nabla\varepsilon(t)\|_{L^2}^2
    \longrightarrow0.
\]
This forces \(\lambda(t)\to0\), contradicting
\(\lambda(t)\to\lambda_\infty>0\); hence the solution blows up in finite time.

For \(E_0=0\), energy conservation under the asymptotic-stability hypothesis
gives only
\[
    \|\nabla\varepsilon(t)\|_{L^2}\longrightarrow0.
\]
To recover the missing \(L^2\)-compactness, we study the localized mass
\[
    \iint\varepsilon^2(t,y)\varrho_D(y_1)\,dy,
\]
where \(\varrho_D\) localizes on a longitudinal scale \(D\). Its physical-time
derivative contains transport, scaling, profile, and nonlinear terms, all
controlled by the refined modulation equations and weighted spacetime
estimates. Repeated time integration, the hierarchy of polynomially weighted
bounds, the high right-hand moment, and a localized interpolation estimate
allow first \(t\to\infty\) and then \(D\to\infty\), yielding
\[
    \|\varepsilon(t)\|_{L^2}\longrightarrow0.
\]
Mass conservation gives \(\|u_0\|_{L^2}=\|Q\|_{L^2}\), and equality in the
sharp Gagliardo--Nirenberg inequality, together with \(E(u_0)=0\), places
\(u_0\) on the solitary-wave orbit generated by scaling, translations, and
sign. Therefore every non-soliton datum with non-positive energy belongs to the
finite-time blow-up branch.

The endpoint is analogous to the zero-energy classification in
\cite{MartelMerleRaphael1}: asymptotic stability forces critical mass, after
which the variational characterization identifies a soliton. The compactness
mechanism is different. The gKdV proof obtains uniform spatial tails for the
physical solution from localized mass and energy monotonicity; here we work
directly with the localized mass of the renormalized error and combine its
evolution with the hierarchy of weighted energy estimates and the high
right-hand moment.

\section{Interval-Arithmetic Verification of the Numerical Inputs}

The virial argument above and the choice of the polynomial weight in
$\mathcal A_{\alpha_0}^{(21)}$ rest on three numerical facts: the
operator $\mathcal L$ has exactly two negative eigenvalues and a
quantitative spectral gap; the even inverse Gram matrix $A$ and the
odd inverse pairing $m_{oo}$ introduced below satisfy $\det A<0$ and
$m_{oo}<0$; and the constant $c$ lies in $(21/13,5/3)$. These facts
are verified in Appendix~\ref{appendixE} by interval arithmetic, and
they are collected in Theorem~\ref{cap:verified-main}. The verification
has two branches. The first
determines the negative index and a spectral gap of $\mathcal L$ and
encloses the inverse pairings required by the generalized Weinstein
criterion of Proposition~\ref{Generalized Weinstein Lemma}; the second
encloses the constant $c$ governing the blow-up rate. We now describe
what the interval certificates recorded in that appendix establish.
Every error bound includes rounding and truncation errors, and every
integral over an unbounded region is accompanied by an explicit
estimate of its tail.

We first describe the enclosure of the ground state; see
Section~\ref{cap:q-section}. Writing $r=|y|$, the radial profile solves
\[
 Q''+\frac1rQ'-Q+Q^3=0,\qquad Q'(0)=0.
\]
The construction treats the three intervals $[0,r_0]$, $[r_0,R]$, and
$[R,\infty)$ separately, the value $a=Q(0)$ being initially unknown.
Near the origin, the regular solution is represented by an even power
series whose first coefficient is $a$; substitution into the equation
determines every subsequent coefficient, and a convergent geometric
majorant bounds the entire omitted series together with its derivatives.
This gives initial interval enclosures at a positive radius $r_0$
without evaluating the singular coefficient $1/r$ at zero.

On $[r_0,R]$, the equation is a regular first-order system for $Q$ and
$Q'$. The CAPD solver encloses the solution on each accepted step and bounds its Taylor coefficients and its entire order-$21$ remainder. The resulting Taylor models therefore enclose every point of the step, and their endpoint enclosures initialize the next step; see Lemma~\ref{cap:q-tube}. The calculation starts
at $r_0 = 1/512$ and reaches $R = 15$ through $184$ cells with Taylor order $20$. A simultaneous integration of the parameter sensitivity encloses the Jacobian needed for matching, and the exported profile is recomputed on the validated matching interval.

On $[R,\infty)$, the decaying solution is constructed from the modified
Bessel functions: an integral equation expresses it as a multiple of the
decaying linear solution $K_0$ plus a nonlinear correction, and the
corresponding map is a contraction in the weighted norm
$\sup_{r\ge R}e^r\sqrt r\,|Q(r)|$ on an explicitly specified ball. Its
fixed point therefore exists and is unique there, and it satisfies
bounds of the form $|Q^{(j)}(r)|\le C_j e^{-r}/\sqrt r$. These estimates
control the whole exterior interval, so that the computation does not
stop at a large finite radius.

The interior parameter $a$ and the tail parameter are then determined by
matching the values of $Q$ and $Q'$ at $R$. Differentiating the interior
and the tail equations with respect to their parameters encloses the
Jacobian of this two-dimensional matching problem, and a regular
interval Newton inclusion proves that the chosen box contains a unique
matching pair. Positivity checks on the finite interval and on the tail
finally identify the resulting decaying solution with the positive
radial ground state, by the uniqueness theorem of \cite{Kwong}; see
Proposition~\ref{cap:q-certificate}. Hence the finite profile enclosures
and the exterior estimates concern the same function $Q$ as the
theoretical argument.

We next turn to the potential and to the spectral bounds. In the
notation of Appendix~\ref{appendixE}, the virial operator $\mathcal L$
introduced above is written as
\[
 \mathcal L=-3\partial_{y_1}^2-\partial_{y_2}^2
                 +\kappa_{\alpha_1}+W,
 \qquad \kappa_{\alpha_1}=\frac{7701}{10201},
\]
where $\kappa_{\alpha_1}=1-\frac{1}{4\alpha_1^2}$ and
\[
 W=-3Q^2+6\alpha_1\left(1+e^{-y_1/\alpha_1}\right)QQ_{y_1}
\]
is the potential of $\mathcal L$; see
\eqref{cap:operator}--\eqref{cap:potential}. The regularity and decay
of $Q$ from Lemma~\ref{Qdecay} imply that $W$ is bounded and tends to
zero, so that the essential spectrum of $\mathcal L$ starts at the
positive number $\kappa_{\alpha_1}$. To obtain quantitative spectral
bounds, the calculation encloses the exact potential directly from the
validated profile: on $\Omega=(-16,16)^2$, it multiplies $W$ by a cutoff
and interpolates the resulting function bilinearly. The explicit
potential $W_h$ obtained in this way has a certified error, given by the
sum of the cutoff, interpolation, and nodal errors, and the uncertainty
in $Q$ is carried by the interval values and derivative bounds entering
these estimates. The exterior estimate follows directly from the
exponential bounds for $Q$ and $Q'$. The corresponding quantitative bounds are given in Proposition~\ref{cap:potential-error} and its recorded certificate.

The Dirichlet and the Neumann problems on this square play different
roles; see Section~\ref{cap:s-section}. Two finite sine trial
functions give an upper bound for the second Dirichlet eigenvalue:
their mass matrix is positive definite, so
that coefficient vectors and trial functions span spaces of the same
dimension, and an interval test shows that the shifted Ritz form is
strictly negative on this two-dimensional space. The min--max principle
therefore gives $\mu_1^D\le\mu_2^D<-5/2$, and no assertion about the
convergence of a sequence of numerical eigenvalues is needed.

For the Neumann problem, the cosine basis splits the space into finitely
many retained modes and their orthogonal complement. The kinetic energy
of the complementary modes has an explicit lower bound, which, after
allowing for the potential, gives a lower threshold $\tau$ for the
complementary block. At a level $t<\tau$, completion of the square
reduces the eigenvalue count to a finite Schur complement, and the
coupling to all complementary modes is bounded by a residual Gram matrix
whose entries are finite integrals of $W_h$ and $W_h^2$. An interval
congruence test shows that the resulting lower matrix has at most two
negative directions at $t=5/8$. Hence $\mu_3^N\ge5/8$, all modes outside
the retained space being accounted for; see
Proposition~\ref{cap:s-schur}.

The potential error changes each interior eigenvalue by at most its
uniform bound. Zero extension of the Dirichlet trial functions produces
at least two negative directions for the whole-space operator;
conversely, Neumann decoupling enlarges the form domain, while the
exterior operator is positive by the tail estimate. Dirichlet--Neumann
bracketing therefore excludes any further negative direction and gives
\[
 \operatorname{ind}_-(\mathcal L)=2,\qquad
 \operatorname{dist}(0,\sigma(\mathcal L))\ge 0.62076895409267;
\]
see Proposition~\ref{cap:spectral-gap}. In particular, $\mathcal L^{-1}$
exists on $L^2(\mathbb R^2)$, and the spectral theorem bounds its norm
by the reciprocal of this gap.

We now describe the enclosure of the inverse pairings. Recall the
constraint functions $f_1,f_2,f_3,f_o$ and their normalization from
\eqref{normalizedortho}; the normalization adopted in
Appendices~\ref{AppendixB} and~\ref{appendixE} is \eqref{cap:constraints}.
The first three functions are even in $y_2$, whereas $f_o$ is odd, and
reflection in $y_2$ commutes with $\mathcal L$ and with its inverse.
Thus the inverse Gram matrix separates into the even matrix $A$, with
entries $A_{ij}=(\mathcal L^{-1}f_i,f_j)$ for $1\le i,j\le3$, and the
odd pairing $m_{oo}=(\mathcal L^{-1}f_o,f_o)$, the cross pairings being
zero.

For each constraint, the inverse solution is approximated by a compact,
globally $C^2$ cubic spline $u_{h,a}$, whose residual is the exact
function $r_a=f_a-\mathcal L u_{h,a}$. Cellwise interval quadrature and
the explicit tails of $f_a$ bound $\|r_a\|_{L^2}$ on the whole plane, and
the global regularity of the spline ensures that its second derivatives
carry no boundary distributions. A symmetric correction then gives
\[
\begin{aligned}
 (\mathcal L^{-1}f_a,f_b)
 ={}&(u_{h,a},f_b)+(f_a,u_{h,b})
       -\mathcal B(u_{h,a},u_{h,b})\\
    &+(\mathcal L^{-1}r_a,r_b),
\end{aligned}
\]
where $\mathcal B$ denotes the symmetric bilinear form of $\mathcal L$,
\[
 \mathcal B(u,v)=3(u_{y_1},v_{y_1})+(u_{y_2},v_{y_2})
 +\bigl((\kappa_{\alpha_1}+W)u,v\bigr),
 \qquad u,v\in H^1(\mathbb R^2),
\]
which coincides with $(\mathcal Lu,v)$ for $u\in H^2(\mathbb R^2)$; see
\eqref{cap:form}. The first three terms are finite integrals, whereas
the remaining error
is at most $\|r_a\|_{L^2}\|r_b\|_{L^2}$ divided by the spectral gap, hence
quadratic in the residual sizes; see Proposition~\ref{cap:m-residual}.
The recorded enclosures give $\det A<0$ and $m_{oo}<0$.

Since $A$ is real and symmetric, its negative determinant supplies a
negative direction among the even constraints, and the odd pairing
supplies a second one, with zero cross pairing by parity. The inverse
constraint form is therefore negative definite on a two-dimensional
space, which matches the negative index of $\mathcal L$. Together with
the gap, these are precisely the numerical inputs of
Lemma~\ref{virialcoercivity}, where Proposition~\ref{Generalized
Weinstein Lemma} turns them into the coercivity of the virial form.

Finally, we describe the enclosure of the constant $c$. For the function
$g_Q$ defined in \eqref{definitionofg}, the constant $c$ of
\eqref{eq:definitionofc} can be written as
\[
 c=\frac{2\bigl((-d^2/dz^2+1)^{-1}g_Q,g_Q\bigr)}{\|g_Q\|_{L^2}^2},
\]
where the inner product and the norm are now those of $L^2(\mathbb R)$.
Radiality expresses $g_Q$ as a one-dimensional integral involving
$Q'(r)/r$, whose value at zero is given by the regular power series.
Validated radial quadrature gives interval values at a finite set of
transverse nodes, and piecewise linear interpolation, bounds for the
second transverse derivative, and explicit exterior tails then produce
an approximation $g_h$ with a certified whole-line $L^2$ error.

The inverse in this quotient has kernel $\tfrac12e^{-|z-t|}$, and an
exact one-sided convolution recurrence integrates the piecewise linear
approximation against this kernel; thus the computation uses the
whole-line inverse, including its exponential tails. Since this inverse
has operator norm one, the $L^2$ error controls the errors in both
quadratic integrals of the quotient, and a positive lower bound for the
denominator permits interval division and yields
\[
 c\in[1.66126332085379,1.66170874135113]
       \subset\left(\frac{21}{13},\frac53\right);
\]
see Proposition~\ref{cap:c-quotient}. In particular,
$\lfloor8c/(c-1)\rfloor+1=21$, which identifies the polynomial weight
required in the theoretical estimates; see \eqref{cap:c-weight-index}.
This last calculation uses only the validated ground state; the
spectral and the inverse-pairing certificates are needed for the virial
branch.

The complete source code and implementation files are available in the public GitHub repository at \href{https://github.com/kk10-genius/ZK2d3p-blowup}{\textcolor{blue}{https://github.com/kk10-genius/ZK2d3p-blowup}}.

\section{Organization of the Book}
The book is organized as follows. In
Chapter~\ref{Construction of the approximate solution Qb} we construct
the approximate profile, and in Chapter~\ref{Modulation Equation} we
derive the modulation equations. The mixed-energy method is developed
in Chapters~\ref{Monotonicity Formulas} and \ref{Closingthebootstrap},
leading to the rigidity theorem in Chapter~\ref{Rigidity near the soliton}.
The openness of the stable blow-up branch is proved in
Chapter~\ref{Stability of Blow-Up}. Strong $L^p$ convergence of the
radiation and the Sobolev obstruction for its limiting profile are
established in Chapter~\ref{sec:asymptotic-radiation}. Finally, the
non-positive-energy case is treated in Chapter~\ref{Blow Up for E0}.

\chapter{Construction of the Approximate Solution}\label{Construction of the approximate solution Qb}

\section{Properties of the Linearized Operator} 

In this section, we denote the spatial variables by $x=(x_1, x_2)$. We begin by recalling well-known spectral properties of the linearized operator defined by
\[
L = -\Delta + 1 - 3Q^2.
\]

\begin{lemma} (\cite{Weinstein}, \cite{Kwong}, \cite{Nakanishi})
The operator $L$ satisfies the following properties:
\begin{itemize}
    \item $L$ is self-adjoint, and its essential spectrum is given by $\sigma_{\mathrm{ess}}(L) = [1, +\infty).$
    \item The kernel of $L$ is nondegenerate and is generated by the translations of the soliton:
    \[
    \ker(L) = \operatorname{span}\{Q_{x_1}, Q_{x_2}\}.
    \]
    \item $L$ possesses a unique negative eigenvalue $-\lambda_{-}$ (with $\lambda_{-} > 0$), associated with a strictly positive, radially symmetric eigenfunction $\chi_0$. Furthermore, $\chi_0$ decays exponentially (see \cite{agmon1982lectures}); there exists $\delta > 0$ such that
    \[
    |\chi_0(x)| \lesssim e^{-\delta|x|} \quad \text{for all } x \in \mathbb{R}^2.
    \]
\end{itemize}
\end{lemma}

In addition to the spectral structure, we recall specific algebraic identities and positivity properties associated with $L$.

\begin{lemma}
The following identities and inequalities hold:
\begin{itemize}
    \item Action on the scaling mode: $L(\Lambda Q) = -2Q$. Moreover, the orthogonality condition $\iint Q \Lambda Q = 0$ holds.
    \item The operator is negative in the direction of the soliton:
    \[
    (LQ, Q) = -2\int Q^4 < 0.
    \]
\end{itemize}
\end{lemma}

We also require precise pointwise decay estimates for the soliton and its derivatives.

\begin{lemma}\label{Qdecay}
For any $0 < \eta \ll 1$ and any multi-index $\alpha \in \mathbb{N}^2$, there exists a constant $C_{\alpha, Q, \eta} > 0$ such that
\[
|\partial^{\alpha}Q(x_1,x_2)| \leq C_{\alpha, Q, \eta} e^{-(1-\eta)|x_1| - \sqrt{2\eta-\eta^2}|x_2|}.
\]
In particular, for any $0 < \eta \ll 1$, there exists $C_{\Lambda, Q, \eta} > 0$ such that
\[
|\Lambda Q(x_1,x_2)| \leq C_{\Lambda, Q, \eta} e^{-(1-\eta)|x_1| - (2\eta-\eta^2)|x_2|}.
\]
\end{lemma}

\begin{proof}
It was established in \cite{BonaLi}, \cite{BonaLi2} that there exists a constant $C_{\alpha, Q}$ such that $|\partial^{\alpha}Q(x_1,x_2)| \leq C_{\alpha, Q} e^{-\sqrt{x_1^2+x_2^2}}$ for all multi-indices $\alpha \in \mathbb{Z}_{+}^2$. Noticing the elementary inequality $\sqrt{x_1^2+x_2^2} \geq (1-\eta)|x_1| + \sqrt{2\eta-\eta^2}|x_2|$, the result follows immediately.

For the estimate on \(\Lambda Q\), we use
\[
|\Lambda Q(x)|\le |Q(x)|+|x||\nabla Q(x)|
\lesssim (1+|x|)e^{-|x|}.
\]
Fix \(0<\eta\ll1\). Since
\[
(1-\eta)^2+(2\eta-\eta^2)^2<(1-\eta/2)^2
\]
for \(\eta>0\) sufficiently small, we have
\[
(1-\eta/2)|x|
\ge
(1-\eta)|x_1|+(2\eta-\eta^2)|x_2|.
\]
Moreover,
\[
(1+|x|)e^{-|x|}
\le C_\eta e^{-(1-\eta/2)|x|}.
\]
Therefore,
\[
|\Lambda Q(x_1,x_2)|
\le
C_{\Lambda,Q,\eta}
e^{-(1-\eta)|x_1|-(2\eta-\eta^2)|x_2|}.
\]
\end{proof}

We now proceed to establish the coercivity property of the operator $L$. Throughout the remainder of this section, we fix the following constants:
\begin{equation}\label{alphas}
\alpha_1 = 1.01 \quad \text{and} \quad \alpha_2 = 1.005 = \alpha_1 - \frac{1}{200}.
\end{equation}
Let us define the weight function $\varphi: \mathbb{R} \to \mathbb{R}_{+}$ by
\begin{equation}\label{definitionofphi}
\varphi(x_1) = 1 + e^{\frac{x_1}{\alpha_1}}.
\end{equation}
Note that $\varphi \in C^{\infty}(\mathbb{R})$. Combining the definition of $\varphi$ with Lemma ~\ref{Qdecay}, it follows that the functions $\varphi(x_1)Q_{x_1}$, $\varphi(x_1)Q_{x_2}$, and $\varphi(x_1)\Lambda Q$ belong to $H^{\infty}(\mathbb{R}^2)$. Furthermore, each of these functions is bounded in absolute value by a constant multiple of 
 \begin{equation}\label{pointwiseboundortho}
    e^{-\frac{1}{200\alpha_1\alpha_2}|x_1| - (1-\frac{1}{\alpha_{2}^{2}})|x_2|}.
\end{equation}

We introduce the matrix $M^* \in \mathbb{R}^{2 \times 2}$ defined by
\begin{equation}\label{M*}
M^* = \begin{bmatrix}
(\Lambda Q, \varphi(x_1)\Lambda Q) & (\Lambda Q, \varphi(x_1)Q_{x_1}) \\
(\Lambda Q, \varphi(x_1)Q_{x_1}) & (Q_{x_1}, \varphi(x_1)Q_{x_1})
\end{bmatrix}.
\end{equation}
Observe that $M^*$ is precisely the Gram matrix associated with the linearly independent pair $\{\Lambda Q, Q_{x_1}\}$ under the weighted inner product defined by $(f, g)_{\varphi} := (f, g\varphi)$. Consequently, $M^*$ is positive definite and, in particular, $\det M^* \neq 0$.

\begin{lemma} \label{lem:CoercivityMain}
Define the subspace $\mathcal X \subset H^1(\mathbb{R}^2)$ by
\[
\mathcal X = \left\{ u \in H^1(\mathbb{R}^2) : (u, Q) = (u, \varphi(x_1)Q_{x_2}) = (u, \varphi(x_1)\Lambda Q) = (u, \varphi(x_1)Q_{x_1}) = 0 \right\},
\]
and define the associated quadratic form of  \(L\) by 
\[
    (Lu,u)
    :=
    \int_{\mathbb R^2}
    \left(
        |\nabla u|^2
        +(1-3Q^2)u^2
    \right)\,dx,
    \qquad u\in H^1(\mathbb R^2).
\]
\begin{itemize}
    \item[a)] \textbf{Coercivity on $\mathcal X$:} There exists $\delta > 0$ such that for all $u \in \mathcal X$,
    \[
    (Lu, u) \geq \delta \|u\|_{H^1}^2.
    \]
    \item[b)] \textbf{Global Coercivity:} There exists $\tilde{\delta} > 0$ such that for any $u \in H^1(\mathbb{R}^2)$,
    \begin{equation}\label{eq:coercivitylemma}
    (Lu, u) \geq \tilde{\delta}\|u\|_{H^1}^2 - \frac{1}{\tilde{\delta}} \left( (u,Q)^2 + (u, \varphi(x_1)Q_{x_2})^2 + (u, \varphi(x_1)\Lambda Q)^2 + (u, \varphi(x_1)Q_{x_1})^2 \right).
    \end{equation}
\end{itemize}
\end{lemma}

The proof of this result can be found in Appendix\DIFaddbegin \DIFadd{~}\DIFaddend \ref{appendixB2}.

\section[Construction and Properties of the First-Order Correction \texorpdfstring{$P$}{P}]{Construction and Properties of the\\
First-Order Correction \texorpdfstring{$P$}{P}}

In this section, we seek an approximation of the soliton $Q(x_1,x_2)$ such that we obtain a sufficient approximate self-similar equation: 
\[
b_s\frac{\partial Q_b}{\partial b}+b\Lambda Q_b+(\Delta Q_b-Q_b+Q_{b}^{3})_{x_1}=0.
\]
By writing $Q_b=Q+bP+\ldots,$ this implies $(LP)_{x_1}=\Lambda Q.$ Therefore, we need to prove that we can invert the operator $\partial_{x_1}L.$ This solvability is ensured by the spectral properties of $L$ and the orthogonality condition $(\Lambda Q, Q)=0.$ Throughout this section, we use $x=(x_1,x_2)\in \mathbb{R}^2$ as spatial variables.

\begin{lemma} The fundamental solution of $-\Delta+1$ in $\mathbb{R}^2$ is $\frac{1}{2\pi}K_0(|\cdot|),$ where $K_0$ is the modified Bessel function of the second kind. It satisfies $K_0\in C^{\infty}(\mathbb{R}^{\ast}_{+})$ and we have the following properties: 
\begin{itemize}
\item[a).] There exists $\delta>0$, such that for all $r>\delta,$  $$\frac{\sqrt{\pi}e^{-r}}{\sqrt{2(r+\frac{1}{4})}}<K_0(r)<\frac{\sqrt{\pi}e^{-r}}{\sqrt{2r}},$$
 $$K_0(r) \sim_{r \rightarrow 0} -\ln(r).$$
\item[b).] For all $r>0$, $K_0(r)>0.$
\item[c).] $\int_{\mathbb{R}}K_0(|x|)dx$ is finite.
\end{itemize}
\end{lemma}
All these results can be found in \cite{AbramowitzStegun} and \cite{ChuYang}. As a corollary of this lemma, we find that 
\[
K_0(|\cdot|)\in L^1(\mathbb R^2).
\] 

\begin{lemma}\label{Greenfunction}
Suppose that $f\in H^2(\mathbb{R}^2)$ satisfies, for some $A>0$ and
$k>0$,
\[
|(-\Delta f+f)(x)|
\leq
A(1+|x|)^ke^{-|x|}
\qquad
\text{for all }x\in\mathbb{R}^2.
\]
Then there exists $C(k)>0$, independent of $A$, $f$, and $x$, such that
\[
|f(x)|
\leq
C(k)A(1+|x|)^{k+\frac32}e^{-|x|}
\qquad
\text{for all }x\in\mathbb{R}^2.
\]
\end{lemma}

\begin{proof}
Let
\[
g=(-\Delta+1)f.
\]
Since $f\in H^2(\mathbb{R}^2)$, we have $g\in L^2(\mathbb{R}^2)$.
Moreover, the assumed pointwise bound implies that
$g\in L^1(\mathbb{R}^2)$. The fundamental-solution representation
therefore gives
\[
f=\frac{1}{2\pi}K_0\ast g.
\]
Using $K_0\geq0$, for every $x\in\mathbb{R}^2$ we obtain
\[
\begin{aligned}
|f(x)|
&\leq
\frac{1}{2\pi}
\norm{
\frac{e^{|\cdot|}g(\cdot)}
     {(1+|\cdot|)^k}
}_{L^\infty(\mathbb{R}^2)}
\int_{\mathbb{R}^2}
K_0(|x-y|)e^{-|y|}(1+|y|)^k\,dy \leq
\frac{A}{2\pi}h(x),
\end{aligned}
\]
where
\[
h(x)
=
\int_{\mathbb{R}^2}
K_0(|x-y|)e^{-|y|}(1+|y|)^k\,dy.
\]
Since $K_0(|\cdot|)$ and
$e^{-|\cdot|}(1+|\cdot|)^k$ are radial, their convolution $h$ is
radial.

The estimates on $K_0$ recalled above imply the global pointwise bound
\[
K_0(r)\lesssim r^{-\frac12}e^{-r},
\qquad r>0.
\]
Indeed, this is the stated estimate for large $r$, while for $0<r\leq1$
it follows from
\[
K_0(r)\lesssim1+|\log r|
\lesssim r^{-\frac12}.
\]

Let
\(
F(r)=e^{-r}(1+r)^k,\)
for \( r\geq0.
\)
Since
\(
F'(r)
=
e^{-r}(1+r)^{k-1}(k-1-r),
\)
we distinguish the following cases.

\begin{itemize}
\item[\textit{i).}]
Suppose that $0<k<1$. Then $F$ is strictly decreasing on
$[0,\infty)$. Since $F(0)=1$ and $F(r)\to0$ as $r\to\infty$, there
exists $\tilde r_0>0$ such that
\[
F(\tilde r_0)=\frac12.
\]
For $0\leq |x|\leq\tilde r_0$, we have
$F(|x|)\geq1/2$. Since $F(|y|)\leq1$ and
$K_0(|\cdot|)\in L^1(\mathbb{R}^2)$,
\[
\begin{aligned}
h(x)
&=
\int_{\mathbb{R}^2}
K_0(|x-y|)F(|y|)\,dy \leq
\int_{\mathbb{R}^2}K_0(|x-y|)\,dy
\lesssim
F(|x|).
\end{aligned}
\]

\item[\textit{ii).}]
Suppose that $k\geq1$. Then $F$ is increasing on $[0,k-1]$ and
decreasing on $[k-1,\infty)$. Since $F(r)\to0$ as $r\to\infty$, there
exists $\tilde r_1\in(k-1,\infty)$ such that
\[
F(\tilde r_1)=\frac{F(k-1)}{2}.
\]
For $0\leq|x|\leq\tilde r_1$, set
\[
\alpha
=
\min\{F(0),F(\tilde r_1)\}>0.
\]
Then $F(|x|)\geq\alpha$, while $F(|y|)\leq F(k-1)$ for every
$y\in\mathbb{R}^2$. Consequently,
\[
\begin{aligned}
h(x)
&=
\int_{\mathbb{R}^2}
K_0(|x-y|)F(|y|)\,dy \leq
F(k-1)
\int_{\mathbb{R}^2}K_0(|x-y|)\,dy \leq
C(k)F(|x|).
\end{aligned}
\]

\item[\textit{iii).}]
It remains to consider either
\[
0<k<1,\qquad |x|\geq\tilde r_0,
\]
or
\[
k\geq1,\qquad |x|\geq\tilde r_1.
\]

In the first case, define
\[
\tilde x_0
=
\inf\left\{
t\in[0,|x|]:
F(t)\leq2F(|x|)
\right\}.
\]
This quantity is well defined because
\(
2F(|x|)
\leq
2F(\tilde r_0)
=
F(0),
\)
and $\tilde x_0<|x|$.

In the second case, define
\[
\tilde x_1
=
\inf\left\{
t\in[k-1,|x|]:
F(t)\leq2F(|x|)
\right\}.
\]
This quantity is well defined because
\(
2F(|x|)
\leq
2F(\tilde r_1)
=
F(k-1),
\)
and again $\tilde x_1<|x|$.

In both cases, the monotonicity of $F$ on the relevant interval gives
\[
F(|y|)
\leq
2F(|x|)
\qquad
\text{whenever }
\tilde x_i\leq|y|\leq|x|,
\quad i=0,1.
\]

We now split the integral defining $h$ into three regions.

\begin{itemize}
\item[\textit{a).}]
On the region $\{|y|\geq|x|\}$, the monotonicity of $F$ gives
$F(|y|)\leq F(|x|)$. Hence
\[
\begin{aligned}
\int_{\{|y|\geq|x|\}}
K_0(|x-y|)e^{-|y|}(1+|y|)^k\,dy
&\leq
F(|x|)
\int_{\mathbb{R}^2}K_0(|x-y|)\,dy \lesssim
F(|x|).
\end{aligned}
\]

\item[\textit{b).}]
On the region
$\{\tilde x_i\leq|y|\leq|x|\}$, we have
$F(|y|)\leq2F(|x|)$. Therefore,
\[
\begin{aligned}
\int_{\{\tilde x_i\leq|y|\leq|x|\}}
K_0(|x-y|)e^{-|y|}(1+|y|)^k\,dy
&\leq
2F(|x|)
\int_{\mathbb{R}^2}K_0(|x-y|)\,dy \lesssim
F(|x|).
\end{aligned}
\]

\item[\textit{c).}]
On the region $\{|y|\leq\tilde x_i\}$, we have
\(
(1+|y|)^k\leq(1+|x|)^k.
\)
Using the global upper bound on $K_0$ and the triangle inequality, we obtain
\[
\begin{aligned}
\int_{\{|y|\leq\tilde x_i\}}
K_0(|x-y|)e^{-|y|}(1+|y|)^k\,dy &\lesssim
\int_{B(0,\tilde x_i)}
\frac{e^{-|x-y|}e^{-|y|}}
     {|x-y|^{1/2}}
(1+|y|)^k\,dy \\
&\qquad\leq
e^{-|x|}(1+|x|)^k
\int_{B(0,\tilde x_i)}
\frac{dy}{|x-y|^{1/2}}.
\end{aligned}
\]
Since $\tilde x_i<|x|$, we have
\(
B(0,\tilde x_i)\subset B(x,2|x|).
\)
Consequently,
\[
\begin{aligned}
\int_{B(0,\tilde x_i)}
\frac{dy}{|x-y|^{1/2}}
&\leq
\int_{B(x,2|x|)}
\frac{dy}{|x-y|^{1/2}} =
2\pi\int_0^{2|x|}r^{1/2}\,dr
\lesssim
|x|^{3/2}.
\end{aligned}
\]
Thus,
\[
\begin{aligned}
\int_{\{|y|\leq\tilde x_i\}}
K_0(|x-y|)e^{-|y|}(1+|y|)^k\,dy
&\lesssim
|x|^{3/2}e^{-|x|}(1+|x|)^k \lesssim
e^{-|x|}(1+|x|)^{k+\frac32}.
\end{aligned}
\]
\end{itemize}
\end{itemize}

Combining the estimates in all regions, we conclude that
\[
h(x)
\lesssim
e^{-|x|}(1+|x|)^{k+\frac32}.
\]
Therefore,
\[
|f(x)|
\leq
C(k)A(1+|x|)^{k+\frac32}e^{-|x|},
\]
which completes the proof.
\end{proof}

Denote by $\mathcal{Y}$ the set of functions
$$
\mathcal{Y} = \left\{ f \in C^{\infty}(\mathbb{R}^2) : 
\begin{aligned}
&\forall i,j \geq 0, \exists r_{i,j}, C_{i,j} > 0 \text{ such that} \\
&\bigl|\partial^{i}_{x_1}\partial^{j}_{x_2}f(x)\bigr| \leq C_{i,j}(1+|x|)^{r_{i,j}}e^{-|x|}, \quad \forall x \in \mathbb{R}^2
\end{aligned} 
\right\}.
$$

\begin{lemma} \label{InversionofL}
Suppose $f \in L^2(\mathbb{R}^2)$ and that $Lf$, interpreted in the distributional sense, belongs to $\mathcal{Y}$. 
Then $f \in C^{\infty}(\mathbb{R}^2)\cap H^{\infty}(\mathbb{R}^2)$ and $f\in \mathcal{Y}.$
\end{lemma}
\begin{proof} 
We start by proving by induction that $f\in H^m(\mathbb{R}^2),$ for all $m\geq 0.$ 
Suppose that $f \in H^m,$ for some $m\geq 0.$ Since multiplication by $Q^2$ preserves every Sobolev space and $Lf\in H^m,$ then 
\[
(-\Delta+1)f=Lf+3Q^2f\in H^m \implies f\in H^{m+2}.
\]
As $f\in L^2,$ iteration gives $f\in H^{\infty}(\R^2)$ and consequently $f\in C^{\infty}(\R^2)$ and every derivative of $f$ is bounded. Now, for every multi-index $\alpha,$ we have 
\[
(-\Delta+1)\partial^{\alpha}f=\partial^{\alpha}Lf+3\sum_{\beta\leq \alpha}\binom{\alpha}{\beta}\partial^{\beta}(Q^2)\partial^{\alpha-\beta}f.
\]
 Every derivative of $f$ in this expression is bounded, while every derivative of $Q^2$ decays exponentially. Since $Lf\in\mathcal{Y},$ 
\[
|(-\Delta+1)\partial^{\alpha}f(x)|\lesssim C_{\alpha}(1+|x|)^{m_{\alpha}}e^{-|x|} \text{ for all }x\in \R^2.
\]
Applying Lemma~\ref{Greenfunction} to $\partial^{\alpha}f$ yields $\mathcal{Y}-$type decay of $\partial^{\alpha}f$ for every multi-index $\alpha,$ thus $f\in \mathcal{Y}.$

\end{proof} 

\begin{lemma}\label{Fredholm}
For any function $h \in L^2(\mathbb{R}^2)$ orthogonal to $\nabla Q$ for the $L^2$ scalar product, there exists a unique function $f \in H^2(\mathbb{R}^2)$ orthogonal to $\nabla Q$ such that $Lf=h.$ Moreover, if $h$ is even in both variables, then $f$ is even in both variables. 
\end{lemma}
This is a consequence of the Fredholm alternative. We now move to the construction of $P.$

\begin{lemma}\label{Pdecaylemma}
There exists a unique smooth function $P$, even in $x_2$, such that
$P_{x_1}\in\mathcal Y$,
\[
(LP)_{x_1}=\Lambda Q,
\qquad
(P,Q_{x_1})=(P,Q_{x_2})=0,
\]
and, for every $x_2\in\mathbb R$,
\[
\lim_{x_1\to+\infty}P(x_1,x_2)
=
\lim_{x_1\to+\infty}P_{x_2x_2}(x_1,x_2)=0.
\]
For every $0<\eta<1$, this function satisfies
\[
|P(x_1,x_2)|\lesssim e^{-(1-\eta)|x_2|}
\qquad\text{for all }(x_1,x_2)\in\R^2,
\]
and
\[
|P(x_1,x_2)|
\lesssim e^{-(1-\eta)x_1-(2\eta-\eta^2)|x_2|},
\qquad x_1>0.
\]
Its pairing with $Q$ is
\[
(P,Q)=\frac14
\norm{\int_{-\infty}^{\infty}\Lambda Q\,dx_1}_{L^2_{x_2}}^2>0.
\]
For $b$ small, the profile $\overline Q_b=Q+bP$ satisfies
\[
\norm{(\Delta\overline Q_b-\overline Q_b+\overline Q_b^3)_{x_1}
+b\Lambda\overline Q_b}_{L^\infty}\lesssim b^2.
\]
\end{lemma}

\begin{proof}
By the Lax--Milgram theorem, there is a unique $U\in H^1(\R^2)$ such
that
\[
(-\Delta+1)U=\Lambda Q.
\]
Elliptic regularity gives $U\in H^\infty(\R^2)$. Since
\[
(-\Delta+1)\partial^\beta U=\partial^\beta\Lambda Q\in\mathcal Y,
\]
Lemma~\ref{Greenfunction} applied to every derivative yields
$U\in\mathcal Y$. Uniqueness also implies that $U$ is even in both
variables.

Set
\[
G(x_1,x_2)=\int_{x_1}^{\infty}U(z,x_2)\,dz,
\qquad
H(x_1,x_2)=\int_{x_1}^{\infty}\Lambda Q(z,x_2)\,dz.
\]
A direct computation gives
\[
(-\Delta+1)G=H,
\qquad
R:=LG-H=-3Q^2G.
\]
The decay of $Q$ and $U$, together with Leibniz's rule, gives
$R\in\mathcal Y$. Moreover,
\[
(R,Q_{x_1})=-(H,Q_{x_1})=-(\Lambda Q,Q)=0,
\qquad
(R,Q_{x_2})=0,
\]
the second identity following by parity. Lemmas~\ref{InversionofL} and
\ref{Fredholm} therefore yield an even-in-$x_2$ function
$\widetilde P\in\mathcal Y$ such that
\[
L\widetilde P=R,
\qquad
(\widetilde P,Q_{x_1})=(\widetilde P,Q_{x_2})=0.
\]
Define
\[
P=\widetilde P-G+\frac{(U,Q)}{\|Q_{x_1}\|_{L^2}^2}Q_{x_1}.
\]
Then
\[
LP=-H,
\qquad
(LP)_{x_1}=\Lambda Q.
\]
Since $(G,Q_{x_1})=(U,Q)$ by integration by parts, we have
$(P,Q_{x_1})=0$, while $(P,Q_{x_2})=0$ by parity. Furthermore,
\[
P_{x_1}=\widetilde P_{x_1}+U
+\frac{(U,Q)}{\|Q_{x_1}\|_{L^2}^2}Q_{x_1x_1}\in\mathcal Y.
\]
The construction gives the stated limits at $+\infty$ and shows that
$P$ is even in $x_2$.

For every $0<\eta<1$, the definition of $\mathcal Y$ gives constants
$C_\eta,b_\eta>0$ such that
\[
|U(z,x_2)|\le C_\eta e^{-(1-\eta)|x_2|}e^{-b_\eta|z|}.
\]
Thus $G$ satisfies the uniform transverse estimate \(|G(x_1,x_2)|\lesssim e^{-(1-\eta)|x_2|}\) for all $(x_1,x_2)\in \R^2$; since
$\widetilde P,Q_{x_1}\in\mathcal Y$, the same is true of $P$. Also,
\[
\lim_{x_1\to-\infty}P(x_1,x_2)
=-\int_{\R}U(z,x_2)\,dz,
\]
which is even and exponentially decaying. For the right-hand estimate,
set $a=1-\eta$ and $d=2\eta-\eta^2$. Since $a^2+d^2<1$, choose
$\sqrt{a^2+d^2}<\rho<1$. From $P_{x_1}\in\mathcal Y$,
$|P_{x_1}(x)|\lesssim e^{-\rho|x|}$, and therefore, for $x_1>0$,
\[
|P(x_1,x_2)|
\le\int_{x_1}^{\infty}|P_{x_1}(z,x_2)|\,dz
\lesssim e^{-ax_1-d|x_2|}.
\]

To prove uniqueness, let $W=P-P_0,$ where $P_0$ is another function with the same
properties, including the right-end normalization. Since
$W_{x_1}\in\mathcal Y$ and $W(+\infty,x_2)=0$, the representation

\[
W(x_1,x_2)=-\int_{x_1}^{\infty}W_{x_1}(z,x_2)\,dz
\]
may be differentiated in $x_2$ under the integral sign. The defining
bounds for $\mathcal Y$ also give convergence of the resulting
integrals as $x_1\to-\infty$. Since $(LW)_{x_1}=0$,
$LW=\Gamma(x_2)$. Passing to the limit at
$+\infty$ gives $\Gamma=0$. The limit
$W_{-}(x_2)=\lim_{x_1\to-\infty}W(x_1,x_2)$ exists and decays exponentially at \(\pm\infty\);
letting $x_1\to-\infty$ in $LW=0$ gives $-W_{-}''+W_{-}=0$, hence $W_{-}=0$.
Using the two zero end limits and $W_{x_1}\in\mathcal Y$ gives
$W\in L^2(\R^2)$. Since
\[
(-\Delta+1)W=3Q^2W,
\]
we have $W\in H^2(\R^2)$, and therefore
\[
W\in\ker L=\operatorname{span}\{Q_{x_1},Q_{x_2}\}.
\]
The orthogonality conditions imply $W=0$.

It remains to compute $(P,Q)$. Since $LP=-H$,
\[
LP(+\infty,x_2)=0,
\qquad
LP(-\infty,x_2)=-\int_{\R}\Lambda Q\,dx_1.
\]
Using $L(\Lambda Q)=-2Q$ and integrating first on finite strips,
\[
\begin{aligned}
-(P,Q)
&=\frac12(P,L\Lambda Q)
=\frac12(LP,(LP)_{x_1})\\
&=\frac14\int_{\R}[(LP)^2]_{x_1=-\infty}^{x_1=+\infty}\,dx_2
=-\frac14\int_{\R}\left(\int_{\R}\Lambda Q\,dx_1\right)^2dx_2.
\end{aligned}
\]
Finally, if $h(x_2)=\int_{\R}Q(x_1,x_2)\,dx_1$, then
\[
\int_{\R}\Lambda Q\,dx_1=x_2h'(x_2).
\]
If this vanished identically, $h$ would be constant; since $Q\in L^1$
and $Q>0$, this is impossible. Hence $(P,Q)>0$.

For $\overline Q_b=Q+bP$, expansion gives
\[
\begin{aligned}
&(\Delta\overline Q_b-\overline Q_b+\overline Q_b^3)_{x_1}
+b\Lambda\overline Q_b\\
&=b\bigl(-(LP)_{x_1}+\Lambda Q\bigr)
+b^2\bigl((3QP^2)_{x_1}+\Lambda P\bigr)
+b^3(P^3)_{x_1}=O_{L^\infty}(b^2),
\end{aligned}
\]
where the last estimate uses $b$ small, the construction of $P$, and
$P_{x_1}\in\mathcal Y$. This completes the proof.
\end{proof}

\chapter{Modulation Equation} \label{Modulation Equation}
\section{Derivation of the Modulation Equation}

Let $u$ be a solution to the equation $u_t + \Delta \partial_{x_1}u + 3u^2\partial_{x_1}u = 0$. We introduce the rescaled function $\mathfrak u$ defined by:
\[
\mathfrak u(t,y_1,y_2) = \lambda(t) u\Big(t, \lambda(t)y_1 + x_1(t), \lambda(t)y_2 + x_2(t)\Big).
\]
Substituting this ansatz into the original equation and multiplying by $\lambda^3$, we express the evolution of $\mathfrak u$ as follows (using the relations $\mathfrak u_{y_i} = \lambda^2 u_{x_i}$ and $\mathfrak u_{y_i y_j y_k} = \lambda^4 u_{x_i x_j x_k}$):
\begin{equation*}
\begin{split}
\lambda^3 \mathfrak u_t &= \lambda^2 \lambda_t \mathfrak u + \lambda^4 u_t + \lambda^2 \lambda_t y_1 \mathfrak u_{y_1} + \lambda^2 \lambda_t y_2 \mathfrak u_{y_2} + \lambda^2 (x_1)_t \mathfrak u_{y_1} + \lambda^2 (x_2)_t \mathfrak u_{y_2} \\
&= \lambda^2 \lambda_t \mathfrak u + \lambda^4 (-u_{x_1 x_1 x_1} - u_{x_2 x_2 x_1} - 3u^2 u_{x_1}) \\
&\quad + \lambda^2 \lambda_t (y \cdot \nabla \mathfrak u) + \lambda^2 (x_1)_t \mathfrak u_{y_1} + \lambda^2 (x_2)_t \mathfrak u_{y_2} \\
&= \lambda^2 \lambda_t \mathfrak u - (\mathfrak u_{y_1 y_1 y_1} + \mathfrak u_{y_2 y_2 y_1} + 3\mathfrak u^2 \mathfrak u_{y_1}) + \lambda^2 \lambda_t (y \cdot \nabla \mathfrak u) + \lambda^2 (x_1)_t \mathfrak u_{y_1} + \lambda^2 (x_2)_t \mathfrak u_{y_2}.
\end{split}
\end{equation*}
We introduce the renormalized time variable $s$ defined by:
\[
s = \int_{0}^{t} \frac{dt'}{\lambda(t')^3} \quad \implies \quad \frac{ds}{dt} = \frac{1}{\lambda^3}.
\] Consequently, the parameter derivatives satisfy $\lambda_s = \lambda^3 \lambda_t$ and $(x_i)_s = \lambda^3 (x_i)_t$. The equation for $\mathfrak u$ in the $s$ variable becomes:
\begin{equation} \label{eq:v_equation}
\mathfrak u_s - \frac{\lambda_s}{\lambda} \Lambda \mathfrak u - \frac{(x_1)_s}{\lambda} \mathfrak u_{y_1} - \frac{(x_2)_s}{\lambda} \mathfrak u_{y_2} + \partial_{y_1} (\Delta \mathfrak u + \mathfrak u^3) = 0,
\end{equation}
where $\Lambda \mathfrak u = \mathfrak u + y_1 \mathfrak u_{y_1} + y_2 \mathfrak u_{y_2}$ is the $L^2$-scaling generator.

\subsection*{Decomposition and Error Equation}

We construct the approximate profile using a localized radiation term. Let $\chi \in C^{\infty}(\mathbb{R})$ be a cut-off function such that $0 \le \chi \le 1$, $\chi' \ge 0$, with $\chi(x) = 0$ for $x \le -2$ and $\chi(x) = 1$ for $x \ge -1$. We define the modulated profile $Q_b$ as:
\[
Q_b(y_1, y_2) = Q(y_1, y_2) + b \chi\left(|b|^{\gamma} y_1\right) P(y_1, y_2),
\]
for some positive $\gamma<1$ and decompose the solution as $\mathfrak u(y_1,y_2,s) = Q_{b(s)}(y_1,y_2) + \varepsilon(y_1,y_2,s)$. Noticing that $\varepsilon_s = \mathfrak u_s - (Q_b)_s$ and using the identity $\partial_{y_1}(-Q + \Delta Q + Q^3) = 0$ in \eqref{eq:v_equation}, the evolution equation for the error $\varepsilon$ is given by:
\begin{equation} \label{eq:ModulationEquation}
\begin{split}
\varepsilon_s - (L \varepsilon)_{y_1} &=\left( \frac{\lambda_s}{\lambda} + b \right) \Lambda Q_b + \left( \frac{(x_1)_s}{\lambda} - 1 \right) (Q_b)_{y_1} + \frac{(x_2)_s}{\lambda} (Q_b)_{y_2} \\
&\quad + \frac{\lambda_s}{\lambda} \Lambda \varepsilon + \left( \frac{(x_1)_s}{\lambda} - 1 \right) \varepsilon_{y_1} + \frac{(x_2)_s}{\lambda} \varepsilon_{y_2} \\
&\quad + \Phi_b + \Psi_b - R_b(\varepsilon)_{y_1} - R_{NL}(\varepsilon)_{y_1},
\end{split}
\end{equation}
where we have defined the following terms:
\begin{align*}
\chi_b(y_1) &= \chi(|b|^{\gamma} y_1),\Phi_b = -b_s \left( \chi_b + \gamma y_1 (\chi_b)_{y_1} \right) P, \Psi_b = (-\Delta Q_b + Q_b - Q_b^3)_{y_1} - b \Lambda Q_b ,\\
R_b(\varepsilon) &= 3(Q_b^2 - Q^2)\varepsilon, 
R_{NL}(\varepsilon) = (\varepsilon + Q_b)^3 - Q_b^3 - 3Q_b^2 \varepsilon.
\end{align*}
The term $\Psi_b$ represents the error generated by the approximate profile $Q_b$. Recalling that $(-\Delta Q + Q - Q^3)_{y_1} = 0$ and $(LP)_{y_1} = \Lambda Q$, we compute $\Psi_b$ as follows:
\begin{align}
\Psi_b &= b \Big[ (\chi_b)_{y_1} P - 3(\chi_b)_{y_1} P_{y_1 y_1} - 3(\chi_b)_{y_1 y_1} P_{y_1} - (\chi_b)_{y_1 y_1 y_1} P  - 3(\chi_b)_{y_1} Q^2 P - (\chi_b)_{y_1} P_{y_2 y_2} \Big]   \label{eq:Psi_b_expansion}\\
&\quad  +b(\chi_b - 1) \Lambda Q+ b^2 \left( -3(\chi_b^2 P^2 Q)_{y_1} - \Lambda(\chi_b P) \right) - b^3 (\chi_b^3 P^3)_{y_1}. \notag
\end{align}

\subsection*{Estimates on the Profile Error}

We now provide pointwise bounds on the error term $\Psi_b$.

\begin{lemma} \label{lemma:Psi_b_estimates}
For $0<|b|\ll1$ and every fixed integer $k\geq0$, the following
estimates hold, with implicit constants depending on $k$. At $b=0$,
$\Psi_0=0$, and the inequalities are understood by continuity:
\begin{align*}
|\Psi_b| &\lesssim |b|^{1+\gamma} \mathbf{1}_{[-\frac{2}{|b|^{\gamma}}, -\frac{1}{|b|^{\gamma}}]}(y_1) e^{-\frac{|y_2|}{2}} + b^2 \left( \mathbf{1}_{[-\frac{2}{|b|^{\gamma}}, 0]}(y_1) + (1+|y_1|)e^{-\frac{|y_1|}{2}} \right) e^{-\frac{|y_2|}{4}}, \\
|(\Psi_b)_{y_2 y_2}| &\lesssim |b|^{1+\gamma} \mathbf{1}_{[-\frac{2}{|b|^{\gamma}}, -\frac{1}{|b|^{\gamma}}]}(y_1) e^{-\frac{|y_2|}{2}} + b^2 \left( \mathbf{1}_{[-\frac{2}{|b|^{\gamma}}, 0]}(y_1) + (1+|y_1|)e^{-\frac{|y_1|}{2}} \right) e^{-\frac{|y_2|}{4}}, \\
|\partial_{y_1}^k \Psi_b| &\lesssim |b|^{1+(k+1)\gamma} \mathbf{1}_{[-\frac{2}{|b|^{\gamma}}, -\frac{1}{|b|^{\gamma}}]}(y_1) e^{-\frac{|y_2|}{2}} + b^2 \left( \mathbf{1}_{[-\frac{2}{|b|^{\gamma}}, 0]}(y_1) + (1+|y_1|)e^{-\frac{|y_1|}{2}} \right) e^{-\frac{|y_2|}{4}}.
\end{align*}
\end{lemma}

\begin{proof}
We estimate the terms in expansion \eqref{eq:Psi_b_expansion} by analyzing the support of the cut-off function derivatives.

\textbf{1. The Cut-off Transition Region:}
Terms involving derivatives of $\chi_b$ are supported in the region where $|b|^\gamma y_1 \in [-2, -1]$. In this region, we have:
\begin{align*}
|b(\chi_b)_{y_1} P| &\lesssim |b|^{1+\gamma} \mathbf{1}_{[-\frac{2}{|b|^{\gamma}}, -\frac{1}{|b|^{\gamma}}]}(y_1) e^{-\frac{|y_2|}{2}}, \\
|b(\chi_b)_{y_1} P_{y_2 y_2}| &\lesssim |b|^{1+\gamma} \mathbf{1}_{[-\frac{2}{|b|^{\gamma}}, -\frac{1}{|b|^{\gamma}}]}(y_1) e^{-\frac{|y_2|}{2}}, \\
|b(\chi_b)_{y_1} P_{y_1 y_1}| &\lesssim |b|^{1+\gamma} \mathbf{1}_{[-\frac{2}{|b|^{\gamma}}, -\frac{1}{|b|^{\gamma}}]}(y_1) e^{-\frac{|y_1| + |y_2|}{2}} \lesssim |b|^3 e^{-\frac{|y_1|}{4}} e^{-\frac{|y_2|}{2}}, \\
|b(\chi_b)_{y_1 y_1} P_{y_1}| &\lesssim |b|^{1+2\gamma} \mathbf{1}_{[-\frac{2}{|b|^{\gamma}}, -\frac{1}{|b|^{\gamma}}]}(y_1) e^{-\frac{|y_1| + |y_2|}{2}} \lesssim |b|^3 e^{-\frac{|y_1|}{4} - \frac{|y_2|}{2}}.
\end{align*}
Similar bounds hold for terms involving $(\chi_b)_{y_1 y_1 y_1} P$ and $(\chi_b)_{y_1} Q^2 P$.

\textbf{2. The Localization Error:}
For the term involving $(1-\chi_b)$, the support is $y_1 \in (-\infty, -|b|^{-\gamma}]$. Using the exponential decay of $Q$, we obtain:
\[
|b(1-\chi_b) \Lambda Q| \lesssim |b| \mathbf{1}_{(-\infty, -\frac{1}{|b|^{\gamma}}]}(y_1) e^{-\frac{|y_1|}{2}} e^{-\frac{|y_2|}{2}} \lesssim |b| e^{-\frac{1}{4|b|^{\gamma}}} e^{-\frac{|y_1|}{4}} e^{-\frac{|y_2|}{2}} \lesssim |b|^3 e^{-\frac{|y_1|}{4}} e^{-\frac{|y_2|}{2}}.
\]

\textbf{3. Higher Order Terms ($O(b^2)$):}
Using the formula $\Lambda(\chi_b P) = \chi_b P + y_1(\chi_b)_{y_1} P + y_1 \chi_b P_{y_1} + y_2 \chi_b P_{y_2}$, we estimate the quadratic terms:
\begin{align*}
|b^2 \chi_b P| &\lesssim b^2 \left( \mathbf{1}_{[-\frac{2}{|b|^{\gamma}}, 0]}(y_1) + (1+|y_1|)e^{-\frac{|y_1|}{2}} \right) e^{-\frac{|y_2|}{2}}, \\
|b^2 y_1 (\chi_b)_{y_1} P| &\lesssim b^2 \frac{2}{|b|^{\gamma}} \mathbf{1}_{[-\frac{2}{|b|^{\gamma}}, -\frac{1}{|b|^{\gamma}}]}(y_1) |b|^{\gamma} |P| \lesssim b^2 \mathbf{1}_{[-\frac{2}{|b|^{\gamma}}, -\frac{1}{|b|^{\gamma}}]}(y_1) e^{-\frac{|y_2|}{2}}.
\end{align*}
The gradient terms satisfy:
\[
|b^2 y_1 \chi_b P_{y_1}| \lesssim b^2 |y_1| e^{-\frac{|y_1| + |y_2|}{2}},
\]
and
\[
|b^2 y_2 \chi_b P_{y_2}| \lesssim b^2 |y_2|e^{-\frac{|y_2|}{2}} \left( \mathbf{1}_{[-\frac{2}{|b|^{\gamma}}, 0]}(y_1)  + \mathbf{1}_{[0, \infty)}(y_1) e^{-\frac{|y_1|}{2}} \right) \lesssim b^2 \left( \mathbf{1}_{[-\frac{2}{|b|^{\gamma}}, 0]}(y_1) + e^{-\frac{|y_1|}{2}} \right) e^{-\frac{|y_2|}{4}}.
\]
Finally, the nonlinear interaction satisfies $|-3b^2 (\chi_b^2 P^2 Q)_{y_1}| \lesssim b^2 e^{-\frac{|y_1| + |y_2|}{2}}$.

The remaining cubic term $-b^3(\chi_b^3P^3)_{y_1}$ satisfies the
same envelope: on the transition interval, derivatives of $\chi_b$
contribute nonnegative powers of $|b|^\gamma$, and away from it one uses
the decay of $P$ and its derivatives, together with $|b|^3\leq b^2$.
 Summing these contributions yields the bound for $|\Psi_b|$. The estimates for $|(\Psi_b)_{y_2 y_2}|$ and $|\partial_{y_1}^k \Psi_b|$ follow from identical computations.
\end{proof}

\subsection*{Modulated Flow}

Since $u_0 \in \mathcal{T}_{\alpha^*}$, we assume there exists $t_0 > 0$ such that for all $t \in [0, t_0]$, $u(t) \in \mathcal{T}_{\alpha^*}$. This implies the existence of parameters $\lambda^\sharp(t) > 0$ and $(x^\sharp_1(t), x^\sharp_2(t)) \in \mathbb{R}^2$ such that:
\[
\left\| Q - \lambda^\sharp(t) u\Big(t, \lambda^\sharp(t)y_1 + x^\sharp_1(t), \lambda^\sharp(t)y_2 + x^\sharp_2(t)\Big) \right\|_{L^2(\mathbb{R}^2)} < \nu < \nu_0,
\]
for some sufficiently small $\nu_0 > 0$.

\begin{lemma} \label{eq:nualpha}(modulated flow) Let $I=[0,t_0]$, and let $u\in C(I;H^1(\mathbb R^2))$.  Assume that for every
$t\in I$ there exist $\lambda^\sharp(t)>0$ and
$ x^\sharp(t)=(x^\sharp_1(t), x^\sharp_2(t))\in\mathbb R^2$ such that
\begin{equation*}
 \left\|Q-\lambda^\sharp(t)
 u\bigl(t,\lambda^\sharp(t)y_1+ x^\sharp_1(t),
          \lambda^\sharp(t)y_2+ x^\sharp_2(t)\bigr)\right\|_{L^2}
 <\nu,
\end{equation*}
where $0<\nu<\nu_0$ and $\nu_0$ is universal and sufficiently small.  Then
there exist differentiable functions
\[
 \lambda:I\to(0,\infty),\qquad x_1,x_2,b:I\to\mathbb R
\]
such that
\begin{equation*}
 \varepsilon(t,y)=\lambda(t)u(t,\lambda(t)y_1+x_1(t),\lambda(t)y_2+x_2(t))-Q_{b(t)}(y_1,y_2)
\end{equation*}
satisfies
\begin{equation}\label{eq:orthogonality_conditions}
 (\varepsilon(t),Q)=(\varepsilon(t),\varphi\Lambda Q)
 =(\varepsilon(t),\varphi Q_{y_1})=(\varepsilon(t),\varphi Q_{y_2})=0.
\end{equation}
Moreover, 
\begin{equation*}
 \left|\frac{\lambda(t)}{\lambda^{\sharp}(t)}-1\right|
 +|b(t)|+\Bigg|\frac{(x_1(t),x_2(t))-(x^{\sharp}_1(t), x^{\sharp}_2(t))}{\lambda^{\sharp}(t)}\Bigg|+\|\varepsilon(t)\|_{L^2}
 \leq \delta(\nu),
\end{equation*}
where $\delta(\nu)\to0$ as $\nu\to0$.

If, in addition, for some $\alpha^\sharp>0$ and for the same admissible parameters,
\begin{equation}\label{eq:H1_tube_assumption_fixed}
 \left\|Q-\lambda^\sharp(t)
 u\bigl(t,\lambda^\sharp(t)y_1+x^\sharp_1(t),
          \lambda^\sharp(t)y_2+ x^\sharp_2(t)\bigr)\right\|_{H^1}
 <\alpha^\sharp,
\end{equation}
then
\begin{equation}\label{eq:H1_mod_bound_fixed}
 \|\varepsilon(t)\|_{H^1}
 \leq C\alpha^\sharp+C\delta(\nu).
\end{equation}
\end{lemma} 
\begin{proof} 
\textbf{Step 1: Local Decomposition around the Soliton.}
Let
\[
\zeta_1=Q,\qquad
\zeta_2=\varphi(y_1)Q_{y_2},\qquad
\zeta_3=\varphi(y_1)\Lambda Q,\qquad
\zeta_4=\varphi(y_1)Q_{y_1}.
\]
For \(\alpha>0\), set
\[
U_\alpha=\{u\in L^2(\mathbb R^2):\|u-Q\|_{L^2}\leq \alpha\}.
\]
For \(u\in L^2(\mathbb R^2)\) and parameters
\[
p=(\lambda_1,\hat x_1,\hat x_2,b)
\]
near
\[
p_0=(1,0,0,0),
\]
define
\[
\varepsilon_p(y_1,y_2)
=
\lambda_1 u(\lambda_1y_1+\hat x_1,\lambda_1y_2+\hat x_2)
-
Q
-
b\chi(|b|^\gamma y_1)P.
\]

We do not use the implicit function theorem for the map
\[
b\mapsto Q_b
\]
as an \(L^2\)-valued map. Instead, we apply it to the finite-dimensional
scalar map
\[
\rho=(\rho^1,\rho^2,\rho^3,\rho^4),
\qquad
\rho^k(u,p):=(\varepsilon_p,\zeta_k),
\quad k=1,\dots,4.
\]
We claim that \(\rho\) is \(C^1\) from a neighborhood of
\((Q,p_0)\) in \(L^2(\mathbb R^2)\times\mathbb R^4\) into \(\mathbb R^4\).

Indeed, for the dependence on \((\lambda_1,\hat x_1,\hat x_2)\), we write
\[
\left(
\lambda_1u(\lambda_1y+\hat x),\zeta_k(y)
\right)_y
=
\lambda_1^{-1}
\int_{\mathbb R^2}
u(z)\,
\zeta_k\!\left(\frac{z-\hat x}{\lambda_1}\right)\,dz,
\]
where $y=(y_1,y_2)$ and $\hat{x}=(\hat{x}_1,\hat{x}_2).$ Since each \(\zeta_k\) is smooth and exponentially decaying, this expression
is \(C^1\) in \((\lambda_1,\hat x_1,\hat x_2)\) for \(u\in L^2\).

The only point requiring care is the \(b\)-dependence. For each \(k\), define
\[
\Theta_k(b)
=
b\iint_{\mathbb R^2}
\chi(|b|^\gamma y_1)P(y)\zeta_k(y)\,dy.
\]
Although \(P\notin L^2(\mathbb R^2)\), the product \(P\zeta_k\) is integrable,
because \(P\) has only a bounded left tail in \(y_1\), with decay in \(y_2\),
whereas \(\zeta_k\) is exponentially localized. Hence
\[
\frac{\Theta_k(b)-\Theta_k(0)}{b}
=
\iint \chi(|b|^\gamma y_1)P\zeta_k
\longrightarrow
\iint P\zeta_k
\qquad\text{as }b\to0.
\]
Moreover, for \(b\neq0\),
\[
\Theta_k'(b)
=
\iint
\left[
\chi(|b|^\gamma y_1)
+
\gamma |b|^\gamma y_1\chi'(|b|^\gamma y_1)
\right]P\zeta_k.
\]
The first term converges to \((P,\zeta_k)\) by dominated convergence. The
second term is supported where
\[
y_1\sim -|b|^{-\gamma},
\]
and therefore tends to zero because \(\zeta_k\) decays exponentially in
\(|y_1|\). Thus \(\Theta_k\in C^1\) near \(b=0\), with
\[
\Theta_k'(0)=(P,\zeta_k).
\]
Consequently,
\[
\partial_b\rho^k(Q,p_0)=-(P,\zeta_k),
\]
where this derivative is understood as a derivative of the scalar functional
\(\rho^k\), not as an \(L^2\)-derivative of \(Q_b\).

\textbf{Step 2: Jacobian Analysis.}
We follow the approach of Lemma 2 in \cite{MerleRaphael05}. We define four functionals corresponding to the orthogonality conditions:
\begin{align*}
\rho^1(u, \lambda_1, \hat{x}_1, \hat{x}_2, b) &= \iint \varepsilon_{\lambda_1, \hat{x}_1, \hat{x}_2, b} \, Q, \\
\rho^2(u, \lambda_1, \hat{x}_1, \hat{x}_2, b) &= \iint \varepsilon_{\lambda_1, \hat{x}_1, \hat{x}_2, b} \, \varphi(y_1) Q_{y_2}, \\
\rho^3(u, \lambda_1, \hat{x}_1, \hat{x}_2, b) &= \iint \varepsilon_{\lambda_1, \hat{x}_1, \hat{x}_2, b} \, \varphi(y_1) \Lambda Q, \\
\rho^4(u, \lambda_1, \hat{x}_1, \hat{x}_2, b) &= \iint \varepsilon_{\lambda_1, \hat{x}_1, \hat{x}_2, b} \, \varphi(y_1) Q_{y_1}.
\end{align*}
We verify the solvability of the system $\rho^i = 0$ near the reference point $p_0 = (\lambda_1=1, \hat{x}_1=0, \hat{x}_2=0, b=0)$ with $u=Q$. At this point, the partial derivatives of $\varepsilon$ are:
\[
\frac{\partial \varepsilon}{\partial \lambda_1}\Big|_{p_0} = \Lambda Q, \quad
\frac{\partial \varepsilon}{\partial \hat{x}_1}\Big|_{p_0} = Q_{y_1}, \quad
\frac{\partial \varepsilon}{\partial \hat{x}_2}\Big|_{p_0} = Q_{y_2}, \quad
\frac{\partial \varepsilon}{\partial b}\Big|_{p_0} = -P.
\]
We compute the Jacobian matrix of the vector $\rho = (\rho^1, \rho^2, \rho^3, \rho^4)$ with respect to $(\lambda_1, \hat{x}_1, \hat{x}_2, b)$ at $p_0$:

\textit{1. Derivatives of $\rho^1$ (Orthogonality against $Q$):}
\begin{align*}
\partial_{\lambda_1} \rho^1 &= (\Lambda Q, Q) = 0, \\
\partial_{\hat{x}_1} \rho^1 &= (Q_{y_1}, Q) = 0, \\
\partial_{\hat{x}_2} \rho^1 &= (Q_{y_2}, Q) = 0, \\
\partial_{b} \rho^1 &= -(P, Q) = -(P, Q).
\end{align*}

\textit{2. Derivatives of $\rho^2$ (Orthogonality against $\varphi Q_{y_2}$):}
Using parity arguments ($\Lambda Q, \varphi(y_1)$ are even in $y_2$, while $Q_{y_2}$ is odd):
\begin{align*}
\partial_{\lambda_1} \rho^2 &= (\Lambda Q, \varphi Q_{y_2}) = 0, \\
\partial_{\hat{x}_1} \rho^2 &= (Q_{y_1}, \varphi Q_{y_2}) = 0, \\
\partial_{\hat{x}_2} \rho^2 &= (Q_{y_2}, \varphi Q_{y_2}) \neq 0, \\
\partial_{b} \rho^2 &= -(P, \varphi Q_{y_2}) = 0 \quad \text{(since $P$ is even in $y_2$)}.
\end{align*}

\textit{3. Derivatives of $\rho^3$ (Orthogonality against $\varphi \Lambda Q$):}
\begin{align*}
\partial_{\lambda_1} \rho^3 &= (\Lambda Q, \varphi \Lambda Q), \\
\partial_{\hat{x}_1} \rho^3 &= (Q_{y_1}, \varphi \Lambda Q), \\
\partial_{\hat{x}_2} \rho^3 &= (Q_{y_2}, \varphi \Lambda Q) = 0, \\
\partial_{b} \rho^3 &= -(P, \varphi \Lambda Q).
\end{align*}

\textit{4. Derivatives of $\rho^4$ (Orthogonality against $\varphi Q_{y_1}$):}
\begin{align*}
\partial_{\lambda_1} \rho^4 &= (\Lambda Q, \varphi Q_{y_1}), \\
\partial_{\hat{x}_1} \rho^4 &= (Q_{y_1}, \varphi Q_{y_1}), \\
\partial_{\hat{x}_2} \rho^4 &= (Q_{y_2}, \varphi Q_{y_1}) = 0, \\
\partial_{b} \rho^4 &= -(P, \varphi Q_{y_1}).
\end{align*}
At \((u,p)=(Q,p_0)\), the Jacobian matrix $J=D_p\rho(Q,p_0)$ is:

\[
J =
\begin{bmatrix}
0 & 0 & 0 & -(P,Q) \\
0 & 0 & (Q_{y_2}, \varphi Q_{y_2}) & 0 \\
(\Lambda Q, \varphi \Lambda Q) & (Q_{y_1}, \varphi \Lambda Q) & 0 & -(P, \varphi \Lambda Q) \\
(\Lambda Q, \varphi Q_{y_1}) & (Q_{y_1}, \varphi Q_{y_1}) & 0 & -(P, \varphi Q_{y_1})
\end{bmatrix}.
\]
Expanding the determinant, we find:
\[
\det J = (P,Q) \cdot (Q_{y_2}, \varphi Q_{y_2}) \cdot \det M^* \neq 0,
\]
where $M^*$ is defined in \eqref{M*}. The non-vanishing of this determinant guarantees the local invertibility.

\textbf{Step 3: Conclusion.}

By the finite-dimensional implicit function theorem applied in $L^2$ to the scalar
map \(\rho\), there exist \(\overline\alpha>0\), a neighborhood
\(V\subset\mathbb R^4\) of \((1,0,0,0)\), and a unique \(C^1\) map
\[
u\mapsto
(\lambda_1(u),\hat x_1(u),\hat x_2(u),b(u))
\in V
\]
for \(u\in U_{\overline\alpha}\), such that
\[
\rho^k(u,\lambda_1(u),\hat x_1(u),\hat x_2(u),b(u))=0,
\qquad k=1,\dots,4.
\]
Equivalently, the corresponding error
\[
\varepsilon_{\lambda_1,\hat x_1,\hat x_2,b}
=
\lambda_1u(\lambda_1y+\hat x)-Q_b
\]
satisfies the four orthogonality conditions
\[
(\varepsilon,Q)
=
(\varepsilon,\varphi Q_{y_2})
=
(\varepsilon,\varphi\Lambda Q)
=
(\varepsilon,\varphi Q_{y_1})
=
0.
\]

Finally, shrinking \(\overline\alpha\) if necessary, the parameters satisfy
\[
|\lambda_1-1|+|\hat x_1|+|\hat x_2|+|b|
\lesssim \|u-Q\|_{L^2}.
\]
Moreover,
\[
\|\varepsilon_{\lambda_1,\hat x_1,\hat x_2,b}\|_{L^2}
\leq \delta(\|u-Q\|_{L^2}).
\]

Now, take $\nu<\min\{\overline{\alpha}, \nu_0\}.$ For all time $t\in [0,t_0]$, there are parameters $\lambda^\sharp(t)>0,$ $(x^\sharp_1(t), x^\sharp_2(t))\in \mathbb{R}^2$ such that 
$$\|Q-\lambda^\sharp(t)u(t, \lambda^\sharp(t)x_1+x^\sharp_1(t), \lambda^\sharp(t)x_2+x^\sharp_2(t))\|_{L^2(\mathbb{R}^2)}<\nu.$$

Applying the preceding claim to the rescaled function $\lambda^\sharp(t)u(t, \lambda^\sharp(t)x_1+x^\sharp_1(t), \lambda^\sharp(t)x_2+x^\sharp_2(t))$, define
\[\lambda(t)=\lambda^\sharp(t)\lambda_1(t), \quad x_1(t)=\lambda^\sharp(t)\hat{x}_1(t)+x^\sharp_1(t),\quad x_2(t)=\lambda^\sharp(t)\hat{x}_2(t)+x^\sharp_2(t).
\]
Therefore 
 $\varepsilon(t,x_1,x_2)=\lambda(t)u(t,\lambda(t)x_1+x_1(t), \lambda(t)x_2+x_2(t))-Q_b(x_1,x_2)$ satisfies 
$$\varepsilon(t)\perp Q, \varepsilon(t)\perp \varphi(y_1)\Lambda Q, \varepsilon(t) \perp \varphi(y_1)Q_{y_1}, \varepsilon(t)\perp \varphi(y_1)Q_{y_2},$$
and
$$\|\varepsilon(t)\|_{L^2}+\Bigg|\frac{\lambda(t)}{\lambda^\sharp(t)}-1\Bigg|+|b(t)|+\Bigg|\frac{(x_1(t),x_2(t))-(x^\sharp_1(t), x^\sharp_2(t))}{\lambda^\sharp(t)}\Bigg|<\delta(\nu),$$ 
 for all $t \in [0,t_0].$

 Finally, assume the additional $H^1$ closeness
\eqref{eq:H1_tube_assumption_fixed}.  The translation and scaling maps are
continuous on $H^1$ near the identity, and
\[
 \|Q_b-Q\|_{H^1}\lesssim |b|^{1-\gamma/2}\lesssim \delta(\nu).
\]
Consequently,
\[
 \|\varepsilon(t)\|_{H^1}
 \leq C\alpha^\sharp+\delta(\nu),
\]
which gives \eqref{eq:H1_mod_bound_fixed}.

Since the IFT gives local uniqueness, parameter maps constructed at nearby times agree on overlaps. This gives continuous global parameters on the interval on which the solution remains in the tube. Differentiability in time can then be obtained from the differentiated orthogonality equations and the evolution equation.
 \end{proof}

\section{Modulation estimates}

We define a locally Lipschitz weight $\vartheta$ with
$|\nabla\vartheta|\leq C$ almost everywhere by
\begin{equation}\label{varthetaweight}
\vartheta(y_1, y_2) = 
\begin{cases} 
\frac{1}{\alpha_1}|y_1| + \left(1 - \frac{1}{\alpha_{2}^{2}}\right)|y_2| & \text{if } y_1 < -1, \\
\frac{1}{200\alpha_1\alpha_2}|y_1| + \left(1 - \frac{1}{\alpha_2^2}\right)|y_2| & \text{if } y_1 \geq 0,
\end{cases}
\end{equation}
where $\alpha_1,\alpha_2$ are defined in \eqref{alphas}, and we use a
smooth interpolation of the $y_1$-dependent part on $[-1,0]$. We further
define the localized mass and energy functionals:
\begin{align*}
\mathcal{M}(s) &= \iint \varepsilon^2(s) e^{-\vartheta(y_1, y_2)} \quad \left(\lesssim \|\varepsilon\|^{2}_{L^2}\right), \\
\widetilde{\mathcal{M}}(s) &= \iint (|\nabla\varepsilon|^2 + \varepsilon^2)(s) e^{-\vartheta(y_1, y_2)}.
\end{align*}
Consequently, all functions depending on $t \in [0, t_0]$ can be viewed as depending on $s \in [0, s_0]$, where $s_0 = s(t_0)$.

We mention a Sobolev-type inequality that we are going to use to bound the nonlinearity terms in the modulation equation. 
 \begin{lemma}[Weighted Sobolev Lemma]\label{SobolevLemma} Let $\theta \in W^{1,1}_{loc}(\mathbb{R}^2)$ be a positive function such that there exists $C_{\theta}>1$ with  $|\nabla \theta|\leq C_{\theta}\theta $. Suppose that $u \in L^2(\mathbb{R}^2)\cap H^1(\mathbb{R}^2;\theta dx_1dx_2),$ then 

$$\iint u^4 \theta dx_1dx_2  \leq 3C_{\theta}^2\|u\|_{L^2}^2 \iint \Big( u^2 +u_{x_1}^2+u_{x_2}^2)\theta dx_1dx_2,$$
$$\iint |u|^3 \theta dx_1dx_2 \leq \sqrt{3}C_{\theta} \|u\|_{L^2} \iint \Big( u^2 +u_{x_1}^2+u_{x_2}^2)\theta dx_1dx_2.$$
\end{lemma}
We include a proof of the lemma in Appendix \ref{appendixA} (see Lemma \ref{lemmasobolev}). 

We now derive rough estimates for the modulation parameters.

\begin{lemma}[Estimates for modulated coefficients]\label{decompositionlemma}
Suppose that, for $t\in[0,t_0]$,
\[
\|\varepsilon(t)\|_{L^2}\leq\hat{\nu},
\]
where $0<\hat{\nu}<\nu_0$ is a sufficiently small universal constant and
$\nu_0$ is defined in Lemma~\ref{eq:nualpha}. Then the map
\[
s\in[0,s_0]\longmapsto
\bigl(\lambda(s),x_1(s),x_2(s),b(s)\bigr)
\]
is $C^1$, and
\[
(\varepsilon,Q)
=(\varepsilon,\varphi(y_1)Q_{y_1})
=(\varepsilon,\varphi(y_1)\Lambda Q)
=(\varepsilon,\varphi(y_1)Q_{y_2})
=0.
\]
Moreover, the modulation parameters satisfy
\begin{equation}\label{eq:modulatedcoefficients}
\begin{split}
|b_s+cb^2|
&\lesssim
|b|\mathcal M^{\frac12}
+\mathcal M
+\|\varepsilon\|_{L^2}\widetilde{\mathcal M}
+|b|^3,\\
|b_s|
&\lesssim
b^2+\mathcal M
+\|\varepsilon\|_{L^2}\widetilde{\mathcal M},\\
\left|\frac{\lambda_s}{\lambda}+b\right|
+\left|\frac{(x_1)_s}{\lambda}-1\right|
+\left|\frac{(x_2)_s}{\lambda}\right|
&\lesssim
b^2+\mathcal M^{\frac12}
+\|\varepsilon\|_{L^2}\widetilde{\mathcal M}.
\end{split}
\end{equation}
\end{lemma}

\begin{proof}
The $C^1$ regularity of the parameters and the orthogonality conditions
follow from Lemma~\ref{eq:nualpha}. It remains to estimate their
derivatives. Set
\[
\text{Mod}(s)
:=
\left|\frac{\lambda_s}{\lambda}+b\right|
+\left|\frac{(x_1)_s}{\lambda}-1\right|
+\left|\frac{(x_2)_s}{\lambda}\right|.
\]

\textbf{Step 1.}
Taking the $L^2$ inner product of \eqref{eq:ModulationEquation} with
$Q$, and using
\[
(\varepsilon,Q)=0,
\qquad
((L\varepsilon)_{y_1},Q)
=-(\varepsilon,LQ_{y_1})=0,
\]
gives
\begin{equation*}
\begin{split}
&\left(\frac{\lambda_s}{\lambda}+b\right)(\Lambda Q_b,Q)
+\left(\frac{(x_1)_s}{\lambda}-1\right)((Q_b)_{y_1},Q)
+\frac{(x_2)_s}{\lambda}((Q_b)_{y_2},Q)\\
&\qquad
-b_s\bigl((\chi_b+\gamma y_1(\chi_b)_{y_1})P,Q\bigr)\\
&=
-\frac{\lambda_s}{\lambda}(\Lambda\varepsilon,Q)
-\left(\frac{(x_1)_s}{\lambda}-1\right)(\varepsilon_{y_1},Q)
-\frac{(x_2)_s}{\lambda}(\varepsilon_{y_2},Q)
-(\Psi_b,Q)\\
&\qquad
+(R_b(\varepsilon)_{y_1},Q)
+(R_{NL}(\varepsilon)_{y_1},Q).
\end{split}
\end{equation*}

The identities $(\Lambda Q,Q)=0$, $(Q,\nabla Q)=0$, and
$(P,\nabla Q)=0$ imply
\[
(\Lambda Q_b,Q)
=-b(\chi_bP,\Lambda Q)
=b(\Lambda P,Q)+O(|b|^4),
\]
\[
|((Q_b)_{y_1},Q)|
=|-b(\chi_bP,Q_{y_1})|
\lesssim |b|^4,
\qquad
|((Q_b)_{y_2},Q)|
=|-b(\chi_bP,Q_{y_2})|
\lesssim |b|^4.
\]
Moreover,
\[
\begin{aligned}
&\bigl((\chi_b+\gamma y_1(\chi_b)_{y_1})P,Q\bigr)\\
&\qquad
=(P,Q)
-\bigl([(1-\chi_b)+\gamma y_1(1-\chi_b)_{y_1}]P,Q\bigr).
\end{aligned}
\]
The function in square brackets vanishes for
$y_1\geq-|b|^{-\gamma}$, equals $1$ for
$y_1\leq-2|b|^{-\gamma}$, and is uniformly bounded on the transition
region. Hence, using the decay of $PQ$,
\begin{equation*}
\begin{split}
&\left|
\iint
[(1-\chi_b)+\gamma y_1(1-\chi_b)_{y_1}]PQ
\right|\\
&\qquad
\lesssim
\iint_{\{y_1\leq-|b|^{-\gamma}\}}
e^{-\frac{|y_1|+|y_2|}{2}}
\,dy_1dy_2
\lesssim
e^{-\frac{1}{2|b|^\gamma}}
\lesssim |b|^4.
\end{split}
\end{equation*}
Consequently,
\[
\begin{aligned}
\bigl((\chi_b+\gamma y_1(\chi_b)_{y_1})P,Q\bigr)
&=(P,Q)+O(|b|^4)\\
&=
\frac14
\int_{-\infty}^{\infty}
\left(
\int_{-\infty}^{\infty}\Lambda Q\,dy_1
\right)^2dy_2
+O(|b|^4).
\end{aligned}
\]
Furthermore, the expansion \eqref{eq:Psi_b_expansion} gives
\[
\begin{aligned}
(\Psi_b,Q)
&=
\bigl((-\Delta Q_b+Q_b-Q_b^3)_{y_1}-b\Lambda Q_b,Q\bigr)\\
&=
-b^2\bigl(\Lambda P+(3QP^2)_{y_1},Q\bigr)
-c_{\Psi}b^3+O(|b|^4).
\end{aligned}
\]

\textit{Claim.}
We have
\[
\bigl(\Lambda P+(3QP^2)_{y_1},Q\bigr)
=
\frac{c}{4}
\int_{-\infty}^{\infty}
\left(
\int_{-\infty}^{\infty}\Lambda Q\,dy_1
\right)^2dy_2
=c(P,Q),
\]
where $c$ is defined in \eqref{eq:definitionofc}.

\textit{Proof of the Claim.}
Since $(LP)_{y_1}=\Lambda Q$, differentiation of $LP$ yields
\[
6QQ_{y_1}P=L(P_{y_1})-(LP)_{y_1}.
\]
Therefore,
\[
\begin{aligned}
-((3QP^2)_{y_1},Q)
&=(3QP^2,Q_{y_1})=\frac12\bigl(L(P_{y_1})-(LP)_{y_1},P\bigr),
\end{aligned}
\]
while, since $\Lambda^*=-\Lambda$ on $\mathbb R^2$,
\[
-(\Lambda P,Q)
=(P,\Lambda Q)
=(P,(LP)_{y_1}).
\]
Adding the two identities, integrating first on
$[-R,R]\times\mathbb R$, and then letting $R\to\infty$, gives
\[
\begin{aligned}
-\bigl(\Lambda P+(3QP^2)_{y_1},Q\bigr)
&=
\frac12\bigl(L(P_{y_1}),P\bigr)
+\frac12\bigl((LP)_{y_1},P\bigr)\\
&=
\frac12
\int_{-\infty}^{\infty}
\left[P(y_1,y_2)LP(y_1,y_2)\right]_{y_1=-\infty}^{y_1=+\infty}
dy_2\\
&=
\frac12
\int_{-\infty}^{\infty}
F(y_2)
\left(
\int_{-\infty}^{\infty}\Lambda Q\,dy_1
\right)dy_2.
\end{aligned}
\]
Here we used the decay at $+\infty$ and the limits at $-\infty$
established in Lemma~\ref{Pdecaylemma}. For every $y_2\in\mathbb R$,
denote
\begin{equation}\label{definitionofF}
F(y_2)
=
\lim_{y_1\to-\infty}P(y_1,y_2).
\end{equation}
The function $F$ is even, exponentially decaying, and satisfies
\begin{equation}\label{Feq}
\begin{cases}
-F_{y_2y_2}+F
=
-\displaystyle\int_{-\infty}^{\infty}\Lambda Q\,dy_1
=-g_Q(y_2),\\
\displaystyle\lim_{y_2\to-\infty}F(y_2)
=
\displaystyle\lim_{y_2\to+\infty}F(y_2)=0,
\end{cases}
\end{equation}
where $g_Q$ is defined in \eqref{definitionofg}. Taking the Fourier
transform in \eqref{Feq} gives
\[
(1+\xi^2)\widehat F(\xi)=-\widehat {g_Q}(\xi).
\]
Thus, by Plancherel's theorem and \eqref{eq:definitionofc},
\begin{equation}\label{otherdefinitionofc}
\frac{
\frac12\displaystyle\int_{-\infty}^{\infty}
F\left(-\int_{-\infty}^{\infty}\Lambda Q\,dy_1\right)dy_2
}{
\frac14\displaystyle\int_{-\infty}^{\infty}
\left(\int_{-\infty}^{\infty}\Lambda Q\,dy_1\right)^2dy_2
}
=
\frac{
\displaystyle\int_{\mathbb R}\frac{2}{1+\xi^2}
|\widehat {g_Q}(\xi)|^2\,d\xi
}{
\displaystyle\int_{\mathbb R}|\widehat {g_Q}(\xi)|^2\,d\xi
}
=c.
\end{equation}
This proves the claim.

It remains to estimate the terms containing $\varepsilon$.

\textit{Linear and drift terms.}
The decay from Lemma~\ref{Qdecay}, integration by parts, and
Cauchy--Schwarz give
\[
\begin{aligned}
|(\varepsilon_{y_1},Q)|
+|(\varepsilon_{y_2},Q)|
&\lesssim
\left(\iint\varepsilon^2e^{-\vartheta}\right)^{\frac12}
\left(\iint|\nabla Q|^2e^\vartheta\right)^{\frac12}\\
&\lesssim \mathcal M^{\frac12}.
\end{aligned}
\]
The same argument gives
\[
|(\varepsilon,\Lambda Q)|
+|(\varepsilon,PQQ_{y_1})|
\lesssim \mathcal M^{\frac12}.
\]
All weighted integrals involving $Q$ converge because its exponential
decay dominates the growth of $e^\vartheta$.

\textit{Linearized modulation error $(R_b)$.}
Since $R_b(\varepsilon)=3(Q_b^2-Q^2)\varepsilon$,
\[
\begin{aligned}
(R_b(\varepsilon)_{y_1},Q)
&=-(R_b(\varepsilon),Q_{y_1})\\
&=-6b(\chi_bPQQ_{y_1},\varepsilon)
-3b^2(\chi_b^2P^2Q_{y_1},\varepsilon)\\
&=-6b(PQQ_{y_1},\varepsilon)
+O(|b|^4+\mathcal M).
\end{aligned}
\]

\textit{Nonlinear terms $(R_{NL})$.}
Using
\[
|R_{NL}(\varepsilon)|
\lesssim |Q_b|\varepsilon^2+|\varepsilon|^3
\]
and Lemma~\ref{SobolevLemma} with
$\theta=e^{-\vartheta}$, for which
$|\nabla\theta|\lesssim\theta$, we obtain
\[
\begin{aligned}
|(R_{NL}(\varepsilon)_{y_1},Q)|
&=|(R_{NL}(\varepsilon),Q_{y_1})|\lesssim
\iint\varepsilon^2e^{-\vartheta}
+\iint|\varepsilon|^3e^{-\vartheta}\lesssim
\mathcal M
+\|\varepsilon\|_{L^2}\widetilde{\mathcal M}.
\end{aligned}
\]

Combining the preceding expansions and remainder estimates yields
\begin{equation}\label{otherbsharpeq}
\begin{split}
&\left|
(b_s+cb^2)
\frac14
\int_{-\infty}^{\infty}
\left(
\int_{-\infty}^{\infty}\Lambda Q\,dy_1
\right)^2dy_2
+\frac{\lambda_s}{\lambda}(\varepsilon,\Lambda Q)\right.\\
&\left.\qquad
-b\left[
6(\varepsilon,PQQ_{y_1})
+\left(\frac{\lambda_s}{\lambda}+b\right)(\Lambda P,Q)
\right]
+c_{\Psi}b^3
\right|\\
&\qquad
\lesssim
\left(
\left|\frac{(x_1)_s}{\lambda}-1\right|
+\left|\frac{(x_2)_s}{\lambda}\right|
\right)
\bigl(|b|^4+\mathcal M^{\frac12}\bigr)
+|b|^4+\mathcal M
+\|\varepsilon\|_{L^2}\widetilde{\mathcal M}.
\end{split}
\end{equation}
Since
\[
\left|\frac{\lambda_s}{\lambda}\right|
\leq
\left|\frac{\lambda_s}{\lambda}+b\right|+|b|,
\]
we infer
\begin{equation}\label{eq:bsestprelim}
|b_s+cb^2|
\lesssim
\text{Mod}(s)\bigl(|b|+\mathcal M^{\frac12}\bigr)
+|b|\mathcal M^{\frac12}
+\mathcal M
+\|\varepsilon\|_{L^2}\widetilde{\mathcal M}
+|b|^3.
\end{equation}
In particular,
\begin{equation}\label{eq:bestprelim}
|b_s|
\lesssim
\text{Mod}(s)\bigl(|b|+\mathcal M^{\frac12}\bigr)
+b^2+\mathcal M
+\|\varepsilon\|_{L^2}\widetilde{\mathcal M}
+|b|^3.
\end{equation}

\textbf{Step 2.}
Let $g\in C^\infty(\mathbb R^2)$ satisfy
\[
\sum_{\ell=0}^3|\nabla^\ell g|
\lesssim
(1+|y|)\varphi(y_1)e^{-|y|},
\qquad
(\varepsilon(s),g)=0
\quad\text{on }[0,s_0].
\]
This class contains the three functions used below. Taking the inner
product of \eqref{eq:ModulationEquation} with $g$ gives
\[
\begin{aligned}
(\varepsilon,L(g_{y_1}))
={}&
\left(\frac{\lambda_s}{\lambda}+b\right)(\Lambda Q_b,g)
+\frac{\lambda_s}{\lambda}(\Lambda\varepsilon,g)\\
&+
\left(\frac{(x_1)_s}{\lambda}-1\right)
((Q_b)_{y_1}+\varepsilon_{y_1},g)
+\frac{(x_2)_s}{\lambda}
((Q_b)_{y_2}+\varepsilon_{y_2},g)\\
&+
(\Phi_b+\Psi_b,g)
+(R_b(\varepsilon)+R_{NL}(\varepsilon),g_{y_1}).
\end{aligned}
\]
For constants $c_\Lambda,c_1,c_2\in\mathbb R$ depending only on $g$,
the localization of $g$ and the decay of $P$ give
\[
(\Lambda Q_b-\Lambda Q,g)
=c_\Lambda b+O(|b|^4),
\]
\[
((Q_b)_{y_1}-Q_{y_1},g)
=c_1b+O(|b|^4),
\qquad
((Q_b)_{y_2}-Q_{y_2},g)
=c_2b+O(|b|^4).
\]

Choose $\beta$ such that
\[
\alpha_1=1+\frac1{100}
>
\alpha_2=1+\frac1{200}
>
\beta
>
\frac{\sqrt2\,\alpha_1}{\sqrt{\alpha_1^2+1}}
\approx1.004963>1.
\]
Then
\[
\frac1{\alpha_1}+\frac1{\alpha_2}-\frac2\beta<0,
\qquad
\frac2{\beta^2}-1-\frac1{\alpha_1^2}<0.
\]
Moreover,
\[
|\Lambda g|
\lesssim
(1+|y|)^3\varphi(y_1)e^{-|y|}
\lesssim
\varphi(y_1)
e^{-\frac{|y_1|}{\beta}
-\left(1-\frac1{\beta^2}\right)|y_2|}.
\]
Using the definition of $\vartheta$, we obtain
\[
\begin{aligned}
\iint|\Lambda g|^2e^\vartheta
&\lesssim
\iint
e^{\left(
\frac1{\alpha_1}+\frac1{\alpha_2}-\frac2\beta
\right)|y_1|
+\left(
\frac2{\beta^2}-1-\frac1{\alpha_1^2}
\right)|y_2|}
<\infty.
\end{aligned}
\]
Consequently, Cauchy--Schwarz and integration by parts yield
\[
|(\Lambda\varepsilon,g)|
=|(\varepsilon,\Lambda g)|
\lesssim\mathcal M^{\frac12},
\]
and
\[
|(\varepsilon_{y_1},g)|
+|(\varepsilon_{y_2},g)|
\lesssim\mathcal M^{\frac12}.
\]
Thus,
\[
\begin{aligned}
&\left|\frac{\lambda_s}{\lambda}\right|
|(\Lambda\varepsilon,g)|
+\left|\frac{(x_1)_s}{\lambda}-1\right|
|(\varepsilon_{y_1},g)|
+\left|\frac{(x_2)_s}{\lambda}\right|
|(\varepsilon_{y_2},g)|
\lesssim
\bigl(\text{Mod}(s)+|b|\bigr)\mathcal M^{\frac12}.
\end{aligned}
\]

Because
\[
\operatorname{supp}
\bigl(1-\chi_b-\gamma y_1(\chi_b)_{y_1}\bigr)
\subset\{y_1\leq-|b|^{-\gamma}\}
\]
and $g$ is exponentially localized,
\[
\begin{aligned}
((\chi_b+\gamma y_1(\chi_b)_{y_1})P,g)
&=(P,g)+O(|b|^4)=c_g+O(|b|^4),
\qquad c_g\in\mathbb R.
\end{aligned}
\]
Hence
\[
\begin{aligned}
(\Phi_b,g)
&=-b_sc_g+O(|b_s||b|^4)=cb^2c_g-(b_s+cb^2)c_g+O(|b_s||b|^4).
\end{aligned}
\]
Using \eqref{eq:bsestprelim} and \eqref{eq:bestprelim}, we obtain
\[
\begin{aligned}
(\Phi_b,g)
=cb^2c_g
+O\Big(
&\text{Mod}(s)(|b|+\mathcal M^{\frac12})
+|b|^3+|b|\mathcal M^{\frac12}\\
&+\mathcal M
+\|\varepsilon\|_{L^2}\widetilde{\mathcal M}
\Big).
\end{aligned}
\]
Recall that
\[
S_P=3(QP^2)_{y_1}+\Lambda P.
\]
The profile error satisfies
\[
(\Psi_b,g)=-b^2(S_P,g)+O(|b|^3).
\]
Furthermore, the decay of $g_{y_1}$, $Q$, and $P$ gives
\[
|(R_b(\varepsilon),g_{y_1})|
\lesssim
|b|\mathcal M^{\frac12},
\]
while Lemma~\ref{SobolevLemma} yields
\[
|(R_{NL}(\varepsilon),g_{y_1})|
\lesssim
\mathcal M
+\|\varepsilon\|_{L^2}\widetilde{\mathcal M}.
\]
Substitution into the projected equation gives
\begin{equation}\label{groughequation}
\begin{aligned}
(\varepsilon,L(g_{y_1}))
={}&
\left(\frac{\lambda_s}{\lambda}+b\right)(\Lambda Q,g)
+\left(\frac{(x_1)_s}{\lambda}-1\right)(Q_{y_1},g)+
\frac{(x_2)_s}{\lambda}(Q_{y_2},g)
+\tilde c_gb^2+\mathcal R_g,
\end{aligned}
\end{equation}
where
\[
\begin{aligned}
\mathcal R_g
=O\Big(
&\text{Mod}(s)(|b|+\mathcal M^{\frac12})
+|b|^3+|b|\mathcal M^{\frac12}+\mathcal M
+\|\varepsilon\|_{L^2}\widetilde{\mathcal M}
\Big).
\end{aligned}
\]

Setting $g=\varphi Q_{y_2}$ in \eqref{groughequation}, using parity and
\[
(Q_{y_2},\varphi Q_{y_2})>0,
\qquad
|(\varepsilon,L((\varphi Q_{y_2})_{y_1}))|
\lesssim\mathcal M^{\frac12},
\]
gives
\begin{equation}\label{eq:x2estprelim}
\left|\frac{(x_2)_s}{\lambda}\right|
\lesssim
\text{Mod}(s)(|b|+\mathcal M^{\frac12})
+b^2+\mathcal M^{\frac12}
+\|\varepsilon\|_{L^2}\widetilde{\mathcal M}.
\end{equation}

Applying \eqref{groughequation} with
$g=\varphi Q_{y_1}$ and $g=\varphi\Lambda Q$, and using the same
weighted Cauchy--Schwarz estimate for
$(\varepsilon,L(g_{y_1}))$, gives
\begin{equation}\label{eq:lambdaestprelim}
\begin{aligned}
&\left|
\left(\frac{\lambda_s}{\lambda}+b\right)
(\Lambda Q,\varphi(y_1)Q_{y_1})
+\left(\frac{(x_1)_s}{\lambda}-1\right)
(\varphi(y_1)Q_{y_1},Q_{y_1})
\right|\\
&\qquad
\lesssim
\mathcal M^{\frac12}+b^2
+\|\varepsilon\|_{L^2}\widetilde{\mathcal M}
+\text{Mod}(s)(|b|+\mathcal M^{\frac12})
\end{aligned}
\end{equation}
and
\begin{equation}\label{eq: x1estprelim}
\begin{aligned}
&\left|
\left(\frac{(x_1)_s}{\lambda}-1\right)
(\varphi(y_1)\Lambda Q,Q_{y_1})
+\left(\frac{\lambda_s}{\lambda}+b\right)
(\varphi(y_1)\Lambda Q,\Lambda Q)
\right|\\
&\qquad
\lesssim
\mathcal M^{\frac12}+b^2
+\|\varepsilon\|_{L^2}\widetilde{\mathcal M}
+\text{Mod}(s)(|b|+\mathcal M^{\frac12}).
\end{aligned}
\end{equation}

\textbf{Step 3.}
Denote
\[
K
=
\text{Mod}(s)(|b|+\mathcal M^{\frac12})
+\mathcal M^{\frac12}+b^2
+\|\varepsilon\|_{L^2}\widetilde{\mathcal M}.
\]
Then \eqref{eq: x1estprelim} and \eqref{eq:lambdaestprelim} become
\begin{equation}\label{eq: x1estprelim2}
\left|
\left(\frac{(x_1)_s}{\lambda}-1\right)
(\varphi(y_1)\Lambda Q,Q_{y_1})
+\left(\frac{\lambda_s}{\lambda}+b\right)
(\varphi(y_1)\Lambda Q,\Lambda Q)
\right|
\lesssim K
\end{equation}
and
\begin{equation}\label{eq:lambdaestprelim2}
\left|
\left(\frac{\lambda_s}{\lambda}+b\right)
(\Lambda Q,\varphi(y_1)Q_{y_1})
+\left(\frac{(x_1)_s}{\lambda}-1\right)
(\varphi(y_1)Q_{y_1},Q_{y_1})
\right|
\lesssim K.
\end{equation}
These inequalities form the linear system with coefficient matrix
\(M^*\) from \eqref{M*}. Applying Cramer's rule gives
\begin{equation}\label{eq:lambdaest}
|\det M^*|
\left|\frac{\lambda_s}{\lambda}+b\right|
\lesssim K
\end{equation}
and
\begin{equation}\label{eq:x1est}
|\det M^*|
\left|\frac{(x_1)_s}{\lambda}-1\right|
\lesssim K.
\end{equation}

Since $\det M^*\neq0$, estimates \eqref{eq:x2estprelim},
\eqref{eq:lambdaest}, and \eqref{eq:x1est} imply
\[
\text{Mod}(s)
\leq
C\left[
\text{Mod}(s)(|b|+\mathcal M^{\frac12})
+\mathcal M^{\frac12}+b^2
+\|\varepsilon\|_{L^2}\widetilde{\mathcal M}
\right].
\]
By Lemma~\ref{eq:nualpha},
$|b|$ is uniformly small, while
$\mathcal M^{1/2}\lesssim\|\varepsilon\|_{L^2}\leq\hat\nu$.
Choosing $\nu_0$ and $\hat\nu$ sufficiently small allows the first term
on the right-hand side to be absorbed. Thus,
\begin{equation}\label{eq:x1x2lambda}
\left|\frac{\lambda_s}{\lambda}+b\right|
+\left|\frac{(x_1)_s}{\lambda}-1\right|
+\left|\frac{(x_2)_s}{\lambda}\right|
\lesssim
\mathcal M^{\frac12}+b^2
+\|\varepsilon\|_{L^2}\widetilde{\mathcal M}.
\end{equation}

Finally, substituting \eqref{eq:x1x2lambda} into
\eqref{eq:bsestprelim} and \eqref{eq:bestprelim}, and using the
smallness of $|b|+\mathcal M^{1/2}$, gives
\begin{equation}\label{eq:b}
|b_s|
\lesssim
\mathcal M+b^2
+\|\varepsilon\|_{L^2}\widetilde{\mathcal M}
\end{equation}
and
\begin{equation}\label{eq:bs}
|b_s+cb^2|
\lesssim
|b|\mathcal M^{\frac12}
+\mathcal M+|b|^3
+\|\varepsilon\|_{L^2}\widetilde{\mathcal M}.
\end{equation}
This proves \eqref{eq:modulatedcoefficients}.
\end{proof}

\section{Sharp Modulation Estimates}
\begin{lemma}\label{sharporthogonalities}
Let $\langle y_2\rangle=(1+y_2^2)^{1/2}$ and define the weight
$\Omega\in C^\infty(\mathbb R^2)$ by
\[
\Omega(y_1,y_2)=
\begin{cases}
e^{-\frac{|y_1|}{\alpha_1}-(1-\frac{1}{\alpha_2^2})\langle y_2\rangle},
& \text{if }y_1<-1,\\
e^{-(1-\frac{1}{\alpha_2^2})\langle y_2\rangle}y_1^2,
& \text{if }y_1>1,
\end{cases}
\]
with any smooth positive interpolation on $[-1,1]$ such that
$e^{-\vartheta}\lesssim\Omega$. Denote
\[
\widehat{\mathcal M}(s)
=\iint\varepsilon^2(s)\Omega(y_1,y_2)\,dy_1dy_2.
\]
Write
\begin{equation}\label{definitionofCQ}
C_Q=\frac14\int_{-\infty}^{\infty}
\left(\int_{-\infty}^{\infty}\Lambda Q\,dy_1\right)^2dy_2>0,
\qquad
\widetilde C_Q=\frac14\int_{-\infty}^{\infty}
\left(\int_{-\infty}^{\infty}Q_{y_2}\,dy_1\right)^2dy_2>0,
\end{equation}
and denote
\[
J_1(s)=\left(\varepsilon(s),\frac{1}{2C_Q}
\int_{-\infty}^{y_1}\Lambda Q\right),
\qquad
J_2(s)=\left(\varepsilon(s),\frac{1}{2\widetilde C_Q}
\int_{-\infty}^{y_1}Q_{y_2}\right).
\]
Under the assumptions of Lemma~\ref{decompositionlemma}, with
$\|\varepsilon\|_{L^2}\leq\delta(\nu^*)$, and the additional bound
\[
\sup_{s\in[0,s_0]}\widehat{\mathcal M}(s)\leq1,
\]
we have:
\begin{itemize}
\item[a)] $|J_1(s)|,|J_2(s)|\lesssim\widehat{\mathcal M}(s)^{1/2}$;
\item[b)]
\[
\left|\frac{d}{ds}J_1(s)+\frac{\lambda_s}{\lambda}J_1(s)
-\left(\frac{\lambda_s}{\lambda}+b\right)\right|
\lesssim b^2(s)+\widehat{\mathcal M}(s)
+\delta(\nu^*)\widetilde{\mathcal M}(s);
\]
\item[c)]
\[
\left|\frac{d}{ds}J_2(s)+\frac{\lambda_s}{\lambda}J_2(s)
-\frac{(x_2)_s}{\lambda}\right|
\lesssim b^2(s)+\widehat{\mathcal M}(s)
+\delta(\nu^*)\widetilde{\mathcal M}(s).
\]
\end{itemize}
\end{lemma}

\begin{proof}
Set
\[
q(y_2)=\int_{\mathbb R}Q(y_1,y_2)\,dy_1.
\]
Then
\[
\int_{\mathbb R}\Lambda Q\,dy_1=y_2q'(y_2),
\qquad
\int_{\mathbb R}Q_{y_2}\,dy_1=q'(y_2).
\]
Since $q$ is positive, even, and tends to zero at infinity, it is not
constant; hence $q'\not\equiv0$, and both constants in
\eqref{definitionofCQ} are strictly positive.

Let $f\in C^\infty(\mathbb R^2)$ be fixed and time-independent, with
\begin{equation}\label{eq:projection-test-hypothesis}
\sum_{|\beta|\leq2}|\partial^\beta f(y)|
\lesssim e^{-\vartheta(y)}.
\end{equation}
The implicit constants below may depend on $f$, but are independent of
$s$ and $b$. Set
\[
h(y_1,y_2)=\int_{-\infty}^{y_1}f(z,y_2)\,dz,
\qquad
J_f=(\varepsilon,h).
\]
Writing $\beta_0=1-\alpha_2^{-2}>0$, the weighted derivative bounds
justify differentiation under the integral and give, for $j=0,1,2$,
\[
|\partial_{y_2}^j h(y_1,y_2)|\lesssim
\begin{cases}
e^{y_1/\alpha_1-\beta_0|y_2|},&y_1<-1,\\
e^{-\beta_0|y_2|},&y_1\geq-1.
\end{cases}
\]
These derivatives converge as $y_1\to+\infty$ to exponentially
decaying functions of $y_2$. The definition of $\Omega$ therefore gives
\[
\iint\frac{h^2+y_1^2f^2+y_2^2h_{y_2}^2+f^2+h_{y_2}^2}{\Omega}<\infty.
\]
Cauchy--Schwarz therefore gives
\[
|J_f|+|(\varepsilon,y_1f)|+|(\varepsilon,y_2h_{y_2})|
+|(\varepsilon,f)|+|(\varepsilon,h_{y_2})|
\lesssim\widehat{\mathcal M}^{1/2}.
\]
The functions $\Lambda Q$ and $Q_{y_2}$ satisfy
\eqref{eq:projection-test-hypothesis}, so these estimates prove part a).
Also, $|Lf|\lesssim e^{-\vartheta}$, and hence
$|(\varepsilon,Lf)|\lesssim\mathcal M^{1/2}$.
Notice that
$e^{-\vartheta}\lesssim\Omega$, and hence
\[
\mathcal M\lesssim\widehat{\mathcal M}.
\]

We now project \eqref{eq:ModulationEquation} onto $h$. Since
\[
(\Lambda\varepsilon,h)
=-(\varepsilon,\Lambda h)
=-(\varepsilon,h)-(\varepsilon,y_1f+y_2h_{y_2}),
\]
and $h_{y_1}=f$, integration by parts gives
\begin{equation*}
\begin{aligned}
(J_f)_s+\frac{\lambda_s}{\lambda}J_f
&=-(\varepsilon,Lf)
+\left(\frac{\lambda_s}{\lambda}+b\right)(\Lambda Q,h)
+\left(\frac{(x_1)_s}{\lambda}-1\right)(Q_{y_1},h)
+\frac{(x_2)_s}{\lambda}(Q_{y_2},h)\\
&\quad-\frac{\lambda_s}{\lambda}(\varepsilon,y_1f+y_2h_{y_2})
-\left(\frac{(x_1)_s}{\lambda}-1\right)(\varepsilon,f)
-\frac{(x_2)_s}{\lambda}(\varepsilon,h_{y_2})\\
&\quad+b\left(\frac{\lambda_s}{\lambda}+b\right)
(\Lambda(\chi_bP),h)
-b\left(\frac{(x_1)_s}{\lambda}-1\right)(\chi_bP,f)
-b\frac{(x_2)_s}{\lambda}(\chi_bP,h_{y_2})\\
&\quad-b_s\bigl([\chi_b+\gamma y_1(\chi_b)_{y_1}]P,h\bigr)
+(\Psi_b,h)+(R_{NL}(\varepsilon)+R_b(\varepsilon),f).
\end{aligned}
\end{equation*}
All terms are well defined by the preceding weighted estimates and the
decay of $P$ from Lemma~\ref{Pdecaylemma}.

Using \eqref{eq:modulatedcoefficients},
$\mathcal M\lesssim\widehat{\mathcal M}$, and
$\widehat{\mathcal M}^{1/2}\leq1$, the terms containing
$\varepsilon$ and the geometric parameters satisfy
\[
\begin{aligned}
&\left|\frac{\lambda_s}{\lambda}\right|
\bigl(|(\varepsilon,y_1f)|+|(\varepsilon,y_2h_{y_2})|\bigr)
+\left|\frac{(x_1)_s}{\lambda}-1\right|\,|(\varepsilon,f)|
+\left|\frac{(x_2)_s}{\lambda}\right|\,|(\varepsilon,h_{y_2})|\\
&\qquad\lesssim |b|\widehat{\mathcal M}^{1/2}+\widehat{\mathcal M}
+\delta(\nu^*)\widetilde{\mathcal M}.
\end{aligned}
\]
The left-hand decay of $h$ and the right-hand and transverse decay of
$P$ make the pairings involving $P$ uniformly bounded; hence
\[
\begin{aligned}
&\left|b\left(\frac{\lambda_s}{\lambda}+b\right)
(\Lambda(\chi_bP),h)\right|
+\left|b\left(\frac{(x_1)_s}{\lambda}-1\right)(\chi_bP,f)\right|\\
&\qquad
+\left|b\frac{(x_2)_s}{\lambda}(\chi_bP,h_{y_2})\right|
\lesssim |b|^3+|b|\widehat{\mathcal M}^{1/2}
+\delta(\nu^*)\widetilde{\mathcal M}.
\end{aligned}
\]
Moreover, the cutoff differences are supported in
$y_1\leq-|b|^{-\gamma}$, where the decay of $h$ makes their pairings
with $P$ exponentially small in $|b|^{-\gamma}$. In particular,
\[
\bigl([\chi_b+\gamma y_1(\chi_b)_{y_1}]P,h\bigr)
=(P,h)+O(|b|^4).
\]
Thus, by \eqref{eq:b} and \eqref{eq:bs},
\[
\begin{aligned}
(\Phi_b,h)
&=-b_s\bigl[(P,h)+O(|b|^4)\bigr]=c_fb^2
+O\bigl(|b|\widehat{\mathcal M}^{1/2}+|b|^3
+\widehat{\mathcal M}+\delta(\nu^*)\widetilde{\mathcal M}\bigr).
\end{aligned}
\]
The same cutoff estimate and \eqref{eq:Psi_b_expansion} give
\[
(\Psi_b,h)=-b^2(S_P,h)+O(|b|^3),
\qquad S_P=3(QP^2)_{y_1}+\Lambda P.
\]
Here the cubic term is $O(|b|^3)$ after integration by parts, since
$P$ is bounded and $f$ is integrable. The weighted Sobolev estimate
and \eqref{eq:projection-test-hypothesis} also yield
\[
\begin{aligned}
|(R_{NL}(\varepsilon),f)|
&\lesssim\mathcal M+\|\varepsilon\|_{L^2}\widetilde{\mathcal M},\\
|(R_b(\varepsilon),f)|
&\lesssim |b|\mathcal M^{1/2}.
\end{aligned}
\]
Combining these estimates, we obtain
\begin{equation}\label{equationprojectionf}
\begin{aligned}
(J_f)_s+\frac{\lambda_s}{\lambda}J_f
&=-(\varepsilon,Lf)
+\left(\frac{\lambda_s}{\lambda}+b\right)(\Lambda Q,h)
+\left(\frac{(x_1)_s}{\lambda}-1\right)(Q_{y_1},h)\\
&\quad+\frac{(x_2)_s}{\lambda}(Q_{y_2},h)
+\widehat c_fb^2+\mathcal R_f,
\end{aligned}
\end{equation}
where
\[
\mathcal R_f
=O\bigl(|b|^3+|b|\widehat{\mathcal M}^{1/2}
+\widehat{\mathcal M}+\delta(\nu^*)\widetilde{\mathcal M}\bigr).
\]

Apply \eqref{equationprojectionf} first with $f=\Lambda Q$. If
$h=\int_{-\infty}^{y_1}\Lambda Q$, then
\[
(\varepsilon,L\Lambda Q)=-2(\varepsilon,Q)=0,
\qquad
(Q_{y_1},h)=-(Q,\Lambda Q)=0,
\qquad
(Q_{y_2},h)=0,
\]
and
\[
(\Lambda Q,h)
=\frac12\int_{\mathbb R}
\left(\int_{\mathbb R}\Lambda Q\,dy_1\right)^2dy_2
=2C_Q.
\]
Dividing by $2C_Q$, for some $c_1\in\mathbb R$ we obtain
\[
\left|\frac{d}{ds}J_1+\frac{\lambda_s}{\lambda}J_1
-\left(\frac{\lambda_s}{\lambda}+b\right)+c_1b^2\right|
\lesssim |b|^3+|b|\widehat{\mathcal M}^{1/2}
+\widehat{\mathcal M}+\delta(\nu^*)\widetilde{\mathcal M}.
\]

Finally, apply \eqref{equationprojectionf} with $f=Q_{y_2}$. If
$h=\int_{-\infty}^{y_1}Q_{y_2}$, then
\[
(\varepsilon,LQ_{y_2})=0,
\qquad
(\Lambda Q,h)=0,
\qquad
(Q_{y_1},h)=-(Q,Q_{y_2})=0,
\]
and
\[
(Q_{y_2},h)
=\frac12\int_{\mathbb R}
\left(\int_{\mathbb R}Q_{y_2}\,dy_1\right)^2dy_2
=2\widetilde C_Q.
\]
Therefore, for some $c_2\in\mathbb R$,
\[
\left|\frac{d}{ds}J_2+\frac{\lambda_s}{\lambda}J_2
-\frac{(x_2)_s}{\lambda}+c_2b^2\right|
\lesssim |b|^3+|b|\widehat{\mathcal M}^{1/2}
+\widehat{\mathcal M}+\delta(\nu^*)\widetilde{\mathcal M}.
\]
Since $|b|$ is small,
\[
|b|^3+|b|\widehat{\mathcal M}^{1/2}
\lesssim b^2+\widehat{\mathcal M},
\]
which proves b) and c).
\end{proof}

We now prove a sharp equation that will control the dynamics of the quantity $b/\lambda^c.$

\begin{lemma}\label{lemmasharpblambda}
    Define 
    $$\rho(y_1,y_2)=\frac{1}{C_Q}\Big(P(y_1,y_2)-F(y_2)-\frac{c}{2}\int_{-\infty}^{y_1}\Lambda Q(z,y_2)dz+\frac{(\Lambda P,Q)}{\|\sqrt{\varphi}\Lambda Q\|_{L^2}^2}\varphi(y_1)\Lambda Q(y_1,y_2)\Big), $$
    where $C_Q, F,$ respectively $c$ are defined in \eqref{definitionofCQ}, \eqref{definitionofF}, respectively \eqref{otherdefinitionofc}. Denote $J(s)=(\varepsilon(s),\rho).$ Then, under the assumptions of Lemma \ref{sharporthogonalities}, 
    \begin{equation}\label{Jinnerproductbound}
    |J(s)|\lesssim \widehat{\mathcal{M}}^{\frac{1}{2}}(s), 
    \end{equation}
    and 
    \begin{equation}\label{blambdasharpeq}
      \Bigg|\frac{d}{ds}\left\{\frac{b}{\lambda^c}\right\}+\frac{b}{\lambda^c}\left(J_s+\frac{\lambda_s}{\lambda}J\right)+\widehat{c}\frac{b^3}{\lambda^c}\Bigg|\lesssim \frac{\widehat{\mathcal{M}}+\widetilde{\mathcal{M}}+b^4}{\lambda^{c}}.  
    \end{equation}
\end{lemma}

\begin{proof}
    From the estimates \eqref{otherbsharpeq} and \eqref{eq:modulatedcoefficients}, we have
    \begin{equation}\label{brough2}
    b_s+cb^2+c_bb^3-C_Q^{-1}b\left[\left(\frac{\lambda_s}{\lambda}+b\right)D_P+6I_P\right]+C_Q^{-1}\frac{\lambda_s}{\lambda}(\varepsilon,\Lambda Q)=\mathcal{R}_b,
    \end{equation}
    where we denote $D_P=(\Lambda P, Q), I_P=(\varepsilon, PQQ_{y_1}),$ and the estimate $\mathcal{R}_b=O(\widetilde{\mathcal{M}}+b^4).$

   Now from \eqref{groughequation} and \eqref{eq:modulatedcoefficients}, 
\begin{equation}\label{lambdaQortho2}
(\varepsilon, L[(\varphi\Lambda Q)_{y_1}])=\left(\frac{\lambda_s}{\lambda}+b\right)(\Lambda Q, \varphi\Lambda Q)+\left(\frac{(x_1)_s}{\lambda}-1\right)(Q_{y_1},\varphi\Lambda Q)+c_1b^2+\mathcal{R}_{\Lambda},
\end{equation}
where $|\mathcal{R}_{\Lambda}|\lesssim |b|^3+\widetilde{\mathcal{M}}+|b|\mathcal{M}^{\frac{1}{2}}.$

We observe from the definition of $F,$ $\lim_{y_1\rightarrow -\infty}\rho(y_1,y_2)=0$ and 
$$\lim_{y_1\rightarrow +\infty}\rho(y_1,y_2)=-C_Q^{-1}\left(F(y_2)+\frac{c}{2}\int_{-\infty}^{+\infty}\Lambda Q(z,y_2)dz\right)=G(y_2),$$
with $|G(y_2)|\lesssim e^{-|y_2|/2},$ for all $y_2\in \R.$ Therefore, 
$$|J(s)|\lesssim \widehat{\mathcal{M}}^{\frac{1}{2}}(s).$$
Using the properties of $P$ and the parity of $Q_{y_1}, Q_{y_2}$ we have 
$$C_Q(Q_{y_1},\rho)=(Q_{y_1},P-F)+\frac{c}{2}(\Lambda Q, Q)+\frac{D_P}{\|\sqrt{\varphi}\Lambda Q\|_{L^2}^2}(\varphi\Lambda Q, Q_{y_1})=\frac{D_P}{\|\sqrt{\varphi}\Lambda Q\|_{L^2}^2}(\varphi\Lambda Q, Q_{y_1}),$$
$$C_Q(Q_{y_2},\rho)=(Q_{y_2},P-F)+\frac{c}{2}(\int_{-\infty}^{y_1}\Lambda Q, Q_{y_2})+\frac{D_P}{\|\sqrt{\varphi}\Lambda Q\|_{L^2}^2}(\varphi\Lambda Q, Q_{y_2})=0.$$
Moreover, using $\Lambda^*=-\Lambda$ and \eqref{otherdefinitionofc},
$$(\Lambda Q, \rho)=C_Q^{-1}[(P,\Lambda Q)+2cC_Q-cC_Q+(\Lambda P, Q)]=c.$$
Since $L(P_{y_1})=\Lambda Q+6PQQ_{y_1},$ we note that 
$$L(\rho_{y_1})=C_Q^{-1}\Big\{\Lambda Q+6PQQ_{y_1}+cQ+\frac{D_P}{\|\sqrt{\varphi}\Lambda Q\|_{L^2}^2}L[(\varphi \Lambda Q)_{y_1}]\Big\}.$$

The function $\rho_{y_1}$ satisfies
\eqref{eq:projection-test-hypothesis}. Since
$\rho(-\infty,y_2)=0$, its primitive is $\rho$, and applying
\eqref{equationprojectionf} with $f=\rho_{y_1}$ gives
\begin{equation}\label{Drho}
D_\rho=J_s+\frac{\lambda_s}{\lambda}J=-(\varepsilon, L(\rho_{y_1}))+\left(\frac{\lambda_s}{\lambda}+b\right)(\Lambda Q,\rho)+\left(\frac{(x_1)_s}{\lambda}-1\right)(Q_{y_1}, \rho)+c_\rho b^2+\mathcal{R}_\rho
\end{equation}
where $\mathcal{R}_\rho=O(|b|\widehat{\mathcal{M}}^{\frac{1}{2}}+\widehat{\mathcal{M}}+\widetilde{\mathcal{M}}+|b|^3)$ and we used that $(Q_{y_2}, \rho)=0.$ 
Plugging \eqref{lambdaQortho2} into \eqref{Drho}, we get 
\begin{equation*}
    D_\rho=-C_Q^{-1}(\varepsilon, \Lambda Q)-C_Q^{-1}6I_P+\left(\frac{\lambda_s}{\lambda}+b\right)(c-C_Q^{-1}D_P)+\tilde{c}b^2+\widetilde{\mathcal{R}}_{\rho}. 
\end{equation*}
with $\widetilde{\mathcal{R}}_{\rho}=O(|b|^3+|b|\widehat{\mathcal{M}}^{\frac{1}{2}}+\widehat{\mathcal{M}}+\widetilde{\mathcal{M}}).$
Therefore, using this equation and \eqref{brough2}, we have
\begin{equation*}
    \begin{split}
        \lambda^c\left(\frac{b}{\lambda^c}\right)_s+bD_{\rho}&=b_s+cb^2-cb\left(\frac{\lambda_s}{\lambda}+b\right)+bD_{\rho}\\
        &=-\hat{c}b^3-C_Q^{-1}\left(\frac{\lambda_s}{\lambda}+b\right)(\varepsilon,\Lambda Q)+\mathcal{R}_{b}^{\rho}
    \end{split}
\end{equation*}
where $\mathcal{R}_{b}^{\rho}=\mathcal{R}_b+b\widetilde{\mathcal{R}}_{\rho}=O(\widehat{\mathcal{M}}+\widetilde{\mathcal{M}}+b^4).$
From \eqref{eq:modulatedcoefficients} we have
$$\left|\left(\frac{\lambda_s}{\lambda}+b\right)(\varepsilon, \Lambda Q)\right|\lesssim |b|^4+\widetilde{\mathcal{M}}.$$
The conclusion \eqref{blambdasharpeq} follows. 
\end{proof}

\begin{remark}
\label{rem:transverse-loss}
The occurrence of $\widehat{\mathcal M}$ in \eqref{blambdasharpeq} is a genuinely two-dimensional feature. Although the correction $\rho$ is localized as $y_1\to-\infty$, as $y_1\to+\infty$ it approaches a transverse profile depending only on $y_2$. Consequently, $J=(\varepsilon,\rho)$ and the remainders in its evolution cannot be controlled solely by a norm localized near the soliton: on the right, one also needs a quadratic moment in $y_1$, together with transverse decay. By contrast, in the critical gKdV analysis of \cite{MartelMerleRaphael1}, the analogous correction belongs to the fully localized class $\mathcal Y$, and the finite-difference estimate for $b/\lambda^2$ is controlled by the first weighted radiation norm. Here the quadratic moment costs one higher weighted level, and its time integration costs one further level (see \eqref{eq:blambdacinequality}). This loss propagates through the bootstrap hierarchy: the basic rigidity argument closes with an $x_1^{18}$ moment, while the additional level required in the zero-energy argument leads to the $x_1^{21}$ localization imposed on the initial data.
\end{remark}

\section{Conservation Laws}
\begin{lemma}[Estimates Induced by Conservation Laws] \label{lemma:conservation_estimates}
The following estimates hold regarding the approximate profile $Q_b$ and the error $\varepsilon$:

\begin{itemize}
    \item[(i)] \textbf{Mass of the Profile:} The mass of the approximate profile satisfies the expansion:
    \begin{equation*}
    \left| \iint Q_b^2 - \iint Q^2 - 2b(P, Q) \right| \lesssim |b|^{2-\gamma}.
    \end{equation*}

    \item[(ii)] \textbf{Energy of the Profile:} The energy of the approximate profile expands as:
    \[
    E(Q_b) = -b(P, Q) + O(|b|^{2-\gamma}).
    \]

    \item[(iii)] \textbf{Mass Control of Error:} The $L^2$ norm of the error is controlled by the scaling parameter and the initial mass difference:
    \begin{equation}\label{eq:massconservation}
    \|\varepsilon\|_{L^2}^2 \lesssim |b| + \left| \iint u_0^2 - \iint Q^2 \right|.
    \end{equation}

    \item[(iv)] \textbf{Energy Conservation relation:} The kinetic energy of the error relates to the system's total energy $E_0$ and the scaling parameter via:
    \begin{equation}\label{eq:energyconservation}
    \left| 2\lambda^2 E_0 + 2b(P, Q) - \|\nabla \varepsilon\|_{L^2}^2 \right| \lesssim |b|^{\frac{3-\gamma}{2}} + \widetilde{\mathcal{M}} + \left( \|\varepsilon\|_{L^2}^2 + |b|^{\frac{1-\gamma}{2}} \right) \|\nabla \varepsilon\|_{L^2}^2.
    \end{equation}
\end{itemize}
\end{lemma}
\begin{proof}

Proof of $i):$  We compute directly 
\[
\begin{aligned}
\left|\iint Q_b^2-\iint Q^2-2b(P,Q)\right|
&\leq 2|b|\,\bigl|((1-\chi_b)P,Q)\bigr|
+b^2\iint\chi_b^2P^2\\
&\lesssim |b|^{2-\gamma}.
\end{aligned}
\]

Proof of $ii):$ Recall the definition of energy, $E(u)=\frac{1}{2}\iint |\nabla u|^2-\frac{1}{4}\iint u^4.$
$$E[(Q_b+\varepsilon)(s)]=E(\mathfrak u(s))=\frac{\lambda^2(s)}{2}\iint |\nabla u|^2-\frac{\lambda^2(s)}{4}\iint u^4=\lambda^2(s)E(u)=\lambda^2(s)E_0$$
Also, using that $E(Q)=0$ and $\Delta Q+Q^3=Q,$ we have 
$$E(Q_b)=E(Q)-b\iint \chi_bP\Delta Q -b\iint \chi_bPQ^3 +O(|b|^{2-\gamma})=-b(P,Q)+O(|b|^{2-\gamma}).$$

Proof of $iii):$ By the conservation of mass, 
$$\iint (Q_b+\varepsilon)^2=\iint u_0^2$$ 
so, using the orthogonality $(\varepsilon, Q)=0,$  
\begin{equation*}
\begin{split}
\iint u_0^2-\iint Q^2&=\iint Q_b^2-\iint Q^2+\iint \varepsilon^2+2(\varepsilon, Q_b)\\
&=\iint \varepsilon^2+b^2\iint \chi_b^2P^2+2b\iint \chi_bPQ+2b\iint \varepsilon \chi_bP
\end{split}
\end{equation*}
Since, 
$$\iint \chi_b^2P^2 \lesssim |b|^{-\gamma}, \text{  } \Big|\iint \chi_bPQ\Big|\lesssim 1, \text{  } \Big|\iint \varepsilon \chi_bP\Big|\lesssim |b|^{-\frac{\gamma}{2}} \|\varepsilon\|_{L^2}$$
we get 
$$\iint \varepsilon^2\lesssim |b|+|b|^{2-\gamma}+|b|^{1-\frac{\gamma}{2}} \|\varepsilon\|_{L^2}+\Big|\iint u_0^2-\iint Q^2\Big|$$
and, using $|b(t)|\leq \hat{\nu}\ll 1$ we get 
$$\|\varepsilon\|_{L^2}^2 \leq C|b|+C|b|^{2-\gamma}+\frac{1}{2}\|\varepsilon\|_{L^2}^2+\Big|\iint u_0^2-\iint Q^2\Big|$$
so 
$$\|\varepsilon\|_{L^2}^2\lesssim |b|+\Big|\iint u_0^2-\iint Q^2\Big|.$$

Proof of $iv):$ Using $ii),$ i.e. $E(Q_b)=-b(P,Q)+O(|b|^{2-\gamma})$ and the orthogonality condition $(\varepsilon,Q)=0,$ we get 
\begin{equation*}
\begin{aligned}
\lambda^2E_0
={}&E(Q_b)
+\frac12\iint|\nabla\varepsilon|^2
-\iint\varepsilon
 \bigl[\Delta(Q_b-Q)+(Q_b^3-Q^3)\bigr]\\
&-\frac14\iint
 \bigl[
 (Q_b+\varepsilon)^4-Q_b^4-4Q_b^3\varepsilon
 \bigr]\\
={}&-b(P,Q)+O(|b|^{2-\gamma})
+\frac12\iint|\nabla\varepsilon|^2
-b\iint\varepsilon\Delta(\chi_bP)
-\iint\varepsilon(Q_b^3-Q^3)\\
&-\frac14\iint
 \bigl[
 (Q_b+\varepsilon)^4-Q_b^4-4Q_b^3\varepsilon
 \bigr].
\end{aligned}
\end{equation*}
 We have
$$ \Big|\iint \varepsilon \Delta(\chi_bP)\Big|=\Big|\iint \nabla \varepsilon\cdot \nabla(\chi_b P)\Big|\lesssim \iint |\varepsilon_{y_1}(\chi_b)_{y_1}P|+\iint |\varepsilon_{y_1}\chi_bP_{y_1}|+\iint |\varepsilon_{y_2}\chi_bP_{y_2}|$$
We estimate each of the terms:
$$\iint |\varepsilon_{y_1}(\chi_b)_{y_1}P|\lesssim |b|^{\frac{\gamma}{2}}\|\varepsilon_{y_1}\|_{L^2}\mbox{,  } \iint |\varepsilon_{y_1}\chi_bP_{y_1}|\lesssim \widetilde{\mathcal{M}}^{\frac{1}{2}}\mbox{,  }\iint |\varepsilon_{y_2}\chi_bP_{y_2}|\lesssim |b|^{-\frac{\gamma}{2}}\|\varepsilon_{y_2}\|_{L^2},$$
which yields
$$\Big|\iint \varepsilon \Delta(b\chi_bP)\Big|\lesssim |b|\widetilde{\mathcal{M}}^{\frac{1}{2}}+|b|^{1+\frac{\gamma}{2}}\|\varepsilon_{y_1}\|_{L^2}+|b|^{1-\frac{\gamma}{2}}\|\varepsilon_{y_2}\|_{L^2}\lesssim |b|^{\frac{3-\gamma}{2}}+\widetilde{\mathcal{M}}+ |b|^{\frac{1-\gamma}{2}}\|\nabla\varepsilon\|_{L^2}^{2}.$$

The other term is 
$$\iint |\varepsilon(Q_b^3-Q^3)|\lesssim |b|\iint |\varepsilon Q^2\chi_bP|+b^2\iint |\varepsilon Q(\chi_bP)^2|+|b|^3\iint |\varepsilon (\chi_bP)^3|$$ 
We estimate each of the terms:
$$|b|\iint |\varepsilon Q^2\chi_bP|\lesssim |b|\widetilde{\mathcal{M}}^{\frac{1}{2}}, b^2\iint |\varepsilon Q(\chi_bP)^2|\lesssim  b^2\widetilde{\mathcal{M}}^{\frac{1}{2}}, |b|^3\iint |\varepsilon (\chi_bP)^3|\lesssim |b|^{3-\gamma}\|\varepsilon\|_{L^2},$$
which yields
$$\iint |\varepsilon(Q_b^3-Q^3)|\lesssim |b|\widetilde{\mathcal{M}}^{\frac{1}{2}}+|b|^{3-\gamma}\|\varepsilon\|_{L^2}.$$

For the last one, from Gagliardo-Nirenberg inequality and Lemma ~\ref{Qdecay}, 
$$\iint |(Q_b+\varepsilon)^4-Q_b^4-4Q_b^3\varepsilon|\lesssim \iint Q^2\varepsilon^2+b^2\iint \varepsilon^2(\chi_bP)^2+\iint \varepsilon^4$$
$$\lesssim \widetilde{\mathcal{M}}+\|\varepsilon\|_{L^2}^2\|\nabla\varepsilon\|_{L^2}^2+b^2\|\varepsilon\|_{L^2}^{2}$$
Combining the profile, interaction, and nonlinear estimates gives
$$\Big|\lambda^2E_0+b(P,Q)-\frac{1}{2}\iint |\nabla \varepsilon|^2\Big|\lesssim |b|^{\frac{3-\gamma}{2}}+\widetilde{\mathcal{M}}+(\| \varepsilon\|_{L^2}^2+|b|^{\frac{1-\gamma}{2}})\|\nabla \varepsilon\|_{L^2}^2.$$
\end{proof}

\chapter{Energy Estimates} \label{Monotonicity Formulas}

\section{Choice of Weights}

Recall the definition of $\alpha_1$ from \eqref{alphas}.
Let \(\widetilde\delta>0\) be the constant in the global coercivity estimate
\eqref{eq:coercivitylemma}.  All implicit constants below are independent of
\(B\), although they may depend on \(i,\alpha_1,\sigma,M\), and on the fixed
cut-off functions.

Choose
\begin{equation}\label{sigmasmall}
    0<\sigma\leq
    \min\Big(\frac{\sqrt{\widetilde{\delta}}}{2},
              \frac{1}{2\alpha_1}\Big),
\end{equation}
and then choose
\begin{equation*}
    \rho\in\bigl(\alpha_1^{-2},1\bigr),
    \qquad
    \kappa=1-\rho>0.
\end{equation*}
Finally choose \(M\gg1\), independently of \(B\), such that
\begin{equation}\label{eq:Mlarge-left-weights}
        \sigma\left(M-\frac12\right)>\frac{1}{\alpha_1}.
\end{equation}

For \(B\ge B_0\), set
\[
        C_B=1+e^{-B/(2\alpha_1)},
        \qquad
        E_B(y_1)=e^{-B/(2\alpha_1)}e^{y_1/\alpha_1},
\]
and define
\[
        S_B(y_1)=\frac{e^{\sigma(y_1+MB)}}
                       {1+e^{\sigma(y_1+MB)}}.
\]
Let \(\varrho\in C^\infty(\mathbb R)\) be non-decreasing, with
\begin{equation*}
        0\leq \varrho\leq1,
        \qquad
        \varrho(r)=0\quad\text{for }r\leq0,
        \qquad
        \varrho(r)=1\quad\text{for }r\geq1,
\end{equation*}
and with all derivatives bounded.  Put
\[
        \varsigma_B(y_1)=\varrho\left(y_1+\frac B2+1\right),
        \qquad
        \widetilde S_B=(1-\varsigma_B)S_B+\varsigma_B.
\]
Then
\[
        \widetilde S_B=S_B\quad\text{for }y_1\leq-\frac B2-1,
        \qquad
        \widetilde S_B=1\quad\text{for }y_1\geq-\frac B2.
\]
On the region where \(\widetilde S_B=S_B\),
\begin{equation*}
        S'_B=\sigma S_B(1-S_B),
\end{equation*}
and
\begin{equation*}
        S''_B=\sigma^2S_B(1-S_B)(1-2S_B),
        \qquad
        S'''_B=\sigma^3S_B(1-S_B)(1-6S_B+6S_B^2).
\end{equation*}
Since \(1-6S+6S^2\le1\) for \(0\le S\le1\), we have
\begin{equation*}
        S'''_B\leq \sigma^2 S'_B.
\end{equation*}
Moreover,
\begin{equation}\label{eq:S-absolute-derivative-bounds}
        |S''_B|+|S'''_B|\lesssim_\sigma S'_B.
\end{equation}

Let
\[
        \mathcal R_B=\left[-\frac B2-1,-\frac B2\right].
\]
By \eqref{eq:Mlarge-left-weights}, on \(\mathcal R_B\),
\begin{equation*}
        1-S_B+|S_B'|+|S_B''|+|S_B'''|
        \lesssim e^{-\sigma(M-1/2)B}
        \ll e^{-B/\alpha_1}.
\end{equation*}
Consequently, using the boundedness of the derivatives of \(\varsigma_B\),
\begin{equation}\label{eq:tildeS-collar-smallness}
        1-\widetilde S_B+|\widetilde S_B'|+|\widetilde S_B''|
        +|\widetilde S_B'''|
        \lesssim e^{-\sigma(M-1/2)B}
        \ll e^{-B/\alpha_1}
        \quad\text{on }\mathcal R_B.
\end{equation}

\subsection*{Weight for the Mass Part}

For \(y_1\le B/2\), define
\begin{equation}\label{eq:new-left-phi-def}
        \phi_{i,B}(y_1)=E_B(y_1)+C_B\widetilde S_B(y_1).
\end{equation}
Since \(\widetilde S_B=1\) on \([-B/2,B/2]\), recalling the definition of $\varphi$ from \eqref{definitionofphi}, this gives
\begin{equation*}
        \phi_{i,B}(y_1)
        =1+\varphi(y_1)e^{-B/(2\alpha_1)}
        \quad\text{for } |y_1|\leq \frac B2.
\end{equation*}

For \(i\ge1\), the full mass weight is
\begin{equation}\label{massweight}
\phi_{i,B}(y_1)=
\begin{cases}
    E_B(y_1)+C_B\widetilde S_B(y_1),
        & y_1\leq \frac B2,\\[1mm]
    \text{the interpolation of Lemma \ref{lem:right_transition_clean_corrected}},
        & \frac B2\leq y_1\leq B,\\[1mm]
    B\left(\dfrac{y_1}{B}\right)^i,
        & y_1\geq B.
\end{cases}
\end{equation}

\begin{lemma}\label{lem:right_transition_clean_corrected}
Let
\[
F_B(y_1)=1+\varphi(y_1)e^{-B/(2\alpha_1)}
      =1+e^{-B/(2\alpha_1)}+e^{(y_1-B/2)/\alpha_1},
\]
and let
\[
P_{i,B}(y_1)=B\left(\frac{y_1}{B}\right)^i .
\]
For \(B\) sufficiently large, there exists a smooth interpolation
\(\phi_{i,B}\) on \([B/2,B]\) such that
\[
\phi_{i,B}=F_B
\quad\text{near } B/2,
\qquad
\phi_{i,B}=P_{i,B}
\quad\text{near } B.
\]
Moreover, on \([B/2,B]\),
\begin{equation}\label{lowerboundderivativetransition}
0<c_i\leq \phi_{i,B}'\leq C_i=\max\{5/2,i\},
\end{equation}
\begin{equation}\label{Sobolevboundrighttransition}
|\phi_{i,B}''|\leq \widetilde{C}_i\phi_{i,B}',
\end{equation}
and
\begin{equation}\label{phithirdderivativerighttransition}
\phi_{i,B}'''
\leq
(1-\kappa)\phi_{i,B}'.
\end{equation}
The constants \(c_i,C_i,\widetilde{C}_i>0\) are independent of \(B\).
\end{lemma}

\begin{proof}
Choose \(r_*\in(2/(3\alpha_1),1)\).  Let
\(\eta\in C^\infty([0,\infty))\) satisfy
\[
0\leq \eta\leq1,\qquad \eta'\geq0,\qquad \eta''\leq0,
\]
\[
\eta(r)=r\quad(0\le r\le r_*),
\qquad
\eta(r)=1\quad(r\ge R),
\]
for some fixed \(R>1\).  Since \(\eta\) is concave and \(\eta(0)=0\),
\[
        r\eta'(r)\le \eta(r),\qquad r\ge0.
\]

For \(H\in[3/2,5/2]\), define
\[
        q_{L,H}(y_1)
        =
        H\,\eta\left(\frac{F_B'(y_1)}{H}\right).
\]
Since \(F_B'(B/2)=1/\alpha_1\), we have
\[
        \frac{F_B'(B/2)}{H}
        \le \frac{2}{3\alpha_1}<r_*.
\]
Thus \(q_{L,H}=F_B'\) near \(B/2\).  Also \(q_{L,H}=H\) once
\(F_B'\ge RH\), so there exists \(K_L>0\), independent of \(B\) and \(H\),
such that \(q_{L,H}=H\) for \(y_1\ge B/2+K_L\).  If
\(r=F_B'/H\), then
\[
        q_{L,H}'=\frac{H}{\alpha_1}r\eta'(r),
        \qquad
        q_{L,H}''
        =
        \frac{H}{\alpha_1^2}
        \bigl(r^2\eta''(r)+r\eta'(r)\bigr).
\]
Therefore
\[
        |q_{L,H}'|\lesssim q_{L,H},
        \qquad
        q_{L,H}''
        \le \alpha_1^{-2}q_{L,H}
        \le \rho q_{L,H}.
\]
Moreover \(H\ge q_{L,H}\ge q_{L,H}(B/2)=1/\alpha_1\).

Next choose a fixed \(\zeta\in C^\infty(\mathbb R)\) such that
\[
0\le\zeta\le1,\qquad
\zeta=0\ \text{for }s\le0,\qquad
\zeta=1\ \text{for }s\ge1.
\]
Let \(\ell\gg1\) be fixed, independently of \(B\), and for
\(y_1\in[B-\ell,B]\) put
\[
        s=\frac{y_1-(B-\ell)}{\ell},
        \qquad
        q_{R,H}(y_1)
        =
        (1-\zeta(s))H+\zeta(s)P_{i,B}'(y_1).
\]
Then \(q_{R,H}=H\) near \(B-\ell\), while \(q_{R,H}=P_{i,B}'\) near \(B\).
For \(B\ge2\ell\),
\[
        P_{i,B}'(y_1)=i\left(\frac{y_1}{B}\right)^{i-1}
        \ge i2^{1-i},
        \qquad B-\ell\le y_1\le B.
\]
Hence \(\max\{H,i\}\ge q_{R,H}\ge \min\{H, i2^{1-i}\}>0\).  Since
\[
        |P_{i,B}'|\lesssim_i1,\qquad
        |P_{i,B}''|\lesssim_i B^{-1},
        \qquad
        |P_{i,B}'''|\lesssim_i B^{-2},
\]
we get
\[
        |q_{R,H}'|\le \widetilde{C}_i q_{R,H}.
\]
Furthermore,
\[
        q_{R,H}''
        =
        \frac{\zeta''(s)}{\ell^2}(P_{i,B}'-H)
        +2\frac{\zeta'(s)}{\ell}P_{i,B}''
        +\zeta(s)P_{i,B}'''.
\]
Choosing \(\ell\) sufficiently large and then \(B\) sufficiently large gives
\[
        q_{R,H}''\le \rho q_{R,H}.
\]

For \(B\) large enough that \(B/2+K_L<B-\ell\), define
\[
q_H(y_1)=
\begin{cases}
q_{L,H}(y_1), & B/2\le y_1\le B/2+K_L,\\[1mm]
H, & B/2+K_L\le y_1\le B-\ell,\\[1mm]
q_{R,H}(y_1), & B-\ell\le y_1\le B.
\end{cases}
\]
Then \(q_H\) is smooth and satisfies
\[
        c_i \le q_H \le C_i=\max\{H,i\},\qquad |q_H'|\le \widetilde{C}_iq_H,\qquad q_H''\le \rho q_H.
\]
Let
\[
        I_B(H)=\int_{B/2}^B q_H(y_1)\,dy_1.
\]
Uniformly for \(H\in[3/2,5/2]\),
\[
        I_B(H)=H\left(\frac B2-K_L-\ell\right)+O_i(1).
\]
Hence
\[
        I_B(3/2)=\frac{3B}{4}+O_i(1),
        \qquad
        I_B(5/2)=\frac{5B}{4}+O_i(1).
\]
On the other hand,
\[
        P_{i,B}(B)-F_B(B/2)
        =
        B-\left(2+e^{-B/(2\alpha_1)}\right)
        =
        B+O(1).
\]
Thus, for \(B\) large, the intermediate value theorem gives
\(H_B\in[3/2,5/2]\) such that
\[
        \int_{B/2}^B q_{H_B}(y_1)\,dy_1
        =
        P_{i,B}(B)-F_B(B/2).
\]
Define
\[
        \phi_{i,B}(y_1)
        =
        F_B(B/2)+\int_{B/2}^{y_1}q_{H_B}(z)\,dz,
        \qquad B/2\le y_1\le B.
\]
Then \(\phi_{i,B}=F_B\) at \(B/2\), \(\phi_{i,B}=P_{i,B}\) at \(B\), and
\[
        \phi_{i,B}'=q_{H_B},\qquad
        \phi_{i,B}''=q_{H_B}',\qquad
        \phi_{i,B}'''=q_{H_B}''.
\]
This proves \eqref{lowerboundderivativetransition},
\eqref{Sobolevboundrighttransition}, and
\eqref{phithirdderivativerighttransition}. 
\end{proof}

\begin{lemma}
For \(B\geq B_0\), the weight defined by \eqref{eq:new-left-phi-def}
satisfies
\begin{equation}\label{eq:phi-positive-increasing}
        \phi_{i,B}>0,
        \qquad
        \phi_{i,B}'>0
        \quad\text{for }y_1\leq B/2.
\end{equation}
Moreover,
\begin{equation}\label{eq:phi-third-one-sided-new}
        \phi_{i,B}'''
        \leq (1-\kappa)\phi_{i,B}'
        \quad\text{for }y_1\leq B/2,
\end{equation}
and
\begin{equation}\label{Sobolevconditionphi2}
        |\phi_{i,B}''|+|\phi_{i,B}'''|
        \lesssim \phi_{i,B}'
        \quad\text{for }y_1\leq B/2.
\end{equation}
Define
\begin{equation}
        A_B(y_1)=1+\varphi(y_1)e^{-B/(2\alpha_1)}-\phi_{i,B}(y_1).
\end{equation}
Then
\begin{equation}\label{centralestimatetotheleft}
\frac{A_B^2(y_1)}{\phi_{i,B}'(y_1)}
\lesssim
\begin{cases}
     e^{-\sigma(y_1+MB)},
        & y_1<-\frac B2-1,\\[1mm]
     e^{-2\sigma(M-\frac12)B}e^{B/\alpha_1},
        & -\frac B2-1\leq y_1\leq-\frac B2,\\[1mm]
     0,
        & -\frac B2<y_1<\frac B2.
\end{cases}
\end{equation}
For all sufficiently large \(B\) and \(y_1<B/2\),
\begin{equation}\label{phiboundderivativephi}
    \phi_{i,B}(y_1)
    \lesssim
    e^{\varpi B}\phi_{i,B}'(y_1),
    \qquad
    \varpi=\frac{\sigma(M+\frac12)}{1+\sigma\alpha_1}.
\end{equation}
Consequently,
\begin{equation}\label{phiderivativephi}
    \phi_{i,B}^2(y_1)
    \lesssim e^{\varpi B}\phi_{i,B}'(y_1),
    \qquad y_1<-B/2.
\end{equation}
Finally, the following spatially localized estimate also holds:
\begin{equation}\label{sharper}
  \phi_{i,B}^2(y_1)
\lesssim
e^{\sigma M B}e^{\sigma y_1}\phi_{i,B}'(y_1),
\qquad
y_1<-B/2.
\end{equation}
\end{lemma}

\begin{proof}
We have
\[
        \phi_{i,B}'=E_B'+C_B\widetilde S_B',
\]
where \(E_B'>0\) and \(\widetilde S_B'\ge0\).  Thus
\eqref{eq:phi-positive-increasing} holds.

On \(y_1\le -B/2-1\), \(\widetilde S_B=S_B\), and hence
\[
        \phi_{i,B}'=E_B'+C_BS_B',
        \qquad
        \phi_{i,B}'''=E_B'''+C_BS_B'''.
\]
Since \(E_B'''=\alpha_1^{-2}E_B'\) and
\(S_B'''\le \sigma^2S_B'\), while
\(\sigma^2<\rho\) and \(\alpha_1^{-2}<\rho\), we obtain
\[
        \phi_{i,B}'''
        \le
        \rho(E_B'+C_BS_B')
        =
        \rho\phi_{i,B}'.
\]
The estimate \eqref{Sobolevconditionphi2} follows there from
\eqref{eq:S-absolute-derivative-bounds} and the corresponding identities for
\(E_B\).

On \([-B/2,B/2]\),
\[
        \phi_{i,B}=1+\varphi e^{-B/(2\alpha_1)},
\]
and therefore
\[
        \phi_{i,B}'''
        =
        \alpha_1^{-2}\phi_{i,B}'
        \le
        \rho\phi_{i,B}',
\]
with the same Sobolev-type estimate.  On the collar \(\mathcal R_B\), we have
\(E_B'\gtrsim e^{-B/\alpha_1}\), whereas
\eqref{eq:tildeS-collar-smallness} gives
\[
        C_B\bigl(|\widetilde S_B''|+|\widetilde S_B'''|\bigr)
        \ll E_B'.
\]
Thus, after increasing \(B_0\),
\[
        C_B|\widetilde S_B'''|
        \le
        (\rho-\alpha_1^{-2})E_B',
\]
which proves \eqref{eq:phi-third-one-sided-new} on \(\mathcal R_B\), and
\eqref{Sobolevconditionphi2} follows similarly.

We next prove \eqref{centralestimatetotheleft}.  On \(y_1\le -B/2-1\),
\[
        A_B=C_B(1-S_B),
\]
so, using \(S_B'=\sigma S_B(1-S_B)\),
\begin{equation*}
        \frac{A_B(y_1)^2}{\phi_{i,B}'(y_1)}
        \lesssim
        \frac{(1-S_B)^2}{S_B(1-S_B)}
        =
        e^{-\sigma(y_1+MB)}.
\end{equation*}
On \(\mathcal R_B\), \(A_B=C_B(1-\widetilde S_B)\), and therefore
\begin{equation*}
        \frac{A_B(y_1)^2}{\phi_{i,B}'(y_1)}
        \lesssim
        e^{-2\sigma(M-1/2)B}e^{B/\alpha_1}.
\end{equation*}
Finally, \(A_B=0\) on \((-B/2,B/2)\).

It remains to prove \eqref{phiboundderivativephi}.  On
\([-B/2-1,B/2]\), \(\phi_{i,B}\lesssim1\), while
\[
        \phi_{i,B}'\ge E_B'
        \gtrsim e^{-B/\alpha_1}.
\]
Since \eqref{eq:Mlarge-left-weights} implies
\[
        \varpi=\frac{\sigma(M+\frac12)}{1+\sigma\alpha_1}
        >
        \frac1{\alpha_1},
\]
we get \(\phi_{i,B}\lesssim e^{\varpi B}\phi_{i,B}'\) on this region.

On \(y_1\le -B/2-1\), write \(z=y_1+MB\).  Then
\begin{equation*}
    \phi_{i,B}'
    =
    E_B'+C_BS_B'
    \gtrsim
    e^{z/\alpha_1-(M+1/2)B/\alpha_1}
    +
    S_B(1-S_B).
\end{equation*}
If \(S_B\le 1/2\), then \(S_B\lesssim S_B'\lesssim\phi_{i,B}'\), and
also \(E_B\lesssim E_B'\lesssim\phi_{i,B}'\).  Hence
\eqref{phiboundderivativephi} holds without any exponential loss.

If \(S_B\ge1/2\), then \(z\ge0\) and \(S_B(1-S_B)\gtrsim e^{-\sigma z}\).
Thus
\[
    \phi_{i,B}'
    \gtrsim
    e^{z/\alpha_1-(M+1/2)B/\alpha_1}
    +
    e^{-\sigma z}.
\]
For all \(z\ge0\),
\[
    \max\left\{
    e^{z/\alpha_1-(M+1/2)B/\alpha_1},
    e^{-\sigma z}
    \right\}
    \ge
    \exp\left(
    -\frac{\sigma(M+\frac12)}{1+\alpha_1\sigma}B
    \right).
\]
Since \(\phi_{i,B}\lesssim1\) on \(y_1<B/2\), this proves
\eqref{phiboundderivativephi}.  The estimate \eqref{phiderivativephi} follows
because \(\phi_{i,B}\lesssim1\) for \(y_1<-B/2\).

Finally, on \(y_1<-B/2\),
\[
        \phi_{i,B}^2\lesssim E_B^2+\widetilde S_B^2.
\]
Since
\[
        \frac{E_B^2}{e^{\sigma y_1}E_B'}
        =
        \alpha_1 E_Be^{-\sigma y_1}
        =
        \alpha_1e^{-B/(2\alpha_1)}
        e^{(1/\alpha_1-\sigma)y_1}
        \lesssim1,
\]
we have
\[
        E_B^2
        \lesssim
        e^{\sigma y_1}\phi_{i,B}'.
\]
On \(y_1\le -B/2-1\),
\[
        \frac{S_B^2}{e^{\sigma y_1}S_B'}
        =
        \frac1\sigma e^{\sigma MB},
\]
so
\[
        S_B^2
        \lesssim
        e^{\sigma MB}e^{\sigma y_1}\phi_{i,B}'.
\]
On the collar \(\mathcal R_B\), \(\widetilde S_B\le1\) and
\[
        e^{\sigma MB}e^{\sigma y_1}\phi_{i,B}'
        \gtrsim
        e^{\sigma(M-\frac12)B}e^{-B/\alpha_1}
        \gtrsim1
\]
by \eqref{eq:Mlarge-left-weights}.  This proves \eqref{sharper}.
\end{proof}

\begin{lemma}\label{NcomparedtotildeN}
For all \(y_1\in\mathbb R\),
\[
        \phi_{i,B}'\lesssim \phi_{i,B},
        \qquad
        \phi_{i,B}'
        \lesssim
        \phi_{i+1,B}'.
\]
Consequently,
\[
        \phi_{i,B}\lesssim \phi_{i+1,B}.
\]
Moreover,
\begin{equation}\label{Sobolevconditionphi}
    |\phi_{i,B}''|
    \lesssim
    \phi_{i,B}'
    \qquad\text{for all }y_1\in\mathbb R.
\end{equation}
\end{lemma}

\begin{proof}
On \(y_1<B/2\), the first estimate follows from
\(|E_B'|\lesssim E_B\) and \(|\widetilde S_B'|\lesssim \widetilde S_B\).
On \(B/2\le y_1\le B\), Lemma
\ref{lem:right_transition_clean_corrected} gives
\(\phi_{i,B}'\le C_i\), while \(\phi_{i,B}\ge2\).  On \(y_1\ge B\),
\[
        \frac{\phi_{i,B}'}{\phi_{i,B}}
        =
        \frac{i}{y_1}
        \le
        \frac{i}{B}.
\]
This proves \(\phi_{i,B}'\lesssim\phi_{i,B}\).

The weights \(\phi_{i,B}\) and \(\phi_{i+1,B}\) agree on \(y_1<B/2\), hence
their derivatives agree there.  On \(B/2\le y_1\le B\), the derivative upper
and lower bounds in Lemma \ref{lem:right_transition_clean_corrected} give
\[
        \phi_{i,B}'
        \le
        \frac{C_i}{c_{i+1}}        
        \phi_{i+1,B}'.
\]
On \(y_1\ge B\),
\[
        \frac{\phi_{i,B}'}
             {\phi_{i+1,B}'}
        =
        \frac{i}{i+1}\frac{B}{y_1}
        \le1.
\]
Integrating the derivative comparison from \(B/2\), where the two weights have
the same value up to harmless constants, gives
\(\phi_{i,B}\lesssim\phi_{i+1,B}\).  Finally,
\eqref{Sobolevconditionphi} follows from \eqref{Sobolevconditionphi2},
\eqref{Sobolevboundrighttransition}, and the explicit polynomial expression on
\(y_1\ge B\).
\end{proof}

\subsection*{Weight for the Gradient Part}

Define
\begin{equation}\label{eq:new-psi-def}
        \psi_B(y_1)=\widetilde S_B(y_1)^2.
\end{equation}
Then \(0<\psi_B\le1\), \(\psi_B'\ge0\), and
\[
        \psi_B=1\quad\text{for }y_1\ge-\frac B2.
\]

\begin{lemma}
The gradient weight satisfies
\begin{equation}\label{psiandphi}
0<\psi_B\leq \phi_{i,B},
\qquad
0\leq \psi_B'\lesssim\phi_{i,B}'
\qquad\text{for all }y_1\in\mathbb R.
\end{equation}
Define
\[
        R_{\psi_B}
        =
        \frac{2\psi_B''\psi_B-[\psi_B']^2}
             {4\psi_B}.
\]
Then
\begin{equation}\label{RpsiB}
|R_{\psi_B}|\leq \sigma^2\psi_B,
\qquad y_1\leq-\frac B2.
\end{equation}
Moreover,
\begin{equation}\label{1minuspsi}
    \frac{(1-\psi_B)^2}{\phi_{i,B}'}
    \lesssim
    \begin{cases}
         e^{-\sigma(y_1+MB)},
            & y_1<-\frac B2-1,\\[1mm]
         e^{-2\sigma(M-\frac12)B}e^{B/\alpha_1},
            & -\frac B2-1\leq y_1\leq-\frac B2,
    \end{cases}
\end{equation}
and
\begin{equation}\label{psiderivativesquared}
    \frac{[\psi_B']^2}{\phi_{i,B}'}
    \lesssim
    \begin{cases}
         1,
            & y_1<-\frac B2-1,\\[1mm]
         e^{-2\sigma(M-\frac12)B}e^{B/\alpha_1},
            & -\frac B2-1\leq y_1\leq-\frac B2.
    \end{cases}
\end{equation}
There exists \(\mathcal C_B>0\) such that
\begin{equation}\label{psiglaeser}
|\psi_B''|^2
\lesssim
\mathcal{C}_B\psi_B'\phi_{i,B}',
\qquad y_1\leq -B/2.
\end{equation}
Also,
\begin{equation}\label{eq:psi-by-phi-prime}
\psi_B(y_1)
\lesssim
e^{\varpi B}\phi_{i,B}'(y_1)
\qquad\text{for all }y_1\in\mathbb R,
\end{equation}
where
\[
        \varpi=\frac{\sigma(M+\frac12)}{1+\sigma\alpha_1}.
\]
On \(y_1<-B/2\),
\begin{equation}\label{psithirdderivative}
    |\psi_B'''|
    \leq
    \frac12\phi_{i,B}'+\psi_B'.
\end{equation}
Finally,
\begin{equation}\label{weightcontroldrift}
    |y_1|\psi_B'
    \lesssim
    B\phi_{i,B}'
    \qquad\text{for all }y_1\in\mathbb R.
\end{equation}
\end{lemma}

\begin{proof}
Since \(0<\widetilde S_B\le1\), we have \(0<\psi_B\le1\).  The estimate
\(\psi_B\le\phi_{i,B}\) follows from the definition on \(y_1<B/2\), from
\(\phi_{i,B}\ge1\) on the central and right regions, and from the monotonicity
of the right interpolation.  The estimate
\(\psi_B'\lesssim\phi_{i,B}'\) is immediate on
\(y_1\le-B/2-1\) from
\[
        \psi_B'=2S_BS_B',
        \qquad
        \phi_{i,B}'\ge C_BS_B',
\]
while on the collar it follows from \eqref{eq:tildeS-collar-smallness} and
\(\phi_{i,B}'\gtrsim e^{-B/\alpha_1}\).  For \(y_1\ge-B/2\),
\(\psi_B'=0\).  This proves \eqref{psiandphi}.

Since
\[
        \psi_B'=2\widetilde S_B\widetilde S_B',
        \qquad
        \psi_B''
        =
        2(\widetilde S_B')^2+2\widetilde S_B\widetilde S_B'',
\]
we have the identity
\begin{equation*}
        R_{\psi_B}=\widetilde S_B\widetilde S_B''.
\end{equation*}
On \(y_1\le-B/2-1\), this becomes
\begin{equation*}
        R_{\psi_B}
        =
        \sigma^2S_B^2(1-S_B)(1-2S_B).
\end{equation*}
Since
\[
        -\frac18
        \le
        (1-S)(1-2S)
        \le
        1
        \qquad(0\le S\le1),
\]
we get
\begin{equation*}
        -\frac{\sigma^2}{8}\psi_B
        \leq
        R_{\psi_B}
        \leq
        \sigma^2\psi_B
        \quad\text{on }y_1\leq -B/2-1.
\end{equation*}
On the collar,
\begin{equation*}
        |R_{\psi_B}|
        \lesssim e^{-\sigma(M-1/2)B}\psi_B.
\end{equation*}
Thus, after increasing \(B_0\),
\begin{equation*}
        |R_{\psi_B}|
        \leq
        \sigma^2\psi_B
        \quad\text{for }y_1\le -B/2,
\end{equation*}
which proves \eqref{RpsiB}.

On \(y_1\le-B/2-1\),
\[
        1-\psi_B=(1-S_B)(1+S_B),
        \qquad
        \psi_B'=2S_BS_B'.
\]
Using \(\phi_{i,B}'\ge C_BS_B'\), we obtain
\[
        \frac{(1-\psi_B)^2}{\phi_{i,B}'}
        \lesssim
        \frac{1-S_B}{S_B}
        =
        e^{-\sigma(y_1+MB)},
\]
and
\[
        \frac{[\psi_B']^2}{\phi_{i,B}'}
        \lesssim
        S_B^2S_B'
        \lesssim1.
\]
On \(\mathcal R_B\), \eqref{eq:tildeS-collar-smallness} and
\(\phi_{i,B}'\gtrsim e^{-B/\alpha_1}\) give the collar bounds in
\eqref{1minuspsi} and \eqref{psiderivativesquared}.

We next prove \eqref{psiglaeser}.  On \(y_1\le-B/2-1\),
\[
        \psi_B'
        =
        2\sigma S_B^2(1-S_B),
\]
and
\[
        \psi_B''
        =
        2\sigma^2S_B^2(1-S_B)(2-3S_B).
\]
Since \(\phi_{i,B}'\gtrsim S_B(1-S_B)\), we get
\[
        \frac{|\psi_B''|^2}
             {\psi_B'\phi_{i,B}'}
        \lesssim
        S_B(2-3S_B)^2
        \lesssim1.
\]
On the collar, Glaeser's inequality applied to
\(F=\psi_B'\ge0\) gives
\[
        |\psi_B''|^2
        \le C(B)\psi_B'.
\]
Since \(\phi_{i,B}'\gtrsim e^{-B/\alpha_1}\) there, this proves
\eqref{psiglaeser} for a suitable \(\mathcal C_B>0\).

The estimate \eqref{eq:psi-by-phi-prime} follows on \(y_1<B/2\) from
\(\psi_B\le\phi_{i,B}\) and \eqref{phiboundderivativephi}.  For
\(B/2\le y_1\le B\), \(\psi_B=1\) and
\(\phi_{i,B}'\ge c_i\).  For \(y_1\ge B\),
\[
        \phi_{i,B}'
        =
        i\left(\frac{y_1}{B}\right)^{i-1}
        \ge i.
\]
Hence \eqref{eq:psi-by-phi-prime} holds globally.

For \eqref{psithirdderivative}, on \(y_1<-B/2-1\) a direct computation gives
\[
        \frac{|\psi_B'''|}{\psi_B'}
        =
        \sigma^2|4-15S_B+12S_B^2|
        \le 4\sigma^2.
\]
Since \(\sigma\le1/2\), this gives
\[
        |\psi_B'''|
        \le
        \psi_B'.
\]
On \(\mathcal R_B\),
\[
        |\psi_B'''|
        \lesssim
        |\widetilde S_B'||\widetilde S_B''|
        +
        |\widetilde S_B'''|
        \ll
        e^{-B/\alpha_1}
        \lesssim
        \phi_{i,B}',
\]
and the factor \(1/2\) follows by increasing \(B_0\).

Finally, on \(y_1\le-B/2-1\),
\[
        \frac{|y_1|\psi_B'}{\phi_{i,B}'}
        \lesssim |y_1|S_B.
\]
If \(-MB\le y_1\le-B/2-1\), then \(|y_1|S_B\lesssim MB\).  If
\(y_1<-MB\), write \(x=-y_1-MB\).  Since \(S_B\lesssim e^{-\sigma x}\),
\[
        |y_1|S_B
        \lesssim
        (MB+x)e^{-\sigma x}
        \lesssim
        B.
\]
On \(\mathcal R_B\), \(|y_1|\lesssim B\),
\(\psi_B'\lesssim e^{-\sigma(M-1/2)B}\), and
\(\phi_{i,B}'\gtrsim e^{-B/\alpha_1}\), so
\[
        |y_1|\psi_B'
        \lesssim
        B\phi_{i,B}'
\]
by \eqref{eq:Mlarge-left-weights}.  For \(y_1\ge-B/2\),
\(\psi_B'=0\).  This proves \eqref{weightcontroldrift}.
\end{proof}

\section{Defining the Norms}

We introduce a set of weighted norms to track the evolution of the error $\varepsilon$:
\begin{align*}
\mathcal{N}_i(s) &= \iint \left[ (\varepsilon_{y_1}^2 + \varepsilon_{y_2}^2)\psi_B + \varepsilon^2\phi_{i,B} \right], \\
\mathcal{N}_{i,loc}(s) &= \iint \varepsilon^2\phi_{i,B}', \\
\widetilde{\mathcal{N}}_{i}(s) &= \iint (|\nabla\varepsilon|^2 + \varepsilon^2)\phi_{i,B}'.
\end{align*}

We have the bounds as a consequence of Lemma \ref{NcomparedtotildeN}:
\[
\mathcal{N}_{i,loc}(s)\lesssim \mathcal{N}_{i+1,loc}(s), \quad \widetilde{\mathcal{N}}_{i}(s)\lesssim \widetilde{\mathcal{N}}_{i+1}(s), \quad \mathcal{N}_{i}(s)\lesssim \mathcal{N}_{i+1}(s).
\]

We also have the important bound 
\begin{equation}\label{importantbound}
    \mathcal{N}_i(s)\lesssim e^{\varpi B}\widetilde{\mathcal{N}}_{i+1}(s).
\end{equation}
Indeed, for the interval $y_1<\frac{B}{2},$ this is a consequence of \eqref{phiboundderivativephi} and \eqref{eq:psi-by-phi-prime}. On the right transition interval, we use \eqref{lowerboundderivativetransition}, while on $y_1>B$ we have $\phi_{i,B}= \frac{B}{i+1}\phi_{i+1,B}'.$ The conclusion follows.

\begin{lemma}For all sufficiently large fixed \(B\) and $i\geq1,$
\begin{equation}\label{McomparedtoN}
\begin{split}
    \mathcal M(s)
    &\leq C e^{B/(2\alpha_1)}\mathcal N_{i,\mathrm{loc}}(s),\\
    \widetilde{\mathcal M}(s)
    &\leq C e^{B/(2\alpha_1)}\widetilde{\mathcal N}_{i}(s),\\
    \mathcal M(s)&\leq \widetilde{\mathcal M}(s)
    \leq C\mathcal N_i(s),
\end{split}
\end{equation}

and
\begin{equation}\label{MhatN}
\widehat{\mathcal{M}}(s) \lesssim B^2\mathcal{N}_{i}(s), \text{ for }i \geq 2.
\end{equation}
\end{lemma}

\begin{proof}
We first prove the first two pointwise estimates from \eqref{McomparedtoN}. Recall the definition of $\vartheta$ from \eqref{varthetaweight}. On \(y_1\leq B/2\),
\[
    \phi_{i,B}'\geq E_B'
    =\frac1\alpha_1 e^{-B/(2\alpha_1)}e^{y_1/\alpha_1}.
\]
Thus, on the left side \(y_1<0\),
\[
    e^{-\vartheta}\lesssim e^{y_1/\alpha_1}
    \lesssim e^{B/(2\alpha_1)}\phi_{i,B}'.
\]
On the bounded central interval \(0\leq y_1\leq B/2\), one has
\(e^{-\vartheta}\lesssim1\), while
\[
    \phi_{i,B}'\geq c e^{-B/(2\alpha_1)}.
\]
Hence again
\[
    e^{-\vartheta}\lesssim e^{B/(2\alpha_1)}\phi_{i,B}'.
\]
On \(y_1>B/2\), the right interpolation and the polynomial branch give
\(e^{-\vartheta}\lesssim \frac{1}{B} \lesssim \phi_{i,B}'\), and this is certainly bounded by
\(Ce^{B/(2\alpha_1)}\phi_{i,B}'\).  Therefore
\[
    e^{-\vartheta}\lesssim e^{B/(2\alpha_1)}\phi_{i,B}'
    \qquad\text{for all }y_1\in \R.
\]

Next we prove the third estimate of \eqref{McomparedtoN}.  The inequality
\(\mathcal M\leq \widetilde{\mathcal M}\) is immediate.  It remains to show
that the weight \(e^{-\vartheta}\) is dominated by the weights in
\(\mathcal N_i\).  For the gradient part, since \(2\sigma\leq1/\alpha_1\), the
far-left asymptotic \(\psi_B\sim e^{2\sigma(y_1+MB)}\) implies
\[
    e^{y_1/\alpha_1}\lesssim \psi_B
    \qquad (y_1\ll-1).
\]
On the collar and central/right regions, \(\psi_B\sim1\), while
\(e^{-\vartheta}\lesssim1\).  Hence
\[
    e^{-\vartheta}\lesssim \psi_B.
\]
For the mass part, since \(\sigma\leq1/\alpha_1\), the far-left asymptotic
\(\phi_{i,B}\gtrsim S_B\sim e^{\sigma(y_1+MB)}\) gives
\[
    e^{y_1/\alpha_1}\lesssim \phi_{i,B}.
\]
Again the collar, central, and right regions are immediate.  Hence
\[
    e^{-\vartheta}\lesssim \phi_{i,B}.
\]
Together these two pointwise estimates imply
\[
    \widetilde{\mathcal M}\lesssim \mathcal N_i.
\]

It remains to prove \eqref{MhatN}. Since the transverse factor in
$\Omega$ is bounded by $1$, it suffices to compare the $y_1$-weights.
For $y_1\leq-B/2-1$, the choice $2\sigma\leq\alpha_1^{-1}$ gives
\[
\Omega(y_1,y_2)\lesssim e^{y_1/\alpha_1}
\lesssim S_B^2(y_1)=\psi_B(y_1)\leq\phi_{i,B}(y_1).
\]
On $[-B/2-1,B]$, one has $\Omega\lesssim B^2\lesssim
B^2\phi_{i,B}$, while for $y_1\geq B$ and $i\geq2$,
\[
\Omega(y_1,y_2)\leq y_1^2\leq B\phi_{i,B}(y_1).
\]
Thus $\Omega\lesssim B^2\phi_{i,B}$, and therefore
\[
\widehat{\mathcal M}(s)
\lesssim B^2\iint\varepsilon^2\phi_{i,B}
\lesssim B^2\mathcal N_i(s).
\]
\end{proof}

\section{The Mixed Energy Functional}
We fix the certified value of $\alpha_1$ and the coercivity constants
first. For the finitely many weight indices and scaling exponents used
below, we also fix the constants $\mathcal C_i>1$ in \eqref{H2}.
We then fix $\sigma,\rho,M$ as above and choose $B$ sufficiently
large. The perturbation size $\nu^*$ is chosen after these weight
parameters, followed by $\alpha^*\ll\nu^*$ and
$\alpha_0\ll\alpha^*$. Constants denoted by $C(B)$ may depend on these
fixed weights, indices, cut-offs, and constants $\mathcal C_i$, but not
on the solution or time.
Terms depending only on the soliton and required to be small as
$B\to\infty$ will retain their explicit decay in $B$; terms with a fixed
$C(B)$ are made small by the subsequent choice of $\nu^*$.

 For $i\in \mathbb{N}$ with $i\geq 1$, we define the Lyapunov functional (mixed energy) by 

\begin{equation} 
\begin{split} 
\mathcal{F}_{i}=\iint (\varepsilon_{y_1}^2+\varepsilon_{y_2}^2)\psi_B+\varepsilon^2 \phi_{i,B}-\frac{1}{2}\Big((\varepsilon+Q_b)^4-Q_b^4-4Q_b^3\varepsilon\Big)\psi_B
\end{split} 
\end{equation}

\begin{proposition}\label{Monotonicity} Fix $i\in \mathbb{N}, i\geq 1$, $j\in \mathbb{R}, j\geq 0$, and a constant $\mathcal C_i>1$. There exists $\mu>0,$ $B>0$ large enough (to be fixed later) and $0<\nu^*<\nu_0$ such that the following holds: Suppose $\varepsilon(t,y_1,y_2)$ satisfies the modulation equation and the orthogonality conditions from Lemma \ref{eq:nualpha} on $[0,t_0]$ and the following a priori estimates hold on $[0,s_0]$ where $s(t_0)=s_0:$
\begin{gather}
\|\varepsilon(s)\|_{L^2}+|b(s)|+\mathcal{N}_{i}(s) \leq \nu^*, \label{H1} \tag{\textbf{H1}} \\
\frac{|b(s)|+\mathcal{N}_i(s)}{\lambda^c(s)} \leq \mathcal C_i, \label{H2} \tag{\textbf{H2}} \\
\int_{y_1>0}y_1^{K_i} \varepsilon^2(s) \leq 2\bigg(1+\frac{1}{\lambda^{K_i}(s)}\bigg) \quad \text{for } K_i=\left\lfloor \frac{c(i+3)}{c-1}\right\rfloor+1. \label{H3} \tag{\textbf{H3}}
\end{gather}
Then we have the following: 
\begin{itemize} 
\item[(i)] Lyapunov control: 
\begin{equation}\label{monotonictyformula1}
\frac{d}{ds}\Bigg\{\frac{\mathcal{F}_{i}}{\lambda^j}\Bigg\}+\frac{\mu}{\lambda^j} \iint (\varepsilon_{y_1}^2+\varepsilon_{y_2}^2+\varepsilon^2)\phi_{i,B}'\lesssim \frac{|b|^4}{\lambda^j}.
\end{equation}
\item[(ii)] Coercivity of $\mathcal{F}_{i}$: 
\begin{equation}\label{coercivityofF}
\mathcal{N}_i\lesssim \mathcal{F}_{i} \lesssim \mathcal{N}_i.
\end{equation}
\end{itemize}
\end{proposition} 
\begin{proof}

\textit{(i)}. For  $i\geq1$, $j\geq0$ we have 
\begin{equation*} 
\begin{split} 
\lambda^j\frac{d}{ds}\Bigg\{\frac{\mathcal{F}_{i}}{\lambda^j}\Bigg\}=&2\iint \psi_B (\varepsilon_{y_1})_s\varepsilon_{y_1}+\psi_B(\varepsilon_{y_2})_s\varepsilon_{y_2}+2\varepsilon_s\{\varepsilon\phi_{i,B}-\psi_B[(\varepsilon+Q_b)^3-Q_b^3]\}\\
& -2\iint \psi_B (Q_b)_s[(Q_b+\varepsilon)^3-Q_b^3-3Q_b^2\varepsilon]-j\frac{\lambda_s}{\lambda}\mathcal{F}_{i}.
\end{split} 
\end{equation*} 

 We use the modulated flow equation \eqref{eq:ModulationEquation},  
 \begin{equation*}
 \begin{split}
 \varepsilon_s-\frac{\lambda_s}{\lambda}\Lambda \varepsilon=&\Big(-\Delta \varepsilon +\varepsilon -[(Q_b+\varepsilon)^3-Q_b^3]\Big)_{y_1}+\Big(\frac{\lambda_s}{\lambda}+b\Big)\Lambda Q_b+ \Big(\frac{(x_1)_s}{\lambda}-1\Big)(Q_b+\varepsilon)_{y_1}\\&+ \Big(\frac{(x_2)_s}{\lambda}\Big)(Q_b+\varepsilon)_{y_2}+\Phi_b+\Psi_b,
 \end{split} 
 \end{equation*} 
 where $\Psi_b=[(-\Delta Q_b+Q_b-Q_b^3)_{y_1}-b\Lambda Q_b]$ and $\Phi_b=-b_s(\chi_b+\gamma y_1(\chi_b)_{y_1})P.$ We decompose the derivative into three contributions:
\[
\lambda^j\frac{d}{ds}\left(\frac{\mathcal F_i}{\lambda^j}\right)
=f_1+f_2+f_3,
\]
where 
$$f_1=2\iint (\varepsilon_s-\frac{\lambda_s}{\lambda}\Lambda \varepsilon)\Big(-(\psi_B\varepsilon_{y_1})_{y_1}-(\psi_B\varepsilon_{y_2})_{y_2}+\varepsilon \phi_{i,B}-\psi_B[(Q_b+\varepsilon)^3-Q_b^3]\Big),$$
 $$f_2=2 \frac{\lambda_s}{\lambda}\iint \Lambda \varepsilon  \Big(-(\psi_B\varepsilon_{y_1})_{y_1}-(\psi_B\varepsilon_{y_2})_{y_2}+\varepsilon \phi_{i,B}-\psi_B[(Q_b+\varepsilon)^3-Q_b^3]\Big)$$
 $$-j\frac{\lambda_s}{\lambda}\mathcal{F}_i,$$
 $$f_3=-2\iint \psi_B (Q_b)_s [(Q_b+\varepsilon)^3-Q_b^3-3Q_b^2\varepsilon].$$ 
 Here $f_1$ contains the terms generated by the equation for $\varepsilon$ which will give the damping term,
$f_2$ contains the scaling drift, and $f_3$ contains the derivative of the
profile $Q_b$. The following three lemmas estimate these contributions.

\begin{lemma}\label{lemma16}
    There exists $\mu>0$ and $C>0$ such that the term $f_{1}$ is bounded by 
\begin{equation*}
f_1\leq -\mu\iint (|\nabla \varepsilon|^2+\varepsilon^2)\phi_{i,B}'+Cb^4.
\end{equation*}
\end{lemma}

\begin{lemma}\label{lemma17}
    For the value $\mu$ appearing in Lemma \ref{lemma16}, there exists $\tilde{C}>0$ such that the term $f_2$ is bounded by 
\begin{equation*}
    f_2\leq \frac{\mu}{4}\iint (|\nabla \varepsilon|^2+\varepsilon^2)\phi_{i,B}'+\tilde{C}b^4.
\end{equation*}
\end{lemma}
The proofs of Lemmas ~\ref{lemma16} and~\ref{lemma17} are given in the following sections. 

\begin{lemma}\label{lemma18}
    For the value $\mu$ appearing in Lemma \ref{lemma16}, the term $f_3$ is bounded by 
\begin{equation*}
    f_3\leq\frac{\mu}{4}\iint (|\nabla\varepsilon|^2+\varepsilon^2)\phi_{i,B}'.
\end{equation*}
\end{lemma}
\begin{proof}[Proof of Lemma \ref{lemma18}]
Recall that 
 $$f_3=-2\iint \psi_B (Q_b)_s [(Q_b+\varepsilon)^3-Q_b^3-3Q_b^2\varepsilon].$$

Estimates \eqref{eq:b} and \eqref{McomparedtoN} give
\[
|(Q_b)_s|\lesssim |b_s|\lesssim e^{B/(2\alpha_1)}\mathcal N_i+b^2.
\]
Furthermore, \eqref{eq:psi-by-phi-prime} and the weighted Sobolev Lemma \ref{SobolevLemma} for the weight $\phi_{i,B}'$ satisfying the gradient bound \eqref{Sobolevconditionphi},
imply 
 \begin{equation*}
\begin{split}
\Big|\iint \psi_B  [(Q_b+\varepsilon)^3-Q_b^3-3Q_b^2\varepsilon]\Big|&\lesssim e^{\varpi B}\Big( \iint |\varepsilon|^3\phi_{i,B}'+\|Q_b\|_{L_{y_1y_2}^{\infty}}\iint \varepsilon^2\phi_{i,B}'\Big)\\
&\lesssim e^{\varpi B}\Big(\|\varepsilon\|_{L^2}\iint (\varepsilon_{y_1}^2+\varepsilon_{y_2}^2+\varepsilon^2)\phi_{i,B}'+\iint \varepsilon^2\phi_{i,B}'\Big).
\end{split}
\end{equation*}
Therefore, after fixing $B$ and choosing $\nu^*$ sufficiently
small,
 \begin{equation*}
\begin{split}
f_3&\lesssim  (e^{\frac{B}{2\alpha_1}}\mathcal{N}_{i}+b^2)e^{\varpi B}\Big(\|\varepsilon\|_{L^2}\iint (\varepsilon_{y_1}^2+\varepsilon_{y_2}^2+\varepsilon^2)\phi_{i,B}'+\iint \varepsilon^2\phi_{i,B}'\Big)\\
&\leq\frac{\mu}{4}\iint (\varepsilon_{y_1}^2+\varepsilon_{y_2}^2+\varepsilon^2)\phi_{i,B}'.
\end{split}
\end{equation*}
\end{proof}

Lemmas~\ref{lemma16}--\ref{lemma18} imply that, for some
$\widetilde\mu>0$, 

$$\lambda^j\frac{d}{ds}\Bigg\{ \frac{\mathcal{F}_{i}}{\lambda^j}\Bigg\}+\widetilde{\mu}\iint (|\nabla\varepsilon|^2+\varepsilon^2)\phi_{i,B}'\lesssim b^4,$$
so 
$$\frac{d}{ds}\Bigg\{ \frac{\mathcal{F}_{i}}{\lambda^j}\Bigg\}+\frac{\widetilde{\mu}}{\lambda^j}\iint (|\nabla \varepsilon|^2+\varepsilon^2)\phi_{i,B}'\lesssim \frac{b^4}{\lambda^j}.$$

\textit{(ii)}. For the upper coercivity bound, the nonlinear part satisfies
 \begin{equation*}
\begin{split}
\Big|\iint \Big((\varepsilon+Q_b)^4-Q_b^4-4Q_b^3\varepsilon\Big)\psi_B\Big|  &\lesssim \iint \Big(\varepsilon^4+Q_b^2\varepsilon^2\Big)\psi_B \\
&\lesssim  \|\varepsilon\|_{L^2}^2 \iint (|\nabla \varepsilon|^2+\varepsilon^2)\psi_B+\|Q_b\|_{L^{\infty}}^2 \iint \varepsilon^2\psi_B\\
&\lesssim  \|\varepsilon\|_{L^2}^2 \iint ( \varepsilon_{y_1}^2+\varepsilon_{y_2}^2)\psi_B+\|Q_b\|_{L^{\infty}}^2 \iint \varepsilon^2\phi_{i,B}\\
&\lesssim \mathcal{N}_i(s).
\end{split}
\end{equation*}
Hence $\mathcal{F}_{i}\lesssim \mathcal{N}_{i}.$
For the lower bound, we write it as 
\begin{equation*} 
\begin{split} 
\mathcal{F}_{i}&=\iint (\varepsilon_{y_1}^{2}+\varepsilon_{y_2}^{2})\psi_B+\varepsilon^2 \phi_{i,B}-\frac{1}{2}\Big((\varepsilon+Q_b)^4-Q_b^4-4Q_b^3\varepsilon\Big)\psi_B\\
&=\iint [(\varepsilon_{y_1}^{2}+\varepsilon_{y_2}^{2})\psi_B+\varepsilon^2 \phi_{i,B}-3Q^2\varepsilon^2\psi_B]\\&-\frac{1}{2}\iint \Big((\varepsilon+Q_b)^4-Q_b^4-4Q_b^3\varepsilon-6Q_b^2\varepsilon^2\Big)\psi_B-3\iint (Q_b^2-Q^2)\varepsilon^2\psi_B.
\end{split} 
\end{equation*} 
Since 
\begin{equation*} 
\begin{split} 
\Big|\iint \Big((\varepsilon+Q_b)^4&-Q_b^4-4Q_b^3\varepsilon-6Q_b^2\varepsilon^2\Big)\psi_B\Big|\lesssim \iint \Big(\varepsilon^4+|Q_b||\varepsilon|^3\Big)\psi_B\\
&\lesssim  \|\varepsilon\|_{L^2}^2 \iint ( \varepsilon_{y_1}^2+\varepsilon_{y_2}^2+\varepsilon^2)\psi_B+ \|Q_b\|_{L^{\infty}}\|\varepsilon\|_{L^2} \iint ( \varepsilon_{y_1}^2+\varepsilon_{y_2}^2+\varepsilon^2)\psi_B
\end{split} 
\end{equation*} 
and 
$$\Big|\iint (Q_b^2-Q^2)\varepsilon^2\psi_B\Big|\lesssim |b|\iint \varepsilon^2\phi_{i,B},$$
it is sufficient to prove that there exists $\overline{\mu}>0$
$$\iint [(\varepsilon_{y_1}^{2}+\varepsilon_{y_2}^{2})\psi_B+\varepsilon^2 \phi_{i,B}-3Q^2\varepsilon^2\psi_B]\geq \overline{\mu} \iint [( \varepsilon_{y_1}^2+\varepsilon_{y_2}^2)\psi_B+\varepsilon^2\phi_{i,B}].$$
Since $\psi_B\leq \phi_{i,B}$, it is sufficient to prove that 
$$\iint [(\varepsilon_{y_1}^{2}+\varepsilon_{y_2}^{2})+\varepsilon^2-3Q^2\varepsilon^2]\psi_B\geq \overline{\mu} \iint ( \varepsilon_{y_1}^2+\varepsilon_{y_2}^2+\varepsilon^2)\psi_B$$
By integration by parts, we compute 

\begin{equation}\label{Lweightedform}
\begin{split}
(L(\varepsilon \sqrt{\psi_B}),\varepsilon\sqrt{\psi_B})&=\iint (\varepsilon_{y_1}^2+\varepsilon_{y_2}^2+\varepsilon^2-3Q^2\varepsilon^2)\psi_B+\varepsilon^2\frac{[\psi_B']^2}{4\psi_B}+\varepsilon\varepsilon_{y_1}\psi_B'\\
&=\iint (\varepsilon_{y_1}^2+\varepsilon_{y_2}^2+\varepsilon^2-3Q^2\varepsilon^2)\psi_B-\iint \varepsilon^2\frac{2\psi_B''\psi_B-[\psi_B']^2}{4\psi_B},
\end{split}
\end{equation}
and also 
\begin{equation}\label{H1weightedform}
    \|\varepsilon\sqrt{\psi_B}\|_{H^1}^2=\iint (\varepsilon_{y_1}^2+\varepsilon_{y_2}^2+\varepsilon^2)\psi_B-\iint \varepsilon^2\frac{2\psi_B''\psi_B-[\psi_B']^2}{4\psi_B}.
\end{equation}

Since $1-\psi_B$ is supported in $\{y_1<-B/2\}$, the orthogonality
conditions \eqref{eq:orthogonality_conditions} and Cauchy--Schwarz give 
$$|(\varepsilon\sqrt{\psi_B},Q)|=|(Q,(1-\sqrt{\psi_B})\varepsilon)|=\iint_{\{y_1<-\frac{B}{2}\}}|Q(1-\sqrt{\psi_B})\varepsilon|$$
$$\leq \Big(\iint \varepsilon^2\psi_B\Big)^{\frac{1}{2}}\Big(\iint_{\{y_1<-\frac{B}{2}\}}\frac{Q^2(1-\sqrt{\psi_B})^2}{\psi_B}\Big)^{\frac{1}{2}}\leq e^{-\frac{B}{2}}\Big(\iint \varepsilon^2\psi_B\Big)^{\frac{1}{2}}.$$
By the definition of $\psi_B,$ 
\begin{equation}
    \frac{(1-\sqrt{\psi_B})^2}{\psi_B}
    \begin{cases}
        =e^{-2\sigma(y_1+MB)}\qquad \text{ for }y_1\leq -B/2-1,\\
        \lesssim e^{-2\sigma(M-\frac{1}{2})B}\qquad \text{ for }-B/2-1\leq y_1\leq -B/2. 
    \end{cases}
\end{equation}
Together with the decay $Q(y_1,y_2)\lesssim e^{-\frac{|y_1|}{\alpha_1}-C_{\alpha_1}|y_2|},$ we obtain 
$$\Big(\iint_{\{y_1<-\frac{B}{2}\}}\frac{Q^2(1-\sqrt{\psi_B})^2}{\psi_B}\Big)^{\frac{1}{2}}\lesssim e^{-\big(\frac{1}{2\alpha_1}+\sigma(M-\frac{1}{2})\big)B}\lesssim e^{-\frac{B}{\alpha_1}}.$$
Therefore, 
$$|(\varepsilon\sqrt{\psi_B},Q)|\lesssim e^{-\frac{B}{\alpha_1}}\Big(\iint \varepsilon^2\psi_B\Big)^{\frac{1}{2}}.$$

Similarly, we have 
$$|(\varepsilon\sqrt{\psi_B},\varphi(y_1)Q_{y_1})|+|(\varepsilon\sqrt{\psi_B},\varphi(y_1)\Lambda Q)|+\Big|(\varepsilon\sqrt{\psi_B},\varphi(y_1)Q_{y_2})\Big|\leq e^{-\frac{B}{\alpha_1}}\Big(\iint \varepsilon^2\psi_B\Big)^{\frac{1}{2}}.$$

Applying the global coercivity estimate \eqref{eq:coercivitylemma} to
$\varepsilon\sqrt{\psi_B}$ and using \eqref{Lweightedform}--\eqref{H1weightedform},
we find a constant $\widetilde\delta>0$ such that
$$\iint (\varepsilon_{y_1}^2+\varepsilon_{y_2}^2+\varepsilon^2-3Q^2\varepsilon^2)\psi_B\geq$$
$$\geq\tilde{\delta}\iint (\varepsilon_{y_1}^2+\varepsilon_{y_2}^2+\varepsilon^2)\psi_B+(1-\tilde{\delta})\iint \varepsilon^2\frac{2\psi_B''\psi_B-[\psi_B']^2}{4\psi_B}-4\frac{e^{-\frac{B}{\alpha_1}}}{\tilde{\delta}}\Big(\iint \varepsilon^2\psi_B\Big)$$
$$\geq (\tilde{\delta}-\sigma^2-4\frac{e^{-\frac{B}{\alpha_1}}}{\tilde{\delta}})\iint (\varepsilon_{y_1}^2+\varepsilon_{y_2}^2+\varepsilon^2)\psi_B,$$
where we used \eqref{RpsiB} so 
\[
\left|\frac{2\psi_B''\psi_B-[\psi_B']^2}{4\psi_B}\right|\begin{cases}
\leq \sigma^2\psi_B &\quad \text{ for } y_1<-\frac{B}{2},\\
=0 &\quad \text{ for } y_1\geq -\frac{B}{2}.\\
\end{cases}
\]

The choice \eqref{sigmasmall} and a sufficiently large value of $B$ ensure
that the coefficient on the right is at least $\widetilde\delta/2$.
Consequently,
$$\iint [(\varepsilon_{y_1}^2+\varepsilon_{y_2}^2)\psi_B+\varepsilon^2 \phi_{i,B}-3Q^2\varepsilon^2\psi_B]\geq \frac{\tilde{\delta}}{2}\iint [(\varepsilon_{y_1}^2+\varepsilon_{y_2}^2)\psi_B+\varepsilon^2 \phi_{i,B}]$$
so $\mathcal{N}_{i}\lesssim \mathcal{F}_{i}.$
 \end{proof}

\section{Proof of Lemma \ref{lemma16}:} 

We first perform the computations for smooth, rapidly decaying functions; the
identities extend to the regularity class used in the proposition by a standard
approximation argument. Throughout this proof, primes denote derivatives with
respect to $y_1$. Set
\[
\mathcal H(\varepsilon)
:=-(\psi_B\varepsilon_{y_1})_{y_1}
  -(\psi_B\varepsilon_{y_2})_{y_2}
  +\phi_{i,B}\varepsilon
  -\psi_B\bigl[(Q_b+\varepsilon)^3-Q_b^3\bigr]
\]
and
\[
    N(\varepsilon):=(Q_b+\varepsilon)^3-Q_b^3.
\]
Using \eqref{eq:ModulationEquation}, we decompose
\[
    f_1=f_{1,1}+f_{1,2}+f_{1,3}+f_{1,4}+f_{1,5}+f_{1,6},
\]
where
\begin{align*}
f_{1,1}
&=2\iint
 \bigl\{-\Delta\varepsilon+\varepsilon-N(\varepsilon)\bigr\}_{y_1}
 \mathcal H(\varepsilon),\\
f_{1,2}
&=2\left(\frac{\lambda_s}{\lambda}+b\right)
  \iint \Lambda Q_b\,\mathcal H(\varepsilon),\\
f_{1,3}
&=2\left(\frac{(x_1)_s}{\lambda}-1\right)
  \iint (Q_b+\varepsilon)_{y_1}\mathcal H(\varepsilon),\\
f_{1,4}
&=2\frac{(x_2)_s}{\lambda}
  \iint (Q_b+\varepsilon)_{y_2}\mathcal H(\varepsilon),\\
f_{1,5}
&=-2b_s\iint
  \bigl(\chi_b+\gamma y_1\chi_b'\bigr)P\,\mathcal H(\varepsilon),\\
f_{1,6}
&=2\iint\Psi_b\,\mathcal H(\varepsilon).
\end{align*}

\paragraph{Step 1: Estimate of $f_{1,1}$.}
Since $\psi_B=\psi_B(y_1)$, we may write
\[
\begin{aligned}
\mathcal H(\varepsilon)
={}&\psi_B\bigl[-\Delta\varepsilon+\varepsilon-N(\varepsilon)\bigr]
-\psi_B'\varepsilon_{y_1}
+(\phi_{i,B}-\psi_B)\varepsilon.
\end{aligned}
\]
Consequently,
\begin{align*}
f_{1,1}
={}&2\iint
 \bigl[-\Delta\varepsilon+\varepsilon-N(\varepsilon)\bigr]_{y_1}
 \bigl[-\Delta\varepsilon+\varepsilon-N(\varepsilon)\bigr]\psi_B\\
&+2\iint
 \bigl[-\Delta\varepsilon+\varepsilon-N(\varepsilon)\bigr]_{y_1}
 \bigl[-\psi_B'\varepsilon_{y_1}
       +(\phi_{i,B}-\psi_B)\varepsilon\bigr]\\
=:{}&f_{1,1}^{(i)}+f_{1,1}^{(ii)}.
\end{align*}
The first term is the direct source of the second-order dissipation. Indeed,
without estimating anything,
\begin{align*}
f_{1,1}^{(i)}
&=-\iint\psi_B'
 \bigl[-\Delta\varepsilon+\varepsilon-N(\varepsilon)\bigr]^2\\
&=-\iint\psi_B'[-\Delta\varepsilon+\varepsilon]^2
-\iint\psi_B'
 \Bigl\{
 \bigl[-\Delta\varepsilon+\varepsilon-N(\varepsilon)\bigr]^2
 -[-\Delta\varepsilon+\varepsilon]^2
 \Bigr\}.
\end{align*}
Thus the nonlinear square is treated as a perturbation of the genuinely
non-positive term in the first line.

We next compute the linear square. Since the weights depend only on $y_1$,
\begin{align*}
-\iint\psi_B'[-\Delta\varepsilon+\varepsilon]^2
={}&-\iint\psi_B'
 \bigl(\varepsilon_{y_1y_1}^2
      +2\varepsilon_{y_1y_2}^2
      +\varepsilon_{y_2y_2}^2
      +2\varepsilon_{y_1}^2
      +2\varepsilon_{y_2}^2\bigr)\\
&+\iint(\psi_B'''-\psi_B')\varepsilon^2
 +\iint\psi_B'''\varepsilon_{y_2}^2.
\end{align*}
Here we used
\begin{align*}
2\iint\psi_B'\varepsilon\varepsilon_{y_1y_1}
&=-2\iint\psi_B'\varepsilon_{y_1}^2
  +\iint\psi_B'''\varepsilon^2,\\
2\iint\psi_B'\varepsilon\varepsilon_{y_2y_2}
&=-2\iint\psi_B'\varepsilon_{y_2}^2,\\
-2\iint\psi_B'\varepsilon_{y_1y_1}\varepsilon_{y_2y_2}
&=-2\iint\psi_B'\varepsilon_{y_1y_2}^2
  +\iint\psi_B'''\varepsilon_{y_2}^2.
\end{align*}

For the second part, separating the linear and nonlinear terms gives
\begin{align*}
f_{1,1}^{(ii)}
={}&2\iint(-\Delta\varepsilon+\varepsilon)_{y_1}
 \bigl[-\psi_B'\varepsilon_{y_1}
       +(\phi_{i,B}-\psi_B)\varepsilon\bigr]\\
&+2\iint\psi_B'N(\varepsilon)_{y_1}\varepsilon_{y_1}
 -2\iint(\phi_{i,B}-\psi_B)N(\varepsilon)_{y_1}\varepsilon.
\end{align*}
A direct integration by parts yields
\begin{align*}
&2\iint(-\Delta\varepsilon+\varepsilon)_{y_1}
 \bigl[-\psi_B'\varepsilon_{y_1}
       +(\phi_{i,B}-\psi_B)\varepsilon\bigr]\\
&\quad=-2\iint\psi_B'
 \bigl(\varepsilon_{y_1y_1}^2+\varepsilon_{y_1y_2}^2\bigr)\\
&\qquad+\iint
 \bigl[\psi_B'''-2\psi_B'
       -3(\phi_{i,B}'-\psi_B')\bigr]\varepsilon_{y_1}^2
 -\iint(\phi_{i,B}'-\psi_B')\varepsilon_{y_2}^2\\
&\qquad+\iint
 \bigl[(\phi_{i,B}'''-\psi_B''')
       -(\phi_{i,B}'-\psi_B')\bigr]\varepsilon^2.
\end{align*}
Finally, the identity
\begin{align*}
-N(\varepsilon)_{y_1}\varepsilon
={}&\left[
\frac{(Q_b+\varepsilon)^4-Q_b^4}{4}
-(Q_b+\varepsilon)^3\varepsilon
\right]_{y_1}-\bigl[N(\varepsilon)-3Q_b^2\varepsilon\bigr](Q_b)_{y_1}
\end{align*}
gives
\begin{align*}
-2\iint(\phi_{i,B}-\psi_B)N(\varepsilon)_{y_1}\varepsilon&=-\frac12\iint(\phi_{i,B}'-\psi_B')
 \Bigl[(Q_b+\varepsilon)^4-Q_b^4
       -4(Q_b+\varepsilon)^3\varepsilon\Bigr]\\
&\qquad-2\iint(\phi_{i,B}-\psi_B)
 \bigl[N(\varepsilon)-3Q_b^2\varepsilon\bigr](Q_b)_{y_1}.
\end{align*}
Combining the preceding identities, we obtain the exact formula
\begin{align*}
f_{1,1}
={}&-\iint \psi_B'
 \bigl(3\varepsilon_{y_1y_1}^2
       +4\varepsilon_{y_1y_2}^2
       +\varepsilon_{y_2y_2}^2\bigr)\\
&+\iint\varepsilon_{y_1}^2
  \bigl[-\psi_B'-3\phi_{i,B}'+\psi_B'''\bigr]
 +\iint\varepsilon_{y_2}^2
  \bigl[-\psi_B'-\phi_{i,B}'+\psi_B'''\bigr]\\
&+\iint\varepsilon^2
  \bigl[-\phi_{i,B}'+\phi_{i,B}'''\bigr]\\
&-\iint \psi_B'
 \Bigl\{[-\Delta\varepsilon+\varepsilon-N(\varepsilon)]^2
       -[-\Delta\varepsilon+\varepsilon]^2\Bigr\}\\
&-\frac12\iint (\phi_{i,B}'-\psi_B')
 \Bigl[(Q_b+\varepsilon)^4-Q_b^4
       -4(Q_b+\varepsilon)^3\varepsilon\Bigr]\\
&-2\iint (\phi_{i,B}-\psi_B)
 \bigl[N(\varepsilon)-3Q_b^2\varepsilon\bigr](Q_b)_{y_1}
 +2\iint \psi_B' N(\varepsilon)_{y_1}\varepsilon_{y_1}.
\end{align*}
We split the spatial integral into
\[
    \{y_1<-B/2\},\qquad
    \{|y_1|\leq B/2\},\qquad
    \{y_1>B/2\},
\]
and denote the corresponding contributions by
$f_{1,1}^{<}$, $f_{1,1}^{\sim}$, and $f_{1,1}^{>}$. 

\medskip
\noindent\textbf{The left region $\{y_1<-B/2\}$.}
We shall use the following estimate only in this region.

\begin{lemma}\label{SobolevLemma2}
There exists $\mathcal D_B>0$, depending only on $B$, such that
\begin{align}\label{eq:main}
&\iint f^2|\nabla f|^2\psi_B'
 +\iint f^4\psi_B'\notag\\
&\qquad\leq
\mathcal D_B\|f\|_{L^2}^2
\left[
 \iint(|D^2f|^2+f^2)\psi_B'
 +\iint(|\nabla f|^2+f^2)\phi_{i,B}'
\right],
\end{align}
and
\begin{align}\label{eq:L6}
\iint |f|^6\psi_B'
\leq
\mathcal D_B\|f\|_{L^2}^4
\left[
 \iint(|D^2f|^2+f^2)\psi_B'
 +\iint(|\nabla f|^2+f^2)\phi_{i,B}'
\right].
\end{align}
\end{lemma}
Its proof is given in Appendix~\ref{appendixA},
Lemma~\ref{lemmasobolev2}.

By \eqref{psithirdderivative} and
\eqref{eq:phi-third-one-sided-new},
\begin{align*}
-\psi_B'-3\phi_{i,B}'+\psi_B'''
&\leq-\frac52\phi_{i,B}',\\
-\psi_B'-\phi_{i,B}'+\psi_B'''
&\leq-\frac12\phi_{i,B}',\\
-\phi_{i,B}'+\phi_{i,B}'''
&\leq-\kappa\phi_{i,B}'.
\end{align*}
We now estimate the four nonlinear remainders.

First, by \eqref{psiandphi}, we have
\[
    |\phi_{i,B}'-\psi_B'|\lesssim\phi_{i,B}'.
\]
Moreover,
\[
\left|(Q_b+\varepsilon)^4-Q_b^4
      -4(Q_b+\varepsilon)^3\varepsilon\right|
\lesssim \varepsilon^4+Q_b^2\varepsilon^2.
\]
Applying Lemma~\ref{SobolevLemma} with the weight
$\theta=\phi_{i,B}'$, which satisfies \eqref{Sobolevconditionphi}, and
using Lemma~\ref{Qdecay}, we obtain
\begin{align*}
&\left|\iint_{y_1<-B/2}(\phi_{i,B}'-\psi_B')
 \Bigl[(Q_b+\varepsilon)^4-Q_b^4
       -4(Q_b+\varepsilon)^3\varepsilon\Bigr]\right|\\
&\qquad\lesssim
\iint\phi_{i,B}'\varepsilon^4
+\iint_{y_1<-B/2}\phi_{i,B}'Q_b^2\varepsilon^2\\
&\qquad\lesssim
\|\varepsilon\|_{L^2}^2\widetilde{\mathcal N}_i
+\bigl(e^{-B/\alpha_1}+b^2\bigr)\mathcal N_{i,\mathrm{loc}}.
\end{align*}

We next estimate the term containing $(Q_b)_{y_1}$. Since
\[
    N(\varepsilon)-3Q_b^2\varepsilon
    =3Q_b\varepsilon^2+\varepsilon^3,
\]
we split it into a quadratic and a cubic contribution. By
\eqref{psiandphi} and \eqref{phiboundderivativephi},
\[
    |\phi_{i,B}-\psi_B|
    \lesssim\phi_{i,B}
    \lesssim C(B)\phi_{i,B}'
    \qquad (y_1<-B/2).
\]
Therefore, using Lemma~\ref{SobolevLemma} with
$\theta=\phi_{i,B}'$,
\begin{align*}
\iint_{y_1<-B/2}
 |\phi_{i,B}-\psi_B|\,|(Q_b)_{y_1}|\,|\varepsilon|^3
&\lesssim
C(B)\iint\phi_{i,B}'|\varepsilon|^3\\
&\leq
C(B)\left(\iint\phi_{i,B}'\varepsilon^2\right)^{1/2}
       \left(\iint\phi_{i,B}'\varepsilon^4\right)^{1/2}\\
&\lesssim
C(B)\|\varepsilon\|_{L^2}\widetilde{\mathcal N}_i.
\end{align*}

For the quadratic contribution, Lemma~\ref{Qdecay} gives
\(
    |Q_b(Q_b)_{y_1}|
    \lesssim
    e^{-2|y_1|/\alpha_2}+b^2.
\)
For the exponentially decaying part,
the explicit left-tail form of the weights gives
\[
|\phi_{i,B}-\psi_B|
\lesssim
e^{\frac{B}{2\alpha_1}-\frac{y_1}{\alpha_1}}\phi_{i,B}'.
\]
Since $y_1\leq-B/2$ and
$\alpha_1-\alpha_2=1/200$ by \eqref{alphas},
\begin{align*}
e^{\frac{B}{2\alpha_1}-\frac{y_1}{\alpha_1}}
 e^{-2|y_1|/\alpha_2}
&=
 e^{\frac{B}{2\alpha_1}
    +y_1(\frac2{\alpha_2}-\frac1{\alpha_1})}\leq
 e^{-\frac{B}{200\alpha_1\alpha_2}}.
\end{align*}
For the $b^2$-part we instead use
$|\phi_{i,B}-\psi_B|\lesssim C(B)\phi_{i,B}'$. Hence
\begin{align*}
&\iint_{y_1<-B/2}
 |\phi_{i,B}-\psi_B|\,|Q_b(Q_b)_{y_1}|\varepsilon^2\lesssim
\left(e^{-\frac{B}{200\alpha_1\alpha_2}}+C(B)b^2\right)
\mathcal N_{i,\mathrm{loc}}.
\end{align*}
Combining the quadratic and cubic contributions, we obtain
\begin{align*}
&\left|\iint_{y_1<-B/2}
 (\phi_{i,B}-\psi_B)
 \bigl[N(\varepsilon)-3Q_b^2\varepsilon\bigr](Q_b)_{y_1}\right|\lesssim
C(B)\|\varepsilon\|_{L^2}\widetilde{\mathcal N}_i
+\left(e^{-\frac{B}{200\alpha_1\alpha_2}}+C(B)b^2\right)
 \mathcal N_{i,\mathrm{loc}}.
\end{align*}
This computation includes the cut-off-transition region $I_b$: there
$Q_b=O(|b|)$ and the cut-off derivative in $(Q_b)_{y_1}$ is
$O(|b|^{1+\gamma})$, so its product is
$O(|b|^{2+\gamma})=O(b^2)$.

We now treat the term containing $N(\varepsilon)_{y_1}$. On
$\{y_1<-B/2\}$, set
\[
    q_B:=e^{-B/(2\alpha_1)}+|b|.
\]
Lemma~\ref{Qdecay} and the definition of $Q_b$ imply
$|Q_b|+|(Q_b)_{y_1}|\lesssim q_B$. Expanding the derivative,
\begin{align*}
N(\varepsilon)_{y_1}
={}&3Q_b^2\varepsilon_{y_1}
 +6Q_b(Q_b)_{y_1}\varepsilon
 +3(Q_b)_{y_1}\varepsilon^2+6Q_b\varepsilon\varepsilon_{y_1}
 +3\varepsilon^2\varepsilon_{y_1}.
\end{align*}
Consequently,
\begin{align*}
&\left|\iint_{y_1<-B/2}
 \psi_B'N(\varepsilon)_{y_1}\varepsilon_{y_1}\right|\lesssim
q_B^2\iint\psi_B'(\varepsilon_{y_1}^2+\varepsilon^2)
+q_B\iint\psi_B'
 \bigl(|\varepsilon|\varepsilon_{y_1}^2
       +\varepsilon^2|\varepsilon_{y_1}|\bigr)
+\iint\psi_B'\varepsilon^2\varepsilon_{y_1}^2.
\end{align*}
By Cauchy--Schwarz and \eqref{eq:main},
\begin{align*}
\iint\psi_B'|\varepsilon|\varepsilon_{y_1}^2
&\leq
\left(\iint\psi_B'\varepsilon^2\varepsilon_{y_1}^2\right)^{1/2}
\left(\iint\psi_B'\varepsilon_{y_1}^2\right)^{1/2},\\
\iint\psi_B'\varepsilon^2|\varepsilon_{y_1}|
&\leq
\left(\iint\psi_B'\varepsilon^2\varepsilon_{y_1}^2\right)^{1/2}
\left(\iint\psi_B'\varepsilon^2\right)^{1/2},
\end{align*}
and both terms are bounded by
\[
\mathcal D_B^{1/2}\|\varepsilon\|_{L^2}
\left[
 \iint(|D^2\varepsilon|^2+\varepsilon^2)\psi_B'
 +\widetilde{\mathcal N}_i
\right].
\]
After decreasing $\nu^*$ so that
$\mathcal D_B^{1/2}\|\varepsilon\|_{L^2}\ll1$, and using
$q_B^2\lesssim q_B$, we conclude that
\begin{align*}
&\left|\iint_{y_1<-B/2}
 \psi_B'N(\varepsilon)_{y_1}\varepsilon_{y_1}\right|\lesssim
\left(\mathcal D_B\|\varepsilon\|_{L^2}^2+q_B\right)
\left[
 \iint(|D^2\varepsilon|^2+\varepsilon^2)\psi_B'
 +\widetilde{\mathcal N}_i
\right].
\end{align*}

Finally, we estimate the nonlinear perturbation of the dissipative square.
Since
\(
N(\varepsilon)=3Q_b^2\varepsilon+3Q_b\varepsilon^2+\varepsilon^3
\)
and
\(
q_B|\varepsilon|^2
\leq\frac12\bigl(|\varepsilon|^3+q_B^2|\varepsilon|\bigr),
\)
we have
\[
    |N(\varepsilon)|
    \lesssim|\varepsilon|^3+q_B^2|\varepsilon|.
\]
Therefore,
\begin{align*}
&\left|
 \bigl[-\Delta\varepsilon+\varepsilon-N(\varepsilon)\bigr]^2
 -[-\Delta\varepsilon+\varepsilon]^2
 \right|\lesssim
(|D^2\varepsilon|+|\varepsilon|)
 (|\varepsilon|^3+q_B^2|\varepsilon|)
+|\varepsilon|^6+q_B^4\varepsilon^2.
\end{align*}
The most delicate term is bounded by Cauchy--Schwarz and \eqref{eq:L6}:
\begin{align*}
\iint\psi_B'|D^2\varepsilon|\,|\varepsilon|^3
&\leq
\left(
 \iint\psi_B'|D^2\varepsilon|^2
\right)^{1/2}
\left(
 \iint\psi_B'|\varepsilon|^6
\right)^{1/2}\\
&\lesssim
\mathcal D_B^{1/2}\|\varepsilon\|_{L^2}^2
\left[
 \iint
 (|D^2\varepsilon|^2+\varepsilon^2)\psi_B'
 +\widetilde{\mathcal N}_i
\right].
\end{align*}
Moreover, by \eqref{eq:main} and \eqref{eq:L6},
\begin{align*}
\iint\psi_B'\varepsilon^4
&\lesssim
\mathcal D_B\|\varepsilon\|_{L^2}^2
\left[
 \iint(|D^2\varepsilon|^2+\varepsilon^2)\psi_B'
 +\widetilde{\mathcal N}_i
\right],\\
\iint\psi_B'|\varepsilon|^6
&\lesssim
\mathcal D_B\|\varepsilon\|_{L^2}^4
\left[
 \iint(|D^2\varepsilon|^2+\varepsilon^2)\psi_B'
 +\widetilde{\mathcal N}_i
\right],
\end{align*}
while Cauchy--Schwarz gives
\[
q_B^2\iint\psi_B'
 \bigl(|D^2\varepsilon||\varepsilon|+\varepsilon^2\bigr)
+q_B^4\iint\psi_B'\varepsilon^2
\lesssim q_B^2
\left[
 \iint(|D^2\varepsilon|^2+\varepsilon^2)\psi_B'
 +\widetilde{\mathcal N}_i
\right].
\]
Combining these bounds yields
\begin{align*}
&\left|\iint_{y_1<-B/2}\psi_B'
 \Bigl\{[-\Delta\varepsilon+\varepsilon-N(\varepsilon)]^2
       -[-\Delta\varepsilon+\varepsilon]^2\Bigr\}\right|\\
&\quad\lesssim
\left[
 \mathcal D_B\bigl(\|\varepsilon\|_{L^2}^2
                   +\|\varepsilon\|_{L^2}^4\bigr)
 +q_B^2
\right]
\left[
 \iint(|D^2\varepsilon|^2+\varepsilon^2)\psi_B'
 +\widetilde{\mathcal N}_i
\right].
\end{align*}
The coefficient is arbitrarily small after $B$ is fixed and $\nu^*$ is
chosen sufficiently small.

Given $\delta_0>0$, choose $B$ sufficiently large and then  $\nu^*$
sufficiently small. The preceding estimates give
\begin{align*}
f_{1,1}^{<}
\leq{}&-2\iint_{y_1<-B/2}\varepsilon_{y_1y_1}^2\psi_B'
       -\frac12\iint_{y_1<-B/2}\varepsilon_{y_2y_2}^2\psi_B'
       -2\iint_{y_1<-B/2}\varepsilon_{y_1y_2}^2\psi_B'\\
&-2\iint_{y_1<-B/2}\varepsilon_{y_1}^2\phi_{i,B}'
 -\frac14\iint_{y_1<-B/2}\varepsilon_{y_2}^2\phi_{i,B}'
 -\frac\kappa2\iint_{y_1<-B/2}\varepsilon^2\phi_{i,B}'
 +\delta_0\widetilde{\mathcal N}_i.
\end{align*}
The last term is global because Lemma~\ref{SobolevLemma2} has a
global right-hand side; it will be absorbed only after the central and
right-region coercive estimates are added.

\medskip
\noindent\textbf{The right region $\{y_1>B/2\}$.}
Here $\psi_B=1$, and hence $\psi_B'=0$. Therefore
\begin{align*}
f_{1,1}^{>}
={}&-3\iint_{y_1>B/2}\varepsilon_{y_1}^2\phi_{i,B}'
    -\iint_{y_1>B/2}\varepsilon_{y_2}^2\phi_{i,B}'
    +\iint_{y_1>B/2}\varepsilon^2(-\phi_{i,B}'+\phi_{i,B}''')\\
&-\frac12\iint_{y_1>B/2}\phi_{i,B}'
 \Bigl[(Q_b+\varepsilon)^4-Q_b^4
       -4(Q_b+\varepsilon)^3\varepsilon\Bigr]\\
&-2\iint_{y_1>B/2}(\phi_{i,B}-1)
 \bigl[N(\varepsilon)-3Q_b^2\varepsilon\bigr](Q_b)_{y_1}.
\end{align*}
By \eqref{phithirdderivativerighttransition},
\[
\iint_{y_1>B/2}\varepsilon^2(-\phi_{i,B}'+\phi_{i,B}''')
\leq-\kappa\iint_{y_1>B/2}\varepsilon^2\phi_{i,B}'.
\]
Using
\[
\left|(Q_b+\varepsilon)^4-Q_b^4
      -4(Q_b+\varepsilon)^3\varepsilon\right|
\lesssim\varepsilon^4+Q_b^2\varepsilon^2,
\]
we estimate the two pieces separately. Lemma~\ref{SobolevLemma}, applied
with the weight $\theta=\phi_{i,B}'$ satisfying
\eqref{Sobolevconditionphi}, gives
\[
\iint_{y_1>B/2}\phi_{i,B}'\varepsilon^4
\lesssim
\|\varepsilon\|_{L^2}^2
\iint(|\nabla\varepsilon|^2+\varepsilon^2)\phi_{i,B}'.
\]
On the same region, Lemma~\ref{Qdecay} and the exponential decay of
$P$ give
\[
    \|Q_b\|_{L^\infty(y_1>B/2)}^2
    \lesssim e^{-B/\alpha_1}+b^2.
\]
Therefore,
\begin{align*}
&\left|\iint_{y_1>B/2}\phi_{i,B}'
 \Bigl[(Q_b+\varepsilon)^4-Q_b^4
       -4(Q_b+\varepsilon)^3\varepsilon\Bigr]\right|\\
&\qquad\lesssim
\|\varepsilon\|_{L^2}^2\widetilde{\mathcal N}_i
+\bigl(e^{-B/\alpha_1}+b^2\bigr)
 \iint_{y_1>B/2}\varepsilon^2\phi_{i,B}'.
\end{align*}

Since $\phi_{i,B}\geq1$ on the right, we have
 $|\phi_{i,B}-1|\leq\phi_{i,B}$. The transition bounds and the
polynomial branch give
\[
\frac{\phi_{i,B}}{\phi_{i,B}'}\leq C_iB
\quad(B/2<y_1<B),\qquad
\frac{\phi_{i,B}}{\phi_{i,B}'}=\frac{y_1}{i}
\quad(y_1\geq B).
\]
Consequently, the decay of $Q_{y_1}$ gives, for all sufficiently large
$B$,
\[
\sup_{y_1>B/2}
\frac{|(\phi_{i,B}-1)Q_{y_1}|}{\phi_{i,B}'}
\leq C_i(1+B)e^{-B/(2\alpha_1)},
\]
whereas the corresponding contribution from $bP_{y_1}$ is at most
$C(B)|b|$. Here $\chi_b=1$ on the right. Using
$N(\varepsilon)-3Q_b^2\varepsilon=3Q_b\varepsilon^2+\varepsilon^3$,
$\|Q_b\|_\infty\lesssim1$, and Lemma~\ref{SobolevLemma}, we obtain
\begin{align*}
&\left|\iint_{y_1>B/2}(\phi_{i,B}-1)
 \bigl[N(\varepsilon)-3Q_b^2\varepsilon\bigr](Q_b)_{y_1}\right|\\
&\quad\lesssim
\left[(1+B)e^{-B/(2\alpha_1)}+C(B)|b|\right]
\left(\mathcal N_{i,\mathrm{loc}}
      +\|\varepsilon\|_{L^2}\widetilde{\mathcal N}_i\right).
\end{align*}

The coefficient involving only the soliton tends to zero as $B$ grows.
Thus, with $B$ large and $\nu^*$ small depending on $\delta_0$, the
right-region estimate is
\begin{align*}
f_{1,1}^{>}
\leq{}&-2\iint_{y_1>B/2}\varepsilon_{y_1}^2\phi_{i,B}'
       -\frac12\iint_{y_1>B/2}\varepsilon_{y_2}^2\phi_{i,B}'
       -\frac\kappa2\iint_{y_1>B/2}\varepsilon^2\phi_{i,B}'
       +\delta_0\widetilde{\mathcal N}_i.
\end{align*}

\medskip
\noindent\textbf{The central region $\{|y_1|<B/2\}$.}
Here
\[
    \psi_B=1,
    \qquad
    \phi_{i,B}=1+e^{-B/(2\alpha_1)}\varphi(y_1).
\]
The exact expansion is
\begin{align*}
f_{1,1}^{\sim}
={}&-e^{-B/(2\alpha_1)}
\Bigg\{
\iint_{|y_1|<B/2}
 (3\varepsilon_{y_1}^2+\varepsilon_{y_2}^2)\varphi'\\
&\qquad
+\iint_{|y_1|<B/2}
 \bigl(\varphi'-\varphi'''-3Q^2\varphi'
       +6QQ_{y_1}\varphi\bigr)\varepsilon^2
\Bigg\}
+R_{\mathrm{Vir}}(\varepsilon),
\end{align*}
where
\begin{align*}
R_{\mathrm{Vir}}(\varepsilon)
=e^{-B/(2\alpha_1)}\iint_{|y_1|<B/2}
\Bigl\{&3(Q_b^2-Q^2)\varepsilon^2\varphi'
-3(Q_b^2-Q^2)_{y_1}\varepsilon^2\varphi\\
&+\frac32\varepsilon^4\varphi'
+4Q_b\varepsilon^3\varphi'
-2(Q_b)_{y_1}\varepsilon^3\varphi
\Bigr\}.
\end{align*}

\begin{lemma}\label{Virial2}
Let $v\in H^1(\mathbb R^2)$ satisfy
\[
(v,Q)=(v,\varphi\Lambda Q)
=(v,\varphi Q_{y_1})=(v,\varphi Q_{y_2})=0.
\]
There exist $\mu_0>0$ and $B_0>0$ such that, for $B\geq B_0$,
\begin{align*}
&\iint_{|y_1|<B/2}
\Bigl[(3v_{y_1}^2+v_{y_2}^2)\varphi'
 +(\varphi'-\varphi'''-3Q^2\varphi'
   +6QQ_{y_1}\varphi)v^2\Bigr]\\
&\qquad\geq
\mu_0\iint_{|y_1|<B/2}(|\nabla v|^2+v^2)\varphi'
-\frac{C}{\mu_0}e^{B/(4\alpha_1)}
 \iint v^2\phi_{i,B}'.
\end{align*}
\end{lemma}
The proof is given in Appendix~\ref{AppendixB}, Lemma~\ref{Virial}.
Applying Lemma~\ref{Virial2} to $v=\varepsilon$ and multiplying by
$e^{-B/(2\alpha_1)}$ gives
\begin{align*}
f_{1,1}^{\sim}
\leq{}&-\mu_0\iint_{|y_1|<B/2}
 (|\nabla\varepsilon|^2+\varepsilon^2)\phi_{i,B}'
+\frac{C}{\mu_0}e^{-B/(4\alpha_1)}
 \mathcal N_{i,\mathrm{loc}}
+R_{\mathrm{Vir}}(\varepsilon).
\end{align*}
The profile estimates and Lemma~\ref{Qdecay} give, uniformly in $B$,
\[
|Q_b^2-Q^2|+|(Q_b^2-Q^2)_{y_1}|\lesssim |b|
\qquad (|y_1|<B/2).
\]
We estimate the five terms in $R_{\mathrm{Vir}}$ separately. Since
$e^{-B/(2\alpha_1)}\varphi'=\phi_{i,B}'$ in the central region,
\begin{align*}
&e^{-B/(2\alpha_1)}
\left|\iint 3(Q_b^2-Q^2)\varepsilon^2\varphi'\right|
\lesssim |b|\mathcal N_{i,\mathrm{loc}}.
\end{align*}
Moreover,
\[
\frac{\varphi(y_1)}{\varphi'(y_1)}
=\alpha_1\bigl(1+e^{-y_1/\alpha_1}\bigr)
\lesssim C(B)
\qquad (|y_1|<B/2),
\]
so
\begin{align*}
&e^{-B/(2\alpha_1)}
\left|\iint 3(Q_b^2-Q^2)_{y_1}\varepsilon^2\varphi\right|\lesssim
C(B)|b|\iint\varepsilon^2\phi_{i,B}'
=C(B)|b|\mathcal N_{i,\mathrm{loc}}.
\end{align*}
For the quartic term, Lemma~\ref{SobolevLemma} with
$\theta=\phi_{i,B}'$ gives

\[
e^{-B/(2\alpha_1)}\iint_{|y_1|<B/2}\varepsilon^4\varphi'
=\iint_{|y_1|<B/2}\varepsilon^4\phi_{i,B}'
\lesssim
\|\varepsilon\|_{L^2}^2\widetilde{\mathcal N}_i.
\]
Since $Q_b$ is bounded on the central region, the same lemma and
Cauchy--Schwarz imply
\begin{align*}
e^{-B/(2\alpha_1)}
\left|\iint Q_b\varepsilon^3\varphi'\right|
&\lesssim
\iint\phi_{i,B}'|\varepsilon|^3\lesssim
\|\varepsilon\|_{L^2}\widetilde{\mathcal N}_i.
\end{align*}
Finally, $(Q_b)_{y_1}$ is bounded by Lemma~\ref{Qdecay}, and
$e^{-B/(2\alpha_1)}\varphi\lesssim C(B)\phi_{i,B}'$; hence
\begin{align*}
e^{-B/(2\alpha_1)}
\left|\iint(Q_b)_{y_1}\varepsilon^3\varphi\right|
&\lesssim
C(B)\iint\phi_{i,B}'|\varepsilon|^3\lesssim
C(B)\|\varepsilon\|_{L^2}\widetilde{\mathcal N}_i.
\end{align*}
Collecting these estimates,
\[
|R_{\mathrm{Vir}}(\varepsilon)|
\lesssim
C(B)\bigl(|b|+\|\varepsilon\|_{L^2}
          +\|\varepsilon\|_{L^2}^2\bigr)
\widetilde{\mathcal N}_i.
\]
For this fixed $B$, choose $\nu^*$ small enough that the last bound
is at most $\delta_0\widetilde{\mathcal N}_i$. We retain it as a global
error:

\[
f_{1,1}^{\sim}
\leq-\mu_0\iint_{|y_1|<B/2}
 (|\nabla\varepsilon|^2+\varepsilon^2)\phi_{i,B}'
+\frac{C}{\mu_0}e^{-B/(4\alpha_1)}\mathcal N_{i,\mathrm{loc}}
+\delta_0\widetilde{\mathcal N}_i.
\]
Now sum the three regional inequalities. Choose $\delta_0$ and then
$B$ so that

\[
\delta_0<\frac1{12}\min\{1/4,\kappa/2,\mu_0\},\qquad
\frac{C}{\mu_0}e^{-B/(4\alpha_1)}
<\frac14\min\{1/4,\kappa/2,\mu_0\}.
\]
The three global errors and the last localized mass term are then
absorbed into the full dissipation. For a sufficiently small $\mu>0$,
this yields
\begin{equation}\label{f11}
f_{1,1}
\leq
-\mu\iint
\Bigl[(\varepsilon_{y_1y_1}^2+\varepsilon_{y_2y_2}^2)\psi_B'
 +(\varepsilon_{y_1}^2+\varepsilon_{y_2}^2+\varepsilon^2)
   \phi_{i,B}'\Bigr].
\end{equation}

\paragraph{Step 2: Estimate of $f_{1,2}$.}
Write
\[
    f_{1,2}
    =2\left(\frac{\lambda_s}{\lambda}+b\right)
      \bigl(\mathcal I_Q+b\mathcal I_P\bigr),
\]
where $\mathcal I_Q$ and $\mathcal I_P$ are obtained by replacing
$\Lambda Q_b$ by $\Lambda Q$ and $\Lambda(\chi_bP)$, respectively.
Since $L\Lambda Q=-2Q$ and $(\varepsilon,Q)=0$, we have
$(L\varepsilon,\Lambda Q)=0$. Integrating the Laplacian terms by parts
and using
\[
3Q^2\varepsilon-\psi_BN(\varepsilon)
=(1-\psi_B)N(\varepsilon)
-\bigl[N(\varepsilon)-3Q_b^2\varepsilon\bigr]
-3(Q_b^2-Q^2)\varepsilon,
\]
we obtain the exact identity
\begin{align*}
\mathcal I_Q
={}&\iint\Bigl[-(\Lambda Q)_{y_1}\psi_B'\varepsilon
 +(1-\psi_B)\Delta(\Lambda Q)\varepsilon
 -\Lambda Q(1-\phi_{i,B})\varepsilon\Bigr]\\
&+\iint\Lambda Q(1-\psi_B)N(\varepsilon)
 -\iint\Lambda Q\bigl[N(\varepsilon)-3Q_b^2\varepsilon\bigr]
 -3\iint\Lambda Q(Q_b^2-Q^2)\varepsilon.
\end{align*}
We repeatedly use the weighted Cauchy--Schwarz inequality
\begin{equation*}
\left|\iint F\varepsilon\right|
\leq
\mathcal N_{i,\mathrm{loc}}^{1/2}
\left(\iint\frac{F^2}{\phi_{i,B}'}\right)^{1/2}.
\end{equation*}
By Lemma~\ref{Qdecay}, \eqref{psiderivativesquared}, and
\eqref{1minuspsi},
\begin{equation}\label{lambdaQCauchypsi}
\left|\iint_{y_1<-B/2}
 (1-\psi_B)\Delta(\Lambda Q)\varepsilon\right|
\lesssim
e^{-\frac B2[\sigma(M-1/2)+1/\alpha_1]}
\mathcal N_{i,\mathrm{loc}}^{1/2},
\end{equation}
and similarly
\[
\left|\iint(\Lambda Q)_{y_1}\psi_B'\varepsilon\right|
\lesssim e^{-B/(2\alpha_1)}
\mathcal N_{i,\mathrm{loc}}^{1/2}.
\]
Using $(\varepsilon,\varphi\Lambda Q)=0$, we write
\[
\iint\Lambda Q(1-\phi_{i,B})\varepsilon
=\iint\Lambda Q
\bigl(1+\varphi e^{-B/(2\alpha_1)}-\phi_{i,B}\bigr)\varepsilon.
\]
The bracket vanishes on $|y_1|<B/2$; on the left we use
\eqref{centralestimatetotheleft}, and on the right we use
Lemma~\ref{Qdecay} and the polynomial definition of $\phi_{i,B}$.
Cauchy--Schwarz therefore gives
\begin{equation}\label{lambdaQCauchy}
\left|\iint\Lambda Q(1-\phi_{i,B})\varepsilon\right|
\lesssim
Be^{-B/(2\alpha_1)}\mathcal N_{i,\mathrm{loc}}^{1/2}.
\end{equation}

For the nonlinear terms, Lemma~\ref{Qdecay} and the weight comparisons
imply $|\Lambda Q|\lesssim C(B)\phi_{i,B}'$. Since
$N(\varepsilon)-3Q_b^2\varepsilon=3Q_b\varepsilon^2+\varepsilon^3$,
\begin{align*}
\left|\iint\Lambda Q
 \bigl[N(\varepsilon)-3Q_b^2\varepsilon\bigr]\right|
&\lesssim
C(B)\iint\phi_{i,B}'
 \bigl(|Q_b|\varepsilon^2+|\varepsilon|^3\bigr).
\end{align*}
The first integral is bounded by $C(B)\mathcal N_{i,\mathrm{loc}}$.
For the second, Cauchy--Schwarz and Lemma~\ref{SobolevLemma}, applied
with $\theta=\phi_{i,B}'$, give
\[
\iint\phi_{i,B}'|\varepsilon|^3
\leq
\left(\iint\phi_{i,B}'\varepsilon^2\right)^{1/2}
\left(\iint\phi_{i,B}'\varepsilon^4\right)^{1/2}
\lesssim
\|\varepsilon\|_{L^2}\widetilde{\mathcal N}_i.
\]
Thus, after retaining the explicit constant from the global comparison
$|\Lambda Q|\lesssim Be^{B/(2\alpha_1)}\phi_{i,B}'$, we obtain
\begin{equation}\label{nonlinearityLambdaQ}
\left|\iint\Lambda Q
 \bigl[N(\varepsilon)-3Q_b^2\varepsilon\bigr]\right|
\lesssim
Be^{B/(2\alpha_1)}
\bigl(\|\varepsilon\|_{L^2}\widetilde{\mathcal N}_i
      +\mathcal N_{i,\mathrm{loc}}\bigr).
\end{equation}

For the profile difference, weighted Cauchy--Schwarz gives
\begin{align*}
\left|\iint\Lambda Q(Q_b^2-Q^2)\varepsilon\right|
&\leq
\mathcal N_{i,\mathrm{loc}}^{1/2}
\left(\iint
\frac{(\Lambda Q)^2(Q_b^2-Q^2)^2}{\phi_{i,B}'}
\right)^{1/2}.
\end{align*}
The profile bounds give $|Q_b^2-Q^2|\lesssim|b|e^{-a|y_2|}$ for
some $a>0$. The longitudinal exponential decay required in this
integral is supplied by $\Lambda Q$, through Lemma~\ref{Qdecay}.
Together with the lower bound for $\phi_{i,B}'$, this gives
\[
\left|\iint\Lambda Q(Q_b^2-Q^2)\varepsilon\right|
\lesssim
|b|e^{B/(4\alpha_1)}\mathcal N_{i,\mathrm{loc}}^{1/2}.
\]

Finally,  $1-\psi_B$ is supported in $\{y_1<-B/2\}$. For the term
linear in $\varepsilon$, write
\[
Q_b^2=Q^2+2b\chi_bPQ+b^2\chi_b^2P^2.
\]
On this left half-plane, the profile estimates imply, for some $a>0$,
\[
|Q|+|\Lambda Q|\leq Ce^{y_1/\alpha_1}e^{-a|y_2|},\qquad
|\chi_bP|\leq Ce^{-a|y_2|},\qquad
\frac1{\phi_{i,B}'}\leq\alpha_1 e^{B/(2\alpha_1)}e^{-y_1/\alpha_1}.
\]

Weighted Cauchy--Schwarz therefore gives
\begin{align*}
&\left|\iint\Lambda Q(1-\psi_B)Q_b^2\varepsilon\right|\\
&\qquad\lesssim
\left(e^{-B/\alpha_1}+|b|e^{-B/(2\alpha_1)}+b^2\right)
\mathcal N_{i,\mathrm{loc}}^{1/2}.
\end{align*}

For the quadratic and cubic terms, the weight comparisons and
Lemma~\ref{SobolevLemma} yield

\[
\left|\iint\Lambda Q(1-\psi_B)
 \bigl(3Q_b\varepsilon^2+\varepsilon^3\bigr)\right|
\lesssim C(B)\left(
 \mathcal N_{i,\mathrm{loc}}
 +\|\varepsilon\|_{L^2}\widetilde{\mathcal N}_i\right).
\]

Keeping the quadratic contribution at this order and combining these
bounds with the preceding estimates, we obtain
\[
|\mathcal I_Q|
\lesssim
Be^{-B/(2\alpha_1)}\mathcal N_{i,\mathrm{loc}}^{1/2}
+C(B)\Bigl(
 \mathcal N_{i,\mathrm{loc}}
 +|b|\mathcal N_{i,\mathrm{loc}}^{1/2}
 +\|\varepsilon\|_{L^2}\widetilde{\mathcal N}_i
\Bigr).
\]

We next estimate $\mathcal I_P$. An integration by parts gives the exact
formula
\begin{align*}
\mathcal I_P
={}&-\iint\Delta\Lambda(\chi_bP)\psi_B\varepsilon
    -\iint[\Lambda(\chi_bP)]_{y_1}\psi_B'\varepsilon\\
&+\iint\Lambda(\chi_bP)\phi_{i,B}\varepsilon
 -\iint\Lambda(\chi_bP)\psi_BN(\varepsilon).
\end{align*}
On the support of $\chi_b$,
\[
|\Lambda(\chi_bP)|+|\Delta\Lambda(\chi_bP)|
+|[\Lambda(\chi_bP)]_{y_1}|
\lesssim
\begin{cases}
(1+|y_2|)e^{-|y_2|/2},
&-2|b|^{-\gamma}\leq y_1\leq0,\\
(1+|y_1|+|y_2|)e^{-(|y_1|+|y_2|)/2},
&y_1\geq0.
\end{cases}
\]
For the first two linear terms, weighted Cauchy--Schwarz gives, for
example,
\begin{align*}
\left|\iint\Delta\Lambda(\chi_bP)\psi_B\varepsilon\right|
&\leq
\mathcal N_{i,\mathrm{loc}}^{1/2}
\left(
\iint\frac{|\Delta\Lambda(\chi_bP)|^2\psi_B^2}
                 {\phi_{i,B}'}
\right)^{1/2}.
\end{align*}
On $[-2|b|^{-\gamma},-B/2-1]$, we use
\begin{equation}\label{psisquaredderivativephi}
    \frac{\psi_B^2}{\phi_{i,B}'}
    \lesssim e^{\sigma(y_1+MB)},
\end{equation}
so the integral in $y_1$ is uniformly bounded. On
$[-B/2-1,-B/2]$, the lower bound
$\phi_{i,B}'\gtrsim e^{-B/\alpha_1}$ is sufficient; on
$[-B/2,B/2]$, we use the explicit formula
$\phi_{i,B}'=\alpha_1^{-1}e^{y_1/\alpha_1-B/(2\alpha_1)}$; and on
$[B/2,+\infty)$, the profile decays exponentially while
$\phi_{i,B}'$ is bounded below by a polynomial weight. The term
containing $\psi_B'$ is handled in exactly the same four regions using
\eqref{psiderivativesquared}. Hence
\begin{equation}\label{lambdachibP}
\left|\iint\Delta\Lambda(\chi_bP)\psi_B\varepsilon\right|
+\left|\iint[\Lambda(\chi_bP)]_{y_1}
              \psi_B'\varepsilon\right|
\lesssim C(B)\mathcal N_{i,\mathrm{loc}}^{1/2}.
\end{equation}

For the mass term, Cauchy--Schwarz gives
\begin{align*}
\left|\iint\Lambda(\chi_bP)\phi_{i,B}\varepsilon\right|
&\leq
\mathcal N_{i,\mathrm{loc}}^{1/2}
\left(
\iint\frac{|\Lambda(\chi_bP)|^2\phi_{i,B}^2}
                 {\phi_{i,B}'}
\right)^{1/2}.
\end{align*}
On $(-2|b|^{-\gamma},-B/2-1]$, the profile is bounded in $y_1$ and
exponentially localized in $y_2$, while \eqref{phiderivativephi}
gives
\[
    \frac{\phi_{i,B}^2}{\phi_{i,B}'}\lesssim e^{\varpi B}.
\]
The interval has length $O(|b|^{-\gamma})$, and hence
\[
\iint_{-2|b|^{-\gamma}<y_1<-B/2-1}
\frac{|\Lambda(\chi_bP)|^2\phi_{i,B}^2}{\phi_{i,B}'}
\lesssim C(B)|b|^{-\gamma}.
\]
The transition, central, and right regions contribute only $C(B)$ by
the same weight bounds used in \eqref{lambdachibP}. Therefore,
\begin{equation}\label{minusgamma}
\left|\iint\Lambda(\chi_bP)\phi_{i,B}\varepsilon\right|
\lesssim
C(B)|b|^{-\gamma/2}\mathcal N_{i,\mathrm{loc}}^{1/2}.
\end{equation}

It remains to estimate the nonlinear term. Using
\[
N(\varepsilon)=3Q_b^2\varepsilon+3Q_b\varepsilon^2+\varepsilon^3,
\]
weighted Cauchy--Schwarz and the same support computation as above give
\[
\left|\iint\Lambda(\chi_bP)\psi_BQ_b^2\varepsilon\right|
\lesssim
C(B)|b|^{-\gamma/2}\mathcal N_{i,\mathrm{loc}}^{1/2}.
\]
For the quadratic term, split $Q_b=Q+b\chi_bP$. The part containing
$Q$ is exponentially localized and is bounded by
$C(B)\mathcal N_{i,\mathrm{loc}}$, hence by
$C(B)\mathcal N_{i,\mathrm{loc}}^{1/2}$ under \eqref{H1}; the tail part
is bounded, using \eqref{eq:psi-by-phi-prime}, by
$C(B)|b|\widetilde{\mathcal N}_i$. Finally,
Lemma~\ref{SobolevLemma}, with $\theta=\phi_{i,B}'$, gives
\[
\left|\iint\Lambda(\chi_bP)\psi_B\varepsilon^3\right|
\lesssim C(B)\|\varepsilon\|_{L^2}\widetilde{\mathcal N}_i.
\]
Consequently,
\[
\left|\iint\Lambda(\chi_bP)\psi_BN(\varepsilon)\right|
\lesssim
C(B)\left(
 |b|^{-\gamma/2}\mathcal N_{i,\mathrm{loc}}^{1/2}
 +(\|\varepsilon\|_{L^2}+|b|)\widetilde{\mathcal N}_i
\right),
\]
and therefore
\[
|\mathcal I_P|
\lesssim
C(B)\left(
 |b|^{-\gamma/2}\mathcal N_{i,\mathrm{loc}}^{1/2}
 +(\|\varepsilon\|_{L^2}+|b|)\widetilde{\mathcal N}_i
\right).
\]

From \eqref{eq:x1x2lambda} and \eqref{McomparedtoN},
\begin{equation}\label{lambadoubleestimate}
\left|\frac{\lambda_s}{\lambda}+b\right|
\lesssim
e^{B/(4\alpha_1)}\mathcal N_{i,\mathrm{loc}}^{1/2}
+b^2
+\min\left\{
 e^{B/(2\alpha_1)}\|\varepsilon\|_{L^2}
       \widetilde{\mathcal N}_i,
 \|\varepsilon\|_{L^2}\mathcal N_i
\right\}.
\end{equation}

The preceding estimates give
\begin{align*}
|\mathcal I_Q+b\mathcal I_P|
\lesssim{}&Be^{-B/(2\alpha_1)}\mathcal N_{i,\mathrm{loc}}^{1/2}
+C(B)\mathcal N_{i,\mathrm{loc}}\\
&+C(B)|b|^{1-\gamma/2}\mathcal N_{i,\mathrm{loc}}^{1/2}
+C(B)(\|\varepsilon\|_{L^2}+|b|)\widetilde{\mathcal N}_i.
\end{align*}
For the first three terms we use the version of
\eqref{lambadoubleestimate} containing
$e^{B/(2\alpha_1)}\|\varepsilon\|_{L^2}\widetilde{\mathcal N}_i$;
for the last term we use the alternative containing
$\|\varepsilon\|_{L^2}\mathcal N_i$. Thus, after multiplication,
\begin{align*}
|f_{1,2}|
\lesssim{}&Be^{-B/(4\alpha_1)}\mathcal N_{i,\mathrm{loc}}
+C(B)\delta(\nu^*)\widetilde{\mathcal N}_i\\
&+C(B)b^2\mathcal N_{i,\mathrm{loc}}^{1/2}
+C(B)\delta(\nu^*)\mathcal N_{i,\mathrm{loc}}.
\end{align*}
Here we used \eqref{H1}, \eqref{importantbound}, and the fact that
$1-\gamma/2>0$. Fixing $B$ large and then reducing $\nu^*$ absorbs all
terms except $b^2\mathcal N_{i,\mathrm{loc}}^{1/2}$, for which
\[
    b^2\mathcal N_{i,\mathrm{loc}}^{1/2}
    \leq\frac{\mu}{200}\mathcal N_{i,\mathrm{loc}}+Cb^4.
\]
Consequently, we obtain
\begin{equation}\label{f12}
|f_{1,2}|
\leq
\frac\mu{100}\iint
 (\varepsilon_{y_1}^2+\varepsilon_{y_2}^2+\varepsilon^2)
 \phi_{i,B}'
+Cb^4.
\end{equation}

\begin{remark}
The orthogonality condition
$(\varepsilon,\varphi\Lambda Q)=0$ is essential in
\eqref{lambdaQCauchy}. A direct estimate of
$\iint\Lambda Q(1-\phi_{i,B})\varepsilon$ would yield a term of order
$e^{-B/(4\alpha_1)}\|\varepsilon\|_{L^2}$. After multiplication by the
modulation coefficient, this produces an error of the size of the
energy rather than a small multiple of the dissipation. The exact
cancellation in the central region removes this leading contribution.
\end{remark}

\paragraph{Step 3: Estimates of $f_{1,3}$ and $f_{1,4}$.}
Set
\[
\mathcal H_0(\varepsilon)
:=-(\psi_B\varepsilon_{y_1})_{y_1}
  -(\psi_B\varepsilon_{y_2})_{y_2}
  +\phi_{i,B}\varepsilon.
\]
We first verify the exact identity used for $f_{1,3}$. Since
\begin{align*}
&(Q_b)_{y_1}\bigl[N(\varepsilon)-3Q_b^2\varepsilon\bigr]
+\varepsilon_{y_1}N(\varepsilon)=
\frac14\partial_{y_1}
\Bigl[(Q_b+\varepsilon)^4-Q_b^4-4Q_b^3\varepsilon\Bigr],
\end{align*}
we have
\begin{align*}
&-\iint\psi_B(Q_b+\varepsilon)_{y_1}N(\varepsilon)=
\frac14\iint\psi_B'
 \Bigl[(Q_b+\varepsilon)^4-Q_b^4-4Q_b^3\varepsilon\Bigr]
-3\iint\psi_BQ_b^2(Q_b)_{y_1}\varepsilon.
\end{align*}
Moreover, using $(L\varepsilon,Q_{y_1})=0$,
\begin{align*}
\iint Q_{y_1}\mathcal H_0(\varepsilon)
={}&\iint Q_{y_1}\Bigl[-\psi_B'\varepsilon_{y_1}
 +(1-\psi_B)\Delta\varepsilon
 -(1-\phi_{i,B})\varepsilon
 +3Q^2\varepsilon\Bigr].
\end{align*}
Combining the last term with
$-3\iint\psi_BQ_b^2(Q_b)_{y_1}\varepsilon$ gives
\begin{align*}
&3\iint Q^2Q_{y_1}\varepsilon
-3\iint\psi_BQ_b^2(Q_b)_{y_1}\varepsilon=
3\iint(1-\psi_B)Q^2Q_{y_1}\varepsilon
-\iint\psi_B\varepsilon\,\partial_{y_1}(Q_b^3-Q^3).
\end{align*}
Therefore,
\begin{align*}
\iint(Q_b+\varepsilon)_{y_1}\mathcal H(\varepsilon)
={}&\frac14\iint\psi_B'
 \Bigl[(Q_b+\varepsilon)^4-Q_b^4-4Q_b^3\varepsilon\Bigr]\\
&+\iint(b\chi_bP+\varepsilon)_{y_1}\mathcal H_0(\varepsilon)\\
&+\iint Q_{y_1}\Bigl[-\psi_B'\varepsilon_{y_1}
 +(1-\psi_B)(\Delta\varepsilon+3Q^2\varepsilon)
 -(1-\phi_{i,B})\varepsilon\Bigr]\\
&-\iint\psi_B\varepsilon\,\partial_{y_1}(Q_b^3-Q^3).
\end{align*}
This displays both the term
$3(1-\psi_B)Q^2\varepsilon$ and the negative sign in the last
integral.

For the first integral, $0\leq\psi_B'\lesssim\phi_{i,B}'$ by
\eqref{psiandphi}. Hence Lemma~\ref{SobolevLemma}, applied with
$\theta=\phi_{i,B}'$, and Lemma~\ref{Qdecay} give
\begin{align*}
&\left|\iint\psi_B'
 \Bigl[(Q_b+\varepsilon)^4-Q_b^4-4Q_b^3\varepsilon\Bigr]\right|\\
&\qquad\lesssim
\|\varepsilon\|_{L^2}^2\widetilde{\mathcal N}_i
+\bigl(e^{-B/\alpha_1}+b^2\bigr)
\mathcal N_{i,\mathrm{loc}}.
\end{align*}

For the profile-gradient term, integrating the second derivatives in
$\mathcal H_0$ onto $(b\chi_bP)_{y_1}$ gives a sum of the same three
weighted pairings that appear in \eqref{lambdachibP} and
\eqref{minusgamma}. The factor $b$, together with the support length
$O(|b|^{-\gamma})$, therefore yields
\[
\left|\iint(b\chi_bP)_{y_1}\mathcal H_0(\varepsilon)\right|
\lesssim
C(B)|b|^{1-\gamma/2}\mathcal N_{i,\mathrm{loc}}^{1/2}.
\]
For the pure radiation term, direct integration by parts gives the exact
identity
\[
\iint\varepsilon_{y_1}\mathcal H_0(\varepsilon)
=-\frac12\iint\psi_B'|\nabla\varepsilon|^2
 -\frac12\iint\phi_{i,B}'\varepsilon^2,
\]
and hence its absolute value is bounded by
$C\widetilde{\mathcal N}_i$.

We now treat the terms containing $Q_{y_1}$. First,
\begin{align*}
\left|\iint Q_{y_1}\psi_B'\varepsilon_{y_1}\right|
&\leq
\left(\iint\phi_{i,B}'\varepsilon_{y_1}^2\right)^{1/2}
\left(\iint Q_{y_1}^2\frac{(\psi_B')^2}{\phi_{i,B}'}\right)^{1/2}\\
&\lesssim e^{-B/(2\alpha_1)}\widetilde{\mathcal N}_i^{1/2},
\end{align*}
by Lemma~\ref{Qdecay} and \eqref{psiderivativesquared}. Next, after
two integrations by parts,
\[
\iint Q_{y_1}(1-\psi_B)\Delta\varepsilon
=
\iint\Delta\bigl[Q_{y_1}(1-\psi_B)\bigr]\varepsilon,
\]
which is controlled by weighted Cauchy--Schwarz, Lemma~\ref{Qdecay},
and \eqref{1minuspsi}, exactly as in \eqref{lambdaQCauchypsi}. The
additional term
$3(1-\psi_B)Q^2Q_{y_1}\varepsilon$ has stronger exponential decay and
is estimated in the same way.

For the weight correction, the orthogonality condition
$(\varepsilon,\varphi Q_{y_1})=0$ gives
\[
\iint Q_{y_1}(1-\phi_{i,B})\varepsilon
=
\iint Q_{y_1}
\bigl(1+\varphi e^{-B/(2\alpha_1)}-\phi_{i,B}\bigr)\varepsilon.
\]
The bracket vanishes on $|y_1|<B/2$, while the two exterior regions are
estimated as in \eqref{lambdaQCauchy}; thus
\[
\left|\iint Q_{y_1}(1-\phi_{i,B})\varepsilon\right|
\lesssim Be^{-B/(2\alpha_1)}\mathcal N_{i,\mathrm{loc}}^{1/2}.
\]
Finally,
\[
\partial_{y_1}(Q_b^3-Q^3)
=3\bigl[Q_b^2(Q_b)_{y_1}-Q^2Q_{y_1}\bigr].
\]
Lemma~\ref{Qdecay} and the definition of $Q_b$ give the weighted bound
\[
\iint\frac{\psi_B^2
 |\partial_{y_1}(Q_b^3-Q^3)|^2}{\phi_{i,B}'}
\lesssim C(B)b^2,
\]
and hence
\[
\left|\iint\psi_B\varepsilon\,
\partial_{y_1}(Q_b^3-Q^3)\right|
\lesssim C(B)|b|\mathcal N_{i,\mathrm{loc}}^{1/2}.
\]
Using the analogue of \eqref{lambadoubleestimate} for
$|(x_1)_s/\lambda-1|$, which follows from
\eqref{eq:x1x2lambda}, we conclude
\begin{equation}\label{f13}
f_{1,3}
\leq
\frac\mu{100}\iint
 (\varepsilon_{y_1}^2+\varepsilon_{y_2}^2+\varepsilon^2)
 \phi_{i,B}'
+Cb^4.
\end{equation}

Since the weights are independent of $y_2$, the transverse identity is
\begin{align*}
\iint(Q_b+\varepsilon)_{y_2}\mathcal H(\varepsilon)
={}&\iint(b\chi_bP+\varepsilon)_{y_2}\mathcal H_0(\varepsilon)\\
&+\iint Q_{y_2}\Bigl[-\psi_B'\varepsilon_{y_1}
 +(1-\psi_B)(\Delta\varepsilon+3Q^2\varepsilon)
 -(1-\phi_{i,B})\varepsilon\Bigr]\\
&-\iint\psi_B\varepsilon\,\partial_{y_2}(Q_b^3-Q^3).
\end{align*}
Here
$\iint\varepsilon_{y_2}\mathcal H_0(\varepsilon)=0$,
$LQ_{y_2}=0$, and
$(\varepsilon,\varphi Q_{y_2})=0$. The preceding estimates and the
modulation bound for $(x_2)_s/\lambda$ give
\begin{equation}\label{f14}
|f_{1,4}|
\leq
\frac\mu{100}\iint
 (|\nabla\varepsilon|^2+\varepsilon^2)\phi_{i,B}'
+Cb^4.
\end{equation}

\paragraph{Step 4: Estimate of $f_{1,5}$.}
Set
\[
    \zeta_b:=\chi_b+\gamma y_1\chi_b'.
\]
We first record the exact formula obtained after integrating the linear
terms by parts. Since $\psi_B$ depends only on $y_1$,
\begin{align*}
\iint\zeta_bP\,\mathcal H(\varepsilon)
={}&-\iint\Delta(\zeta_bP)\psi_B\varepsilon
    -\iint(\zeta_bP)_{y_1}\psi_B'\varepsilon+\iint\zeta_bP\phi_{i,B}\varepsilon
 -\iint\zeta_bP\psi_BN(\varepsilon).
\end{align*}
Indeed,
\begin{align*}
-\iint\zeta_bP(\psi_B\varepsilon_{y_1})_{y_1}
&=-\iint\bigl[\psi_B(\zeta_bP)_{y_1}\bigr]_{y_1}\varepsilon,\\
-\iint\zeta_bP(\psi_B\varepsilon_{y_2})_{y_2}
&=-\iint\psi_B(\zeta_bP)_{y_2y_2}\varepsilon.
\end{align*}
Consequently,
\begin{align*}
f_{1,5}
=-2b_s\Bigg[&-\iint\Delta(\zeta_bP)\psi_B\varepsilon
 -\iint(\zeta_bP)_{y_1}\psi_B'\varepsilon
 +\iint\zeta_bP\phi_{i,B}\varepsilon-\iint\zeta_bP\psi_BN(\varepsilon)\Bigg].
\end{align*}

For the three linear terms, weighted Cauchy--Schwarz gives
\begin{align*}
\left|\iint\Delta(\zeta_bP)\psi_B\varepsilon\right|
&\leq\mathcal N_{i,\mathrm{loc}}^{1/2}
 \left(\iint
 \frac{|\Delta(\zeta_bP)|^2\psi_B^2}{\phi_{i,B}'}
 \right)^{1/2},\\
\left|\iint(\zeta_bP)_{y_1}\psi_B'\varepsilon\right|
&\leq\mathcal N_{i,\mathrm{loc}}^{1/2}
 \left(\iint
 \frac{|(\zeta_bP)_{y_1}|^2(\psi_B')^2}{\phi_{i,B}'}
 \right)^{1/2},\\
\left|\iint\zeta_bP\phi_{i,B}\varepsilon\right|
&\leq\mathcal N_{i,\mathrm{loc}}^{1/2}
 \left(\iint
 \frac{|\zeta_bP|^2\phi_{i,B}^2}{\phi_{i,B}'}
 \right)^{1/2}.
\end{align*}
On $[-2|b|^{-\gamma},-B/2]$, the functions
$\zeta_bP$ and their required derivatives are bounded and exponentially
localized in $y_2$. The estimates
\eqref{psisquaredderivativephi}, \eqref{psiderivativesquared}, and
\eqref{sharper} make the corresponding $y_1$-integrals uniformly
finite. On $[-B/2,+\infty)$, the conclusion follows from the defining
bounds for the weights and the exponential decay of $P$ and its
$y_1$-derivatives. Hence
\begin{align*}
&\left|\iint\Delta(\zeta_bP)\psi_B\varepsilon\right|
+\left|\iint(\zeta_bP)_{y_1}\psi_B'\varepsilon\right|
+\left|\iint\zeta_bP\phi_{i,B}\varepsilon\right|\lesssim C(B)\mathcal N_{i,\mathrm{loc}}^{1/2}.
\end{align*}

For the nonlinear term, expand
\[
N(\varepsilon)=3Q_b^2\varepsilon+3Q_b\varepsilon^2+\varepsilon^3.
\]
The term linear in $\varepsilon$ is estimated by weighted
Cauchy--Schwarz:
\[
\left|\iint\zeta_bP\psi_BQ_b^2\varepsilon\right|
\lesssim C(B)\mathcal N_{i,\mathrm{loc}}^{1/2}.
\]
For the quadratic term, split $Q_b=Q+b\chi_bP$. The part containing $Q$
is exponentially localized and is bounded by
$C(B)\mathcal N_{i,\mathrm{loc}}$, hence by
$C(B)\mathcal N_{i,\mathrm{loc}}^{1/2}$ under \eqref{H1}. The tail part
is bounded by
\[
C(B)|b|\iint\phi_{i,B}'\varepsilon^2
\lesssim C(B)|b|\widetilde{\mathcal N}_i,
\]
where we used \eqref{eq:psi-by-phi-prime}. Finally,
Cauchy--Schwarz and Lemma~\ref{SobolevLemma}, applied with
$\theta=\phi_{i,B}'$, give
\begin{align*}
\iint|\zeta_bP|\psi_B|\varepsilon|^3
&\lesssim C(B)\iint\phi_{i,B}'|\varepsilon|^3\lesssim C(B)\|\varepsilon\|_{L^2}
\widetilde{\mathcal N}_i.
\end{align*}
Therefore,
\begin{align*}
\left|\iint\zeta_bP\psi_BN(\varepsilon)\right|
\lesssim
C(B)\left[
 \mathcal N_{i,\mathrm{loc}}^{1/2}
 +(\|\varepsilon\|_{L^2}+|b|)\widetilde{\mathcal N}_i
\right].
\end{align*}

By \eqref{eq:b} and \eqref{McomparedtoN},
\begin{align*}
|b_s|
\lesssim{}&e^{B/(2\alpha_1)}\mathcal N_{i,\mathrm{loc}}+b^2+\min\left\{
 e^{B/(2\alpha_1)}\|\varepsilon\|_{L^2}
       \widetilde{\mathcal N}_i,
 \|\varepsilon\|_{L^2}\mathcal N_i
\right\}.
\end{align*}
Multiplying the last two estimates, all terms containing at least one
factor $\mathcal N_{i,\mathrm{loc}}$, $\widetilde{\mathcal N}_i$,
$\|\varepsilon\|_{L^2}$, or $|b|$ are bounded, by \eqref{H1}, by a
small multiple of $\widetilde{\mathcal N}_i$, except for
$b^2\mathcal N_{i,\mathrm{loc}}^{1/2}$. For this term we use
\[
    b^2\mathcal N_{i,\mathrm{loc}}^{1/2}
    \leq\frac{\mu}{200}\mathcal N_{i,\mathrm{loc}}+Cb^4.
\]
Thus
\begin{equation}\label{f15}
|f_{1,5}|
\leq
\frac\mu{100}\iint
 (|\nabla\varepsilon|^2+\varepsilon^2)\phi_{i,B}'
+Cb^4.
\end{equation}

\paragraph{Step 5: Estimate of $f_{1,6}$.}
Integrating the linear terms by parts gives the exact identity
\begin{align*}
f_{1,6}
=2\Bigg[&-\iint\Delta\Psi_b\,\psi_B\varepsilon
 -\iint(\Psi_b)_{y_1}\psi_B'\varepsilon
 +\iint\Psi_b\phi_{i,B}\varepsilon-\iint\Psi_b\psi_BN(\varepsilon)\Bigg].
\end{align*}

Weighted Cauchy--Schwarz yields
\begin{align*}
\left|\iint\Delta\Psi_b\,\psi_B\varepsilon\right|
&\leq\mathcal N_{i,\mathrm{loc}}^{1/2}
 \left(\iint\frac{|\Delta\Psi_b|^2\psi_B^2}{\phi_{i,B}'}\right)^{1/2},\\
\left|\iint(\Psi_b)_{y_1}\psi_B'\varepsilon\right|
&\leq\mathcal N_{i,\mathrm{loc}}^{1/2}
 \left(\iint\frac{|(\Psi_b)_{y_1}|^2(\psi_B')^2}{\phi_{i,B}'}\right)^{1/2},\\
\left|\iint\Psi_b\phi_{i,B}\varepsilon\right|
&\leq\mathcal N_{i,\mathrm{loc}}^{1/2}
 \left(\iint\frac{|\Psi_b|^2\phi_{i,B}^2}{\phi_{i,B}'}\right)^{1/2}.
\end{align*}
Lemma~\ref{lemma:Psi_b_estimates}, together with
\eqref{psisquaredderivativephi}, \eqref{psiderivativesquared}, and
\eqref{sharper}, implies
\[
\iint\frac{|\Delta\Psi_b|^2\psi_B^2
              +|(\Psi_b)_{y_1}|^2(\psi_B')^2
              +|\Psi_b|^2\phi_{i,B}^2}
             {\phi_{i,B}'}
\lesssim C(B)b^4.
\]
Indeed, the core part of $\Psi_b$ is $O(b^2)$. Its transition part is
$O(|b|^{1+\gamma})$ and is supported where
$y_1\sim-|b|^{-\gamma}$; the exponential factors in
\eqref{psisquaredderivativephi}, \eqref{psiderivativesquared}, and
\eqref{sharper} then give
$e^{-c|b|^{-\gamma}}\lesssim |b|^m$ for every fixed $m>0$.
Consequently,
\begin{align*}
&\left|\iint\Delta\Psi_b\,\psi_B\varepsilon\right|
+\left|\iint(\Psi_b)_{y_1}\psi_B'\varepsilon\right|
+\left|\iint\Psi_b\phi_{i,B}\varepsilon\right|\lesssim C(B)b^2\mathcal N_{i,\mathrm{loc}}^{1/2}.
\end{align*}

For the nonlinear term, expand $N(\varepsilon)$ as above. The part
linear in $\varepsilon$ is controlled by the same weighted
Cauchy--Schwarz estimate and contributes
$C(B)b^2\mathcal N_{i,\mathrm{loc}}^{1/2}$. For the quadratic and cubic
parts, the core component of $\Psi_b$ is again $O(b^2)$, whereas the
transition component is $O(|b|^{1+\gamma})$. Using
\eqref{eq:psi-by-phi-prime} and Lemma~\ref{SobolevLemma} with
$\theta=\phi_{i,B}'$, we obtain
\begin{align*}
\left|\iint\Psi_b\psi_BN(\varepsilon)\right|
\lesssim{}&C(B)b^2\mathcal N_{i,\mathrm{loc}}^{1/2}+C(B)|b|^{1+\gamma}\|\varepsilon\|_{L^2}
\widetilde{\mathcal N}_i.
\end{align*}
Finally,
\[
    b^2\mathcal N_{i,\mathrm{loc}}^{1/2}
    \leq\frac{\mu}{200}\mathcal N_{i,\mathrm{loc}}+Cb^4,
\]
while the second term is a small multiple of
$\widetilde{\mathcal N}_i$ by \eqref{H1}. Therefore,
\begin{equation}\label{f16}
|f_{1,6}|
\leq
\frac\mu{100}\iint
 (|\nabla\varepsilon|^2+\varepsilon^2)\phi_{i,B}'
+Cb^4.
\end{equation}

Combining \eqref{f11}, \eqref{f12}, \eqref{f13}, \eqref{f14},
\eqref{f15}, and \eqref{f16}, and decreasing $\mu$ once more, yields
\[
    f_1
    \leq
    -\mu\iint
    (|\nabla\varepsilon|^2+\varepsilon^2)\phi_{i,B}'
    +Cb^4.
\]

\section{Proof of Lemma \ref{lemma17}:}

It remains to estimate the scaling-drift contribution $f_2$. Recall 
 $$f_2=2 \frac{\lambda_s}{\lambda}\iint \Lambda \varepsilon  \Big(-(\psi_B\varepsilon_{y_1})_{y_1}-(\psi_B\varepsilon_{y_2})_{y_2}+\varepsilon \phi_{i,B}-\psi_B[(Q_b+\varepsilon)^3-Q_b^3]\Big)-j\frac{\lambda_s}{\lambda}\mathcal{F}_i.$$
 Integration by parts gives
$$2\iint \Lambda \varepsilon ([-(\psi_B\varepsilon_{y_1})]_{y_1}+[-(\psi_B\varepsilon_{y_2})]_{y_2})=\iint (2\psi_B-y_1\psi_B')(\varepsilon_{y_1}^2+\varepsilon_{y_2}^2),$$
and
$$2\iint \Lambda \varepsilon \phi_{i,B}\varepsilon=-\iint y_1\phi_{i,B}'\varepsilon^2.$$
We have the following identity for $p>1$ and $\Lambda^{(p)}=\frac{2}{p-1}+y\cdot\nabla$:
\begin{equation*}
\begin{split}
\iint \Lambda^{(p)} \varepsilon \psi_B [(Q_b+\varepsilon)^p-Q_b^p] 
&= \frac{1}{p+1}\iint \Big( \frac{4}{p-1}\psi_B-y_1\psi_B'\Big)\\
&\qquad \times [(Q_b+\varepsilon)^{p+1}-Q_b^{p+1}-(p+1)Q_b^p\varepsilon] \\
&\quad -\iint \psi_B\Lambda^{(p)} Q_b [(Q_b+\varepsilon)^p-Q_b^p-pQ_b^{p-1}\varepsilon].
\end{split}
\end{equation*}
This identity allows us to rewrite $f_2$ as
\begin{equation*}
\begin{split}
f_{2}&= \frac{\lambda_s}{\lambda}\iint [(2-j)\psi_B-y_1\psi_B'](\varepsilon_{y_1}^2+\varepsilon_{y_2}^2)- \frac{\lambda_s}{\lambda}\iint [j\phi_{i,B}+y_1\phi_{i,B}']\varepsilon^2\\
&-\frac{\lambda_s}{\lambda}\iint \frac{1}{2}[(2-j)\psi_B-y_1\psi_B'][(Q_b+\varepsilon)^4-Q_b^4-4Q_b^3\varepsilon]\\
&+\frac{\lambda_s}{\lambda}\iint 2\psi_B \Lambda Q_b[(Q_b+\varepsilon)^3-Q_b^3-3Q_b^2\varepsilon].
\end{split}
\end{equation*}

\textbf{1. Estimates of Non-Problematic Terms.}
We begin by bounding the terms in regions where the weights do not exhibit dangerous polynomial growth. 
Using the bound on the weights \eqref{eq:psi-by-phi-prime} and \eqref{weightcontroldrift}, specifically $|(2-j)\psi_B-y_1\psi_B'| \lesssim e^{\varpi B}\phi_{i,B}',$ we have:
\[ \Big|\iint [(2-j)\psi_B-y_1\psi_B'](\varepsilon_{y_1}^2+\varepsilon_{y_2}^2)\Big| \leq C(B) \iint (\varepsilon_{y_1}^2+\varepsilon_{y_2}^2)\phi_{i,B}', \]
and for the nonlinear term involving the potential we employ the weighted Sobolev Lemma \ref{SobolevLemma} to get 
\begin{align*}
\Big|\iint [(2-j)\psi_B-&y_1\psi_B'][(Q_b+\varepsilon)^4-Q_b^4-4Q_b^3\varepsilon]\Big| \\
&\leq C(B)\iint \phi_{i,B}'\varepsilon^4 + C(B)\iint \phi_{i,B}'Q_b^2\varepsilon^2 \\
&\leq C(B)\|\varepsilon\|_{L^2}^2\widetilde{\mathcal{N}}_{i} + C(B)\|Q_b\|_{L^{\infty}}^2\iint \varepsilon^2\phi_{i,B}'.
\end{align*}
Similarly, for the interaction with $\Lambda Q_b$:
\begin{align*}
\Big|\iint 2\psi_B \Lambda Q_b&[(Q_b+\varepsilon)^3-Q_b^3-3Q_b^2\varepsilon]\Big| \\
&\lesssim \|\Lambda Q_b\|_{L^{\infty}}C(B)\iint |\varepsilon|^3\phi_{i,B}' + \|Q_b\Lambda Q_b\|_{L^{\infty}}e^{\frac{B}{\alpha_1}}\iint \varepsilon^2\phi_{i,B}' \\
&\lesssim C(B)\|\varepsilon\|_{L^2}\widetilde{\mathcal{N}}_{i} +C(B)\iint \varepsilon^2\phi_{i,B}'.
\end{align*}

\textbf{2. Analysis of the Weight $j\phi_{i,B}+y_1\phi_{i,B}'$.}
We first observe that $j\phi_{i,B}+y_1\phi_{i,B}' = B(j+i)\frac{y_1^i}{B^i}$ for $y_1>B.$
\begin{itemize}
    \item For $y_1 < B/2$: We bound the term using \eqref{phiboundderivativephi}, interpolation, and the fact that $\phi_{i,B}'$ decays exponentially on $(-\infty,0]$:
    \begin{align*}
    \iint_{\{y_1<\frac{B}{2}\}} |j\phi_{i,B}+y_1\phi_{i,B}'|\varepsilon^2 &\lesssim C(B)\iint_{\{y_1<\frac{B}{2}\}} \varepsilon^2\phi_{i,B}'\\
    &+ \Big( \iint_{\{y_1<\frac{B}{2}\}} y_1^4 \phi_{i,B}'\varepsilon^2\Big)^{\frac{1}{4}}\Big( \iint \varepsilon^2\phi_{i,B}'\Big)^{\frac{3}{4}} \\
    &\lesssim C(B)\iint \phi_{i,B}'\varepsilon^2+ \|\varepsilon\|_{L^2}^{\frac{1}{2}} \Big( \iint \phi_{i,B}'\varepsilon^2\Big)^{\frac{3}{4}}.
    \end{align*}
    \item For $B/2< y_1 < B$: Since the domain is compact, then 
    \[ \iint_{\{B/2<y_1<B\}} |j\phi_{i,B}+y_1\phi_{i,B}'|\varepsilon^2 \lesssim  C(B)\iint \phi_{i,B}'\varepsilon^2. \]
\end{itemize}

\textbf{3. The Problematic Region.}
On \(\{y_1>B\}\), the remaining contribution is
\[
    \frac{\lambda_s}{\lambda}(i+j)
    \iint_{y_1>B}\varepsilon^2\phi_{i,B}.
\]
Here the weight grows polynomially as \(y_1^i\) and could overwhelm the
dissipation. We control this term using the a priori conditions
\eqref{H2} and \eqref{H3}, together with interpolation. Denote
$$\mathcal{N}_i^>=\iint_{y_1>B}y_1^i\varepsilon^2, \qquad \widetilde{\mathcal{N}}_i^>=\iint_{y_1>B}y_1^{i-1}\varepsilon^2,\qquad \mathcal{N}_{K_i}^>=\iint_{y_1>B}y_1^{K_i}\varepsilon^2.$$
Using \eqref{lambadoubleestimate}, we have 
$$\big|\frac{\lambda_s}{\lambda}\Big|\lesssim|b|+e^{\frac{B}{4\alpha_1}}\widetilde{\mathcal{N}}_{i}^{\frac{1}{2}}.$$
Set $m=K_i-i$. Interpolation gives
\begin{equation}\label{interpolation}
\mathcal{N}_i^>\leq (\widetilde{\mathcal{N}}_i^>)^{\frac{m}{m+1}}(\mathcal{N}_{K_i}^>)^{\frac{1}{m+1}},
\end{equation}
therefore 
$$|b|\mathcal{N}_i^>\leq \frac{\mu}{200B^{i}}\widetilde{\mathcal{N}}_i^>+C_{\mu, B,i}|b|^{m+1}\mathcal{N}_{K_i}^>.$$

Assumptions \eqref{H1}--\eqref{H3} then imply 
\begin{align*}
|b|^{m+1}\mathcal{N}_{K_i}^>
&\leq 2|b|^{m+1}(1+\lambda^{-K_i})\\
&\leq 2\bigl(|b|^{m+1}+\mathcal C_i^{\frac{K_i}{c}}
|b|^{m+1-\frac{K_i}{c}}\bigr)\lesssim |b|^4,
\end{align*}
since by the definition of $K_i>\frac{c(i+3)}{c-1}$, we have $m+1, m+1-\frac{K_i}{c}\geq 4,$ hence 
$$|b|\mathcal{N}_i^>\leq \frac{\mu}{200B^{i}}\widetilde{\mathcal{N}}_i^>+C_Bb^4\leq \frac{\mu}{200B}\widetilde{\mathcal{N}}_i+C_Bb^4.$$
From \eqref{interpolation}, we have 
$$(\mathcal{N}_i^>)^2\leq \widetilde{\mathcal{N}}_i^>[(\mathcal{N}_i^>)^{m-1}\mathcal{N}_{K_i}^>]^{\frac{1}{m}},$$
thus 

$$\widetilde{\mathcal{N}}_i^{\frac{1}{2}}\mathcal{N}_i^>\leq \frac{\mu}{400Be^{\frac{B}{4\alpha_1}}}\widetilde{\mathcal{N}}_i+C_{\mu, B, i}\widetilde{\mathcal{N}}_i^>[(\mathcal{N}_i^>)^{m-1}\mathcal{N}_{K_i}^>]^{\frac{1}{m}}.$$
From \eqref{H1} and \eqref{H2}, we have that 
\begin{equation}\label{eqeq}
    \mathcal{N}_i^>\lesssim B^{i-1}\mathcal{N}_i\leq B^{i-1}\min\{\nu^*, \mathcal C_i\lambda^c\},
\end{equation}
For $\lambda\geq 1,$ using \eqref{H3} and \eqref{eqeq} we have 
$$(\mathcal{N}_i^>)^{m-1}\mathcal{N}_{K_i}^>\leq 4B^{(i-1)(m-1)}(\nu^*)^{m-1}.$$
For $0<\lambda\leq 1,$ we have
\[
\sup_{0<\lambda\leq 1}\min\{\nu^*,\mathcal C_i\lambda^c\}^{m-1}
\lambda^{-K_i}\leq\mathcal C_i^{\frac{K_i}{c}}
(\nu^*)^{m-1-\frac{K_i}{c}}.
\]
We conclude 
\[
\sup_{s\in [0,s_0]}(\mathcal{N}_i^>)^{m-1}\mathcal{N}_{K_i}^>
\leq 4B^{(i-1)(m-1)}\bigl[(\nu^*)^{m-1}
+\mathcal C_i^{\frac{K_i}{c}}(\nu^*)^{m-1-\frac{K_i}{c}}\bigr],
\]
with $m-1-\frac{K_i}{c}>0$ from the definition of $K_i.$ 
Consequently, after reducing $\nu^*$ depending on the fixed $\mathcal C_i$,

$$\widetilde{\mathcal{N}}_i^{\frac{1}{2}}\mathcal{N}_i^>\leq \frac{\mu}{400Be^{\frac{B}{4\alpha_1}}}\widetilde{\mathcal{N}}_i+\frac{\mu}{400e^{\frac{B}{\alpha_1}}B^{i}}\widetilde{\mathcal{N}}_{i}^>\leq \frac{\mu}{200Be^{\frac{B}{4\alpha_1}}}\widetilde{\mathcal{N}}_i.$$
Here we used $\widetilde{\mathcal N}_i^>\leq
B^{i-1}\widetilde{\mathcal N}_i/i$, which follows from
$\phi_{i,B}'=iB^{1-i}y_1^{i-1}$ on $y_1>B$.
We conclude 
$$(i+j)\Bigg|\frac{\lambda_s}{\lambda}\Bigg|\iint_{y_1>B} \varepsilon^2\phi_{i,B}\leq \frac{\mu(i+j)}{50B}\widetilde{\mathcal{N}}_i+C_Bb^4\leq \frac{\mu}{100}\widetilde{\mathcal{N}}_i+C_Bb^4.$$

Combining all the estimates for the terms in $f_2$, using the bound \eqref{eq:modulatedcoefficients} and the a priori estimate \eqref{H1} so $|\frac{\lambda_s}{\lambda}|\lesssim |b|+\mathcal{N}_{i,loc}^{\frac{1}{2}}+\widetilde{\mathcal{N}}_i$ and $\|\varepsilon\|_{L^2}\lesssim \delta(\nu^*),$ by Young's inequality, we finalize the estimate 
\begin{equation*}
\begin{split} 
f_2\leq&\delta(\nu^*)|b|\Big(\iint \varepsilon^2\phi_{i,B}'\Big)^{\frac{3}{4}}+Cb^4+\frac{\mu}{100}\widetilde{\mathcal{N}}_{i}\\
&\leq \frac{\mu}{4}\iint (|\nabla \varepsilon|^2+\varepsilon^2)\phi_{i,B}'+Cb^4.
\end{split}
\end{equation*}
The proof of Lemma \ref{lemma17} is complete. 

\chapter{Closing the Bootstrap} \label{Closingthebootstrap}

\section{Initial Data Decomposition}

We first discuss whether the solution of the equation \eqref{eq:ZK} remains inside the $L^2-$modulated tube if we start with small initial data. Take universal constant $\alpha_0\ll \alpha^*\ll \nu_0,$ with $\nu_0$ taken as in Lemma \ref{eq:nualpha}. Fix a natural number $n\geq 2$ and let $u_0 \in H^1(\mathbb{R}^2)\cap \mathcal{A}_{\alpha_0}^{(n)}$ and $u(t)$ be solution associated with this initial data on the maximal existence interval $[0,T).$ Define the distance to the soliton manifold 
$$d(v)=\inf_{\lambda_0>0, x_1,x_2\in \R}\norm{v-\frac{1}{\lambda_0}Q\Bigg(\frac{\cdot-x_1}{\lambda_0}, \frac{\cdot-x_2}{\lambda_0}\Bigg)}_{L^2(\mathbb{R}^2)}.$$ 
Since $d(u_0)<\alpha_0<\alpha^*,$ thus $u_0 \in \mathcal{T}_{\alpha^*}.$  Now define $t^*=\sup\{t\in [0,T): u(t')\in \mathcal{T}_{\alpha^*}\text{ for all }0\leq t'<t\}$ and we observe that $t^*\geq 0.$ We know that the distance function $d$ is $1-$Lipschitz in $L^2,$ i.e. $|d(v)-d(w)|\leq \|v-w\|_{L^2}.$ By local well-posedness, there exists $\tau>0$ such that on $[0,\tau],$ $\|u(t)-u_0\|_{L^2}<\frac{\alpha^*-\alpha_0}{2}.$ Therefore, on $[0,\tau],$ $d(u(t))<\alpha^*,$ hence $t^{*}>0.$ Denote the corresponding $s-$variable of $t^*$ by $s(t^*)=s^*.$ 

\begin{lemma}\label{initialparametersmall}
    Let $u_0 \in H^1(\mathbb{R}^2)\cap \mathcal{A}_{\alpha_0}^{(n)}$ and consequently $u_0\in \mathcal{T}_{\alpha^*}.$ After applying  decomposition from Proposition \ref{eq:nualpha}, we have 
    $$\|\varepsilon(0)\|_{L^2}+|b(0)|+|1-\lambda(0)|+\mathcal{N}_m(0)\lesssim \delta(\alpha_0), \text{ for }1\leq m\leq n-1, $$
    $$\iint_{y_1>0}y_1^{n}\varepsilon^2(0)\leq 1+\delta(\alpha_0).$$
\end{lemma}
\begin{proof}
    First, write $u_0=Q+\eta_0,$ with $\|\eta_0\|_{H^1}<\alpha_0$ and $\iint_{y_1>0}y_1^{n}\eta_0^2(y_1,y_2)<1.$ It means there exists a decomposition of $u_0$ with $\lambda^\sharp(0)=1, x^\sharp_1(0)=0, x^\sharp_2(0)=0.$ By applying Proposition \ref{eq:nualpha} with $\nu, \alpha\sim \alpha_0$, we obtain a decomposition $(\lambda, x_1,x_2, b, \varepsilon)$ with 
    $$\|\varepsilon(0)\|_{H^1}+|b(0)|+|1-\lambda(0)|+|(x_1(0),x_2(0))|\lesssim \delta(\alpha_0).$$
    Denote $y=(y_1,y_2), b(0)=b_0, \lambda(0)=\lambda_0,  x_0=(x_{1,0},x_{2,0})=(x_1(0),x_2(0)), \varepsilon(0,y)=\varepsilon_0(y).$ 
    Write $\varepsilon_0(y)=a(y)+r(y)$ with $a(y)=\lambda_0\eta_0(\lambda_0y+x_0), r(y)=\lambda_0Q(\lambda_0y+x_0)-Q_{b_0}(y)$ and $z=(z_1,z_2)=\lambda_0y+x_0.$ First, we note that 
    $$\iint_{y_1>0}y_1^{n}a^2(y)dy=\lambda_0^{-n}\iint_{z_1>x_{1,0}}(z_1-x_{1,0})^{n}\eta_0(z)^2dz.$$
    For \(p,q\in\mathbb R\), write \(p_+=\max\{p,0\}\). The elementary
inequality
\[
    \bigl|(p-q)_+^n-p_+^n\bigr|
    \leq
    C_n|q|\bigl(p_+^{n-1}+|q|^{n-1}\bigr)
\]
gives
\begin{align*}
\iint_{y_1>0}y_1^n a^2(y)\,dy
&=
\lambda_0^{-n}
\iint_{\mathbb R^2}
(z_1-x_{1,0})_+^n\eta_0^2(z)\,dz\\
&\leq
\lambda_0^{-n}
\bigg[
\iint z_{1,+}^n\eta_0^2
+
C_n|x_{1,0}|
\iint z_{1,+}^{n-1}\eta_0^2
+
C_n|x_{1,0}|^n\|\eta_0\|_{L^2}^2
\bigg].
\end{align*}
By interpolation,
\[
\iint z_{1,+}^{n-1}\eta_0^2
\leq
\left(
\iint z_{1,+}^n\eta_0^2
\right)^{\frac{n-1}{n}}
\|\eta_0\|_{L^2}^{\frac{2}{n}}
\lesssim
\alpha_0^{\frac{2}{n}}.
\]
Since
\(
\iint z_{1,+}^n\eta_0^2<1,
\quad
|\lambda_0-1|+|x_{1,0}|
\lesssim\delta(\alpha_0),
\)
we obtain
\[
    \iint_{y_1>0}y_1^n a^2(y)\,dy
    \leq 1+\delta(\alpha_0).
\]

Furthermore, by the exponential decay of \(Q\) and of \(P\) on
\(\{y_1>0\}\),
\[
    \iint_{y_1>0}y_1^n r^2(y)\,dy
    \lesssim
    |\lambda_0-1|^2+|x_0|^2+|b_0|^2
    \lesssim\delta(\alpha_0).
\]
Therefore, using \(\varepsilon_0=a+r\),
\begin{align}\label{boundextralocalization}
\iint_{y_1>0}y_1^n\varepsilon_0^2
&\leq
\iint_{y_1>0}y_1^n a^2
+
2\left(
\iint_{y_1>0}y_1^n a^2
\right)^{1/2}
\left(
\iint_{y_1>0}y_1^n r^2
\right)^{1/2}
+
\iint_{y_1>0}y_1^n r^2\\
&\leq 1+\delta(\alpha_0),\notag
\end{align}
after redefining the generic function \(\delta(\alpha_0)\).
    Now, since $0\leq \psi_B\leq 1,$ we have 
    $$\iint |\nabla\varepsilon_0|^2\psi_B\lesssim \|\varepsilon_0\|_{H^1}^2\lesssim \delta(\alpha_0).$$
    On $y_1<0,$ we have $\iint_{y_1<0}\varepsilon_0^2\phi_{m,B}\lesssim \|\varepsilon_0\|_{L^2}^2\lesssim \delta(\alpha_0).$
    On $y_1>0,$ we have $\phi_{m,B}(y_1)\lesssim 1+y_1^{m},$ then 
    $\iint_{y_1>0}\varepsilon_0^2\phi_{m,B}\lesssim \|\varepsilon_0\|_{L^2}^2+\iint_{y_1>0}y_1^{m}\varepsilon_0^2.$ 
    By H\"older's inequality and \eqref{boundextralocalization}, for $m\leq n-1$ we have 
    $$\iint_{y_1>0}y_1^{m}\varepsilon_0^2\leq \Big(\iint_{y_1>0}y_1^{n}\varepsilon_0^2\Big)^{\frac{m}{n}}\Big(\iint \varepsilon_0^2\Big)^{1-\frac{m}{n}}\lesssim \delta(\alpha_0).$$
    We conclude that $\mathcal{N}_m(0)\lesssim \delta(\alpha_0).$
\end{proof}

\section{Monotonicity of Energy Functional}

Fix $i\geq 1$ and $\mathcal C_i>1$, and recall $K_i=\lfloor \frac{c(i+3)}{c-1}\rfloor+1.$ We suppose that $u_0 \in \mathcal{A}_{\alpha_0}^{(K_i)}.$ We define the a priori conditions 
\begin{gather}
|b(s)| + \mathcal{N}_i(s) + \|\varepsilon(s)\|_{L^2} \leq \nu^*, \label{H1condi} \tag{\textbf{H1$_i$}} \\
\frac{|b(s)|+\mathcal{N}_i(s)}{\lambda^c(s)} \leq \mathcal C_i, \label{H2condi} \tag{\textbf{H2$_i$}} \\
\iint_{y_1>0} \varepsilon^2(s)y_1^{K_i} \leq 2\Big(1+\lambda^{-K_i}(s)\Big). \label{H3condi} \tag{\textbf{H3$_i$}}
\end{gather}
with $\nu^*$ chosen small enough for the applications of
Proposition~\ref{Monotonicity} at the finitely many levels $m\leq i$
as explained below.

We define the maximal time $s^{**}_{i}$ compatible with the trapped regime as
\[
    s^{**}_{i} = \sup \Big\{ 0\leq s<s^*, \, u \text{ satisfies }\textbf{H1}_i-\textbf{H2}_i-\textbf{H3}_i \text{ on }[0,s]\Big\}.
\]
As a corollary of Lemma \ref{initialparametersmall}, $s_i^{**}>0.$  For the $t-$variable, we define $t_i^{**}>0$ such that $s(t_i^{**})=s_i^{**}.$

\begin{proposition}\label{prop:energy_estimates}
Let $c$ be the constant defined in \eqref{eq:definitionofc}. Then the
following estimates hold.
\begin{itemize}
    \item[\textbf{i)}] \textbf{Dispersive bounds.}
    For every $1\leq m\leq i$ and $0\leq s_1\leq s_2<s_i^{**}$,
    \begin{equation}\label{eq:Dispersiveestimates1}
    \mathcal N_m(s_2)
    +\int_{s_1}^{s_2}\iint
    \bigl(|\nabla\varepsilon|^2+\varepsilon^2\bigr)
    \phi_{m,B}'\,ds
    \lesssim
    \mathcal N_m(s_1)+|b^3(s_1)|+|b^3(s_2)|.
    \end{equation}
    Moreover, for every $0<\alpha<3$, there exists $\nu_\alpha>0$
    such that, if $0<\nu^*\leq\nu_\alpha$ in \eqref{H1condi}, then
    \begin{equation}\label{eq:Dispersiveestimates2}
    \begin{aligned}
    \frac{\mathcal N_m(s_2)}{\lambda^{\alpha c}(s_2)}
    &+\int_{s_1}^{s_2}
    \frac{\iint
    \bigl(|\nabla\varepsilon|^2+\varepsilon^2\bigr)\phi_{m,B}'}
         {\lambda^{\alpha c}(s)}\,ds \\
    &\lesssim
    \frac{\mathcal N_m(s_1)}{\lambda^{\alpha c}(s_1)}
    +\frac{|b^3(s_1)|}{\lambda^{\alpha c}(s_1)}
    +\frac{|b^3(s_2)|}{\lambda^{\alpha c}(s_2)}.
    \end{aligned}
    \end{equation}

    \item[\textbf{ii)}] \textbf{Control of the scaling dynamics.}
    Let
    \[
        \tilde\lambda(s)=\lambda(s)(1-J_1(s)),
    \]
    where $J_1$ is defined in Lemma~\ref{sharporthogonalities}.
    If $i\geq2$, then, on $[0,s_i^{**})$,
    \begin{equation}\label{eq:lambdatildeineq}
    \left|
    \frac{\tilde\lambda_s}{\tilde\lambda}+b
    \right|
    \lesssim
    b^2+\mathcal N_{3,\mathrm{loc}}
    +\widetilde{\mathcal N}_1.
    \end{equation}

    \item[\textbf{iii)}] \textbf{Control of the dynamics of $b$.}
    For $0\leq s_1\leq s_2<s_i^{**}$,
    \begin{equation}\label{eq:controldynamics}
    \int_{s_1}^{s_2}b^2(s)\,ds
    \lesssim
    \mathcal N_1(s_1)+|b(s_1)|+|b(s_2)|.
    \end{equation}
    If, in addition, $i\geq3$, then
    \begin{equation}\label{eq:blambdacinequality}
    \left|
    \frac{b(s_2)}{\lambda^c(s_2)}
    -
    \frac{b(s_1)}{\lambda^c(s_1)}
    \right|
    \lesssim
    \frac{\mathcal N_3(s_1)}{\lambda^c(s_1)}
    +\frac{b^2(s_1)}{\lambda^c(s_1)}
    +\frac{b^2(s_2)}{\lambda^c(s_2)}.
    \end{equation}
\end{itemize}
\end{proposition}

\begin{proof}
\textit{i).}
For $m\leq i$, the weight comparisons give
$\mathcal N_m\leq C_{m,i}\mathcal N_i$, with $C_{i,i}=1$,
so \eqref{H2condi}
supplies \eqref{H2} at level $m$ with the fixed constant
$\mathcal C_m=\max\{1,C_{m,i}\}\mathcal C_i>1$.
Choose the common $\nu^*$ so that $\max\{1,C_{m,i}\}\nu^*$
lies below the level-$m$ smallness threshold for every $m\leq i$.
Then \eqref{H1} holds at these levels, while
H\"older's inequality transfers \eqref{H3condi} with the same
coefficient $2$, since $K_m\leq K_i$ and $\|\varepsilon\|_{L^2}\leq1$.
Choosing $B$ sufficiently large and $\nu^*$ sufficiently small for
this finite collection, we may apply
Proposition~\ref{Monotonicity} at each level $m\leq i$. 

By \eqref{eq:modulatedcoefficients}, \eqref{McomparedtoN}, and
\eqref{H1condi}, for every $0<\epsilon<c$ there exists
$C_\epsilon>0$ such that, after choosing $\nu^*$ sufficiently small
depending on $\epsilon$,

\begin{equation}\label{differentialinequalityb}
    (c-\epsilon)b^2
    \leq
    -b_s+C_\epsilon\mathcal N_{1,\mathrm{loc}}
    +C_1\nu^*\widetilde{\mathcal N}_1.
\end{equation}
 Integrating \eqref{monotonictyformula1} with $j=0$ and using the
coercivity \eqref{coercivityofF}, we obtain
\begin{align*}
\mathcal N_m(s_2)
&+\int_{s_1}^{s_2}\iint
\bigl(|\nabla\varepsilon|^2+\varepsilon^2\bigr)
\phi_{m,B}'\,ds \\
&\lesssim
\mathcal F_m(s_2)
+\mu\int_{s_1}^{s_2}\iint
\bigl(|\nabla\varepsilon|^2+\varepsilon^2\bigr)
\phi_{m,B}'\,ds \\
&\lesssim
\mathcal F_m(s_1)
+\int_{s_1}^{s_2}b^4(s)\,ds \\
&\lesssim
\mathcal N_m(s_1)
+\int_{s_1}^{s_2}b^4(s)\,ds.
\end{align*}
Taking $\epsilon=c-1$ in \eqref{differentialinequalityb},
multiplying by $b^2$, and using \eqref{H1condi}, we find
\begin{align*}
\int_{s_1}^{s_2}b^4(s)\,ds
&\leq
-\int_{s_1}^{s_2}b^2b_s\,ds
+C_\epsilon(\nu^*)^2
\int_{s_1}^{s_2}\mathcal N_{1,\mathrm{loc}}\,ds
+C_1(\nu^*)^3
\int_{s_1}^{s_2}\widetilde{\mathcal N}_1\,ds \\
&\leq
\left[-\frac{b^3(s)}{3}\right]_{s_1}^{s_2}
+C_2(\nu^*)^2
\int_{s_1}^{s_2}
\bigl(\mathcal N_{1,\mathrm{loc}}
+\widetilde{\mathcal N}_1\bigr)\,ds \\
&\leq
|b^3(s_1)|+|b^3(s_2)|
+C_3(\nu^*)^2
\int_{s_1}^{s_2}\iint
\bigl(|\nabla\varepsilon|^2+\varepsilon^2\bigr)
\phi_{m,B}'\,ds.
\end{align*}
Choosing $\nu^*$ sufficiently small allows the last term to be
absorbed, which proves \eqref{eq:Dispersiveestimates1}.

For \eqref{eq:Dispersiveestimates2}, integrate
\eqref{monotonictyformula1} with $j=\alpha c$ and again use
\eqref{coercivityofF}:
\begin{align*}
\frac{\mathcal N_m(s_2)}{\lambda^{\alpha c}(s_2)}
&+\int_{s_1}^{s_2}
\frac{\iint
\bigl(|\nabla\varepsilon|^2+\varepsilon^2\bigr)
\phi_{m,B}'}
{\lambda^{\alpha c}(s)}\,ds \\
&\lesssim
\frac{\mathcal F_{m}(s_1)}
     {\lambda^{\alpha c}(s_1)}
+\int_{s_1}^{s_2}
\frac{b^4(s)}{\lambda^{\alpha c}(s)}\,ds \\
&\lesssim
\frac{\mathcal N_m(s_1)}
     {\lambda^{\alpha c}(s_1)}
+\int_{s_1}^{s_2}
\frac{b^4(s)}{\lambda^{\alpha c}(s)}\,ds.
\end{align*}
Set
\[
    \epsilon_\alpha
    =\frac c2\left(1-\frac\alpha3\right).
\]
Multiplying \eqref{differentialinequalityb} by
$b^2/\lambda^{\alpha c}$ and integrating by parts gives
\begin{align*}
(c-\epsilon_\alpha)
\int_{s_1}^{s_2}
\frac{b^4}{\lambda^{\alpha c}}\,ds
&\leq
\left[
-\frac{b^3}{3\lambda^{\alpha c}}
\right]_{s_1}^{s_2}
-\frac{\alpha c}{3}
\int_{s_1}^{s_2}
\frac{b^3}{\lambda^{\alpha c}}
\frac{\lambda_s}{\lambda}\,ds \\
&\quad+
\widetilde C_{\epsilon_\alpha}(\nu^*)^2
\int_{s_1}^{s_2}
\frac{
\mathcal N_{1,\mathrm{loc}}
+\widetilde{\mathcal N}_1}
{\lambda^{\alpha c}}\,ds.
\end{align*}
Since
\[
\begin{aligned}
-\frac{\alpha c}{3}
\int_{s_1}^{s_2}
\frac{b^3}{\lambda^{\alpha c}}
\frac{\lambda_s}{\lambda}\,ds
&=
\frac{\alpha c}{3}
\int_{s_1}^{s_2}
\frac{b^4}{\lambda^{\alpha c}}\,ds \\
&\quad-
\frac{\alpha c}{3}
\int_{s_1}^{s_2}
\frac{b^3}{\lambda^{\alpha c}}
\left(
\frac{\lambda_s}{\lambda}+b
\right)\,ds,
\end{aligned}
\]
the modulation estimate \eqref{eq:modulatedcoefficients} and
\eqref{H1condi} yield
\begin{align*}
(c-\epsilon_\alpha)
\int_{s_1}^{s_2}
\frac{b^4}{\lambda^{\alpha c}}\,ds
&\leq
\left|
\frac{b^3(s_1)}
     {3\lambda^{\alpha c}(s_1)}
\right|
+
\left|
\frac{b^3(s_2)}
     {3\lambda^{\alpha c}(s_2)}
\right| \\
&\quad+
\frac{\alpha c}{3}
\int_{s_1}^{s_2}
\frac{b^4}{\lambda^{\alpha c}}\,ds
+
C_4(\nu^*)^{1/2}
\int_{s_1}^{s_2}
\frac{b^4}{\lambda^{\alpha c}}\,ds \\
&\quad+
C_4(\nu^*)^{1/2}
\int_{s_1}^{s_2}
\frac{\iint
\bigl(|\nabla\varepsilon|^2+\varepsilon^2\bigr)
\phi_{m,B}'}
{\lambda^{\alpha c}}\,ds.
\end{align*}
Consequently,

\begin{align*}
c\left[
\frac12\left(1-\frac\alpha3\right)
-\frac{C_4(\nu^*)^{1/2}}{c}
\right]
\int_{s_1}^{s_2}
\frac{b^4}{\lambda^{\alpha c}}\,ds
&\lesssim
\frac{|b^3(s_1)|}
     {\lambda^{\alpha c}(s_1)}
+
\frac{|b^3(s_2)|}
     {\lambda^{\alpha c}(s_2)} \\
&\quad+
\delta(\nu^*)
\int_{s_1}^{s_2}
\frac{\iint
\bigl(|\nabla\varepsilon|^2+\varepsilon^2\bigr)
\phi_{m,B}'}
{\lambda^{\alpha c}}\,ds.
\end{align*}
For $\nu^*\leq\nu_\alpha$ sufficiently small, the coefficient on
the left is positive. Substitution into the preceding energy
estimate and absorption of the last term prove
\eqref{eq:Dispersiveestimates2}. Notice that, as $\nu^*\to0$,
$\alpha$ may be chosen arbitrarily close to $3$; below it is enough
to fix one $\alpha>2$.

\medskip
\textit{ii).}
By the definition of $J_1$,
\[
|J_1|
\lesssim
\mathcal N_{3,\mathrm{loc}}^{1/2}
\lesssim
\delta(\nu^*).
\]
Set
\[
E_{J_1}
:=
(J_1)_s+\frac{\lambda_s}{\lambda}J_1
-\left(\frac{\lambda_s}{\lambda}+b\right).
\]
Lemma~\ref{sharporthogonalities} gives
\[
|E_{J_1}|
\lesssim
b^2+\mathcal N_{3,\mathrm{loc}}
+\delta(\nu^*)\widetilde{\mathcal N}_1.
\]
Since $|J_1|\ll1$,
\[
\left|
\frac{\tilde\lambda_s}{\tilde\lambda}+b
\right|
=
\left|
\frac{E_{J_1}+bJ_1}{1-J_1}
\right|
\lesssim
b^2+\mathcal N_{3,\mathrm{loc}}
+\widetilde{\mathcal N}_1
\lesssim
b^2+\mathcal N_2,
\]
which proves \eqref{eq:lambdatildeineq}. In particular,
$\tilde\lambda(s)>0$ on $[0,s_i^{**})$ for every $i\geq2$.

\medskip
\textit{iii).}
Integrating \eqref{differentialinequalityb} on $[s_1,s_2]$ and
using \eqref{eq:Dispersiveestimates1}, we obtain
\begin{equation}\label{equationb2bis}
\begin{aligned}
\int_{s_1}^{s_2}b^2\,ds
&\leq
-\int_{s_1}^{s_2}b_s\,ds
+C\int_{s_1}^{s_2}\mathcal N_{1,\mathrm{loc}}\,ds
+\delta(\nu^*)
\int_{s_1}^{s_2}\widetilde{\mathcal N}_1\,ds \\
&\lesssim
|b(s_1)|+|b(s_2)|
+\int_{s_1}^{s_2}\iint
\bigl(|\nabla\varepsilon|^2+\varepsilon^2\bigr)
\phi_{1,B}'\,ds \\
&\lesssim
|b(s_1)|+|b(s_2)|+\mathcal N_1(s_1),
\end{aligned}
\end{equation}
which proves \eqref{eq:controldynamics}.

For the second estimate, recall $J$ from
Lemma~\ref{lemmasharpblambda}. By \eqref{Jinnerproductbound},
\[
|J(s)|
\lesssim
\mathcal N_{3,\mathrm{loc}}^{1/2}(s)
\leq(\nu^*)^{1/2}.
\]
Hence \eqref{blambdasharpeq} and
\eqref{eq:modulatedcoefficients} imply
\begin{equation}\label{eq:minmax}
\begin{aligned}
\left|
\frac{d}{ds}
\left(
\frac{b}{\lambda^c}e^J
\right)
\right|
&=
e^J
\left|
\frac{d}{ds}
\left(
\frac{b}{\lambda^c}
\right)
+\frac{b}{\lambda^c}J_s
\right| \\
&\lesssim
\left|
\frac{b}{\lambda^c}
\frac{\lambda_s}{\lambda}
\mathcal N_{3,\mathrm{loc}}^{1/2}
\right|
+
\frac{
|b|^3+\widetilde{\mathcal N}_1
+\mathcal N_{3,\mathrm{loc}}
}
{\lambda^c} \\
&\lesssim
\frac{
|b|^3+\widetilde{\mathcal N}_3
}
{\lambda^c}.
\end{aligned}
\end{equation}
Let $s,s'\in[s_1,s_2]$. Integrating \eqref{eq:minmax}, using
\eqref{eq:Dispersiveestimates2} with $\alpha=1$, together with
\[
|1-e^J|
\lesssim
|J|
\lesssim
\mathcal N_2^{1/2},
\]
\eqref{equationb2bis}, \eqref{eq:Dispersiveestimates1}, and
\eqref{H1condi}, gives
\begin{align*}
\left|
\frac{b(s)}{\lambda^c(s)}
-
\frac{b(s')}{\lambda^c(s')}
\right|
&\lesssim
\int_s^{s'}
\frac{\widetilde{\mathcal N}_3(s'')}
     {\lambda^c(s'')}\,ds''
+
\int_s^{s'}
\frac{|b^3(s'')|}
     {\lambda^c(s'')}\,ds'' \\
&\quad+
\frac{|b(s)|}{\lambda^c(s)}
|1-e^{J(s)}|
+
\frac{|b(s')|}{\lambda^c(s')}
|1-e^{J(s')}| \\
&\lesssim
\frac{\mathcal N_3(s_1)}
     {\lambda^c(s_1)}
+
\delta(\nu^*)
\sup_{\tau\in[s_1,s_2]}
\frac{|b(\tau)|}{\lambda^c(\tau)}.
\end{align*}
Taking the supremum and minimum over $[s_1,s_2]$, and then
absorbing the last term, yields
\[
\sup_{[s_1,s_2]}
\frac{|b|}{\lambda^c}
\lesssim
\min_{[s_1,s_2]}
\frac{|b|}{\lambda^c}
+
\frac{\mathcal N_3(s_1)}
     {\lambda^c(s_1)}.
\]
Applying the preceding estimate with $s=s_1$ and $s'=s_2$, and
retaining the endpoint bounds produced by $|1-e^J|$, we obtain
\begin{align*}
\left|
\frac{b(s_2)}{\lambda^c(s_2)}
-
\frac{b(s_1)}{\lambda^c(s_1)}
\right|
&\lesssim
\frac{\mathcal N_3(s_1)}
     {\lambda^c(s_1)}
+
\frac{|b^3(s_1)|}
     {\lambda^c(s_1)}
+
\frac{|b^3(s_2)|}
     {\lambda^c(s_2)} \\
&\quad+
\left(
\min_{[s_1,s_2]}
\frac{|b|}{\lambda^c}
+
\frac{\mathcal N_3(s_1)}
     {\lambda^c(s_1)}
\right)
\Bigl(
|b(s_1)|+|b(s_2)|
+\mathcal N_2(s_1) \\
&\hspace{43mm}
+\mathcal N_2^{1/2}(s_1)
+
\bigl(
\mathcal N_2(s_1)
+|b^3(s_1)|
+|b^3(s_2)|
\bigr)^{1/2}
\Bigr) \\
&\lesssim
\frac{\mathcal N_3(s_1)}
     {\lambda^c(s_1)}
+
\frac{b^2(s_1)}
     {\lambda^c(s_1)}
+
\frac{b^2(s_2)}
     {\lambda^c(s_2)},
\end{align*}
which proves \eqref{eq:blambdacinequality}.
\end{proof}

As a corollary, we have from \eqref{importantbound} and \eqref{eq:Dispersiveestimates1}, as $0\leq s_1\leq s_2 <s^{**}_{i},$ for $i\geq2 $ and $1\leq m\leq i-1,$
\begin{equation}\label{boundiplusone}
    \int_{s_1}^{s_2}\mathcal{N}_m(s)ds\lesssim \mathcal{N}_{m+1}(s_1)+|b^3(s_1)|+|b^3(s_2)|, 
\end{equation}
and from \eqref{eq:Dispersiveestimates2}, 
\begin{equation}\label{boundiplusone2}
    \int_{s_1}^{s_2}\frac{\mathcal{N}_m(s)}{\lambda^{c}(s)}ds\lesssim \frac{\mathcal{N}_{m+1}(s_1)+|b^3(s_1)|}{\lambda^c(s_1)}+\frac{|b^3(s_2)|}{\lambda^c(s_2)}.
\end{equation}

\section{Bootstrap argument}
In this section, we prove the propagation of the a priori estimates until the exit time from the modulated tube. From now on, we choose parameters $\alpha_0\ll \alpha^*\ll \nu^*\ll 1$ and $u_0 \in \mathcal{A}_{\alpha_0}^{(K_i)}$ for $i\geq 4.$ 

\begin{lemma}\label{lambdaboundedbefores**}
    For every $S<s^{**}_i$ and every $0\leq r\leq s \leq S,$ we have 
    $$\tilde{\lambda}(r)\leq 4\max\{1,\tilde{\lambda}(s)\}.$$
    Consequently, there exists $C_0>4$ such that
    $$\lambda(r)\leq C_0\max\{1,\lambda(s)\}.$$
\end{lemma}
\begin{proof} 
    First, we assign specific names to the constants appearing in Proposition \ref{prop:energy_estimates}. Let $L_1,L_2, L_3>0$ such that 
    $$ \frac{\mathcal{N}_3(s_2)}{\lambda^{c}(s_2)}  \leq L_1\Bigg( \frac{\mathcal{N}_3(s_1)}{\lambda^{c}(s_1)} + \frac{|b^3(s_1)|}{\lambda^{c}(s_1)} + \frac{|b^3(s_2)|}{\lambda^{c}(s_2)}\Bigg),$$
      $$\left| \frac{(\tilde{\lambda})_s}{\tilde{\lambda}} + b \right| \leq L_2\left( b^2+ \mathcal{N}_2 \right),$$
      $$\left| \frac{b(s_2)}{\lambda^c(s_2)} - \frac{b(s_1)}{\lambda^c(s_1)} \right| \leq L_3\Bigg(\frac{b^2(s_2)}{\lambda^c(s_2)} + \frac{b^2(s_1)+ \mathcal{N}_3(s_1)}{\lambda^c(s_1)}\Bigg).$$
       Now choose a new constant $\tilde{L}$ large enough such that 
      \begin{equation}\label{tildeL}
      \tilde{L}\geq 8\max\{L_3, L_1L_2\}.
      \end{equation}
      We choose $\nu^*$ small enough such that $L_j\nu^*\leq \frac{1}{4}$ for $j=1,2,3$ and $\tilde{L}\nu^*\leq \frac{1}{2}.$

    \textbf{Case 1.} Assume 
    $$|b(s)|\leq \tilde{L}(\mathcal{N}_3(s)+b^2(s)) \text{ for all }0\leq s\leq S.$$
    Since $\nu^*$ is sufficiently small, we obtain 
    $$|b(s)|\leq 2\tilde{L}\mathcal{N}_3(s).$$
    We have  from \eqref{eq:lambdatildeineq}
    $$\Bigg|\frac{\tilde{\lambda}_s}{\tilde{\lambda}}(s)\Bigg|\lesssim  |b(s)|+\mathcal{N}_2(s)\lesssim \mathcal{N}_3(s),$$
    hence by \eqref{eq:Dispersiveestimates1}, \eqref{H1} and Lemma \ref{initialparametersmall},
    \begin{equation}\label{eq1}
    \Bigg|\log \Big(\frac{\tilde{\lambda}(s)}{\tilde{\lambda}(0)}\Big)\Bigg|\lesssim \mathcal{N}_4(0)+|b^3(s)|+|b^3(0)|\lesssim \delta(\alpha_0)+\delta(\nu^*)\leq \log 2.
    \end{equation}
    Again, from Lemma \ref{initialparametersmall}, $|\tilde{\lambda}(0)-1|\lesssim \delta(\alpha_0),$ together with \eqref{eq1},
    $$\frac{1}{4}\leq\tilde{\lambda}(s)\leq 4. $$

    \textbf{Case 2.} Define $\hat{s}=\inf\{s\in [0,S]:|b(s)|\geq \tilde{L}(\mathcal{N}_3(s)+b^2(s))\}.$ If the inequality already holds at $s=0,$ then $\hat{s}=0.$ By continuity, at $\hat{s}$ we have 
    $$\mathcal{N}_3(\hat{s})+b^2(\hat{s})\leq \frac{|b(\hat{s})|}{\tilde{L}}.$$
    On $[0,\hat{s}),$ the argument of Case 1 applies. Now, for $s\in [\hat{s},S],$ we have 
    $$\left| \frac{b(s)}{\lambda^c(s)} - \frac{b(\hat{s})}{\lambda^c(\hat{s})} \right|  \leq L_3\Bigg(\frac{b^2(s)}{\lambda^c(s)} + \frac{b^2(\hat{s})+\mathcal{N}_3(\hat{s})}{\lambda^c(\hat{s})}\Bigg)\leq L_3\nu^*\left| \frac{b(s)}{\lambda^c(s)}\right|+\frac{L_3}{\tilde{L}}\left|\frac{b(\hat{s})}{\lambda^c(\hat{s})} \right|,$$
    and therefore we get 
    $$\left| \frac{b(s)}{\lambda^c(s)} - \frac{b(\hat{s})}{\lambda^c(\hat{s})} \right|\leq \frac{1}{2}\left| \frac{b(\hat{s})}{\lambda^c(\hat{s})} \right|.$$
    Since $\lambda>0,$ we obtain \(\operatorname{sign} b(s)=\operatorname{sign} b(\hat{s}),\) for all $s\in [\hat{s},S).$
    Moreover, we have that 
    $$\frac{1}{2}\left| \frac{b(\hat{s})}{\lambda^c(\hat{s})} \right|\leq \left| \frac{b(s)}{\lambda^c(s)} \right|\leq \frac{3}{2}\left| \frac{b(\hat{s})}{\lambda^c(\hat{s})} \right|.$$
    The stronger preceding inequality also gives
    
\[
    \left|\frac{b(\hat{s})}{\lambda^c(\hat{s})}\right|
    \leq\frac{1+L_3\nu^*}{1-L_3/\tilde L}
    \left|\frac{b(s)}{\lambda^c(s)}\right|
    \leq\frac{10}{7}\left|\frac{b(s)}{\lambda^c(s)}\right|
    \leq\frac32\left|\frac{b(s)}{\lambda^c(s)}\right|.
    \]
    From \eqref{eq:Dispersiveestimates2}, we have
    $$\frac{\mathcal{N}_3(s)}{\lambda^c(s)}\leq L_1\Bigg(\frac{\mathcal{N}_3(\hat{s})+|b^3(\hat{s})|}{\lambda^c(\hat{s})}+\frac{|b^3(s)|}{\lambda^c(s)}\Bigg)\leq \Big(\frac{3L_1}{2\tilde{L}}+L_1(\nu^*)^2\Big)\left| \frac{b(s)}{\lambda^c(s)} \right|,$$
    therefore,  from \eqref{eq:lambdatildeineq} and \eqref{tildeL}, 
    $$\Bigg|\frac{\tilde{\lambda}_s}{\tilde{\lambda}}+b\Bigg|\leq L_2\Big(\frac{3L_1}{2\tilde{L}}+L_1(\nu^*)^2+\nu^*\Big)|b(s)|\leq \frac{1}{2}|b|.$$
    This in turn implies that if $b(\hat{s})>0,$ then $\tilde{\lambda}$ decreases after $\hat{s}$ and if $b(\hat{s})<0,$ then $\tilde{\lambda}$ increases after $\hat{s}.$ Together with $\tilde{\lambda}\sim 1$ before $\hat{s},$ this proves that 
    $$\tilde{\lambda}(r)\leq 4\max\{1, \tilde{\lambda}(s)\}, \text{ for }0\leq r\leq s\leq S.$$ 
    Since $\tilde{\lambda}\sim\lambda,$ we obtain the corresponding estimate for $\lambda$. 

\end{proof}

We introduce a dynamic control of the tail similar to the result of Lemma 3.7 in \cite{MartelMerleRaphael1}. We prove it for completeness. 

\begin{lemma}\label{lemma21}
Let $N\in\mathbb N$, $N\geq1$. Define the $C^3$ function
\begin{equation*}
    \phi_N=
    \begin{cases}
        0, &\quad y_1\leq 0,\\
        y_1^{N} &\quad y_1\geq 1,
    \end{cases}
\end{equation*}
with $\phi_N'\geq 0.$ 
    Under the assumptions of Proposition \ref{Monotonicity}, we have 
    $$\frac{1}{\lambda^N}\frac{d}{ds}\left\{\lambda^N\iint \phi_N\varepsilon^2\right\}\lesssim b^4+\widetilde{\mathcal{N}}_{1}.$$
\end{lemma}

\begin{proof}
     We use the modulated flow equation \eqref{eq:ModulationEquation},  
\begin{align*}
\frac12\frac{d}{ds}\iint\phi_N\varepsilon^2
={}&
\iint\phi_N\varepsilon
\left[
 \frac{\lambda_s}{\lambda}\Lambda\varepsilon
 +\bigl(
   -\Delta\varepsilon+\varepsilon
   -(Q_b+\varepsilon)^3+Q_b^3
  \bigr)_{y_1}
\right]\\
&+
\iint\phi_N\varepsilon
\left[
 \left(\frac{\lambda_s}{\lambda}+b\right)\Lambda Q_b
 +\left(\frac{(x_1)_s}{\lambda}-1\right)
  (Q_b+\varepsilon)_{y_1}
\right.\\
&\hspace{5.5cm}\left.
 +\frac{(x_2)_s}{\lambda}(Q_b+\varepsilon)_{y_2}
 +\Phi_b+\Psi_b
\right].
\end{align*}
 where $\Psi_b=[(-\Delta Q_b+Q_b-Q_b^3)_{y_1}-b\Lambda Q_b]$ and $\Phi_b=-b_s(\chi_b+\gamma y_1(\chi_b)_{y_1})P.$

For the linear term, we compute by integration by parts and using the fact $y_1\phi_N'=N\phi_N$ for $y_1\geq 1$ and $\phi_N'''\ll\phi_N'$ for $y_1$ large enough to derive the bound
\begin{align*}
&\iint\phi_N\varepsilon
\left[
 \frac{\lambda_s}{\lambda}\Lambda\varepsilon
 +(-\Delta\varepsilon+\varepsilon)_{y_1}
\right]\\
&\quad=
-\frac12\frac{\lambda_s}{\lambda}
 \iint y_1\phi_N'\varepsilon^2
-\frac12\iint\phi_N'
 \bigl(3\varepsilon_{y_1}^2+\varepsilon_{y_2}^2\bigr)
+\frac12\iint(\phi_N'''-\phi_N')\varepsilon^2\\
&\quad\leq
-\frac N2\frac{\lambda_s}{\lambda}
 \iint\phi_N\varepsilon^2
-\frac14\iint\phi_N'
 \bigl(|\nabla\varepsilon|^2+\varepsilon^2\bigr)
+C_N\mathcal N_{1,\mathrm{loc}}.
\end{align*}

Recall that $|Q|+|Q_{y_1}|+|P|+|P_{y_1}|\lesssim e^{-\frac{|y_1|}{\alpha_1}-C_{\alpha_1}|y_2|},$ for $y_1>0,$ and Lemma \DIFaddbegin \DIFadd{~\ref{decompositionlemma}
}\DIFaddend \[
\text{Mod}(s) := \left| \frac{\lambda_s}{\lambda} + b \right| + \left| \frac{(x_1)_s}{\lambda} - 1 \right| + \left| \frac{(x_2)_s}{\lambda} \right|\lesssim b^2+\mathcal{N}_{1,loc}^{\frac{1}{2}}+\nu^*\widetilde{\mathcal{N}}_1.
\]
In addition, \eqref{McomparedtoN} gives
$\widetilde{\mathcal M}\lesssim\mathcal N_i$, so that
\eqref{eq:modulatedcoefficients} and \eqref{H1} imply the uniform bound
$\operatorname{Mod}(s)\lesssim\nu^*$. Also,
$$\Big|\iint \phi_N\varepsilon\Lambda Q_b\Big|+\Big|\iint \phi_N\varepsilon (Q_b)_{y_1}\Big|+\Big|\iint \phi_N\varepsilon (Q_b)_{y_2}\Big|\lesssim \mathcal{N}_{1,loc}^{\frac{1}{2}},$$
and $|\iint \phi_N\varepsilon\varepsilon_{y_1}|=|\frac{1}{2}\iint \phi_N'\varepsilon^2|.$ 
Therefore 
$$\text{Mod}(s)\Big(\mathcal{N}_{1,loc}^{\frac{1}{2}}+\iint \phi_N'\varepsilon^2\Big)\lesssim b^4+\mathcal{N}_{1,loc}+\delta(\nu^*)\Big(\widetilde{\mathcal{N}}_1+\iint \phi_N'\varepsilon^2\Big) .$$

From \eqref{lemma:Psi_b_estimates}, we obtain 
$$\Big|\iint \phi_N\varepsilon\Psi_b\Big|\lesssim b^4+\mathcal{N}_{1,loc}.$$

From Lemma \ref{decompositionlemma}, 
$$\Big|\iint \phi_N\varepsilon\Phi_b\Big|\lesssim |b_s|\mathcal{N}_{1,loc}^{\frac{1}{2}}\lesssim b^4+\widetilde{\mathcal{N}}_{1}.$$

For the nonlinear term, using that $|Q_b|+|(Q_b)_{y_1}|\lesssim e^{-\frac{|y_1|}{2}}$ for $y_1>0,$ we bound 
$$\left|\iint \phi_{N}\varepsilon[(Q_b+\varepsilon)^3-Q_b^3]_{y_1}\right|\lesssim \iint_{y_1>0} e^{-\frac{|y_1|}{2}}(\varepsilon^2+|\varepsilon|^3)+\iint \phi'_N\varepsilon^4,$$
$$\lesssim \widetilde{\mathcal{N}}_1+\iint \phi'_N\varepsilon^4,$$
where in the last inequality we used the weighted Sobolev Lemma \ref{SobolevLemma}.
For the last term, set
\[
\theta_N(y_1)=\phi_N'(y_1)+\omega(y_1),\qquad
\omega(y_1)=\operatorname{sech}(y_1/\alpha_1).
\]
The function $\omega$ is positive and has bounded logarithmic derivative.
On $[0,1]$ it has a positive minimum, whereas $\phi_N''$ is bounded;
for $y_1\geq1$ the logarithmic derivative of the polynomial
$\phi_N'$ is bounded, and for $y_1\leq0$ we have $\phi_N'=0$.
It follows that $|\theta_N'|\leq C_N\theta_N$. Moreover,
$\omega\leq C(B)\phi_{1,B}'$: this follows from the exponential
lower bound on the left, positivity on bounded intervals, and the
constant positive slope on the right. Lemma~\ref{SobolevLemma},
applied to the weight $\theta_N$ on $\mathbb R^2$, therefore gives
\[
\iint\phi_N'\varepsilon^4
\lesssim_{N,B}\|\varepsilon\|_{L^2}^2
\left(\iint\phi_N'(|\nabla\varepsilon|^2+\varepsilon^2)
+\widetilde{\mathcal N}_1\right).
\]
Hence, using $\|\varepsilon\|_{L^2}\leq \nu^*,$ we obtain 
$$\left|\iint \phi_{N}\varepsilon[(Q_b+\varepsilon)^3-Q_b^3]_{y_1}\right|\lesssim \widetilde{\mathcal{N}}_1+\delta(\nu^*)\iint \phi'_N(|\nabla \varepsilon|^2+\varepsilon^2).$$

These estimates prove the claimed differential inequality.

\end{proof}

\begin{lemma}\label{bootstrap}
Using the notation above, we have $t_i^{**} = t^{*}$.
\end{lemma}
\begin{proof}

Take $t<t_i^{**}.$ Since $u(t) \in \mathcal{T}_{\alpha^*}$, by Lemma \ref{eq:nualpha} with  $\nu= \alpha^*,$ we get that $|b(t)| \lesssim \delta(\alpha^*) \ll \nu^*$. As $\left| \iint u_0^2 - \iint Q^2 \right| \leq \alpha_0 \ll \nu^*$ by the choice of the initial data, then the conservation of mass \eqref{eq:massconservation} gives
\[
\|\varepsilon\|_{L^2}^2 \lesssim |b| + \left| \iint u_0^2 - \iint Q^2 \right| \lesssim \delta(\alpha^*) + \alpha_0 \ll \frac{(\nu^*)^2}{4}.
\]
From Lemma \ref{initialparametersmall}, we have $\mathcal{N}_i(0)\lesssim \delta(\alpha_0).$ We conclude by \eqref{eq:Dispersiveestimates1},  
\[
\mathcal{N}_i(s) \lesssim \mathcal{N}_i(0) + |b^3(0)| + |b^3(s)| \lesssim \delta(\alpha_0)+\delta(\alpha^*) \ll \nu^*.
\]
So we improved \eqref{H1condi} on $[0,t_i^{**}].$ 
From \eqref{eq:blambdacinequality}, Lemma \ref{initialparametersmall} and the improved bound $|b(s)|\lesssim \delta(\alpha^*)$, then 
$$\frac{|b(s)|}{\lambda^c(s)}\lesssim \delta(\alpha_0).$$
From \eqref{eq:Dispersiveestimates2} and Lemma ~\ref{initialparametersmall}, we obtain 
$$\frac{\mathcal{N}_i(s)}{\lambda^c(s)}\lesssim \delta(\alpha_0),$$
so, after choosing $\alpha_0$ sufficiently small for the fixed
$\mathcal C_i$, we obtain
\[
\frac{|b(s)|+\mathcal N_i(s)}{\lambda^c(s)}<\frac{\mathcal C_i}{2},
\]
which improves \eqref{H2condi} on $[0,t^{**}].$
Since $\phi_{K_i}$ and $y_{1,+}^{K_i}$ agree outside $[0,1]$ and
have bounded difference there, integrating Lemma~\ref{lemma21} gives
\begin{align*}
\iint y_{1,+}^{K_i}\varepsilon^2(s)
\leq{}&\frac{\lambda^{K_i}(0)}{\lambda^{K_i}(s)}
\iint y_{1,+}^{K_i}\varepsilon^2(0)
+C\|\varepsilon(s)\|_{L^2}^2
+C\frac{\lambda^{K_i}(0)}{\lambda^{K_i}(s)}
\|\varepsilon(0)\|_{L^2}^2\\
&+C\int_0^s
\left(\frac{\lambda(\tau)}{\lambda(s)}\right)^{K_i}
\bigl(b^4+\widetilde{\mathcal N}_1\bigr)(\tau)\,d\tau.
\end{align*}
By Lemma~\ref{lambdaboundedbefores**},
\[
\left(\frac{\lambda(\tau)}{\lambda(s)}\right)^{K_i}
\leq C_0^{K_i}\left(1+\lambda(s)^{-K_i}\right).
\]
Since $|b|\ll1$, \eqref{equationb2bis} gives
\[
\int_0^s b^4\,ds'
\lesssim |b(s)|+|b(0)|+\mathcal N_1(0)
\lesssim\delta(\alpha_0)+\delta(\alpha^*).
\]
Likewise, \eqref{eq:Dispersiveestimates1} yields
\[
\int_0^s\widetilde{\mathcal N}_1\,ds'
\lesssim\mathcal N_1(0)+|b(0)|^3+|b(s)|^3
\lesssim\delta(\alpha_0)+\delta(\alpha^*).
\]
The improved mass bound controls the $L^2$ norm
at time $s$, and Lemma~\ref{initialparametersmall} controls its initial
counterpart. Together with $|\lambda(0)-1|\lesssim\delta(\alpha_0)$
and the initial moment bound, these estimates imply

\[
\iint y_{1,+}^{K_i}\varepsilon^2(s)
\leq\bigl[1+\delta(\alpha_0)+\delta(\alpha^*)\bigr]
\bigl[1+\lambda(s)^{-K_i}\bigr].
\]

Here the fixed constants, including $C_0^{K_i}$, are included in the
smallness choices. Taking $\alpha^*$ and then $\alpha_0$ sufficiently
small makes the coefficient strictly less than $2$, improving
 \eqref{H3condi} on $[0,t_i^{**})$.
 By improving \eqref{H1condi}, \eqref{H2condi}, \eqref{H3condi} on $[0,t_i^{**}]$, we deduce that $t_i^{**} = t^*$ by a continuity argument.
\end{proof}

\chapter{Rigidity near the soliton} \label{Rigidity near the soliton}

 We combine the modulation estimates and monotonicity properties of the
preceding chapters to classify the dynamics of solutions starting near
the soliton  manifold.

 Theorem~\ref{cap:verified-main} gives
\(\frac{21}{13}<c<\frac53\), hence \(K_5=21\). We write
\(\mathcal A_{\alpha_0}:=\mathcal A_{\alpha_0}^{(21)}\), fix
$\mathcal C_5>1$, and use the following a priori conditions:
\begin{gather}
|b(s)| + \mathcal{N}_5(s) + \|\varepsilon(s)\|_{L^2} \leq \nu^*, \label{H1cond} \tag{\textbf{H1'}} \\
\frac{|b(s)|+\mathcal{N}_5(s)}{\lambda^c(s)} \leq \mathcal C_5, \label{H2cond} \tag{\textbf{H2'}} \\
\iint_{y_1>0} \varepsilon^2(s)y_1^{21} \leq 2\big(1+\lambda^{-21}(s)\big). \label{H3cond} \tag{\textbf{H3'}}
\end{gather}
with $\nu^*$ defined as in Proposition \ref{Monotonicity}. Recall $t^*$ to be the exit time of the tube $\mathcal{T}_{\alpha^*}.$

\begin{theorem} \label{rigiditytheorem}
There exist constants $0 < \alpha_{0}^{*} \ll \alpha^* \ll \nu^*$, depending on the fixed $\mathcal C_5$, such that the following holds. Let $u_0 \in \mathcal{A}_{\alpha_0}$ with $0 < \alpha_0 < \alpha_{0}^{*}$; then $u(t)$ satisfies the assumptions \eqref{H1cond}--\eqref{H2cond}--\eqref{H3cond} on $[0,t^*)$.

Furthermore, the following trichotomy holds:

\textbf{(Asymptotic Stability)} Suppose $t^* = T = +\infty$. There exists $\lambda_{\infty}$ satisfying $|\lambda_{\infty}-1| \leq \delta(\alpha_0)$, $x_{2,\infty} \in \mathbb{R}$, and $x_1(t) \in C^1$ such that
\[
\lambda_{\infty}u(t, \lambda_{\infty}\cdot+x_1(t), \lambda_{\infty}\cdot+x_{2,\infty}) \rightarrow Q \quad \mbox{in } H^{1}_{loc} \mbox{ as } t \rightarrow +\infty.
\]
Moreover,
\begin{equation*}
\begin{split}
&\mathcal{N}_5(t) \rightarrow 0, \quad b(t) \rightarrow 0 \quad \text{as } t \rightarrow +\infty,\\
&\lim_{t\rightarrow\infty} \lambda(t) = \lambda_{\infty}, \quad x_1(t) = \frac{t}{\lambda_{\infty}^{2}}(1 + o_{t\rightarrow \infty}(1)), \quad \lim_{t\rightarrow\infty} x_2(t) = x_{2,\infty} \in \mathbb{R}.
\end{split}
\end{equation*}
Furthermore, there exists $C^* > 0$ such that $|b(t)| \lesssim C^* \mathcal{N}_3(t)$ for all $t \geq 0$.

\textbf{(Stable Blow-Up)} Suppose $t^* = T < +\infty$. There exists $0 < \ell_0 < \delta(\alpha_0)$ such that
\[
\lim_{t \rightarrow T} \frac{\lambda(t)}{(T-t)^{\frac{1}{3-c}}} = \ell_0, \quad \lim_{t \rightarrow T} \frac{b(t)}{(T-t)^{\frac{c}{3-c}}} = \frac{\ell_0^3}{3-c}, \quad \lim_{t\rightarrow T}(T-t)^{\frac{c-1}{3-c}}x_1(t)=
\frac{3-c}{c-1}\cdot\frac{1}{\ell_0^2},
\]
\[
\text{ and }\lim_{t\rightarrow T} x_2(t) = x_{2,\infty}, \quad \text{for some } x_{2,\infty} \in \mathbb{R}.
\]
Moreover, the following error bounds hold:
\[
\|\nabla \varepsilon(t)\|_{L^2} \asymp \lambda^{\frac{c}{2}}(t), \quad \|\varepsilon(t)\|_{L^2} \lesssim \delta(\alpha_0).
\]

\textbf{(Exit of Tube)} Suppose $t^* < T$. We have
\[
\inf_{\lambda_1>0, x_1, x_2 \in \mathbb{R}} \left\| u(t^*, \cdot, \cdot) - \frac{1}{\lambda_1} Q\left( \frac{\cdot -x_1}{\lambda_1}, \frac{\cdot-x_2}{\lambda_1} \right) \right\|_{L^2} = \alpha^*.
\]
In addition,
\[
\frac{(\alpha^*)^{\frac{2}{c}}}{\delta(\alpha_0)} \lesssim \lambda(t^*), \quad b(t^*) \lesssim -(\alpha^*)^2.
\]
\end{theorem}
\begin{remark}
Since  \(c>1\) by Proposition~\ref{c>1}, the stable blow-up
asymptotics imply \(x_1(t)\to+\infty\) as \(t\uparrow T\).
\end{remark}

Since $K_5=21,$ by Lemma \ref{bootstrap} we have $t_5^{**}=t^*,$ proving the first part of Theorem \ref{rigiditytheorem}. From now on, denote $t^{**}\coloneqq t_{5}^{**}>0,$ and $s^{**}\coloneqq s_{5}^{**}=s(t_{5}^{**})= s(t^*)\coloneqq s^*>0$.

We first record an elementary convergence criterion.

\begin{lemma}\label{convergencelemma}
Let \(m\in(0,+\infty]\) and let \(f\in AC_{\mathrm{loc}}([0,m))\). If
\(f'\in L^1([0,m))\), then
\[
    \lim_{t\uparrow m}f(t)
    =f(0)+\int_0^m f'(\tau)\,d\tau
\]
exists and is finite.
\end{lemma}

Lemma~\ref{bootstrap} gives \(t^{**}=t^*\), and therefore
\(s^{**}=s^*\). Applying \eqref{eq:blambdacinequality} with
\(s_1=0\) and \(s_2=s<s^*\), and using \eqref{H1cond}, we obtain
\[
\frac{|b(s)|}{\lambda^c(s)}
\lesssim
\frac{|b(0)|}{\lambda^c(0)}(1+|b(0)|)
+\frac{\mathcal N_3(0)}{\lambda^c(0)}.
\]
Hence Lemma~\ref{initialparametersmall} gives
\[
    \limsup_{s\uparrow s^*}
    \frac{|b(s)|}{\lambda^c(s)}
    \lesssim\delta(\alpha_0).
\]
Together with \eqref{eq:minmax}, \eqref{eq:controldynamics}, \eqref{boundiplusone2} and
\eqref{H1cond}, this yields
\[
    \int_0^{s^*}
    \left|
    \frac{d}{ds}
    \left(
       \frac{b}{\lambda^c}e^{J(s)}
    \right)
    \right|\,ds
    \lesssim\delta(\alpha_0).
\]
Therefore, by Lemma~\ref{convergencelemma},
\begin{equation}\label{definitionofc0}
    \frac{b(s)}{\lambda^c(s)}e^{J(s)}
    \longrightarrow c_0\in\mathbb R
    \qquad\text{as }s\uparrow s^*,
    \qquad |c_0|\lesssim\delta(\alpha_0).
\end{equation}
If \(c_0\neq0\), choose \(s_c<s^*\) so that
\[
    \left|
    \frac{b(s)}{\lambda^c(s)}e^{J(s)}-c_0
    \right|
    \leq\frac{|c_0|}{2},
    \qquad s_c\leq s<s^*.
\]

\section{Exit Case and the characterization \texorpdfstring{\(c_0<0\)}{c0<0}}

\subsection*{Exit implies \(c_0<0\)}

Assume that \(t^*<T\). By continuity of the distance to the soliton
manifold and the definition of \(t^*\),
\[
\begin{aligned}
(\alpha^*)^2
&=
\inf_{\lambda_1>0,\,a_1,a_2\in\mathbb R}
\left\|
 u(t^*)-
 \frac1{\lambda_1}
 Q\left(
   \frac{x_1-a_1}{\lambda_1},
   \frac{x_2-a_2}{\lambda_1}
 \right)
\right\|_{L^2_{x_1x_2}}^2\\
&\lesssim
|b(t^*)|^{2-\gamma}
+\|\varepsilon(t^*)\|_{L^2}^2.
\end{aligned}
\]
By \eqref{eq:massconservation} and
\(u_0\in\mathcal A_{\alpha_0}\),
\[
    (\alpha^*)^2\lesssim |b(t^*)|+\alpha_0.
\]
Since \(\alpha_0\ll\alpha^*\),
\begin{equation}\label{lowerboundonb}
    (\alpha^*)^2\lesssim |b(t^*)|.
\end{equation}

We claim that \(b(t^*)<0\). Suppose otherwise. Then
\eqref{lowerboundonb} gives \(b(t^*)>0\). Define
\[
s_0^*
=
\begin{cases}
0,
&\text{if }b(s)>0\text{ on }[0,s^*],\\[1mm]
\displaystyle
\sup\{s<s^*:b(s)=0\},
&\text{otherwise}.
\end{cases}
\]
Thus \(b\geq0\) on \([s_0^*,s^*]\). If \(s_0^*=0\), then
\(|\lambda(s_0^*)-1|\lesssim\delta(\alpha_0)\). If
\(s_0^*>0\), then \(b(s_0^*)=0\). Applying
\eqref{eq:blambdacinequality} with \(s_1=s\) and \(s_2=s_0^*\), and
absorbing the \(b^2(s)\)-term by \eqref{H1cond}, gives
\begin{equation}\label{Ndomb}
    |b(s)|\leq C^*\mathcal N_3(s),
    \qquad 0\leq s\leq s_0^*.
\end{equation}
Hence, by \eqref{eq:lambdatildeineq}, \eqref{Ndomb}, and
\eqref{boundiplusone},
\[
    \int_0^{s_0^*}
    \left|
    \frac{(\tilde\lambda)_s}{\tilde\lambda}
    \right|\,ds
    \lesssim
    \int_0^{s_0^*}\mathcal N_3(s)\,ds
    \lesssim\delta(\nu^*).
\]
Using Lemma~\ref{initialparametersmall} and \(|J_1|\ll1\), we obtain
\[
    |\lambda(s_0^*)-1|\lesssim\delta(\nu^*).
\]

Next, \eqref{eq:blambdacinequality} with \(s_1=0\) and
\(s_2=s\in[s_0^*,s^*]\), together with \eqref{H1cond}, yields
\begin{equation}\label{b1}
    \frac{|b(s)|}{\lambda^c(s)}
    \lesssim
    \frac{|b(0)|}{\lambda^c(0)}
    +\frac{\mathcal N_3(0)}{\lambda^c(0)}
    \lesssim\delta(\alpha_0).
\end{equation}
Since \(b\geq0\) on \([s_0^*,s^*]\), \eqref{eq:lambdatildeineq}
gives
\[
    \frac{(\tilde\lambda)_s}{\tilde\lambda}
    \leq-b(1-Cb)+C\mathcal N_2
    \leq C\mathcal N_2.
\]
Integrating and using \eqref{boundiplusone},
\[
    \log\frac{\tilde\lambda(s)}{\tilde\lambda(s_0^*)}
    \leq\delta(\nu^*),
\]
so that
\[
    \tilde\lambda(s)
    \leq2\tilde\lambda(s_0^*)
    \leq3,
    \qquad s_0^*\leq s\leq s^*.
\]
Together with \(|J_1|\ll1\) and \eqref{b1}, this gives
\(|b(t^*)|\lesssim\delta(\alpha_0)\), contradicting
\eqref{lowerboundonb}. Therefore \(b(t^*)<0\).

Since the modulation parameters extend continuously to \(t=t^*\),
\eqref{definitionofc0} gives
\[
    c_0
    =\frac{b(t^*)}{\lambda^c(t^*)}e^{J(s^*)}<0.
\]
Moreover, \eqref{lowerboundonb} yields
\[
    b(t^*)\lesssim-(\alpha^*)^2.
\]
Finally, \eqref{eq:Dispersiveestimates2} and \eqref{H1cond} imply
\[
    \frac{|b(s)|}{\lambda^c(s)}
    \lesssim\delta(\alpha_0),
    \qquad 0\leq s\leq s^*,
\]
and hence, after redefining the generic function \(\delta\),
\[
    \frac{(\alpha^*)^{2/c}}{\delta(\alpha_0)}
    \lesssim\lambda(t^*).
\]

\subsection*{\(c_0<0\) implies Exit}

Assume that \(c_0<0\) and, by contradiction, that \(t^*=T\). By
increasing \(s_c\) if necessary and using \(|J|\ll1\) in
\eqref{definitionofc0}, we may assume
\begin{equation}\label{boundsonb1}
    \frac{5c_0}{3}
    \leq
    \frac{b(t)}{\lambda^c(t)}
    \leq
    \frac{c_0}{3}<0,
    \qquad t_c\leq t<T,
\end{equation}
where \(s(t_c)=s_c\).

Since \(b<0\), \eqref{eq:lambdatildeineq} gives
\[
    \frac{(\tilde\lambda)_s}{\tilde\lambda}
    \geq-b(1+Cb)-C\mathcal N_2
    \geq-C\mathcal N_2.
\]
Thus, for \(s_c\leq s_1\leq s_2<s(T)\),
\[
    \log\frac{\tilde\lambda(s_2)}{\tilde\lambda(s_1)}
    \geq-C\int_{s_1}^{s_2}\mathcal N_2(s)\,ds
    \geq-\delta(\nu^*).
\]
Using \(|J_1|\ll1\), we obtain
\begin{equation}\label{almostmonotonicitylambda1}
    \lambda(s_2)\geq\frac12\lambda(s_1),
    \qquad s_c\leq s_1\leq s_2<s(T).
\end{equation}

The identity \(\tilde\lambda=\lambda(1-J_1)\), the relation
\(ds/dt=\lambda^{-3}\), and \eqref{eq:lambdatildeineq} give
\[
\left|
\tilde\lambda^{2-c}(\tilde\lambda)_t
+\frac{b}{\lambda^c}
\right|
\lesssim
\frac{b^2+\mathcal N_2+|b||J_1|}{\lambda^c}.
\]
Indeed,
\[
\frac1{\lambda^c}
\left(
\frac{(\tilde\lambda)_s}{\tilde\lambda}+b
\right)
=
\frac{\tilde\lambda^{2-c}(\tilde\lambda)_t}
     {(1-J_1)^{3-c}}
+\frac{b}{\lambda^c}.
\]
By \eqref{almostmonotonicitylambda1} and \eqref{boundiplusone},
\[
\int_{t_c}^{t}
\frac{\mathcal N_2(t')}{\lambda^c(t')}\,dt'
=
\int_{s_c}^{s(t)}
\lambda^{3-c}(s')\mathcal N_2(s')\,ds'
\lesssim
\delta(\nu^*)\tilde\lambda^{3-c}(t).
\]
Integrating the preceding differential estimate, using
\eqref{boundsonb1}, and absorbing this last term gives
\[
\tilde\lambda^{3-c}(t)
\geq
\frac12\tilde\lambda^{3-c}(t_c)
+\tilde{\kappa}|c_0|(t-t_c)
\]
for some \(\tilde{\kappa}>0\). The same integrated identity gives
\(\sup_{t_c\leq t<T}\lambda(t)<+\infty\) whenever \(T<+\infty\).
Thus, if \(T<+\infty\), \eqref{almostmonotonicitylambda1} shows that
\(\lambda\) is also bounded below on \([t_c,T)\); then
\eqref{boundsonb1},
\eqref{eq:massconservation}, and \eqref{eq:energyconservation} give a
uniform \(H^1\)-bound for \(u(t)\), contradicting the maximality of
\(T\). If \(T=+\infty\), the lower bound gives
\(\lambda(t)\to+\infty\), and \eqref{boundsonb1} then implies
\[
    |b(t)|\gtrsim |c_0|\lambda^c(t)\to+\infty,
\]
contradicting \eqref{H1cond}. Hence \(t^*<T\).

We have proved
\[
    t^*<T
    \quad\Longleftrightarrow\quad
    c_0<0. 
\]

\section{Blow-Up -- the case \texorpdfstring{\(c_0>0\)}{c0>0}}

Assume \(c_0>0\). The preceding equivalence excludes Exit, so
\(t^*=T\). Increasing \(s_c\), if necessary, we may assume
\[
    0<\frac{c_0}{3}
    \leq
    \frac{b(t)}{\lambda^c(t)}
    \leq
    \frac{5c_0}{3},
    \qquad t_c\leq t<T.
\]
Since \(b>0\), \eqref{eq:lambdatildeineq} gives
\[
    \frac{(\tilde\lambda)_s}{\tilde\lambda}
    \leq-b(1-Cb)+C\mathcal N_2
    \leq C\mathcal N_2.
\]
Consequently,
\begin{equation}\label{almostmonotonicitylambda2}
    \lambda(s_2)\leq2\lambda(s_1),
    \qquad s_c\leq s_1\leq s_2<s(T).
\end{equation}

As above,
\[
\left|
\tilde\lambda^{2-c}(\tilde\lambda)_t
+\frac{b}{\lambda^c}
\right|
\lesssim
\frac{b^2+\mathcal N_2+|b||J_1|}{\lambda^c}.
\]
Moreover, \eqref{almostmonotonicitylambda2} and
\eqref{boundiplusone} give
\[
\int_{t_c}^{t}
\frac{\mathcal N_2(t')}{\lambda^c(t')}\,dt'
=
\int_{s_c}^{s(t)}
\lambda^{3-c}(s')\mathcal N_2(s')\,ds'
\lesssim
\delta(\nu^*)\lambda^{3-c}(t_c).
\]
Integrating the differential estimate and decreasing \(\nu^*\), we
obtain the signed upper bound
\[
    \tilde\lambda^{3-c}(t)
    \leq
    2\tilde\lambda^{3-c}(t_c)
    -\frac{3-c}{4}c_0(t-t_c).
\]
If \(T=+\infty\), the right-hand side becomes negative in finite time,
whereas the left-hand side is positive. Hence \(T<+\infty\).

We claim that \(\lambda(t)\to0\) as \(t\uparrow T\). Otherwise,
\eqref{almostmonotonicitylambda2} would give a positive lower bound for
\(\lambda\) on \([t_c,T)\). Since \(b\), \(\mathcal N_5\), and
\(\|\varepsilon\|_{L^2}\) remain controlled in the tube, while
\eqref{eq:energyconservation} controls \(\|\nabla\varepsilon\|_{L^2}\),
the \(H^1\)-norm of \(u(t)\) would remain bounded, contradicting the
maximality of \(T\). Thus
\[
    \lambda(t)\to0,
    \qquad
    b(t)\to0
    \qquad\text{as }t\uparrow T.
\]

Fix \(0<\eta<1\) for which \eqref{eq:Dispersiveestimates2} is
available. For \(s_c\leq s<s(T)\),
\[
\begin{aligned}
\frac{\mathcal N_5(s)}{\lambda^{(3-\eta)c}(s)}
&\lesssim
\frac{\mathcal N_5(s_c)}{\lambda^{(3-\eta)c}(s_c)}
+
\frac{b^3(s)}{\lambda^{(3-\eta)c}(s)}
+
\frac{b^3(s_c)}{\lambda^{(3-\eta)c}(s_c)}\\
&\lesssim
1+\lambda^{\eta c}(s)+\lambda^{\eta c}(s_c)
\lesssim1.
\end{aligned}
\]
Extending the constant over the compact initial interval gives
\begin{equation}\label{sharpN}
    \mathcal N_5(s)
    \lesssim_{s_c}
    \lambda^{(3-\eta)c}(s),
    \qquad 0\leq s<s(T).
\end{equation}
In particular, \(\mathcal N_5(t)\to0\) as \(t\uparrow T\).

Define
\[
    \ell_0:=\bigl((3-c)c_0\bigr)^{1/(3-c)}>0.
\]
Since \(|J(t)|\lesssim\mathcal N_2^{1/2}(t)\to0\),
\eqref{definitionofc0} gives
\[
    \frac{b(t)}{\lambda^c(t)}
    \longrightarrow
    c_0
    =\frac{\ell_0^{3-c}}{3-c},
    \qquad
    \ell_0\lesssim\delta(\alpha_0).
\]

We next derive the sharp law for \(\lambda\). From
\eqref{eq:lambdatildeineq},
\[
    \left|
    (\tilde\lambda)_t\tilde\lambda^{2-c}
    +\frac{b}{\lambda^c}
    \right|
    \lesssim
    \frac{\mathcal N_2+b^2+|b||J_1|}{\lambda^c}.
\]
Since \(\tilde\lambda(t)\to0\), integration on \([t,T)\) gives
\[
\begin{aligned}
\left|
\frac{\tilde\lambda^{3-c}(t)}{3-c}
-
\int_t^T\frac{b(t')}{\lambda^c(t')}\,dt'
\right|
&\lesssim
\int_t^T
\frac{\mathcal N_2(t')+o(|b(t')|)}{\lambda^c(t')}\,dt'\\
&\lesssim
\lambda^{3-c}(t)
\int_{s(t)}^{s(T)}\mathcal N_2(s')\,ds'
+o(T-t),
\end{aligned}
\]
where \eqref{almostmonotonicitylambda2} was used in the last line. Set
\[
    X(t):=\frac{\lambda^{3-c}(t)}{T-t},
    \qquad
    \beta(t):=\frac1{T-t}
    \int_t^T\frac{b(t')}{\lambda^c(t')}\,dt'.
\]
By \eqref{boundiplusone}, \eqref{sharpN}, and \(b(t)\to0\), there are
nonnegative functions \(\omega_1(t),\omega_2(t)\to0\) such that
\[
    \left|
    \frac{(1-J_1(t))^{3-c}}{3-c}X(t)-\beta(t)
    \right|
    \leq\omega_1(t)X(t)+\omega_2(t).
\]
Moreover, Ces\`aro averaging gives \(\beta(t)\to c_0\). For \(t\) close
to \(T\),
\[
    (1-J_1(t))^{3-c}\geq\frac12,
    \qquad
    \omega_1(t)\leq\frac1{4(3-c)}.
\]
Hence
\[
    \frac1{4(3-c)}X(t)
    \leq |\beta(t)|+\omega_2(t),
\]
so \(X(t)\) is bounded. Consequently \(\omega_1(t)X(t)\to0\), and
passing to the limit yields
\begin{equation}\label{eq:tlawforlambda}
    \lim_{t\rightarrow T}
    \frac{\lambda^{3-c}(t)}{T-t}
    =\ell_0^{3-c},
    \qquad
    \lim_{t\rightarrow T}
    \frac{\lambda(t)}{(T-t)^{1/(3-c)}}
    =\ell_0.
\end{equation}
Combining this with \(b/\lambda^c\to\ell_0^{3-c}/(3-c)\), we obtain
\begin{equation}\label{eq:tlawforb}
    \lim_{t\rightarrow T}
    \frac{b(t)}{(T-t)^{c/(3-c)}}
    =\frac{\ell_0^3}{3-c}.
\end{equation}

The law \eqref{eq:tlawforlambda} implies \(s(T)=+\infty\). Since
\[
    \left|
    \frac{(x_1)_s}{\lambda}-1
    \right|
    \lesssim
    \mathcal N_2^{1/2}+b^2
    \longrightarrow0,
\]
we obtain
\[
    (T-t)^{2/(3-c)}(x_1)_t(t)
    \longrightarrow\frac1{\ell_0^2}.
\]
Because \(c>1\),
\begin{equation*}
    x_1(t)
    =
    \frac{3-c}{c-1}\frac1{\ell_0^2}
    (T-t)^{(1-c)/(3-c)}
    \bigl(1+o_{t\rightarrow T}(1)\bigr).
\end{equation*}

The sharp equation for \(x_2\) in
Lemma~\ref{sharporthogonalities} gives
\[
    \int_0^\infty
    |(x_2-\lambda J_2)_s|\,ds
    \lesssim
    \int_0^\infty
    \lambda(s)(\mathcal N_2(s)+b^2(s))\,ds
    <+\infty.
\]
Since \(J_2(s)\lesssim\mathcal N_2^{1/2}(s)\to0\),
Lemma~\ref{convergencelemma} yields
\begin{equation*}
    x_2(t)\longrightarrow x_{2,\infty}\in\mathbb R
    \qquad\text{as }t\uparrow T.
\end{equation*}

In the \(s\)-variable,
\[
    s
    =\frac{3-c}{c}\frac1{\ell_0^3}
    (T-t)^{-c/(3-c)}\bigl(1+o(1)\bigr),
\]
\[
    \lambda(s)
    =
    \left(\frac{3-c}{c}\ell_0^{c-3}\right)^{1/c}
    s^{-1/c}\bigl(1+o(1)\bigr),
    \qquad
    b(s)=\frac1{cs}\bigl(1+o(1)\bigr),
\]
\[
    x_1(s)
    =
    \frac{3-c}{c-1}
    \left(\frac{c}{3-c}\right)^{(c-1)/c}
    \ell_0^{(c-3)/c}
    s^{(c-1)/c}\bigl(1+o(1)\bigr),
\]
and \(x_2(s)\to x_{2,\infty}\).

From \eqref{eq:massconservation}, \eqref{eq:energyconservation},
\eqref{H1cond}, \eqref{sharpN}, and \(c<2\),  we obtain
\begin{equation}\label{eq:estimateL2}
    \|\varepsilon(t)\|_{L^2}
    \lesssim\delta(\alpha_0),
\end{equation}
and 
\begin{equation}\label{eq:estimateL2derivative}
    \|\nabla\varepsilon(t)\|_{L^2}
    \asymp\lambda^{c/2}(t)
    \qquad\text{as }t\uparrow T.
\end{equation}

A change of variables gives
\[
    \|\partial_{x_1}u(t)\|_{L^2_{x}}^2
    =
    \frac{
    \|\partial_{y_1}(Q_{b(t)}+\varepsilon(t))\|_{L^2_{y}}^2
    }{\lambda^2(t)}.
\]
Moreover,
\[
\begin{aligned}
\|\partial_{y_1}(b\chi_bP)\|_{L^2}^2&=O(b^2),
&
\left|
\iint \partial_{y_1}Q\,\partial_{y_1}(b\chi_bP)
\right|&=O(|b|),\\
\left|
\iint \partial_{y_1}\varepsilon\,\partial_{y_1}Q
\right|&=O(\lambda^{c/2}),
&
\left|
\iint \partial_{y_1}\varepsilon\,
\partial_{y_1}(b\chi_bP)
\right|&=O(\lambda^{c/2}|b|).
\end{aligned}
\]
Thus \eqref{eq:tlawforlambda} gives
\begin{equation}\label{convergencepartial}
    \lim_{t\rightarrow T}
    (T-t)^{1/(3-c)}
    \|\partial_{x_1}u(t)\|_{L^2}
    =
    \frac{\|\partial_{x_1}Q\|_{L^2}}{\ell_0(u_0)}.
\end{equation}
The same calculation in the \(x_2\)-direction yields
\begin{equation}\label{convergencenabla}
    \lim_{t\rightarrow T}
    (T-t)^{1/(3-c)}
    \|\nabla u(t)\|_{L^2}
    =
    \frac{\|\nabla Q\|_{L^2}}{\ell_0(u_0)}.
\end{equation}

\section{Asymptotic Stability -- the case
\texorpdfstring{\(c_0=0\)}{c0=0}}

Assume \(c_0=0\). The Exit characterization gives \(t^*=T\). Since
\(e^{J(s)}\) is bounded above and below, letting \(s_2\uparrow s(T)\)
in \eqref{eq:blambdacinequality} yields a constant \(C^*>0\) such that
\begin{equation}\label{Ndominatingb}
    |b(s_1)|\leq C^*\mathcal N_3(s_1),
    \qquad 0\leq s_1<s(T).
\end{equation}
By \eqref{eq:lambdatildeineq}, \eqref{Ndominatingb}, and
\eqref{boundiplusone},
\[
\begin{aligned}
\int_0^{s_1}
\left|
\frac{(\tilde\lambda)_s}{\tilde\lambda}
\right|\,ds
&\lesssim
\int_0^{s_1}(|b|+\mathcal N_2)\,ds\lesssim
\int_0^{s_1}\mathcal N_3(s)\,ds
\lesssim\delta(\nu^*).
\end{aligned}
\]
Hence
\[
    \left|
    \log\frac{\tilde\lambda(s)}{\tilde\lambda(0)}
    \right|
    \lesssim\delta(\nu^*).
\]
Since \(|\tilde\lambda/\lambda-1|=|J_1|\lesssim\mathcal N_2^{1/2}\),
Lemma~\ref{initialparametersmall} gives
\begin{equation}\label{lambdacloseto1}
    |\lambda(s)-1|\lesssim\delta(\nu^*),
    \qquad 0\leq s<s(T).
\end{equation}
Mass conservation also gives
\begin{equation}\label{errorL2normsmall1}
    \|\varepsilon(s)\|_{L^2}
    \lesssim\delta(\nu^*),
    \qquad 0\leq s<s(T).
\end{equation}

By \eqref{eq:energyconservation}, \eqref{H1cond},
\eqref{lambdacloseto1}, and \eqref{errorL2normsmall1},
\[
    \|\nabla\varepsilon(t)\|_{L^2}^2
    \leq2|E_0|+1,
    \qquad 0\leq t<T.
\]
Thus \(\|u(t)\|_{H^1}\) is uniformly bounded, so \(T=+\infty\).
Since \(\lambda\) is bounded above and below,
\[
    s(T)=\int_0^\infty\lambda(t)^{-3}\,dt=+\infty.
\]

From \eqref{Ndominatingb} and \eqref{boundiplusone},
\[
    \int_0^\infty |b(s)|\,ds<+\infty.
\]
Moreover, \eqref{eq:b}, \eqref{McomparedtoN}, and
\eqref{eq:controldynamics} give
\[
    \int_0^\infty |b_s(s)|\,ds
    \lesssim
    \int_0^\infty(b^2+\mathcal N_1)\,ds
    <+\infty.
\]
Hence Lemma~\ref{convergencelemma} gives a finite limit for \(b(s)\),
and the integrability of \(b\) forces this limit to be zero:
\begin{equation}\label{bdecay1}
    b(t)\longrightarrow0
    \qquad\text{as }t\rightarrow+\infty.
\end{equation}

Similarly,
\[
    \int_0^\infty
    \left|
    \frac{(\tilde\lambda)_s}{\tilde\lambda}
    \right|\,ds
    \lesssim
    \int_0^\infty(|b|+\mathcal N_2)\,ds
    <+\infty.
\]
Lemma~\ref{convergencelemma} and \eqref{lambdacloseto1} imply
\[
    \tilde\lambda(s)\longrightarrow
    \lambda_\infty\in(0,+\infty),
    \qquad
    |\lambda_\infty-1|\lesssim\delta(\alpha_0).
\]

From \eqref{eq:Dispersiveestimates1} with \(i=5\),
\[
    \int_0^\infty
    \iint
    (|\nabla\varepsilon|^2+\varepsilon^2)
    (\phi_{5,B})_{y_1}\,ds
    \lesssim\delta(\nu^*).
\]
Hence there exists \(s_n\to+\infty\), with \(t_n=t(s_n)\), such that
\[
    \iint
    (|\nabla\varepsilon|^2+\varepsilon^2)(t_n)
    (\phi_{5,B})_{y_1}
    \longrightarrow0.
\]
By \eqref{phiboundderivativephi},
\eqref{lowerboundderivativetransition}, \(K_5=21\),
\eqref{H3cond}, and \eqref{lambdacloseto1},
\[
    \iint_{y_1<B}
    \varepsilon^2(t_n)\phi_{5,B}
    \longrightarrow0,
\]
\[
\begin{aligned}
\iint_{y_1>B}
\varepsilon^2(t_n)\phi_{5,B}
&\lesssim
\left(
\iint_{y_1>B}
\varepsilon^2(t_n)(\phi_{5,B})_{y_1}
\right)^{16/17}
\left(
\iint_{y_1>B}y_1^{21}\varepsilon^2(t_n)
\right)^{1/17}\\
&\longrightarrow0,
\end{aligned}
\]
and
\[
    \iint|\nabla\varepsilon|^2(t_n)\psi_B
    \lesssim
    \iint|\nabla\varepsilon|^2(t_n)(\phi_{5,B})_{y_1}
    \longrightarrow0.
\]
Therefore \(\mathcal N_5(t_n)\to0\). Using
\eqref{eq:Dispersiveestimates1} and \eqref{bdecay1},
\begin{equation}\label{N6decay}
    \mathcal N_5(t)
    \lesssim
    \mathcal N_5(t_n)+|b^3(t_n)|+|b^3(t)|
    \longrightarrow0
    \qquad\text{as }t\rightarrow+\infty.
\end{equation}
Since \(\tilde\lambda(t)\to\lambda_\infty\), \eqref{N6decay} gives
\[
    \lambda(t)\longrightarrow\lambda_\infty.
\]

From \eqref{eq:modulatedcoefficients},
\[
    \left|
    \frac{(x_1)_s}{\lambda}-1
    \right|
    \lesssim
    \mathcal N_2^{1/2}+b^2
    \longrightarrow0.
\]
Thus
\[
    (x_1)_t(t)\longrightarrow\frac1{\lambda_\infty^2},
    \qquad
    x_1(t)=\frac{t}{\lambda_\infty^2}
    \bigl(1+o_{t\rightarrow+\infty}(1)\bigr).
\]
The sharp modulation equation for \(x_2\) gives
\[
    \int_0^\infty
    |(x_2-\lambda J_2)_s|\,ds
    \lesssim
    \int_0^\infty
    \lambda(s)(\mathcal N_2(s)+b^2(s))\,ds
    <+\infty.
\]
Together with Lemma~\ref{convergencelemma} and \eqref{N6decay}, this
implies
\[
    x_2(t)\longrightarrow x_{2,\infty}\in\mathbb R.
\]

Finally,
\[
\begin{aligned}
&\left\|
\lambda(t)
 u\bigl(
 t,\lambda(t)y_1+x_1(t),
   \lambda(t)y_2+x_2(t)
 \bigr)-Q(y_1,y_2)
\right\|_{H^1_{\mathrm{loc}}}\\
&\qquad\lesssim
\mathcal N_2^{1/2}(t)+|b(t)|^{1-\gamma/2}
\longrightarrow0.
\end{aligned}
\]
Since \(\lambda(t)\to\lambda_\infty\) and
\(x_2(t)\to x_{2,\infty}\),
\[
    \left\|
    \lambda_\infty
    u\bigl(
      t,\lambda_\infty y_1+x_1(t),
        \lambda_\infty y_2+x_{2,\infty}
    \bigr)-Q(y_1,y_2)
    \right\|_{H^1_{\mathrm{loc}}}
    \longrightarrow0.
\]

\begin{remark}\label{21vs18}
The rigidity argument uses only the level \(i=4\) in
 \eqref{H1condi}--\eqref{H3condi}, which suffices for
Lemma~\ref{lambdaboundedbefores**}. We retain \(i=5\) because the
 zero-energy blow-up argument uses
\(\mathcal N_5(t)\to0\) in the asymptotic-stability regime; see
Lemma~\ref{masslemma}.
\end{remark}

\begin{remark}\label{nominimalmassblowup}
 In the stable blow-up regime of Theorem~\ref{rigiditytheorem}, the
initial mass is strictly greater than the ground-state mass:
\(\|u_0\|_{L^2}>\|Q\|_{L^2}\). Indeed, local well-posedness,
conservation of energy, and the sharp Gagliardo--Nirenberg inequality
exclude finite-time blow-up below the ground-state mass. Suppose now
that \(\|u_0\|_{L^2}=\|Q\|_{L^2}\) and that the solution belongs to
the stable blow-up regime. Mass conservation  and
\((\varepsilon,Q)=0\) give
\[
0=2b(P,Q)+o(b)+\|\varepsilon\|_{L^2}^2
  +2b(\varepsilon,\chi_bP)+b^2\|\chi_bP\|_{L^2}^2
\]
as \(t\uparrow T\), where \(b(t)>0\) and \(b(t)\to0\).
Since \(\gamma<1\), the estimates
\[
b^2\|\chi_bP\|_{L^2}^2\lesssim b^{2-\gamma}=o(b),
\qquad
2b|(\varepsilon,\chi_bP)|
\leq\frac12\|\varepsilon\|_{L^2}^2+Cb^{2-\gamma}
\]
imply

\[
0\geq2b(P,Q)+\frac12\|\varepsilon\|_{L^2}^2+o(b).
\]

Dividing by \(b>0\) and letting \(t\uparrow T\) contradicts
\((P,Q)>0\).
\end{remark}

\chapter{Openness of the blow-up branch} \label{Stability of Blow-Up}

For a solution in the tube, write
\[
 u(t,x)=\frac1{\lambda(t)}(Q_{b(t)}+\varepsilon(t))
 \left(\frac{x_1-x_1(t)}{\lambda(t)},\frac{x_2-x_2(t)}{\lambda(t)}\right),
 \qquad \frac{ds}{dt}=\frac1{\lambda^3(t)},
\]
with the orthogonality conditions of the modulation lemma.  We set
\[
         A(s)=\frac{b(s)}{\lambda^c(s)} .
\]
We use the following consequences of the previous chapters.  First, the modulation
map is unique and $C^1$ in a neighborhood of the soliton manifold.  Second, for
solutions in the bootstrap regime, we have from \eqref{eq:Dispersiveestimates1} and \eqref{eq:blambdacinequality} that 
\begin{equation}\label{eq:N4}
 \mathcal{N}_4(s_2)\le C\bigl(\mathcal{N}_4(s_1)+|b(s_1)|^3+|b(s_2)|^3\bigr),\qquad s_2\ge s_1,
\end{equation}
while
\begin{equation}\label{eq:almost}
 \left|A(s_2)-A(s_1)\right|
 \le C_0\left\{
 \frac{\mathcal{N}_3(s_1)+b^2(s_1)}{\lambda^c(s_1)}
 +\frac{b^2(s_2)}{\lambda^c(s_2)}\right\} .
\end{equation}
Finally, the rigidity theorem says that the branch is determined by the sign of
$\lim A(s)$; in particular $\liminf A(s)>0$ implies stable finite-time blow-up.

\begin{lemma}\label{lem:transfer}
Let $v_0\in\mathcal{A}_{\alpha_0}$ and let $v(t)$ remain in $\mathcal{T}_{\alpha^*}$ on $[0,t_0]$.  Let
$w_{0,n}\in\mathcal{A}_{\alpha_0}$ and $w_{0,n}\to v_0$ in $H^1(\mathbb{R}^2)$.  Then, after discarding
finitely many indices, $w_n$ is modulated on $[0,t_0]$ and
\[
  \lambda_n(t_0)\to\lambda_v(t_0),\qquad
  b_n(t_0)\to b_v(t_0),\qquad
  x_{j,n}(t_0)\to x_{j,v}(t_0),\quad j=1,2 .
\]
Moreover,
\begin{equation*}
\varepsilon_n(t_0)\rightarrow \varepsilon_v(t_0) \text{ in }H^1. 
\end{equation*}
Furthermore, with $s_{0,n}=s_n(t_0)$ and $s_{0,v}=s_v(t_0)$,
\begin{equation}\label{eq:transfer_conc}
   \mathcal{N}_{3,n}(s_{0,n})\to\mathcal{N}_{3,v}(s_{0,v}).
\end{equation}
\end{lemma}

\begin{proof}
Local well-posedness gives $w_n\to v$ in $C([0,t_0];H^1)$.  Since
$v([0,t_0])$ is compactly contained in the tube $\mathcal{T}_{\tilde{\alpha}}$ for some $\tilde{\alpha}<\alpha^*$, then $w_n([0,t_0])\in \mathcal{T}_{\alpha^*}$ for
all large $n$.  Applying the $C^1$ modulation map at the fixed time
$t_0$ gives the convergence of the parameters.  The formula defining the error
then gives $\varepsilon_n(t_0)\to\varepsilon_v(t_0)$ in $H^1$.

It remains only to justify the weighted part of \eqref{eq:transfer_conc}.  The
condition $w_{0,n}\in\mathcal{A}_{\alpha_0}$ gives uniformly
\[
   \mathcal{N}_{4,n}(0)+|b_n(0)|\le C.
\]
Since $w_n$ stays in the bootstrap regime on $[0,t_0]$,
\eqref{eq:N4} yields
\begin{equation}\label{eq:uniformN6}
   \sup_n\sup_{0\le t\le t_0}\mathcal{N}_{4,n}(s_n(t))<\infty .
\end{equation}

Set $f_n=\varepsilon_n(t_0)-\varepsilon_v(t_0)$.  Then $f_n\to 0$ in $L^2$, and
\eqref{eq:uniformN6} gives
\[
   \sup_n\int_{y_1>0}y_1^4 |f_n|^2<\infty .
\]
Hence, for each  $0\leq p<4$,
\[
   \int_{y_1>0}y_1^p|f_n|^2
   \le \left(\int |f_n|^2\right)^{1-p/4}
       \left(\int_{y_1>0}y_1^4|f_n|^2\right)^{p/4}\to 0 .
\]
Taking the exponents appearing in $\mathcal{N}_3$ gives $\mathcal{N}_{3,n}(s_{0,n})\to
\mathcal{N}_{3,v}(s_{0,v})$; the gradient part follows from strong $H^1$ convergence.
\end{proof}

\begin{lemma}\label{lem:trap}
There are universal constants $K_*>1$ and $\sigma_*>0$ with the following
property.  Suppose that, at some renormalized time $s_0$ in the bootstrap
regime,
\begin{equation}\label{eq:trap_ass}
   A(s_0)\ge 4\sigma,
   \qquad
   \frac{\mathcal{N}_3(s_0)+b^2(s_0)}{\lambda^c(s_0)}\le \frac{\sigma}{K_*},
   \qquad 0<\sigma<\sigma_* .
\end{equation}
Then, as long as the solution remains in the tube for $s\ge s_0$,
\[
       A(s)\ge 2\sigma .
\]
Consequently this solution belongs to the stable blow-up branch.
\end{lemma}

\begin{proof}
Let $I=[s_0,S]$ be any compact subinterval of $[0,s^*]$, where $s^*$ is the exit time, and set $M_I=\sup_I |A|$.  Since $|b|\le\nu^*$ in this regime,
\[
  \frac{b^2(s)}{\lambda^c(s)}\le \nu^* |A(s)| .
\]
From \eqref{eq:almost}, for $s\in I$,
\[
  |A(s)|\le |A(s_0)|+C_0\frac{\mathcal{N}_3(s_0)+b^2(s_0)}{\lambda^c(s_0)}+C_0\nu^*M_I .
\]
Taking $\nu^*$ so small that $C_0\nu^*\le1/2$ gives
\begin{equation}\label{eq:MI}
   M_I\le 2\left(|A(s_0)|+C_0\frac{\mathcal{N}_3(s_0)+b^2(s_0)}{\lambda^c(s_0)}\right).
\end{equation}
Using \eqref{eq:almost} again, together with \eqref{eq:MI},
\[
  A(s)
  \ge A(s_0)-C_0\frac{\mathcal{N}_3(s_0)+b^2(s_0)}{\lambda^c(s_0)}-C_0\nu^*M_I .
\]
Choose $K_*\ge16C_0$ and reduce $\nu^*$ so that $2C_0\nu^*\le1/4$.  Then
\eqref{eq:trap_ass} implies
\begin{equation}\label{positivityatexit}
    A(s)\ge \frac34 A(s_0)-\frac54\frac{C_0\sigma}{K_*}>2\sigma .
\end{equation}

Since $S\leq s^*$ was arbitrary, the lower bound holds for the whole future interval in
the tube. By Theorem~\ref{rigiditytheorem}, the lower bound
\eqref{positivityatexit}  excludes the exit and asymptotic-stability
alternatives. Hence the solution belongs to the stable blow-up branch.
\end{proof}

\begin{proposition}[Openness of the stable blow-up branch]
Let $v_0\in\mathcal{A}_{\alpha_0}$ generate a solution $v(t)$ which blows up in finite time $T_v$ and
satisfies $v(t)\in\mathcal{T}_{\alpha^*}$ for $0\le t<T_v$.  Then there exists $\rho>0$ such
that, if $w_0\in\mathcal{A}_{\alpha_0}$ and $\|w_0-v_0\|_{H^1}<\rho$, the corresponding solution
$w(t)$ is in the stable blow-up branch.
\end{proposition}

\begin{proof}
For the base solution, the blow-up asymptotics give
\[
      A_v(s)=\frac{b_v(s)}{\lambda_v^c(s)}\to c_0>0,
      \qquad
      \mathcal{N}_{3,v}(s)\lesssim \lambda_v^{(3-\eta)c}(s)
\]
for some small fixed $\eta>0$.  Hence
\[
   \frac{\mathcal{N}_{3,v}(s)+b_v^2(s)}{\lambda_v^c(s)}\to 0 .
\]
Choose $0<\sigma<\min\{\sigma_*,c_0/16\}$ and then choose a time
$t_0<T_v$ so close to $T_v$ that, with $s_{0,v}=s_v(t_0)$,
\begin{equation*}
   A_v(s_{0,v})\ge 8\sigma,
   \qquad
   \frac{\mathcal{N}_{3,v}(s_{0,v})+b_v^2(s_{0,v})}{\lambda_v^c(s_{0,v})}
        \le \frac{\sigma}{2K_*} .
\end{equation*}

Assume there are $w_{0,n}\in\mathcal{A}_{\alpha_0}$, $w_{0,n}\to v_0$ in $H^1$.  Lemma \ref{lem:transfer}, applied at the fixed time $t_0$ (and the associated time $s_{0,n}=s_n(t_0))$, implies for all large $n$,
\[
   A_n(s_{0,n})\ge 4\sigma,
   \qquad
   \frac{\mathcal{N}_{3,n}(s_{0,n})+b_n^2(s_{0,n})}{\lambda_n^c(s_{0,n})}
        \le \frac{\sigma}{K_*} .
\]
Lemma \ref{lem:trap} therefore places $w_n$ in the stable blow-up branch.  This proves the relative openness in $\mathcal{A}_{\alpha_0}$ for the
$H^1$ topology.
\end{proof}

\chapter{Asymptotic Radiation in the Blow-Up Regime}
\label{sec:asymptotic-radiation}

\section{Strong Convergence in \texorpdfstring{$L^p$}{Lg} of the Asymptotic Profile} 
Assume that the solution belongs to the stable blow-up branch of
Theorem~\ref{rigiditytheorem}. We prove that there exists a nonzero function
\[
    u^*\in\bigcap_{2\leq p<\infty}L^p(\mathbb R^2)
\]
such that, for every fixed \(2\leq p<\infty\),
\[
 u(t,x_1,x_2)
 -\frac1{\lambda(t)}Q\left(
   \frac{x_1-x_1(t)}{\lambda(t)},
   \frac{x_2-x_2(t)}{\lambda(t)}
 \right)
 \longrightarrow u^*
 \quad\text{in }L^p(\mathbb R^2)
 \quad\text{as }t\uparrow T.
\]
The endpoint \(p=2\) is treated by the refined inhomogeneous
Strichartz estimate, whereas the extension to \(2<p<\infty\) uses the
interpolated dispersive estimate on a shrinking time interval.

Write
\[
    u=\mathcal Q+\widetilde u,
\]
where, with
\[
    y_i=\frac{x_i-x_i(t)}{\lambda(t)},\qquad i=1,2,
\]
we set
\[
\begin{aligned}
\mathcal Q(t,x_1,x_2)
&=\frac1{\lambda(t)}Q_{b(t)}(y_1,y_2),\\
\widetilde u(t,x_1,x_2)
&=\frac1{\lambda(t)}\varepsilon(t,y_1,y_2).
\end{aligned}
\]
Denote the exact modulated ground state by
\[
    S(t,x_1,x_2)
    :=\frac1{\lambda(t)}Q(y_1,y_2).
\]
For every fixed \(2\leq p<\infty\), the change of variables from
\(x=(x_1,x_2)\) to \(y=(y_1,y_2)\), the bound
\(\|\chi_bP\|_{L^p_y}\lesssim |b|^{-\gamma/p}\), and
\eqref{eq:tlawforb} give
\begin{equation}\label{eq:QSL2estimate}
\begin{aligned}
\|\mathcal Q(t)-S(t)\|_{L^p_{x}}
&=\lambda(t)^{\frac2p-1}|b(t)|
  \|\chi_bP\|_{L^p_{y}}\lesssim
\lambda(t)^{\frac2p-1}|b(t)|^{1-\frac\gamma p}\\
&\lesssim
\lambda(t)^{c-1+\frac{2-c\gamma}{p}}
\longrightarrow0
\qquad\text{as }t\uparrow T.
\end{aligned}
\end{equation}
The exponent is positive because \(c>1\) and \(c\gamma<2\). It remains
to prove that \(\widetilde u(t)\) converges in every finite \(L^p\).

The physical radiation satisfies
\[
    \partial_t\widetilde u
    +\partial_{x_1}\Delta\widetilde u
    +f(\widetilde u)_{x_1}
    +\mathfrak F=0,
\]
where
\[
    f(\widetilde u):=(\mathcal Q+\widetilde u)^3-\mathcal Q^3
\]
and
\begin{equation}\label{mathfrakF}
\begin{aligned}
\mathfrak F(t,x_1,x_2)
={}&-\frac1{\lambda^4(t)}\Psi_{b(t)}(y_1,y_2)
+\frac{b_t(t)}{\lambda(t)}\zeta_{b(t)}(y_1)P(y_1,y_2)\\
&-\frac{\lambda^2(t)\lambda_t(t)+b(t)}{\lambda^4(t)}
  \Lambda Q_{b(t)}(y_1,y_2)\\
&-\frac{\lambda^2(t)(x_1)_t(t)-1}{\lambda^4(t)}
  (Q_{b(t)})_{y_1}(y_1,y_2)\\
&-\frac{\lambda^2(t)(x_2)_t(t)}{\lambda^4(t)}
  (Q_{b(t)})_{y_2}(y_1,y_2).
\end{aligned}
\end{equation}
Here
\[
    \zeta_bP=\frac{\partial Q_b}{\partial b},
    \qquad
    \Psi_b=
    \bigl(-\Delta Q_b+Q_b-Q_b^3\bigr)_{y_1}-b\Lambda Q_b.
\]

For \(0<\tau<T-t_0\), define
\[
    v^\tau(t):=\widetilde u(t+\tau)-\widetilde u(t),
    \qquad t\in[t_0,T-\tau).
\]
Then
\[
    \partial_tv^\tau+\partial_{x_1}\Delta v^\tau+G^\tau+H^\tau=0,
\]
where
\[
\begin{aligned}
G^\tau(t)
&=f(\widetilde u)_{x_1}(t+\tau)-f(\widetilde u)_{x_1}(t),\\
H^\tau(t)
&=\mathfrak F(t+\tau)-\mathfrak F(t).
\end{aligned}
\]
Let \(U(t)\) be the unitary group associated with the linear ZK equation:
\begin{equation}\label{eq:Airy}
\begin{aligned}
U(t)g_0(x_1,x_2)
=(2\pi)^{-2}\iint_{\mathbb R^2}
&e^{i\left[t(\xi_1^3+\xi_1\xi_2^2)
+x_1\xi_1+x_2\xi_2\right]}\widehat g_0(\xi_1,\xi_2)\,d\xi_1d\xi_2.
\end{aligned}
\end{equation}
Duhamel's formula gives, for \(t_0\leq t<T-\tau\),
\begin{equation}\label{eq:Duhamel}
 v^\tau(t)
 =U(t-t_0)v^\tau(t_0)
 -\int_{t_0}^{t}U(t-s)
 \bigl(G^\tau(s)+H^\tau(s)\bigr)\,ds.
\end{equation}

\section*{Linear and refined Strichartz estimates}

Following the strategy used in \cite{Lan} and
\cite{MerleRaphaelSzeftel}, we compare two time translates of the
radiation and control their difference through Duhamel's formula. We
first recall the abstract inhomogeneous estimate of Foschi
\cite{Foschi}.

\begin{theorem}[Foschi's inhomogeneous Strichartz estimate]
Let \((X,d\mu)\) be a measure space, let \(H\) be a Hilbert space, and
let \(V(t):H\to L^2_X\), \(t\in\mathbb R\), be a family of linear
operators satisfying
\[
    \|V(t)h\|_{L^2_X}\lesssim\|h\|_H
\]
for every \(t\in\mathbb R\) and \(h\in H\), and
\[
    \|V(t)V(s)^*g\|_{L^\infty_X}
    \lesssim |t-s|^{-\sigma}\|g\|_{L^1_X}
\]
for every \(t\neq s\) and \(g\in L^1_X\cap L^2_X\), where
\(0<\sigma<1\). A pair \((q,r)\in[1,\infty]^2\) is called
\(\sigma\)-acceptable if
\[
    \frac1q<2\sigma\left(\frac12-\frac1r\right),
\]
or if \((q,r)=(\infty,2)\). If \((q_1,r_1)\) and
\((q_2,r_2)\) are \(\sigma\)-acceptable and satisfy
\[
    \frac1{q_1}+\frac{\sigma}{r_1}
    +\frac1{q_2}+\frac{\sigma}{r_2}
    =\sigma,
\]
then
\[
\left\|
    \int_{s<t}V(t)V(s)^*F(s)\,ds
\right\|_{L_t^{q_1}L_X^{r_1}}
\lesssim
\|F\|_{L_t^{q_2'}L_X^{r_2'}}.
\]
\end{theorem}

We also recall the linear estimate for the ZK group, stated as
Lemma~2.3 in \cite{LinaresPastor2}.

\begin{theorem}[Linear ZK estimates]
Let \(U(t)\) be the group defined in \eqref{eq:Airy}. Then
\[
    \|U(t)h\|_{L^2_{x}}
    \lesssim\|h\|_{L^2_{x}},
\]
and, for \(t\neq0\),
\[
    \|U(t)h\|_{L^\infty_{x}}
    \lesssim |t|^{-\frac23}\|h\|_{L^1_{x}}.
\]
\end{theorem}

Since \(U(t)U(s)^*=U(t-s)\), Foschi's theorem with
\(\sigma=2/3\) gives the following consequence.

\begin{corollary}[Refined inhomogeneous Strichartz estimate]
Let \((q_1,r_1)\) and \((q_2,r_2)\) be
\(\frac23\)-acceptable pairs satisfying
\[
    \frac1{q_1}+\frac2{3r_1}
    +\frac1{q_2}+\frac2{3r_2}
    =\frac23.
\]
Then
\begin{equation}\label{eq:Refined}
\left\|
    \int_0^tU(t-s)h(s)\,ds
\right\|_{L_t^{q_1}L_{x}^{r_1}}
\lesssim
\|h\|_{L_t^{q_2'}L_{x}^{r_2'}}.
\end{equation}
\end{corollary}

From now on, fix \(0<\eta<1\) for which \eqref{sharpN} holds. By the
hierarchy of the weighted norms,
\begin{equation}\label{eq:sharp-hierarchy-Lp}
    \mathcal N_i(t)
    \lesssim\mathcal N_5(t)
    \lesssim\lambda(t)^{(3-\eta)c},
    \qquad 1\leq i\leq5.
\end{equation}
The restriction \(\eta<1\) gives
\((3-\eta)/2>1\), which is the strict decay margin required below. It
also permits the choice of a finite exponent
\(r_2>2\gamma/(1-\eta)\) in the \(L^2\)-argument.

\section*{Pointwise bounds for the forcing}

We first establish bounds that are valid for every fixed \(r>2\). Set
\[
    a_r
    :=\min\left\{
       \frac{3-\eta}{2},
       1+\gamma,
       1+\frac1r
    \right\},
    \qquad
    d_r
    :=\min\left\{
       \frac{3-\eta}{2},
       1+\frac\gamma r
    \right\}.
\]
Since \(0<\eta<1\), we have \(a_r>1\) and \(d_r>1\).

\paragraph{Step 1: Estimate of \(f(\widetilde u)_{x_1}\).}
The change of variables from \(x=(x_1,x_2)\) to \(y=(y_1,y_2)\) gives
\[
\|f(\widetilde u)_{x_1}(t)\|_{L^{r'}_{x}}
=
\lambda(t)^{-2-\frac2r}
\left\|
\partial_{y_1}\bigl[(Q_b+\varepsilon)^3-Q_b^3\bigr]
\right\|_{L^{r'}_{y}}.
\]
The complete expansion is
\begin{align*}
\partial_{y_1}\bigl[(Q_b+\varepsilon)^3-Q_b^3\bigr]
={}&3Q_b^2\varepsilon_{y_1}
+6Q_b(Q_b)_{y_1}\varepsilon
+3(Q_b)_{y_1}\varepsilon^2\\
&+6Q_b\varepsilon\varepsilon_{y_1}
+3\varepsilon^2\varepsilon_{y_1}.
\end{align*}
Choose
\[
    R=\frac{2r}{r-2},
    \qquad
    \frac1{r'}=\frac12+\frac1R.
\]
The exponent \(R\) is only an auxiliary H\"older exponent; the target
space for the nonlinear terms is \(L^{r'}\). 

For the term involving \((Q_b)_{y_1}\varepsilon^2\), write
\(
    (Q_b)_{y_1}=A_b+B_b,
\)
where
\[
    A_b:=Q_{y_1}+b\chi_bP_{y_1},
    \qquad
    B_b:=b|b|^\gamma
    \chi'\bigl(|b|^\gamma y_1\bigr)P.
\]
H\"older's inequality, the weighted Sobolev estimate from Lemma~\ref{SobolevLemma} and \eqref{eq:sharp-hierarchy-Lp} give
\[
\begin{aligned}
\|A_b\varepsilon^2\|_{L^{r'}_y}
&\leq
\left\|
\varepsilon^2\sqrt{|A_b|}
\right\|_{L^2_y}
\left\|
\sqrt{|A_b|}
\right\|_{L^R_y}\lesssim
\mathcal N_1^{1/2}(t)
\lesssim
\lambda(t)^{\frac{(3-\eta)c}{2}}.
\end{aligned}
\]
On the other hand, since \(P\in L^\infty(\mathbb R^2)\), together with \eqref{eq:estimateL2},
\eqref{eq:estimateL2derivative}, and \(b(t)\sim c_0\lambda^c(t)\),
we obtain
\[
\begin{aligned}
\|B_b\varepsilon^2\|_{L^{r'}_y}
&\lesssim
|b|^{1+\gamma}
\|\varepsilon\|_{L^{2r'}_y}^2\lesssim
|b|^{1+\gamma}
\|\nabla\varepsilon\|_{L^2_y}^{2/r}
\|\varepsilon\|_{L^2_y}^{2-2/r}\lesssim
\lambda(t)^{c\left(1+\gamma+\frac1r\right)}.
\end{aligned}
\]
Consequently,
\begin{equation*}
\|(Q_b)_{y_1}\varepsilon^2\|_{L^{r'}_y}
\lesssim
\lambda(t)^{
c\min\left\{
\frac{3-\eta}{2},
1+\gamma+\frac1r
\right\}
}.
\end{equation*}

The other two terms containing only one factor of \(\varepsilon\) or
\(\varepsilon_{y_1}\) are estimated in the same way. It follows that
\begin{align*}
\|Q_b^2\varepsilon_{y_1}\|_{L^{r'}_y}
&\lesssim
\lambda(t)^{c\frac{3-\eta}{2}},
\\
\|Q_b(Q_b)_{y_1}\varepsilon\|_{L^{r'}_y}
&\lesssim
\lambda(t)^{c\min\left\{\frac{3-\eta}{2},1+\gamma\right\}}.
\end{align*}

Also, by the two-dimensional Gagliardo--Nirenberg inequality,
\[
    \|\varepsilon\|_{L^R_y}
    \lesssim
    \|\nabla\varepsilon\|_{L^2_y}^{2/r}
    \|\varepsilon\|_{L^2_y}^{1-2/r}
    \lesssim\lambda^{c/r}.
\]
Therefore,
\begin{align*}
\|Q_b\varepsilon\varepsilon_{y_1}\|_{L^{r'}_y}
&\lesssim
\mathcal N_1(t)
+|b|\|\nabla\varepsilon\|_{L^2_y}
  \|\varepsilon\|_{L^R_y}\lesssim
\lambda^{c\min\left\{3-\eta,\frac32+\frac1r\right\}}.
\end{align*}
Finally,
\begin{align*}
\|\varepsilon^2\varepsilon_{y_1}\|_{L^{r'}_y}
&\lesssim
\|\nabla\varepsilon\|_{L^2_y}
\|\varepsilon\|_{L^{2R}_y}^2\lesssim
\|\nabla\varepsilon\|_{L^2_y}^{2+2/r}
\|\varepsilon\|_{L^2_y}^{1-2/r}\lesssim
\lambda^{c\left(1+\frac1r\right)}.
\end{align*}
Collecting the five estimates yields
\begin{equation}\label{eq:pointwise-nonlinearity-Lp}
\|f(\widetilde u)_{x_1}(t)\|_{L^{r'}_{x}}
\lesssim
\lambda(t)^{-2-\frac2r+ca_r}.
\end{equation}

\paragraph{Step 2: Estimate of \(\mathfrak F\).}
The profile bounds give
\[
\begin{aligned}
&\|\zeta_bP\|_{L^{r'}_{y}}
\lesssim |b|^{-\gamma/r'},
\qquad
\|\Psi_b\|_{L^{r'}_{y}}
\lesssim |b|^{1+\gamma/r},\\
&\|\Lambda Q_b\|_{L^{r'}_y}
+\|(Q_b)_{y_1}\|_{L^{r'}_y}
+\|(Q_b)_{y_2}\|_{L^{r'}_y}
\lesssim1.
\end{aligned}
\]
Moreover,
\[
\lambda^3b_t=b_s,
\quad
\lambda^2\lambda_t+b=\frac{\lambda_s}{\lambda}+b,
\quad
\lambda^2(x_1)_t-1=\frac{(x_1)_s}{\lambda}-1,
\quad
\lambda^2(x_2)_t=\frac{(x_2)_s}{\lambda}.
\]
Hence \eqref{eq:modulatedcoefficients} imply
\[
|b_s|\lesssim b^2+\mathcal N_1,
\qquad
\left|\frac{\lambda_s}{\lambda}+b\right|
+\left|\frac{(x_1)_s}{\lambda}-1\right|
+\left|\frac{(x_2)_s}{\lambda}\right|
\lesssim b^2+\mathcal N_1^{1/2}.
\]
Using \(b\sim\lambda^c\) and \eqref{eq:sharp-hierarchy-Lp}, we obtain
\begin{align*}
\|\mathfrak F(t)\|_{L^{r'}_{x}}
&\lesssim
\lambda(t)^{-2-\frac2r}
\Bigl[
|b|^{1+\gamma/r}
+(b^2+\mathcal N_1)|b|^{-\gamma/r'}
+b^2+\mathcal N_1^{1/2}
\Bigr]\lesssim
\lambda(t)^{-2-\frac2r+cd_r}.
\end{align*}
Indeed, the powers of \(\lambda^c\) inside the brackets are
\[
1+\frac\gamma r,
\qquad
2-\gamma+\frac\gamma r,
\qquad
3-\eta-\gamma+\frac\gamma r,
\qquad
2,
\qquad
\frac{3-\eta}{2},
\]
and their minimum is \(d_r\), because \(0<\gamma,\eta<1\). Therefore,
\begin{equation}\label{eq:pointwise-profile-forcing-Lp}
\|\mathfrak F(t)\|_{L^{r'}_{x}}
\lesssim
\lambda(t)^{-2-\frac2r+cd_r}.
\end{equation}

\section*{The \(L^2\)-limit}

Choose
\[
    (q_1,r_1)=(\infty,2),
    \qquad
    \frac1{q_2}+\frac2{3r_2}=\frac13,
\]
with
\[
    r_2>\max\left\{2,\frac{2\gamma}{1-\eta}\right\}.
\]
The condition \(r_2>2\) makes \((q_2,r_2)\)
\(\frac23\)-acceptable, while
\(r_2>2\gamma/(1-\eta)\) gives
\[
    \frac{3-\eta}{2}>1+\frac\gamma{r_2},
    \qquad
    d_{r_2}=1+\frac\gamma{r_2}.
\]
Moreover,
\[
    \frac3{q_2'}=2+\frac2{r_2},
\]
and \(a_{r_2},d_{r_2}>1>1/q_2'\). From
\eqref{eq:pointwise-nonlinearity-Lp},
\eqref{eq:pointwise-profile-forcing-Lp}, and
\(\lambda(t)\sim(T-t)^{1/(3-c)}\), we find
\begin{align*}
&\|f(\widetilde u)_{x_1}\|_{L^{q_2'}([t,T);L^{r_2'}_{x})}\quad\lesssim
\left(
\int_t^T
\lambda(t')^{-q_2'\left(\frac3{q_2'}-ca_{r_2}\right)}dt'
\right)^{1/q_2'}\quad\lesssim
(T-t)^{\frac{c}{3-c}
\left(a_{r_2}-\frac1{q_2'}\right)},
\end{align*}
and similarly for \(\mathfrak F\). Consequently,
\begin{equation}\label{eq:fuestimate}
\|f(\widetilde u)_{x_1}\|_{L^{q_2'}([t,T);L^{r_2'}_{x})}
\lesssim
(T-t)^{\frac{c}{3-c}\left(a_{r_2}-\frac1{q_2'}\right)}
\longrightarrow0,
\end{equation}
and
\begin{equation}\label{eq:Festimate}
\|\mathfrak F\|_{L^{q_2'}([t,T);L^{r_2'}_{x})}
\lesssim
(T-t)^{\frac{c}{3-c}\left(d_{r_2}-\frac1{q_2'}\right)}
\longrightarrow0.
\end{equation}

For \(0\leq t'<T-\tau-t\), rewrite \eqref{eq:Duhamel} with base time
\(t\) as
\[
\begin{aligned}
v^\tau(t+t')
={}&U(t')v^\tau(t)
-\int_0^{t'}U(t'-s)
\bigl(G^\tau(t+s)+H^\tau(t+s)\bigr)\,ds.
\end{aligned}
\]
Using \eqref{eq:Refined} and the unitarity of \(U(t)\) on \(L^2\),
we obtain
\begin{equation}\label{Strichartz1}
\begin{aligned}
\|v^\tau\|_{L^\infty([t,T-\tau);L^2_{x})}
\lesssim{}&
\|v^\tau(t)\|_{L^2_{x}}\\
&+\|G^\tau\|_{L^{q_2'}([t,T-\tau);L^{r_2'}_{x})}
+\|H^\tau\|_{L^{q_2'}([t,T-\tau);L^{r_2'}_{x})}.
\end{aligned}
\end{equation}
By the triangle inequality and translation in time,
\begin{align*}
\|G^\tau\|_{L^{q_2'}([t,T-\tau);L^{r_2'}_x)}
&\leq
\|f(\widetilde u)_{x_1}\|_{L^{q_2'}([t,T);L^{r_2'}_x)}
+\|f(\widetilde u)_{x_1}\|_{L^{q_2'}([t+\tau,T);L^{r_2'}_x)},\\
\|H^\tau\|_{L^{q_2'}([t,T-\tau);L^{r_2'}_x)}
&\leq
\|\mathfrak F\|_{L^{q_2'}([t,T);L^{r_2'}_x)}
+\|\mathfrak F\|_{L^{q_2'}([t+\tau,T);L^{r_2'}_x)}.
\end{align*}
Thus \eqref{eq:fuestimate} and \eqref{eq:Festimate} imply that, for
every \(\epsilon>0\), there exists \(t_\epsilon<T\) such that,
uniformly in \(0<\tau<T-t_\epsilon\),
\begin{equation}\label{eq:fuandFestimate}
\begin{aligned}
&\|G^\tau\|_{L^{q_2'}([t_\epsilon,T-\tau);L^{r_2'}_{x})}
+\|H^\tau\|_{L^{q_2'}([t_\epsilon,T-\tau);L^{r_2'}_{x})}
\leq\epsilon.
\end{aligned}
\end{equation}
By \eqref{eq:QSL2estimate}, after increasing \(t_\epsilon\) if
necessary,
\begin{equation}\label{eq:QSestimate}
    \|\mathcal Q(t)-S(t)\|_{L^2_{x}}
    \leq\epsilon,
    \qquad t_\epsilon\leq t<T.
\end{equation}

We next estimate \(v^\tau(t_\epsilon)\) in \(L^2_{x}\). The triangle
inequality gives
\begin{align*}
\|v^\tau(t_\epsilon)\|_{L^2_{x}}
\leq{}&
\|u(t_\epsilon+\tau)-u(t_\epsilon)\|_{L^2_{x}}+\|\mathcal Q(t_\epsilon+\tau)-S(t_\epsilon+\tau)\|_{L^2_{x}}\\
&+\|\mathcal Q(t_\epsilon)-S(t_\epsilon)\|_{L^2_{x}}+\|S(t_\epsilon+\tau)-S(t_\epsilon)\|_{L^2_{x}}.
\end{align*}
Since \(u\in C([0,T);H^1_{x})\), there exists
\(\tau_1=\tau_1(t_\epsilon)>0\) such that
\begin{equation}\label{H1theory}
    \|u(t_\epsilon+\tau)-u(t_\epsilon)\|_{L^2_{x}}
    \leq\epsilon,
    \qquad 0<\tau\leq\tau_1.
\end{equation}

It remains to control the last term. On every compact interval
\(I\Subset[0,T)\), the parameters \(\lambda,x_1,x_2\) are absolutely
continuous and \(\lambda\) is bounded away from zero. Hence
\[
    \frac{\lambda_t}{\lambda^2},
    \quad
    \frac{(x_1)_t}{\lambda^2},
    \quad
    \frac{(x_2)_t}{\lambda^2}
    \in L^1(I).
\]
Differentiating \(S\) gives, in \(L^2_{x}(\mathbb R^2)\),
\[
\partial_tS
=-\frac{\lambda_t}{\lambda^2}\Lambda Q(y_1,y_2)
-\frac{(x_1)_t}{\lambda^2}Q_{y_1}(y_1,y_2)
-\frac{(x_2)_t}{\lambda^2}Q_{y_2}(y_1,y_2).
\]
After changing variables from \(x=(x_1,x_2)\) to \(y=(y_1,y_2)\),
\[
\begin{aligned}
\|\partial_tS(t)\|_{L^2_{x}}
\leq{}&
\left|\frac{\lambda_t}{\lambda}\right|\|\Lambda Q\|_{L^2_{y}}
+\left|\frac{(x_1)_t}{\lambda}\right|\|Q_{y_1}\|_{L^2_{y}}+\left|\frac{(x_2)_t}{\lambda}\right|\|Q_{y_2}\|_{L^2_{y}},
\end{aligned}
\]
which belongs to \(L^1(I)\). Therefore
\(S\in AC(I;L^2_{x}(\mathbb R^2))\). The fundamental theorem of calculus
in \(L^2_{x}\) and Minkowski's integral inequality give
\begin{align*}
\|S(t_\epsilon+\tau)-S(t_\epsilon)\|_{L^2_{x}}=\norm{\int_{t_\epsilon}^{t_\epsilon+\tau}\partial_tS(t')\,dt'}_{L^2_{x}}
&\leq
\int_{t_\epsilon}^{t_\epsilon+\tau}
\|\partial_tS(t')\|_{L^2_{x}}\,dt'
\longrightarrow0
\end{align*}
as \(\tau\to0\). Thus, after decreasing \(\tau_1\) if necessary,
\begin{equation}\label{eq:modulatedQconvergence}
    \|S(t_\epsilon+\tau)-S(t_\epsilon)\|_{L^2_{x}}
    \leq\epsilon,
    \qquad 0<\tau\leq\tau_1.
\end{equation}
Combining \eqref{eq:QSestimate}, \eqref{H1theory}, and
\eqref{eq:modulatedQconvergence}, we obtain
\begin{equation}\label{L2vtau}
    \|v^\tau(t_\epsilon)\|_{L^2_{x}}
    \leq4\epsilon,
    \qquad 0<\tau\leq\tau_1.
\end{equation}

Applying \eqref{Strichartz1} with base time \(t=t_\epsilon\), and using
\eqref{eq:fuandFestimate} and \eqref{L2vtau}, yields
\[
    \|v^\tau\|_{L^\infty([t_\epsilon,T-\tau);L^2_{x})}
    \lesssim\epsilon.
\]
Choose \(t_0\in(\max\{T-\tau_1,t_\epsilon\},T)\). For
\(t_0<t_1<t_2<T\), set \(\tau=t_2-t_1\). Then
\(0<\tau<\tau_1\) and \(t_1<T-\tau\), so
\[
    \|\widetilde u(t_2)-\widetilde u(t_1)\|_{L^2_{x}}
    =\|v^\tau(t_1)\|_{L^2_{x}}
    \lesssim\epsilon.
\]
Thus there exists \(u^*\in L^2_{x}(\mathbb R^2)\) such that
\[
    \widetilde u(t)\longrightarrow u^*
    \quad\text{in }L^2_{x}(\mathbb R^2)
    \quad\text{as }t\uparrow T.
\]

\section*{Extension to every \(2<p<\infty\)}

Fix \(2<p<\infty\), set \(p'=p/(p-1)\), and define
\[
    \delta_p:=\frac23\left(1-\frac2p\right).
\]
For fixed \(t\neq0\), interpolate by the Riesz--Thorin theorem to obtain 
\begin{equation}\label{eq:interpolated-dispersive-Lp}
    \|U(t)h\|_{L^p_{x}}
    \lesssim
    |t|^{-\delta_p}\|h\|_{L^{p'}_{x}},
    \qquad t\neq0.
\end{equation}

Apply the pointwise forcing bounds with \(r=p\). Since
\[
    d_p
    =\min\left\{\frac{3-\eta}{2},1+\frac\gamma p\right\}
    \leq a_p,
\]
we have
\[
\begin{aligned}
\|f(\widetilde u)_{x_1}(t)\|_{L^{p'}_{x}}
+\|\mathfrak F(t)\|_{L^{p'}_{x}}
&\lesssim
\lambda(t)^{-2-\frac2p+cd_p}\lesssim
(T-t)^{-\beta_p},
\end{aligned}
\]
where
\[
    \beta_p
    :=\max\left\{0,
    \frac{2+\frac2p-cd_p}{3-c}\right\}.
\]
We claim that
\begin{equation}\label{eq:key-c-three-halves}
    \beta_p<1-\delta_p.
\end{equation}
If \(2+2/p-cd_p\leq0\), this is immediate. Otherwise,
\eqref{eq:key-c-three-halves} is equivalent to
\begin{equation}\label{eq:key-algebraic-c-condition}
    3cd_p
    >3+c+\frac{4c-6}{p}.
\end{equation}
If \(d_p=(3-\eta)/2\), then, because \(c>3/2\) (see Theorem \ref{cap:verified-main}), \(p\geq2\), and
\(\eta<1\),
\[
    3+c+\frac{4c-6}{p}
    \leq3+c+\frac{4c-6}{2}=3c
    <3c\frac{3-\eta}{2}.
\]
If \(d_p=1+\gamma/p\), then
\begin{align*}
&3c\left(1+\frac\gamma p\right)
-\left(3+c+\frac{4c-6}{p}\right)=
2c-3+\frac{3c\gamma-4c+6}{p}.
\end{align*}
If \(3c\gamma-4c+6\geq0\), this is at least \(2c-3>0\). If
\(3c\gamma-4c+6<0\), then \(p\geq2\) gives
\[
2c-3+\frac{3c\gamma-4c+6}{p}
\geq
2c-3+\frac{3c\gamma-4c+6}{2}
=\frac{3c\gamma}{2}>0.
\]
Thus \eqref{eq:key-algebraic-c-condition}, and hence
\eqref{eq:key-c-three-halves}, holds for every finite \(p>2\).

Choose \(q=q(p)>1\) such that
\begin{equation*}
    \frac1{1-\delta_p}<q<\frac1{\beta_p},
\end{equation*}
where the upper bound is omitted when \(\beta_p=0\). This is possible
by \eqref{eq:key-c-three-halves}. Then
\[
    \delta_pq'<1,
    \qquad
    \beta_pq<1.
\]
Consequently,
\begin{equation}\label{eq:all-p-forcing-tail}
\begin{aligned}
&\|f(\widetilde u)_{x_1}\|_{L^q([t,T);L^{p'}_{x})}
+\|\mathfrak F\|_{L^q([t,T);L^{p'}_{x})}\lesssim
(T-t)^{\frac1q-\beta_p}
\longrightarrow0
\qquad\text{as }t\uparrow T.
\end{aligned}
\end{equation}
The corresponding difference terms satisfy, by the triangle inequality,
\begin{align*}
\|G^\tau\|_{L^q([t,T-\tau);L^{p'}_{x})}
&\leq
\|f(\widetilde u)_{x_1}\|_{L^q([t,T);L^{p'}_{x})}
+\|f(\widetilde u)_{x_1}\|_{L^q([t+\tau,T);L^{p'}_{x})},\\
\|H^\tau\|_{L^q([t,T-\tau);L^{p'}_{x})}
&\leq
\|\mathfrak F\|_{L^q([t,T);L^{p'}_{x})}
+\|\mathfrak F\|_{L^q([t+\tau,T);L^{p'}_{x})}.
\end{align*}

We now estimate the Duhamel formula with base time \(t_\epsilon\). By
Sobolev embedding and the unitarity of \(U(t)\) on \(H^1_{x}\),
\[
\begin{aligned}
\|U(t-t_\epsilon)v^\tau(t_\epsilon)\|_{L^p_{x}}
&\lesssim_p
\|U(t-t_\epsilon)v^\tau(t_\epsilon)\|_{H^1_{x}}=\|v^\tau(t_\epsilon)\|_{H^1_{x}}.
\end{aligned}
\]
For the inhomogeneous term, Minkowski's integral inequality and
\eqref{eq:interpolated-dispersive-Lp} give
\begin{align*}
\left\|
\int_{t_\epsilon}^{t}U(t-s)
\bigl(G^\tau(s)+H^\tau(s)\bigr)\,ds
\right\|_{L^p_{x}}\quad&\leq
\int_{t_\epsilon}^{t}
\left\|U(t-s)
\bigl(G^\tau(s)+H^\tau(s)\bigr)\right\|_{L^p_{x}}\,ds\\&\lesssim
\int_{t_\epsilon}^{t}
(t-s)^{-\delta_p}
\bigl(
\|G^\tau(s)\|_{L^{p'}_{x}}
+\|H^\tau(s)\|_{L^{p'}_{x}}
\bigr)\,ds.
\end{align*}
Thus \eqref{eq:Duhamel} yields
\begin{align*}
\|v^\tau(t)\|_{L^p_{x}}
\lesssim{}&
\|v^\tau(t_\epsilon)\|_{H^1_{x}}+\int_{t_\epsilon}^{t}
(t-s)^{-\delta_p}
\bigl(
\|G^\tau(s)\|_{L^{p'}_{x}}
+\|H^\tau(s)\|_{L^{p'}_{x}}
\bigr)\,ds.
\end{align*}
Since \(\delta_pq'<1\), H\"older's inequality in time gives
\begin{align*}
&\int_{t_\epsilon}^{t}
(t-s)^{-\delta_p}
\bigl(
\|G^\tau(s)\|_{L^{p'}_{x}}
+\|H^\tau(s)\|_{L^{p'}_{x}}
\bigr)\,ds\\
&\quad\leq
\left(
\int_{t_\epsilon}^{t}
(t-s)^{-\delta_pq'}\,ds
\right)^{1/q'}
\bigl(
\|G^\tau\|_{L^q([t_\epsilon,t);L^{p'}_{x})}
+\|H^\tau\|_{L^q([t_\epsilon,t);L^{p'}_{x})}
\bigr)\\
&\quad\lesssim
(T-t_\epsilon)^{\frac1{q'}-\delta_p}
\bigl(
\|G^\tau\|_{L^q([t_\epsilon,T-\tau);L^{p'}_{x})}
+\|H^\tau\|_{L^q([t_\epsilon,T-\tau);L^{p'}_{x})}
\bigr).
\end{align*}
Taking the supremum over \(t_\epsilon\leq t<T-\tau\), we obtain
\begin{equation}\label{eq:terminal-Duhamel-all-p}
\begin{aligned}
\|v^\tau\|_{L^\infty([t_\epsilon,T-\tau);L^p_{x})}
\lesssim{}&
\|v^\tau(t_\epsilon)\|_{H^1_{x}}\\
&+(T-t_\epsilon)^{\frac1{q'}-\delta_p}
\bigl(
\|G^\tau\|_{L^q([t_\epsilon,T-\tau);L^{p'}_{x})}
+\|H^\tau\|_{L^q([t_\epsilon,T-\tau);L^{p'}_{x})}
\bigr).
\end{aligned}
\end{equation}

Choose \(t_\epsilon<T\) sufficiently close to \(T\) so that
\eqref{eq:all-p-forcing-tail} makes the second term on the right of
\eqref{eq:terminal-Duhamel-all-p} smaller than \(\epsilon\), uniformly
in \(0<\tau<T-t_\epsilon\). Since
\(\widetilde u\in C_{\mathrm{loc}}([0,T);H^1_{x}(\mathbb R^2))\), there
exists \(\tau_p=\tau_p(t_\epsilon)>0\) such that
\begin{equation*}
    \|v^\tau(t_\epsilon)\|_{H^1_{x}}
    \leq\epsilon,
    \qquad 0<\tau\leq\tau_p.
\end{equation*}
Therefore,
\[
    \|v^\tau\|_{L^\infty([t_\epsilon,T-\tau);L^p_{x})}
    \lesssim\epsilon,
    \qquad 0<\tau\leq\tau_p.
\]
Choose \(t_0\in(\max\{T-\tau_p,t_\epsilon\},T)\). If
\(t_0<t_1<t_2<T\) and \(\tau=t_2-t_1\), then
\[
\begin{aligned}
\|\widetilde u(t_2)-\widetilde u(t_1)\|_{L^p_{x}}
&=\|v^\tau(t_1)\|_{L^p_{x}}
\lesssim\epsilon.
\end{aligned}
\]
Thus \(\widetilde u(t)\) is a Cauchy family in \(L^p_{x}(\mathbb R^2)\) as
\(t\uparrow T\). Let its limit be \(u_p^*\). Since the already proved
\(L^2_{x}\)-convergence and the \(L^p_{x}\)-convergence both imply convergence
in distributions, uniqueness of the distributional limit gives
\[
    u_p^*=u^*.
\]
Together with \eqref{eq:QSL2estimate}, this proves that, for every
\(2\leq p<\infty\),
\begin{equation}\label{eq:strong-all-finite-p}
    u(t)-S(t)
    \longrightarrow u^*
    \qquad\text{strongly in }L^p_{x}(\mathbb R^2)
    \quad\text{as }t\uparrow T.
\end{equation}

It remains to identify the mass of the limiting radiation. Since
\((\varepsilon,Q)=0\),
\[
\begin{aligned}
(\mathcal Q(t),\widetilde u(t))_{L^2_{x}}
&=(Q_b,\varepsilon)_{L^2_{y}}
=b(\chi_bP,\varepsilon),
\end{aligned}
\]
and therefore
\[
    |(\mathcal Q(t),\widetilde u(t))|
    \lesssim
    |b(t)|^{1-\gamma/2}\|\varepsilon(t)\|_{L^2_{y}}
    \longrightarrow0.
\]
Mass conservation and \eqref{eq:QSL2estimate} with \(p=2\) give
\[
    \|u_0\|_{L^2_{x}}^2
    =\|Q\|_{L^2_{y}}^2+\|u^*\|_{L^2_{x}}^2.
\]
By Remark~\ref{nominimalmassblowup},
\(\|u_0\|_{L^2_{x}}>\|Q\|_{L^2_{y}}\), and consequently
\[
    u^*\not\equiv0.
\]

Finally, the convergence cannot be strong in \(H^1_{x}\). Set
\[
    r(t,x_1,x_2):=u(t,x_1,x_2)-S(t,x_1,x_2).
\]
Then
\[
    r(t,x_1,x_2)
    =\frac1{\lambda(t)}
    \bigl[
      \varepsilon(t,y_1,y_2)+b(t)\chi_b(y_1)P(y_1,y_2)
    \bigr].
\]
By the triangle inequality,
\[
\begin{aligned}
\|\nabla r(t)\|_{L^2_{x}}
&\geq
\frac1{\lambda(t)}
\left(
\|\nabla\varepsilon(t)\|_{L^2_{y}}
-\|\nabla(b\chi_bP)\|_{L^2_{y}}
\right)\\
&\gtrsim
\lambda(t)^{\frac c2-1}
-\lambda(t)^{c(1-\gamma/2)-1}\\
&=
\lambda(t)^{\frac{c-2}{2}}
\left[
1-O\left(
\lambda(t)^{\frac{c(1-\gamma)}2}
\right)
\right]
\longrightarrow+\infty,
\end{aligned}
\]
where we used \eqref{eq:estimateL2derivative},
\(b(t)\sim c_0\lambda^c(t)\), \(c<2\), and \(\gamma<1\).

Hence the convergence in \eqref{eq:strong-all-finite-p} is not strong in
$H^1_{x}(\mathbb R^2)$. The next section proves the stronger statement
that the limiting profile itself satisfies
\[
    u^*\notin H^\varsigma_{x}(\mathbb R^2)
    \qquad\text{for every }\varsigma\geq\frac c2.
\]

\section{Sobolev Obstruction for the Limiting Radiation}

The following lemma quantifies the left-hand convergence of $P$ and is needed below to isolate the leading cut-off transition term in $\Psi_b$ and control the remaining profile error.

\begin{lemma}[Quantitative left-tail expansion of $P$]
\label{lem:P-left-tail}
Recall from \eqref{definitionofF} that 
\[
    F(y_2):=\lim_{y_1\to-\infty}P(y_1,y_2)
\]
and from \eqref{Feq}
\[
-F_{y_2y_2}+F
=
-\displaystyle\int_{-\infty}^{\infty}\Lambda Q\,dy_1
=-g_Q(y_2).
\]
There exist constants $C,\kappa>0$ such that, for every multi-index
$\beta$ with $|\beta|\leq3$,
\[
    \left|
    \partial^\beta\bigl(P(y_1,y_2)-F(y_2)\bigr)
    \right|
    \leq
    Ce^{\kappa y_1}e^{-\kappa|y_2|}
\]
for all $y_1\leq-1$. Consequently,
\[
\left|P-P_{y_2y_2}+g_Q\right|
+|P_{y_1}|+|P_{y_1y_1}|
\leq
Ce^{\kappa y_1}e^{-\kappa|y_2|}
\qquad (y_1\leq-1).
\]
\end{lemma}

\begin{proof}
From the construction of $P$,
\[
P-F
=
\widetilde P
+
\int_{-\infty}^{y_1}U(z,y_2)\,dz
+
\beta_PQ_{y_1},
\qquad
\beta_P=\frac{(U,Q)}{\|Q_{y_1}\|_{L^2_y}^2}.
\]
The functions $U,\widetilde P$, and $Q_{y_1}$, together with all their
derivatives, belong to $\mathcal Y$. The stated estimate therefore
follows after slightly decreasing the exponential decay rate. Finally,
\[
P-P_{y_2y_2}+g_Q
=
(P-F)-(P_{y_2y_2}-F''),
\]

and the last assertion follows.
\end{proof}

\begin{theorem}
\label{thm:radiation-not-Hs}
Assume that $u$ belongs to the stable blow-up regime of
Theorem~\ref{rigiditytheorem}, and let $u^*$ be the asymptotic radiation
constructed in Chapter~\ref{sec:asymptotic-radiation}.
Then
\[
    u^*\notin H^\varsigma(\mathbb R^2)
    \qquad\text{for every }\varsigma\geq\frac c2.
\]
In particular, since $c<2$, we have $u^*\notin H^1(\mathbb R^2)$.
\end{theorem}

\begin{proof}
For $t$ sufficiently close to $T$, we have $b(t)>0$ and $\lambda$ is
strictly decreasing. Set
\[
S(t,x)=\frac1{\lambda(t)}
Q\left(\frac{x-x(t)}{\lambda(t)}\right),
\qquad
\mathcal Q(t,x)=\frac1{\lambda(t)}
Q_{b(t)}\left(\frac{x-x(t)}{\lambda(t)}\right),
\]
\[
\widetilde u(t,x)=u(t,x)-\mathcal Q(t,x)
=\frac1{\lambda(t)}
\varepsilon\left(t,\frac{x-x(t)}{\lambda(t)}\right),
\qquad x(t)=(x_1(t),x_2(t)).
\]
By the strong $L^2$ convergence proved in
Chapter~\ref{sec:asymptotic-radiation} and
\eqref{eq:QSL2estimate},
\begin{equation}\label{eq:radiation-L2-limit}
    \widetilde u(t)\longrightarrow u^*
    \qquad\text{strongly in }L^2_x(\mathbb R^2)
    \quad\text{as }t\uparrow T.
\end{equation}
We use the Fourier convention of \eqref{eq:Airy} and write
\[
    \omega(\xi)=\xi_1^3+\xi_1\xi_2^2.
\]

\medskip
\noindent\textbf{Step 1: asymptotics of the geometric parameters.}
Fix $0<\eta<1$ for which \eqref{sharpN} holds, and set
\[
    a=\frac{3-\eta}{2},
    \qquad
    \nu_m=a-1=\frac{1-\eta}{2}>0.
\]
By \eqref{sharpN}, the hierarchy of the $\mathcal N_i$, and
\eqref{McomparedtoN},
\[
\mathcal M^{1/2}\lesssim\mathcal N_1^{1/2}
\lesssim\lambda^{ac}\sim b^a,
\qquad
\widetilde{\mathcal M}\lesssim\mathcal N_1
\lesssim\lambda^{(3-\eta)c}\sim b^{3-\eta}.
\]
Since $|J|\lesssim\mathcal N_2^{1/2}\to0$, the convergence in
\eqref{definitionofc0} implies
\[
    \frac{b(t)}{\lambda^c(t)}\longrightarrow c_0>0.
\]
Hence \eqref{eq:modulatedcoefficients} and \eqref{eq:estimateL2} give
\begin{equation}\label{eq:physical-modulation-small}
\left|\lambda^2\lambda_t+b\right|
+\left|\lambda^2(x_1)_t-1\right|
+\left|\lambda^2(x_2)_t\right|
\lesssim b^{1+\nu_m},
\qquad
|b_s|\lesssim b^2.
\end{equation}
In particular,
\[
    \lambda_t=-c_0\lambda^{c-2}(1+o(1)).
\]
Integrating this identity, and then dividing the equation for $(x_1)_t$
by that for $\lambda_t$, yields
\begin{equation}\label{eq:radiation-phase-asymptotics}
T-t=\frac{\lambda^{3-c}(t)}{(3-c)c_0}(1+o(1)),
\qquad
x_1(t)=\frac{\lambda^{1-c}(t)}{(c-1)c_0}(1+o(1)).
\end{equation}
Moreover, by \eqref{eq:physical-modulation-small},
\[
|(x_2)_t|\lesssim\lambda^{-2}b^{1+\nu_m},
\qquad
-\lambda_t\sim b\lambda^{-2},
\]
and therefore
\[
\left|\frac{dx_2}{d\lambda}\right|
=\left|\frac{(x_2)_t}{\lambda_t}\right|
\lesssim b^{\nu_m}\lesssim\lambda^{c\nu_m}.
\]
Since $x_2(t)\to x_{2,\infty}$,
\begin{equation}\label{eq:x2-rate-radiation}
|x_2(t)-x_{2,\infty}|
\lesssim\int_0^{\lambda(t)}\mu^{c\nu_m}\,d\mu
\lesssim\lambda^{1+c\nu_m}(t).
\end{equation}

\medskip
\noindent\textbf{Step 2: transition source and time-integrable remainders.}

By Lemma~\ref{lem:P-left-tail}, on the transition region
$-2b^{-\gamma}\leq y_1\leq-b^{-\gamma}$,
\[
|P-P_{y_2y_2}+g_Q|+|P_{y_1}|+|P_{y_1y_1}|
\lesssim e^{-\kappa b^{-\gamma}}e^{-\kappa|y_2|}.
\]
Using the exact expansion \eqref{eq:Psi_b_expansion}, we obtain
\begin{equation}\label{eq:Psi-transition-decomposition}
\Psi_b(y)
=-b^{1+\gamma}\chi'(b^\gamma y_1)g_Q(y_2)
+\Psi_b^{\rm rem}(y).
\end{equation}
More precisely,
\[
\|\Psi_b^{\rm rem}\|_{L^1_y}
\lesssim
be^{-\kappa b^{-\gamma}}
+b^{1+2\gamma}+b^{2-\gamma}+b^2.
\]
Here the first term contains the contributions of $P-F$, $P_{y_1}$,
$P_{y_1y_1}$, $Q^2P$, and the localization error; the term
$b(\chi_b)_{y_1y_1y_1}P$ is $O(b^{1+2\gamma})$ in $L^1_y$,
$b^2\Lambda(\chi_bP)$ is $O(b^{2-\gamma})$, and the remaining nonlinear
terms are $O(b^2)$. Thus, with
\[
    \nu_{\rm tr}=\min\{2\gamma,1-\gamma\}>0,
\]
\begin{equation}\label{eq:Psi-rem-L1}
    \|\Psi_b^{\rm rem}\|_{L^1_y}\lesssim b^{1+\nu_{\rm tr}}.
\end{equation}

Recall the forcing $\mathfrak F$
introduced in \eqref{mathfrakF}. We write
\[
    \mathfrak F=\mathfrak F_{\rm tr}+\mathfrak F_{\rm rem},
\]
where
\begin{equation}\label{eq:F-transition-physical}
\mathfrak F_{\rm tr}(t,x)
=\frac{b^{1+\gamma}(t)}{\lambda^4(t)}
\chi'\bigl(b^\gamma(t)y_1\bigr)g_Q(y_2).
\end{equation}
A change of variables gives
\begin{equation}\label{eq:F-transition-Fourier}
\widehat{\mathfrak F_{\rm tr}}(t,\xi)
=b(t)\lambda^{-2}(t)e^{-ix(t)\cdot\xi}
\widehat{\chi'}\left(\frac{\lambda(t)\xi_1}{b^\gamma(t)}\right)
\widehat {g_Q}\bigl(\lambda(t)\xi_2\bigr),
\end{equation}
and $\widehat{\chi'}(0)=\int\chi'=1$.

The nonlocalized $y_1$-tail has effective length $O(b^{-\gamma})$,
every $y_1$-derivative of $P$ is localized, and $P$ and its derivatives
decay exponentially in $y_2$. Hence
\[
\|\zeta_bP\|_{L^1_y}\lesssim b^{-\gamma},
\qquad
\|\Lambda Q_b\|_{L^1_y}
+\|(Q_b)_{y_1}\|_{L^1_y}
+\|(Q_b)_{y_2}\|_{L^1_y}\lesssim1.
\]
Using \eqref{eq:Psi-rem-L1}, \eqref{eq:physical-modulation-small}, and
$|b_t|=|b_s|/\lambda^3\lesssim b^2/\lambda^3$, we find
\[
\begin{aligned}
\|\lambda^{-4}\Psi_b^{\rm rem}(y)\|_{L^1_x}\lesssim \lambda^{-2}\|\Psi_b^{\rm rem}\|_{L^1_y}
&\lesssim\lambda^{-2}b^{1+\nu_{\rm tr}},\\
\left\|\frac{b_t}{\lambda}\zeta_bP(y)\right\|_{L^1_x}\lesssim |b_t||\lambda|\|\zeta_bP\|_{L^1_y}
&\lesssim\lambda^{-2}b^{2-\gamma},\\
\left\|\frac{\lambda^2\lambda_t+b}{\lambda^4}
\Lambda Q_b(y)\right\|_{L^1_x}
+\sum_{j=1}^2
\left\|\frac{\lambda^2(x_j)_t-\delta_{1j}}{\lambda^4}
(Q_b)_{y_j}(y)\right\|_{L^1_x}
&\lesssim\lambda^{-2}b^{1+\nu_m}.
\end{aligned}
\]
Set
\[
    \nu_{\text{prof}}=\min\{\nu_{\rm tr},\nu_m\}>0.
\]
Since
\begin{equation}\label{dlambdadt}
    dt=-\frac1{c_0}\lambda^{2-c}(1+o(1))\,d\lambda,
\end{equation}
or, in estimates, $|dt|\lesssim\lambda^{2-c}|d\lambda|$, we obtain
\begin{equation}\label{eq:profile-rem-time-L1}
\int_{t_\star}^{T}\|\mathfrak F_{\rm rem}(t)\|_{L^1_x}\,dt
\lesssim\lambda^{1+c\nu_{\text{prof}}}(t_\star).
\end{equation}

We next estimate the nonlinearity
$f(\widetilde u)=\lambda^{-3}[(Q_b+\varepsilon)^3-Q_b^3](y)$.
By \eqref{sharpN} and \eqref{eq:estimateL2derivative},
\[
\|Q\varepsilon\|_{L^2_y}\lesssim\lambda^{ac},
\qquad
\|\varepsilon\|_{L^4_y}\lesssim\lambda^{c/4},
\qquad
\|\varepsilon\|_{L^{8/3}_y}^2\lesssim\lambda^{c/4}.
\]
Expanding $Q_b=Q+b\chi_bP$, we have
\[
\begin{aligned}
\|Q_b^2\varepsilon\|_{L^{4/3}_y}
&\lesssim
\|Q^2\varepsilon\|_{L^{4/3}_y}
+b\|Q\chi_bP\varepsilon\|_{L^{4/3}_y}
+b^2\|\chi_b^2P^2\varepsilon\|_{L^{4/3}_y}\lesssim
\lambda^{ac}+b^{2-\gamma/4}.
\end{aligned}
\]
Indeed,
$\|Q^2\varepsilon\|_{L^{4/3}_y}\leq\|Q\varepsilon\|_{L^2_y}\|Q\|_{L^4_y}$.
Moreover, the same weighted estimate gives
$\|\sqrt Q\,\varepsilon\|_{L^2_y}\lesssim \mathcal N_1^{1/2}(t)\lesssim\lambda^{ac}$, while
$\|\sqrt Q\,\chi_bP\|_{L^4_y}\lesssim1$; hence
$\|Q\chi_bP\varepsilon\|_{L^{4/3}_y}\lesssim\lambda^{ac}$. Finally,
$\|\chi_b^2P^2\|_{L^4_y}\lesssim b^{-\gamma/4}$.
Similarly,
\[
\|Q_b\varepsilon^2\|_{L^{4/3}_y}
\lesssim
\|Q\varepsilon\|_{L^2_y}\|\varepsilon\|_{L^4_y}
+b\|P\|_{L^\infty_y}\|\varepsilon\|_{L^{8/3}_y}^2
\lesssim\lambda^{ac+c/4}+b\lambda^{c/4},
\]
while from Gagliardo-Nirenberg inequality and \eqref{eq:estimateL2derivative},
\[
\|\varepsilon^3\|_{L^{4/3}_y}
=\|\varepsilon\|_{L^4_y}^3
\lesssim\lambda^{3c/4}.
\]
Since $a>1$, $0<\gamma<1$, $b\lesssim\lambda^c$, and
$0<\lambda\leq1$, all these terms are bounded by
$C\lambda^{3c/4}$. Therefore
\begin{equation}\label{eq:nonlinearity-y-L43}
\|(Q_b+\varepsilon)^3-Q_b^3\|_{L^{4/3}_y}
\lesssim\lambda^{3c/4}.
\end{equation}
After returning to physical variables and using the preceding time
change,
\begin{equation}\label{eq:nonlinearity-time-L43}
\int_{t_R}^{T}\|f(\widetilde u)(t)\|_{L^{4/3}_x}\,dt
\lesssim\lambda^{3/2-c/4}(t_R).
\end{equation}

Choose
\begin{equation}\label{eq:kappa-choice-radiation}
\max\left\{
\frac1{1+c\nu_{\text{prof}}},
\frac{3(2-c)}{6-c}
\right\}<\varkappa<1,
\end{equation}
and, for $R$ sufficiently large, define $t_R<T$ by
\[
    \lambda(t_R)=R^{-\varkappa}.
\]
For a fixed compact interval $K\Subset\mathbb R$, set
\[
\sigma(\xi)=\frac{\xi_1\xi_2^{c-1}}{c_0},
\qquad
\Gamma_R^K
=\{\xi:\ R\leq\xi_2\leq2R,\ \sigma(\xi)\in K\}.
\]
For $t_R<t_1<T$, the local $H^1$ theory and the equation for
$\widetilde u$ give the finite-time Duhamel formula
\begin{equation*}
\begin{aligned}
\widehat{\widetilde u}(t_1)
={}&
e^{i(t_1-t_R)\omega}
\widehat{\widetilde u}(t_R)-
\int_{t_R}^{t_1}e^{i(t_1-t)\omega}
\left[
i\xi_1\widehat{f(\widetilde u)}(t)
+\widehat{\mathfrak F}(t)
\right]dt.
\end{aligned}
\end{equation*}
For \(t_1<T\), Duhamel's formula holds in \(H^{-2}(\mathbb R^2)\).
Let \(\Pi_R^K\) denote the Fourier multiplier with symbol
\(\mathbf1_{\Gamma_R^K}\). Since \(\Gamma_R^K\) is bounded for fixed
\(R\), \(\Pi_R^K:H^{-2}_x\to L^2_x\) is bounded. Applying \(\Pi_R^K\) to
Duhamel's formula gives 
\begin{equation}\label{eq:finite-Duhamel-radiation}
\begin{aligned}
\mathbf 1_{\Gamma_R^K}\widehat{\widetilde u}(t_1)
={}&
\mathbf 1_{\Gamma_R^K}e^{i(t_1-t_R)\omega}
\widehat{\widetilde u}(t_R)-\mathbf 1_{\Gamma_R^K}
\int_{t_R}^{t_1}e^{i(t_1-t)\omega}
\left[
i\xi_1\widehat{f(\widetilde u)}(t)
+\widehat{\mathfrak F}(t)
\right]dt
\end{aligned}
\end{equation}
in
\(L^2_\xi\).

Define
\[
\mathcal G_R(t,\xi)
=
\mathbf 1_{\Gamma_R^K}(\xi)
\left[
i\xi_1\widehat{f(\widetilde u)}(t,\xi)
+\widehat{\mathfrak F}(t,\xi)
\right].
\]
We claim that
\begin{equation}\label{eq:GR-L1L2}
    \mathcal G_R\in L^1_t([t_R,T);L^2_\xi).
\end{equation}
For the transition source,
\[
\|\mathfrak F_{\rm tr}(t)\|_{L^2_x}
\lesssim\lambda^{-3}b^{1+\gamma/2},
\]
and therefore, using \eqref{dlambdadt}, 
\[
\int_{t_R}^{T}
\|\mathfrak F_{\rm tr}(t)\|_{L^2_x}\,dt
\lesssim \int_0^{\lambda(t_R)}
\mu^{-3+c(1+\gamma/2)}\mu^{2-c}\,d\mu=
\int_0^{\lambda(t_R)}
\mu^{-1+c\gamma/2}\,d\mu<\infty.
\]
For the remainder,
\[
\|\mathbf 1_{\Gamma_R^K}
\widehat{\mathfrak F_{\rm rem}}(t)\|_{L^2_\xi}
\leq
|\Gamma_R^K|^{1/2}
\|\mathfrak F_{\rm rem}(t)\|_{L^1_x},
\]
which is time-integrable by \eqref{eq:profile-rem-time-L1}. Finally,
Hausdorff--Young and H\"older on the finite-measure set $\Gamma_R^K$
give
\[
\|\mathbf 1_{\Gamma_R^K}i\xi_1
\widehat{f(\widetilde u)}(t)\|_{L^2_\xi}
\lesssim
\|\xi_1\|_{L^\infty_\xi(\Gamma_R^K)}
|\Gamma_R^K|^{1/4}
\|f(\widetilde u)(t)\|_{L^{4/3}_x},
\]
which is time-integrable by \eqref{eq:nonlinearity-time-L43}. This
proves \eqref{eq:GR-L1L2}.

We can now let $t_1\uparrow T$ in
\eqref{eq:finite-Duhamel-radiation}. Indeed,
\eqref{eq:radiation-L2-limit} and Plancherel give convergence of the
left-hand side in $L^2_\xi(\Gamma_R^K)$. Since $\Gamma_R^K$ is bounded,
$e^{i(t_1-T)\omega}\to1$ uniformly there. Moreover, if
\[
I_{t_1}
=
\int_{t_R}^{t_1}
e^{i(t_1-t)\omega}\mathcal G_R(t)\,dt,
\qquad
I_T
=
\int_{t_R}^{T}
e^{i(T-t)\omega}\mathcal G_R(t)\,dt,
\]
then
\[
\begin{aligned}
\|I_{t_1}-I_T\|_{L^2_\xi}
&\leq
\sup_{\Gamma_R^K}
|e^{i(t_1-T)\omega}-1|
\int_{t_R}^{t_1}\|\mathcal G_R(t)\|_{L^2_\xi}\,dt+
\int_{t_1}^{T}\|\mathcal G_R(t)\|_{L^2_\xi}\,dt
\longrightarrow0.
\end{aligned}
\]
Also, 
\[
\norm{1_{\Gamma_R^K}\big[e^{i(t_1-t_R)\omega}-e^{i(T-t_R)\omega}\big]\widehat{\widetilde{u}}(t_R)}_{L^2_\xi}\leq \sup_{\Gamma_R^K}
|e^{i(t_1-T)\omega}-1|\|\widetilde{u}(t_R)\|_{L^2_x}\rightarrow 0.
\]
 Finally, by the strong $L^2$ convergence of $\widetilde{u}$ to $u^*$ we have that 
\[
\|1_{\Gamma_R^K}(\widehat{\widetilde{u}}(t_1)-\widehat{u^*})\|_{L^2_\xi}\rightarrow 0. 
\]
Thus, for every fixed $R$ and $K$,
\begin{equation}\label{eq:terminal-Duhamel-radiation}
\begin{aligned}
\mathbf 1_{\Gamma_R^K}\widehat{u^*}
={}&
\mathbf 1_{\Gamma_R^K}e^{i(T-t_R)\omega}
\widehat{\widetilde u}(t_R)-\mathbf 1_{\Gamma_R^K}
\int_{t_R}^{T}e^{i(T-t)\omega}
\left[
i\xi_1\widehat{f(\widetilde u)}(t)
+\widehat{\mathfrak F}(t)
\right]dt
\end{aligned}
\end{equation}
in $L^2_\xi(\Gamma_R^K)$.

\medskip
\noindent\textbf{Step 3: the transition-frequency contribution.}
For $q>0$, define
\[
\Theta(q)
=
\frac{q^{3-c}}{3-c}
-
\frac{q^{1-c}}{c-1}.
\]
Then
\[
\Theta'(q)=q^{-c}(1+q^2)>0,
\]
and $\Theta$ maps $(0,\infty)$ bijectively onto $\mathbb R$. Set
\begin{equation}\label{eq:amplitude-definition}
    \mathscr{A}(\sigma)
    =
    \int_0^\infty e^{i\sigma\Theta(q)}\widehat {g_Q}(q)\,dq.
\end{equation}
Since $g_Q\in W^{2,1}(\mathbb R)$,
\[
|\widehat {g_Q}(q)|\lesssim(1+q^2)^{-1},
\]
and hence $\widehat {g_Q}\in L^1_q$. The function $g_Q$ is real and even.
Hence $\widehat {g_Q}$ is real and even. Since $g_Q\not\equiv0$ from \eqref{definitionofCQ}, its Fourier
transform cannot vanish identically on $(0,\infty)$. After the change
of variables $z=\Theta(q)$,
\[
\mathscr{A}(\sigma)
=
\int_{\mathbb R}e^{i\sigma z}H(z)\,dz,
\qquad
H(z)
=
\frac{\widehat {g_Q}(\Theta^{-1}(z))}
{\Theta'(\Theta^{-1}(z))}.
\]
Here $H\in L^1_z$ and $H\not\equiv0$, thus Fourier injectivity
implies $\mathscr{A}\not\equiv0$. By continuity, there exist a compact
interval $I\Subset\mathbb R$ and $a_0>0$ such that
\begin{equation}\label{eq:amplitude-lower}
    |\mathscr{A}(\sigma)|\geq2a_0
    \qquad(\sigma\in I).
\end{equation}

Let
\begin{equation}\label{eq:transition-Duhamel-contribution}
\mathcal T_R(\xi)
=
-\int_{t_R}^{T}e^{i(T-t)\omega(\xi)}
\widehat{\mathfrak F_{\rm tr}}(t,\xi)\,dt.
\end{equation}

\begin{lemma}
\label{lem:uniform-transition-limit}
Let $K\Subset\mathbb R$. Then
\[
\sup_{\substack{R\leq\xi_2\leq2R\\ \sigma(\xi)\in K}}
\left|
\xi_2e^{ix_{2,\infty}\xi_2}\mathcal T_R(\xi)
+\mathscr{A}(\sigma(\xi))
\right|
\longrightarrow0
\]
as $R\to\infty$.
\end{lemma}

\begin{proof}
Using \eqref{eq:F-transition-Fourier}, changing variables from $t$ to
$\lambda$, and then setting $q=\lambda\xi_2$, we obtain
\begin{equation}\label{eq:transition-q-representation}
\begin{aligned}
\xi_2e^{ix_{2,\infty}\xi_2}\mathcal T_R(\xi)
={}&-\int_0^{\lambda(t_R)\xi_2}
 e^{i\Upsilon_R(q,\xi)}
 \frac{b}{-\lambda^2\lambda_t}
 \widehat{\chi'}\left(\frac{\lambda\xi_1}{b^\gamma}\right)
 \widehat {g_Q}(q)\,dq,
\end{aligned}
\end{equation}
where $t=t(q,\xi_2)$ is determined by
$\lambda(t)=q/\xi_2$, and
\[
\Upsilon_R(q,\xi)
=
(T-t)\omega(\xi)-x_1(t)\xi_1
-(x_2(t)-x_{2,\infty})\xi_2.
\]
Fix $\epsilon>0$. Since $\widehat {g_Q}\in L^1_q$, choose
$0<\delta<1<M$ such that
\[
\int_0^\delta|\widehat {g_Q}(q)|\,dq
+
\int_M^\infty|\widehat {g_Q}(q)|\,dq
<\epsilon.
\]
For $R$ large,
\[
\lambda(t_R)\xi_2\geq R^{1-\varkappa}>M.
\]
The ratio $b/(-\lambda^2\lambda_t)$ is uniformly bounded near $T$, and
\[
|\widehat{\chi'}|\leq\|\chi'\|_{L^1_{y_1}}.
\]
Hence the contributions of
$(0,\delta)\cup(M,\infty)$ to the difference between
\eqref{eq:transition-q-representation} and
$-\mathscr{A}(\sigma)$ are $O(\epsilon)$, uniformly in the stated region.

It remains to consider $q\in[\delta,M]$. There
\[
\lambda=\frac q{\xi_2}\longrightarrow0
\]
uniformly for $R\leq\xi_2\leq2R$. If $\sigma(\xi)\in K$, then
\[
    \xi_1=c_0\sigma(\xi)\xi_2^{1-c}.
\]
Using \eqref{eq:radiation-phase-asymptotics}, uniformly on
$K\times[\delta,M]$,
\[
(T-t)\xi_1\xi_2^2
=
\sigma(\xi)\frac{q^{3-c}}{3-c}(1+o_R(1)),
\]
\[
x_1(t)\xi_1
=
\sigma(\xi)\frac{q^{1-c}}{c-1}(1+o_R(1)).
\]
Moreover,
\[
|(T-t)\xi_1^3|
\lesssim_{K,M}R^{-2c}\longrightarrow0,
\]
\[
|(x_2(t)-x_{2,\infty})\xi_2|
\lesssim
q^{1+c\nu_m}\xi_2^{-c\nu_m}
\longrightarrow0,
\]
and
\[
\frac{\lambda\xi_1}{b^\gamma}
=
C_{\sigma,q}\xi_2^{-c(1-\gamma)}(1+o_R(1))
\longrightarrow0
\]
uniformly. Finally,
\[
    \frac{b}{-\lambda^2\lambda_t}\longrightarrow1
\]
uniformly for $0<\lambda\leq M/R$. Hence the integrand in
\eqref{eq:transition-q-representation} converges uniformly on
$K\times[\delta,M]$ to
\[
e^{i\sigma\Theta(q)}\widehat {g_Q}(q).
\]
The compact part therefore converges uniformly, and the integrable tails
complete the proof.
\end{proof}

Define
\begin{equation}\label{eq:Gamma-R}
\Gamma_R
=
\left\{
\xi:\ R\leq\xi_2\leq\frac32R,
\quad
\sigma(\xi)\in I
\right\}.
\end{equation}
For fixed $\xi_2$, the admissible $\xi_1$-interval has length
comparable to $\xi_2^{1-c}$, and hence
\begin{equation}\label{eq:Gamma-measure}
    |\Gamma_R|\asymp \int_R^{\frac{3R}{2}}\xi_2^{1-c}d\xi_2\asymp R^{2-c}.
\end{equation}
By Lemma~\ref{lem:uniform-transition-limit} and
\eqref{eq:amplitude-lower}, for all sufficiently large $R$,
\[
\left|
\xi_2e^{ix_{2,\infty}\xi_2}\mathcal T_R(\xi)
+\mathscr{A}(\sigma(\xi))
\right|
\leq a_0
\qquad(\xi\in\Gamma_R).
\]
Thus
\[
|\mathcal T_R(\xi)|
\geq\frac{a_0}{\xi_2}
\gtrsim R^{-1}
\qquad(\xi\in\Gamma_R),
\]
and therefore
\begin{equation}\label{eq:transition-frequency-lower}
\|\mathbf 1_{\Gamma_R}\mathcal T_R\|_{L^2_\xi}^2
\gtrsim
R^{-2}|\Gamma_R|
\gtrsim R^{-c}.
\end{equation}

\medskip
\noindent\textbf{Step 4: comparison with the remaining Duhamel terms.}
Since $|\xi|\gtrsim R$ on $\Gamma_R$, Plancherel and
\eqref{eq:estimateL2derivative} give
\begin{equation}\label{eq:initial-frequency-small}
\begin{aligned}
\|\mathbf 1_{\Gamma_R}\widehat{\widetilde u}(t_R)\|_{L^2_\xi}
&\lesssim
R^{-1}\|\nabla\widetilde u(t_R)\|_{L^2_x}\\
&=
R^{-1}\lambda^{-1}(t_R)
\|\nabla\varepsilon(t_R)\|_{L^2_y}\\
&\lesssim
R^{-1+\varkappa(1-c/2)}
=o(R^{-c/2}),
\end{aligned}
\end{equation}
because $c<2$ and $\varkappa<1$.

By \eqref{eq:profile-rem-time-L1} and
\eqref{eq:Gamma-measure},
\begin{equation}\label{eq:profile-frequency-small}
\left\|
\mathbf 1_{\Gamma_R}
\int_{t_R}^{T}e^{i(T-t)\omega}
\widehat{\mathfrak F_{\rm rem}}(t)\,dt
\right\|_{L^2_\xi}
\lesssim
R^{(2-c)/2-\varkappa(1+c\nu_{\text{prof}})}
=o(R^{-c/2})
\end{equation}
by \eqref{eq:kappa-choice-radiation}. Finally, Hausdorff--Young,
$|\xi_1|\lesssim R^{1-c}$ on $\Gamma_R$, and
\eqref{eq:nonlinearity-time-L43} yield
\begin{equation}\label{eq:nonlinear-frequency-small}
\begin{aligned}
&\left\|
\mathbf 1_{\Gamma_R}
\int_{t_R}^{T}e^{i(T-t)\omega}
i\xi_1\widehat{f(\widetilde u)}(t)\,dt
\right\|_{L^2_\xi}\\
&\qquad\lesssim
R^{1-c+(2-c)/4-\varkappa(3/2-c/4)}
=o(R^{-c/2}),
\end{aligned}
\end{equation}
again by \eqref{eq:kappa-choice-radiation}.

Using the identity
\eqref{eq:terminal-Duhamel-radiation} with $K=I$, restricting it to
$\Gamma_R$, and applying the triangle inequality, we conclude from
\eqref{eq:transition-frequency-lower}--\eqref{eq:nonlinear-frequency-small}
that
\[
\begin{aligned}
\|\mathbf 1_{\Gamma_R}\widehat{u^*}\|_{L^2_\xi}
&\geq
\|\mathbf 1_{\Gamma_R}\mathcal T_R\|_{L^2_\xi}
-\|\mathbf 1_{\Gamma_R}\widehat{\widetilde u}(t_R)\|_{L^2_\xi}\\
&\quad-
\left\|
\mathbf 1_{\Gamma_R}
\int_{t_R}^{T}e^{i(T-t)\omega}
\widehat{\mathfrak F_{\rm rem}}(t)\,dt
\right\|_{L^2_\xi}\\
&\quad-
\left\|
\mathbf 1_{\Gamma_R}
\int_{t_R}^{T}e^{i(T-t)\omega}
i\xi_1\widehat{f(\widetilde u)}(t)\,dt
\right\|_{L^2_\xi}\\
&\gtrsim R^{-c/2}.
\end{aligned}
\]
Consequently,
\begin{equation}\label{eq:u-star-frequency-lower}
    \int_{\Gamma_R}|\widehat{u^*}(\xi)|^2\,d\xi
    \gtrsim R^{-c}.
\end{equation}

Let $R_n=4^n$. The sets $\Gamma_{R_n}$ are pairwise disjoint and
$|\xi|\asymp R_n$ on $\Gamma_{R_n}$. Therefore, for every
$\varsigma\geq c/2$,
\[
\begin{aligned}
\|u^*\|_{\dot H^\varsigma}^2
&\geq
\sum_{n\gg1}
\int_{\Gamma_{R_n}}
|\xi|^{2\varsigma}|\widehat{u^*}(\xi)|^2\,d\xi\gtrsim
\sum_{n\gg1}R_n^{2\varsigma-c}
=+\infty.
\end{aligned}
\]
Since $u^*\in L^2_x$, this proves
\[
u^*\notin H^\varsigma(\mathbb R^2)
\qquad\text{for every }\varsigma\geq\frac c2.
\]
In particular, $u^*\notin H^1(\mathbb R^2)$.
\end{proof}

 \chapter{Blow-Up for  \DIFaddbegin \texorpdfstring{$E_0\leq 0$}{nonpositive energy}\DIFaddend }\label{Blow Up for E0}

 We start with an orbital stability result that appears in \cite{Merle01} (Lemma 1) and \cite{MartelMerle02} (Lemma 1) for the one-dimensional case regarding the gKdV equation. We reiterate the result for the two-dimensional case.

\begin{theorem} \label{orbitalstability}
There exists $\alpha_{\text{orb}}>0$ such that the following holds. For all $0<\alpha'\leq \alpha_{\text{orb}},$ there exists $\delta=\delta(\alpha')>0,$ with $\delta(\alpha')\rightarrow 0$ as $\alpha'\rightarrow 0,$ such that for all $u \in H^1(\mathbb{R}^2), u \not\equiv 0,$ if
$$\|u\|_{L^2}^2-\|Q\|_{L^2}^2=\alpha(u)\leq \alpha' \quad \text{and} \quad E(u)\leq \alpha'\iint  |\nabla u|^2,$$
then there exists $\tilde{y}_1, \tilde{y}_2$ $\in \mathbb{R}$ and $\epsilon_0\in \{-1,1\}$ such that
$$\|Q-\epsilon_0\lambda_0u(\lambda_0 y_1 + \tilde{y}_1,\lambda_0 y_2 + \tilde{y}_2)\|_{H^1}\leq \delta(\alpha'),$$
with $$\lambda_0=\frac{\|\nabla Q\|_{L^2}}{\|\nabla u\|_{L^2}}.$$
\end{theorem}

We include the proof of this theorem in Appendix \ref{appendixC} for the sake of completeness.

Let $u_0\in\mathcal A_{\alpha_0}$ satisfy $E(u_0)\le0$. Choose the
orbital-stability threshold below $\alpha^*$. The orbital-stability theorem \ref{orbitalstability}
then keeps the solution in $\mathcal T_{\alpha^*}$ throughout its lifespan,
so the exit alternative is excluded. Let $(\lambda(t), b(t), x_1(t), x_2(t))$ be the geometrical quantities arising from a decomposition of $u(t)$ as in Lemma \ref{decompositionlemma}.

\textbf{Case 1: \(E_0<0\).}

Suppose, by contradiction, that the solution belongs to the
Asymptotic Stability branch. Then
\[
    b(s)\to0,\qquad \mathcal N_2(s)\to0,
    \qquad \lambda(s)\to\lambda_\infty>0.
\]
On the other hand, conservation of energy \eqref{eq:energyconservation} gives
\[
    \lambda^2(s)|E_0|
    +\|\nabla\varepsilon(s)\|_{L^2}^2
    \lesssim |b(s)|+\mathcal N_2(s)\to0,
\]
so \(\lambda(s)\to0\), a contradiction.

\textbf{Case 2:} $E_0 = 0$.

 Suppose, for a contradiction, that the solution belongs to the
asymptotic-stability branch of Theorem~\ref{rigiditytheorem} and that
\(u_0\) is not equivalent to \(Q\) up to symmetries. Then
\(b(s)\to0\), \(\mathcal N_2(s)\to0\), and
\(\lambda(s)\to\lambda_\infty>0\). The energy estimate
\eqref{eq:energyconservation}  gives
\begin{equation}\label{sharpboundforgradient}
    \|\nabla \varepsilon(s)\|_{L^2}^{2} \lesssim |b(s)| + \mathcal{N}_2(s) \rightarrow 0 \mbox{ as } s \rightarrow +\infty.
\end{equation}

The mass estimate \eqref{eq:massconservation} also gives
\(\|\varepsilon(t)\|_{L^2}\lesssim\delta(\alpha_0)\) for all
\(t\geq0\). Moreover, by \eqref{Ndominatingb}, there exists
\(C^*>0\) such that
\begin{equation}\label{bdominatedbyN}
|b(t)| \leq C^* \mathcal{N}_3(t)
\end{equation}
for all $t \in [0, +\infty)$.
Thus, by the bounds \eqref{sharpboundforgradient}, \eqref{bdominatedbyN} and the estimate \eqref{boundiplusone}, we obtain
\begin{equation}\label{boundforgradient}
\int_{t}^{+\infty} \|\nabla \varepsilon(t')\|_{L^2}^{2} dt' \lesssim \int_{t}^{+\infty} |b(t')| dt' + \int_{t}^{+\infty} \mathcal{N}_2(t') dt' \lesssim \mathcal{N}_4(t) \rightarrow 0, \mbox{ as } t \rightarrow +\infty.
\end{equation}

We begin by giving a rough bound for the localized mass. The next lemma appears in Appendix C, \cite{MerleRaphael06}, and here it is adapted to our setting.

\begin{lemma} 
Let $w \in H^1(\mathbb{R}^2)$ and $D \geq 1$. Then we have the following estimate:
\[
\iint_{|y_1| \leq D} w^2 \leq C D^2 \left( \iint w^2 e^{-|y_1|} + \iint |\nabla w|^{2} \right).
\]
\end{lemma}
\begin{proof}
 Suppose first that \(v\in C_c^\infty(\mathbb R)\). Averaging over
\([0,1]\) gives a point \(y_0\in[0,1]\) such that
\( |v(y_0)|^2\leq3\int v^2(y)e^{-|y|}\,dy\).
The fundamental theorem of calculus and Cauchy--Schwarz yield

\[
|v(y)|^2\leq2|v(y_0)|^2
+2|y-y_0|\int_{\min\{y,y_0\}}^{\max\{y,y_0\}}|v_y(x)|^2\,dx.
\]
Since \(D\geq1\) and \(y_0\in[0,1]\), integration over
\([-D,D]\) gives

\[
\int_{|y|\leq D}v^2(y)\,dy
\leq CD^2\left(\int v^2e^{-|y|}+\int v_y^2\right).
\]
Density extends this estimate to \(H^1(\mathbb R)\). Apply it to
\(w(\cdot,y_2)\) for almost every \(y_2\), and integrate in \(y_2\),
to obtain the lemma.
\end{proof}

 The weighted mass term in the lemma is bounded by
\(C\mathcal N_1(t)\). In the present zero-energy
asymptotic-stability regime, \eqref{sharpboundforgradient},
 \eqref{bdominatedbyN}, and the hierarchy of the \(\mathcal N_i\) give
\[
\|\nabla\varepsilon(t)\|_2^2
\lesssim |b(t)|+\mathcal N_2(t)
\lesssim\mathcal N_3(t).
\]
Consequently, for every \(D\geq1\),
\begin{equation}\label{eq:MerleLemma}
\iint_{|y_1| \leq D} \varepsilon^2(t) \leq C D^2 \mathcal{N}_3(t),
\end{equation}
where \(C\) is independent of \(D\) and \(t\).

\begin{lemma}\label{eq:Improvement}
Let $\varrho\in C_c^4(\mathbb R)$ be nonnegative, with
$\operatorname{supp}\varrho\subset[-1,1]$, and define
\[
    \varrho_D(y_1):=\varrho\left(\frac{y_1}{D}\right),
    \qquad D\geq1.
\]
There exist $T_0>0$ and $C>0$, independent of $D\geq1$, such that, for every
$t_0\geq T_0$,
\[
\begin{aligned}
\iint \varepsilon^2(t_0,y_1,y_2)\varrho_D(y_1)\,dy_1dy_2
\leq{}& C\mathcal N_4(t_0)
+C\min\left(\mathcal N_2^{1/2}(t_0),D\mathcal N_4(t_0)\right)\\
&+C\int_{t_0}^{\infty}
\left|
\iint \varepsilon^2(t,y_1,y_2)(\varrho_D)_{y_1}(y_1)\,dy_1dy_2
\right|dt.
\end{aligned}
\]
The constant $C$ may depend on $\varrho$, but not on $D$ or $t_0$.
\end{lemma}

\begin{proof}
Throughout the proof, all spatial integrals are over $\mathbb R^2$. In the
asymptotic-stability regime, \eqref{lambdacloseto1} gives $\lambda\sim1$;
therefore $dt=\lambda^3ds\sim ds$, and we use \eqref{boundiplusone} in either
time variable without further comment.

Recall that
\[
\begin{aligned}
\varepsilon_t(t,y_1,y_2)
={}&-\frac1{\lambda^3}\partial_{y_1}\Delta\varepsilon
-\mathfrak H-b_t\zeta_bP
-\frac1{\lambda^3}f(\varepsilon)_{y_1}
+\frac{\lambda_t}{\lambda}\Lambda\varepsilon\\
&+\frac{(x_1)_t}{\lambda}\varepsilon_{y_1}
+\frac{(x_2)_t}{\lambda}\varepsilon_{y_2},
\end{aligned}
\]
where
\[
    f(\varepsilon)=(Q_b+\varepsilon)^3-Q_b^3,
    \qquad
    \zeta_b=\chi_b+\gamma y_1(\chi_b)_{y_1},
\]
and
\[
\mathfrak H(t,y_1,y_2)
=\frac1{\lambda^3}
\Big[
-\Psi_b-(\lambda^2\lambda_t+b)\Lambda Q_b
-(\lambda^2(x_1)_t-1)(Q_b)_{y_1}
-\lambda^2(x_2)_t(Q_b)_{y_2}
\Big].
\]
Set
\[
    M_D(t):=\iint\varepsilon^2(t,y_1,y_2)\varrho_D(y_1)\,dy_1dy_2.
\]
A direct integration by parts gives
\begin{align*}
\frac12M_D'(t)
={}&-\frac{3}{2\lambda^3}\iint\varepsilon_{y_1}^2(\varrho_D)_{y_1}
-\frac{1}{2\lambda^3}\iint\varepsilon_{y_2}^2(\varrho_D)_{y_1}
+\frac{1}{2\lambda^3}\iint\varepsilon^2(\varrho_D)_{y_1y_1y_1}\\
&-\frac1{\lambda^3}\iint f(\varepsilon)_{y_1}\varepsilon\varrho_D
-\iint\mathfrak H\varepsilon\varrho_D
-b_t\iint\zeta_bP\varepsilon\varrho_D\\
&+\frac{\lambda_t}{\lambda}\iint\varepsilon\Lambda\varepsilon\varrho_D
+\frac{(x_1)_t}{\lambda}\iint\varepsilon\varepsilon_{y_1}\varrho_D
+\frac{(x_2)_t}{\lambda}\iint\varepsilon\varepsilon_{y_2}\varrho_D.
\end{align*}
Since $\varrho_D$ depends only on $y_1$,
\[
    \iint\varepsilon\varepsilon_{y_2}\varrho_D=0,
    \qquad
    \iint\varepsilon\Lambda\varepsilon\varrho_D
    =-\frac12\iint\varepsilon^2y_1(\varrho_D)_{y_1},
\]
and
\[
    \iint\varepsilon\varepsilon_{y_1}\varrho_D
    =-\frac12\iint\varepsilon^2(\varrho_D)_{y_1}.
\]
Using $\lambda\sim1$ and the boundedness of the derivatives of $\varrho$, we
therefore obtain
\begin{align}
\frac12M_D'(t)
\geq{}&-C\|\nabla\varepsilon(t)\|_{L^2}^2
-C\iint\varepsilon^2|(\varrho_D)_{y_1y_1y_1}|
-C\left|\iint f(\varepsilon)_{y_1}\varepsilon\varrho_D\right|
-C\left|\iint\mathfrak H\varepsilon\varrho_D\right|
\notag\\
&-b_t\iint\zeta_bP\varepsilon\varrho_D
-\frac12\frac{\lambda_t}{\lambda}
 \iint\varepsilon^2y_1(\varrho_D)_{y_1}
-C\left|\iint\varepsilon^2(\varrho_D)_{y_1}\right|.
\label{eq:localized-mass-lower-bound}
\end{align}

\paragraph{Step 1: The third-derivative term.}
By \eqref{eq:MerleLemma},
\[
\begin{aligned}
\iint\varepsilon^2|(\varrho_D)_{y_1y_1y_1}|
&=\frac1{D^3}\iint\varepsilon^2
\left|\varrho_{y_1y_1y_1}\left(\frac{y_1}{D}\right)\right|\\
&\lesssim \frac1{D^3}\iint_{|y_1|\leq D}\varepsilon^2
\lesssim \frac1D\mathcal N_3(t).
\end{aligned}
\]

\paragraph{Step 2: The nonlinear term.}
Since
\[
    f(\varepsilon)=3Q_b^2\varepsilon+3Q_b\varepsilon^2+\varepsilon^3,
\]
an integration by parts yields the exact identity
\begin{align*}
\iint f(\varepsilon)_{y_1}\varepsilon\varrho_D
={}&3\iint Q_b(Q_b)_{y_1}\varepsilon^2\varrho_D
-\frac32\iint Q_b^2\varepsilon^2(\varrho_D)_{y_1}\\
&+\iint (Q_b)_{y_1}\varepsilon^3\varrho_D
-2\iint Q_b\varepsilon^3(\varrho_D)_{y_1}
-\frac34\iint\varepsilon^4(\varrho_D)_{y_1}.
\end{align*}
The pointwise estimates for $Q_b$, the two-dimensional
Gagliardo--Nirenberg inequality, and the smallness of
$\|\varepsilon\|_{L^2}$ give
\begin{align*}
&\left|\iint Q_b(Q_b)_{y_1}\varepsilon^2\right|
+\iint Q_b^2\varepsilon^2
\lesssim \mathcal N_1+b^2,\\
&\left|\iint (Q_b)_{y_1}\varepsilon^3\right|
+\left|\iint Q_b\varepsilon^3\right|
\lesssim \mathcal N_1
+\bigl(|b|+|b|^{1+\gamma}\bigr)\|\nabla\varepsilon\|_{L^2}
+|b|\mathcal N_1,\\
&\iint\varepsilon^4
\lesssim \|\varepsilon\|_{L^2}^2\|\nabla\varepsilon\|_{L^2}^2
\lesssim \|\nabla\varepsilon\|_{L^2}^2.
\end{align*}
Consequently, by \eqref{sharpboundforgradient}, \eqref{bdominatedbyN}, and
$\mathcal N_1\lesssim\mathcal N_3$,
\[
    \left|\iint f(\varepsilon)_{y_1}\varepsilon\varrho_D\right|
    \lesssim \mathcal N_1+b^2+\|\nabla\varepsilon\|_{L^2}^2
    \lesssim \mathcal N_3(t).
\]

\paragraph{Step 3: The term involving $\mathfrak H$.}
Lemma~\ref{lemma:Psi_b_estimates} and the decay of $Q$ and $P$ give
\begin{align*}
&\left|\iint\Lambda Q_b\,\varepsilon\varrho_D\right|
+\left|\iint(Q_b)_{y_2}\varepsilon\varrho_D\right|
\lesssim \mathcal N_1^{1/2}+|b|^{1-\gamma/2},\\
&\left|\iint(Q_b)_{y_1}\varepsilon\varrho_D\right|
\lesssim \mathcal N_1^{1/2}+|b|^{1+\gamma/2},
\qquad
\left|\iint\Psi_b\varepsilon\varrho_D\right|
\lesssim |b|^{1+\gamma/2}.
\end{align*}
Using \eqref{eq:modulatedcoefficients}, $ds/dt=\lambda^{-3}$,
$\lambda\sim1$, and \eqref{bdominatedbyN}, we obtain
\[
\begin{aligned}
\left|\iint\mathfrak H\varepsilon\varrho_D\right|
&\lesssim |b|^{1+\gamma/2}
+(b^2+\mathcal N_2^{1/2})
 (\mathcal N_2^{1/2}+|b|^{1-\gamma/2})\\
&\lesssim \mathcal N_2+|b|^{2-\gamma}+|b|^{1+\gamma/2}
\lesssim \mathcal N_3(t).
\end{aligned}
\]

\paragraph{Step 4: The longitudinal transport term.}
Since $\lambda(t)\to\lambda_\infty>0$ and
$(x_1)_t(t)\to\lambda_\infty^{-2}$, the coefficient
$|(x_1)_t/\lambda|$ is bounded for $t\geq T_0$. Hence
\[
    \left|\frac{(x_1)_t}{\lambda}
    \iint\varepsilon\varepsilon_{y_1}\varrho_D\right|
    \lesssim
    \left|\iint\varepsilon^2(\varrho_D)_{y_1}\right|.
\]

\paragraph{Step 5: Integration in time.}
Integrating \eqref{eq:localized-mass-lower-bound} over
$[t_0,+\infty)$, and using \eqref{boundforgradient},
\eqref{boundiplusone}, Steps 1--4, and $\lambda\sim1$, gives
\begin{align*}
\lim_{t\to\infty}M_D(t)-M_D(t_0)
\geq{}&-C\mathcal N_4(t_0)
-C\left|\int_{t_0}^{\infty}b_t
 \iint\zeta_bP\varepsilon\varrho_D\,dt\right|\\
&-C\left|\int_{t_0}^{\infty}\frac{\lambda_t}{\lambda}
 \iint\varepsilon^2y_1(\varrho_D)_{y_1}\,dt\right|
-C\int_{t_0}^{\infty}
 \left|\iint\varepsilon^2(\varrho_D)_{y_1}\right|dt.
\end{align*}
Indeed, the terms bounded pointwise by $\mathcal N_3$ are bounded by
$\mathcal N_4(t_0)$ after time integration, while the Step 1 term is at
most $D^{-1}\mathcal N_4(t_0)\lesssim\mathcal N_4(t_0)$.
Moreover, \eqref{eq:MerleLemma}, the hierarchy
$\mathcal N_3\lesssim\mathcal N_5$, and \eqref{N6decay} imply
\[
    |M_D(t)|\lesssim D^2\mathcal N_3(t)\longrightarrow0
    \qquad\text{as }t\to+\infty.
\]

Here \(D\) is fixed when taking the time limit. Therefore,
\begin{align}
M_D(t_0)
\lesssim{}&\mathcal N_4(t_0)
+\left|\int_{t_0}^{\infty}b_t
 \iint\zeta_bP\varepsilon\varrho_D\,dt\right|
+\left|\int_{t_0}^{\infty}\frac{\lambda_t}{\lambda}
 \iint\varepsilon^2y_1(\varrho_D)_{y_1}\,dt\right|
\notag\\
&+\int_{t_0}^{\infty}
 \left|\iint\varepsilon^2(\varrho_D)_{y_1}\right|dt.
\label{eq:localized-mass-after-time-integration}
\end{align}

\paragraph{Step 6: The $b_t$-term.}
Define
\[
    f_D(t):=\iint\chi_b(y_1)\varrho_D(y_1)
    P(y_1,y_2)\varepsilon(t,y_1,y_2)\,dy_1dy_2.
\]
By Cauchy--Schwarz,
\begin{equation}\label{boundforfD}
\begin{aligned}
|f_D(t)|
&\leq
\left(\iint\varepsilon^2(t,y_1,y_2)\varrho_D(y_1)\,dy_1dy_2\right)^{1/2}
\left(\iint\varrho_D\chi_b^2P^2\,dy_1dy_2\right)^{1/2}\\
&\lesssim \|\varepsilon(t)\|_{L^2}
\min\left(D,|b(t)|^{-\gamma}\right)^{1/2}
\leq\delta(\alpha_0)D^{1/2}.
\end{aligned}
\end{equation}
Furthermore,
\begin{equation}\label{decayofbfD}
    |b(t)f_D(t)|
    \lesssim
    \left(\iint\varepsilon^2(t,y_1,y_2)\varrho_D(y_1)\,dy_1dy_2\right)^{1/2}
    |b(t)|^{1-\gamma/2}
    \longrightarrow0
\end{equation}
as $t\to+\infty$.

Set
\[
    \widetilde\varrho(t,y_1,y_2)
    :=\varrho_D(y_1)\chi_b(y_1)P(y_1,y_2),
    \qquad
    h_D(t):=\iint\widetilde\varrho\,\varepsilon_t\,dy_1dy_2.
\]
The regularity of the modulation parameters and the equation for $\varepsilon$ imply that $bf_D$ is locally absolutely continuous. Since $\partial_t(b\chi_bP)=b_t\zeta_bP$,
\begin{equation}\label{bfD}
    (bf_D)_t
    =b_t\iint\zeta_bP\varepsilon\varrho_D\,dy_1dy_2
    +b\iint\chi_bP\varepsilon_t\varrho_D\,dy_1dy_2.
\end{equation}

\begin{lemma}\label{lem:uniform-cut-off-pairing}
Let $\theta\in C_c^3(\mathbb R)$ satisfy
$\operatorname{supp}\theta\subset[-1,1]$, and define
\[
    \theta_D(y_1):=\theta\left(\frac{y_1}{D}\right),
    \qquad
    \Theta_D(t,y_1,y_2)
    :=\theta_D(y_1)\chi_b(y_1)P(y_1,y_2).
\]
Then, for all
$t\geq T_0$,
\begin{equation*}
    \left|
    b(t)\iint\Theta_D(t,y_1,y_2)
    \varepsilon_t(t,y_1,y_2)\,dy_1dy_2
    \right|
    \lesssim\mathcal N_3(t),
\end{equation*}
where the implicit constant is independent of $D$ and $t$.
\end{lemma}

\begin{proof}
Set
\[
    m_D(t):=\min\bigl(D,|b(t)|^{-\gamma}\bigr),
\]
with the convention $m_D(t)=D$ when $b(t)=0$.

The decay of $P$ and the support properties of $\chi_b$ and $\theta_D$ give
\begin{equation}\label{eq:Theta-cut-off-bounds}
\begin{aligned}
&\|(\Theta_D)_{y_1}\|_{L^2}
+\|(\Theta_D)_{y_1y_1y_1}\|_{L^2}
+\|(\Theta_D)_{y_1y_2y_2}\|_{L^2}\lesssim1,\\
&\|\Theta_D\|_{L^2}+\|(\Theta_D)_{y_2}\|_{L^2}
+\|\Theta_D+y_1(\Theta_D)_{y_1}+y_2(\Theta_D)_{y_2}\|_{L^2}
\lesssim m_D^{1/2},\\
&\iint|\Theta_D\zeta_bP|\lesssim m_D.
\end{aligned}
\end{equation}
Indeed, a $y_1$-derivative falling on $\theta_D$ produces a factor
$D^{-1}$, while one falling on $\chi_b$ produces a factor $|b|^\gamma$.
All $y_1$-derivatives of $P$ are exponentially localized; the only
nonlocalized contribution is the undifferentiated left-hand tail of $P$,
whose effective length is bounded by $m_D$.

Using the equation for $\varepsilon_t$ and integrating by parts in
$(y_1,y_2)$, we obtain
\begin{align}
\iint\Theta_D\varepsilon_t
={}&\frac1{\lambda^3}\iint
\bigl[(\Theta_D)_{y_1y_1y_1}+(\Theta_D)_{y_1y_2y_2}\bigr]\varepsilon
+\frac1{\lambda^3}\iint(\Theta_D)_{y_1}f(\varepsilon)
\notag\\
&-\iint\Theta_D\mathfrak H
-b_t\iint\Theta_D\zeta_bP
-\frac{(x_1)_t}{\lambda}\iint(\Theta_D)_{y_1}\varepsilon
-\frac{(x_2)_t}{\lambda}\iint(\Theta_D)_{y_2}\varepsilon
\notag\\
&+\frac{\lambda_t}{\lambda}\iint\Theta_D\Lambda\varepsilon,
\label{eq:cut-off-pairing-identity}
\end{align}
where
\[
    \iint\Theta_D\Lambda\varepsilon
    =-\iint
    \bigl[\Theta_D+y_1(\Theta_D)_{y_1}+y_2(\Theta_D)_{y_2}\bigr]\varepsilon.
\]
Moreover,
\[
\begin{aligned}
\left|\iint(\Theta_D)_{y_1}f(\varepsilon)\right|
&\lesssim\|(\Theta_D)_{y_1}\|_{L^2}
 \bigl(\|Q_b^2\varepsilon\|_{L^2}+\|\varepsilon^3\|_{L^2}\bigr)
 +\|(\Theta_D)_{y_1}Q_b\|_{L^\infty}\|\varepsilon\|_{L^2}^2\\
&\lesssim1+\|\varepsilon\|_{L^2}\|\nabla\varepsilon\|_{L^2}^2
\lesssim1.
\end{aligned}
\]
By Lemma~\ref{lemma:Psi_b_estimates}, \eqref{eq:modulatedcoefficients},
and $\lambda\sim1$,
\[
\|\mathfrak H\|_{L^2}\lesssim|b|^{1+\gamma/2}+b^2+\mathcal N_2^{1/2},
\quad |b_t|\lesssim b^2+\mathcal N_2,
\quad
\left|\frac{\lambda_t}{\lambda}\right|
+\left|\frac{(x_2)_t}{\lambda}\right|
\lesssim|b|+\mathcal N_2^{1/2},
\quad
\left|\frac{(x_1)_t}{\lambda}\right|\lesssim1.
\]
Thus \eqref{eq:cut-off-pairing-identity} and
\eqref{eq:Theta-cut-off-bounds} imply
\[
\begin{aligned}
\left|\iint\Theta_D\varepsilon_t\right|
\lesssim{}&1
+m_D^{1/2}\bigl(|b|^{1+\gamma/2}+b^2+\mathcal N_2^{1/2}\bigr)
+(b^2+\mathcal N_2)m_D\\
&+(|b|+\mathcal N_2^{1/2})m_D^{1/2}.
\end{aligned}
\]
For $b\neq0$, $m_D\leq|b|^{-\gamma}$, and hence
\[
    \left|\iint\Theta_D\varepsilon_t\right|
    \lesssim1+(\mathcal N_2m_D)^{1/2}+\mathcal N_2m_D
    \lesssim1+\mathcal N_2m_D.
\]
Multiplying by $|b|$ and using \eqref{bdominatedbyN} and
$\mathcal N_2\lesssim\mathcal N_3$, we obtain
\[
    \left|b\iint\Theta_D\varepsilon_t\right|
    \lesssim |b|+|b|^{1-\gamma}\mathcal N_2
    \lesssim\mathcal N_3.
\]
If $b(t)=0$, the conclusion is immediate.
\end{proof}

Applying Lemma~\ref{lem:uniform-cut-off-pairing} with $\theta=\varrho$ gives
\begin{equation}\label{bfDestimate}
    |b(t)h_D(t)|\lesssim\mathcal N_3(t).
\end{equation}
Integrating \eqref{bfD} over $[t_0,R]$, letting $R\to+\infty$, and using
\eqref{decayofbfD}, \eqref{bfDestimate}, and \eqref{boundiplusone}, yields
\[
\begin{aligned}
\left|\int_{t_0}^{\infty}b_t
\iint\zeta_bP\varepsilon\varrho_D\,dy_1dy_2dt\right|
&\lesssim |b(t_0)f_D(t_0)|+\int_{t_0}^{\infty}\mathcal N_3(t)\,dt\\
&\lesssim |b(t_0)|^{1-\gamma/2}M_D(t_0)^{1/2}
+\mathcal N_4(t_0).
\end{aligned}
\]

\paragraph{Step 7: The scaling term.}
We claim that
\begin{equation}
\begin{aligned}
&\left|\int_{t_0}^{\infty}\frac{\lambda_t}{\lambda}
\iint\varepsilon^2(t,y_1,y_2)y_1(\varrho_D)_{y_1}(y_1)
\,dy_1dy_2dt\right|\\
&\qquad\lesssim
\mathcal N_4(t_0)
+\min\left(\mathcal N_2^{1/2}(t_0),D\mathcal N_4(t_0)\right).
\end{aligned}
\end{equation}
We prove two complementary bounds.

\emph{First estimate.}
Let $C_\varrho:=\|r\varrho'(r)\|_{L^\infty}$. By
\eqref{eq:MerleLemma}, \eqref{bdominatedbyN}, and
\eqref{sharpboundforgradient},
\[
\begin{aligned}
\left|\frac{\lambda_t}{\lambda}
\iint\varepsilon^2y_1(\varrho_D)_{y_1}\right|
&\lesssim (|b|+\mathcal N_1^{1/2})C_\varrho
\left(\iint_{|y_1|\leq D}\varepsilon^2\right)^{1/2}
\|\varepsilon\|_{L^2}\\
&\lesssim D(|b|+\mathcal N_1^{1/2})\mathcal N_3^{1/2}
\lesssim D\mathcal N_3(t).
\end{aligned}
\]
The last inequality uses \(\mathcal N_1\lesssim\mathcal N_3\),
\(|b|\lesssim\mathcal N_3\), and the uniform smallness of
\(\mathcal N_3\).
 Therefore, by \eqref{boundiplusone},
\begin{equation}
\left|\int_{t_0}^{\infty}\frac{\lambda_t}{\lambda}
\iint\varepsilon^2y_1(\varrho_D)_{y_1}\,dt\right|
\lesssim D\mathcal N_4(t_0).
\end{equation}

\emph{Second estimate.}
Set
\[
    \iota(r):=r\varrho'(r),
    \qquad
    \iota_D(y_1):=\iota\left(\frac{y_1}{D}\right).
\]
Then $\iota\in C_c^3(\mathbb R)$ and
$y_1(\varrho_D)_{y_1}=\iota_D$. The sharp scaling law in
Lemma~\ref{sharporthogonalities} reads
\begin{equation}\label{sharpJlambdaeq}
    \frac{\lambda_t}{\lambda}
    =-\frac{b}{\lambda^3}+(J_1)_t
    +\frac{\lambda_t}{\lambda}J_1
    +O\left(\frac{b^2+\mathcal N_2}{\lambda^3}\right),
\end{equation}
where
\[
    J_1(t)=\frac1{2C_Q}
    \left(\varepsilon(t),\int_{-\infty}^{y_1}\Lambda Q\right),
    \qquad |J_1(t)|\lesssim\mathcal N_2^{1/2}(t).
\]
Substitution into the scaling term gives
\begin{align*}
\int_{t_0}^{\infty}\frac{\lambda_t}{\lambda}
\iint\varepsilon^2\iota_D
={}&-\int_{t_0}^{\infty}\frac{b}{\lambda^3}
 \iint\varepsilon^2\iota_D
+\int_{t_0}^{\infty}\frac{\lambda_t}{\lambda}J_1
 \iint\varepsilon^2\iota_D\\
&+\int_{t_0}^{\infty}(J_1)_t\iint\varepsilon^2\iota_D
+\int_{t_0}^{\infty}O\left(\frac{b^2+\mathcal N_2}{\lambda^3}\right)
 \iint\varepsilon^2\iota_D.
\end{align*}
Using \eqref{bdominatedbyN}, \eqref{boundiplusone},
\eqref{eq:controldynamics}, and $\lambda\sim1$, the three terms not
containing $(J_1)_t$ satisfy
\begin{align*}
&\left|\int_{t_0}^{\infty}\frac{b}{\lambda^3}
 \iint\varepsilon^2\iota_D\right|
\lesssim\mathcal N_4(t_0),\\
&\left|\int_{t_0}^{\infty}\frac{\lambda_t}{\lambda}J_1
 \iint\varepsilon^2\iota_D\right|
+\left|\int_{t_0}^{\infty}O\left(\frac{b^2+\mathcal N_2}{\lambda^3}\right)
 \iint\varepsilon^2\iota_D\right|
\lesssim\mathcal N_3(t_0).
\end{align*}

For the remaining term, we integrate by parts on $[t_0,R]$ and then let $R\to+\infty$, obtaining
\begin{align*}
\int_{t_0}^{\infty}(J_1)_t\iint\varepsilon^2\iota_D
={}&-J_1(t_0)\iint\varepsilon^2(t_0,y_1,y_2)\iota_D(y_1)\,dy_1dy_2\\
&-\int_{t_0}^{\infty}J_1(t)
\frac d{dt}\left(\iint\varepsilon^2(t,y_1,y_2)\iota_D(y_1)\,dy_1dy_2\right)dt.
\end{align*}
The boundary term at $+\infty$ vanishes because
$\mathcal N_2(t)\to0$, and the boundary term at $t_0$ is bounded by
$C\mathcal N_2^{1/2}(t_0)$.

The localized-mass identity with $\iota_D$ in place of $\varrho_D$ is
\begin{align*}
\frac12\frac d{dt}\iint\varepsilon^2\iota_D
={}&-\frac{3}{2\lambda^3}\iint\varepsilon_{y_1}^2(\iota_D)_{y_1}
-\frac{1}{2\lambda^3}\iint\varepsilon_{y_2}^2(\iota_D)_{y_1}
+\frac{1}{2\lambda^3}\iint\varepsilon^2(\iota_D)_{y_1y_1y_1}\\
&-\frac1{\lambda^3}\iint f(\varepsilon)_{y_1}\varepsilon\iota_D
-\iint\mathfrak H\varepsilon\iota_D
-b_t\iint\zeta_bP\varepsilon\iota_D\\
&-\frac12\frac{\lambda_t}{\lambda}
 \iint\varepsilon^2y_1(\iota_D)_{y_1}
+\frac{(x_1)_t}{\lambda}\iint\varepsilon\varepsilon_{y_1}\iota_D.
\end{align*}
Using part (a) of Lemma~\ref{sharporthogonalities},
\eqref{sharpboundforgradient}, \eqref{eq:MerleLemma},
\eqref{eq:modulatedcoefficients}, Steps 1--4, and Cauchy--Schwarz, all
terms except the $b_t$-term obey
\begin{align*}
&|J_1|\left(
\left|\iint\varepsilon_{y_1}^2(\iota_D)_{y_1}\right|
+\left|\iint\varepsilon_{y_2}^2(\iota_D)_{y_1}\right|
+\left|\iint\varepsilon^2(\iota_D)_{y_1y_1y_1}\right|
\right)\lesssim\mathcal N_3,\\
&|J_1|\left(
\left|\iint f(\varepsilon)_{y_1}\varepsilon\iota_D\right|
+\left|\iint\mathfrak H\varepsilon\iota_D\right|
+\left|\frac{\lambda_t}{\lambda}\iint\varepsilon^2y_1(\iota_D)_{y_1}\right|
+\left|\frac{(x_1)_t}{\lambda}\iint\varepsilon\varepsilon_{y_1}\iota_D\right|
\right)\lesssim\mathcal N_3.
\end{align*}
In particular, the third-derivative contribution is bounded by
\[
|J_1|D^{-3}\iint_{|y_1|\leq D}\varepsilon^2
\lesssim |J_1|D^{-1}\mathcal N_3
\lesssim\mathcal N_3.
\]
For the scaling contribution in the second line, use the global
\(L^2\) bound and
\( |J_1|\,|\lambda_t/\lambda|
\lesssim\mathcal N_2^{1/2}(|b|+\mathcal N_1^{1/2})
\lesssim\mathcal N_3\).

It remains to estimate
\[
    \int_{t_0}^{\infty}J_1(t)b_t(t)
    \iint\zeta_bP\varepsilon\iota_D\,dy_1dy_2dt.
\]
Define
\[
    g_D(t):=b(t)\iint\chi_b(y_1)\iota_D(y_1)
    P(y_1,y_2)\varepsilon(t,y_1,y_2)\,dy_1dy_2.
\]
Then $g_D$ is locally absolutely continuous and
\[
    (g_D)_t
    =b_t\iint\zeta_bP\varepsilon\iota_D
    +b\iint\chi_bP\varepsilon_t\iota_D,
\]
and, by the Cauchy--Schwarz estimate used in \eqref{boundforfD},
\[
    |g_D(t)|\lesssim |b(t)|m_D(t)^{1/2}
    \lesssim |b(t)|^{1-\gamma/2}.
\]
A second integration by parts on $[t_0,R]$, followed by $R\to+\infty$, gives
\begin{align*}
\int_{t_0}^{\infty}J_1b_t\iint\zeta_bP\varepsilon\iota_D
={}&-J_1(t_0)g_D(t_0)
-\int_{t_0}^{\infty}(J_1)_t g_D\,dt\\
&-\int_{t_0}^{\infty}J_1b
 \iint\chi_bP\varepsilon_t\iota_D\,dy_1dy_2dt.
\end{align*}
Here $J_1(t)g_D(t)\to0$ as $t\to+\infty$. By
\eqref{sharpJlambdaeq},
\[
    (J_1)_t
    =\frac{\lambda_t}{\lambda}+\frac b{\lambda^3}
    -\frac{\lambda_t}{\lambda}J_1
    +O\left(\frac{b^2+\mathcal N_2}{\lambda^3}\right).
\]
Thus, using part (a) of Lemma~\ref{sharporthogonalities},
\eqref{eq:modulatedcoefficients}, \eqref{bdominatedbyN}, and
\eqref{boundiplusone},
\[
    |J_1(t_0)g_D(t_0)|\lesssim\mathcal N_3(t_0),
    \qquad
    \int_{t_0}^{\infty}|(J_1)_t g_D|\,dt
    \lesssim\mathcal N_4(t_0).
\]
Finally, Lemma~\ref{lem:uniform-cut-off-pairing}, applied with
$\theta=\iota$, gives
\[
    \left|b(t)\iint\chi_bP\varepsilon_t\iota_D\,dy_1dy_2\right|
    \lesssim\mathcal N_3(t),
\]
so the last integral is also bounded by $C\mathcal N_4(t_0)$. We have
proved
\[
    \left|\int_{t_0}^{\infty}(J_1)_t
    \iint\varepsilon^2\iota_D\,dt\right|
    \lesssim\mathcal N_2^{1/2}(t_0)+\mathcal N_4(t_0),
\]
and hence
\[
    \left|\int_{t_0}^{\infty}\frac{\lambda_t}{\lambda}
    \iint\varepsilon^2\iota_D\,dt\right|
    \lesssim\mathcal N_2^{1/2}(t_0)+\mathcal N_4(t_0).
\]
Combining the first and second estimates gives
\[
\begin{aligned}
&\left|\int_{t_0}^{\infty}\frac{\lambda_t}{\lambda}
\iint\varepsilon^2y_1(\varrho_D)_{y_1}\,dt\right|\\
&\qquad\lesssim
\min\left(D\mathcal N_4(t_0),
\mathcal N_2^{1/2}(t_0)+\mathcal N_4(t_0)\right)\\
&\qquad\lesssim
\mathcal N_4(t_0)
+\min\left(\mathcal N_2^{1/2}(t_0),D\mathcal N_4(t_0)\right).
\end{aligned}
\]

\paragraph{Step 8: Closing the estimate.}
Substituting Steps 6 and 7 into
\eqref{eq:localized-mass-after-time-integration}, we obtain
\begin{align*}
\frac12M_D(t_0)
\leq{}&C\mathcal N_4(t_0)
+C|b(t_0)|^{1-\gamma/2}M_D(t_0)^{1/2}
+C\min\left(\mathcal N_2^{1/2}(t_0),D\mathcal N_4(t_0)\right)\\
&+C\int_{t_0}^{\infty}
\left|\iint\varepsilon^2(t,y_1,y_2)(\varrho_D)_{y_1}(y_1)\,dy_1dy_2\right|dt.
\end{align*}
Young's inequality gives
\[
    C|b(t_0)|^{1-\gamma/2}M_D(t_0)^{1/2}
    \leq\frac14M_D(t_0)+C|b(t_0)|^{2-\gamma}.
\]
By \eqref{bdominatedbyN}, the hierarchy of the $\mathcal N_i$, and the
smallness of the solution,
$|b(t_0)|^{2-\gamma}\lesssim\mathcal N_4(t_0)$. Absorbing the first term
on the right and renaming $t_0$ as $t$ proves the lemma.
\end{proof}

We now apply Lemma~\ref{eq:Improvement} to prove global $L^2$ decay of the
error.

\begin{lemma}\label{masslemma}
We have
\[
    \|\varepsilon(t)\|_{L^2_{y_1y_2}}\longrightarrow0
    \qquad\text{as }t\to+\infty.
\]
\end{lemma}

\begin{proof}
Choose $\rho\in C_c^\infty(\mathbb R)$ such that
$\operatorname{supp}\rho\subset[-1,1]$, $0\leq\rho\leq1$,
$\rho=1$ on $[-1/2,1/2]$, and $\rho$ is monotone on each transition
interval. We may also arrange that $|\rho'|\in C_c^\infty(\mathbb R)$.
For example, let
\[
\varsigma(r)=
\begin{cases}
    e^{-1/r},&r>0,\\
    0,&r\leq0,
\end{cases}
\qquad
\Xi(r)=\frac{\varsigma(r)}{\varsigma(r)+\varsigma(1-r)},
\]
and set
\[
    \rho(y_1)=\Xi(2(y_1+1))\Xi(2(1-y_1)).
\]
Then $\rho=1$ on $[-1/2,1/2]$, all derivatives vanish at the endpoints
of the transition intervals, and $|\rho'|$ is smooth and compactly
supported.

Apply Lemma~\ref{eq:Improvement} with $\varrho=|\rho'|$ and multiply the
result by $D^{-1}$. Since
$|(\rho_D)_{y_1}|=D^{-1}|\rho'(y_1/D)|$, we obtain
\begin{align*}
\iint\varepsilon^2(t,y_1,y_2)|(\rho_D)_{y_1}(y_1)|\,dy_1dy_2
\lesssim{}&\frac1D\mathcal N_4(t)
+\frac1D\min\left(\mathcal N_2^{1/2}(t),D\mathcal N_4(t)\right)\\
&+\frac1{D^2}\int_t^{\infty}
\left|\iint\varepsilon^2(t',y_1,y_2)
\bigl(|\rho'|\bigr)'\left(\frac{y_1}{D}\right)\,dy_1dy_2\right|dt'.
\end{align*}
Since  \(\bigl|(|\rho'|)'\bigr|\leq|\rho''|\) and
\(\operatorname{supp}\rho''\subset[-1,1]\),
\eqref{eq:MerleLemma} and \eqref{boundiplusone} bound the last term by
 
\[
\frac{C}{D^2}\int_t^{\infty}
   \iint_{|y_1|\leq D}\varepsilon^2(t')\,dt'
\lesssim\int_t^{\infty}\mathcal N_3(t')\,dt'
\lesssim\mathcal N_4(t).
\]
Also,
\(D^{-1}\min\{\mathcal N_2^{1/2}(t),D\mathcal N_4(t)\}
\leq\mathcal N_4(t)\). Hence
 \begin{equation}
    \iint\varepsilon^2(t,y_1,y_2)|(\rho_D)_{y_1}(y_1)|\,dy_1dy_2
    \lesssim\mathcal N_4(t),
    \qquad D\geq1.
\end{equation}

Applying Lemma~\ref{eq:Improvement} once more, now with
$\varrho=\rho$, and using the preceding estimate, we find
\begin{align*}
\iint_{-D/2\leq y_1\leq D/2}\varepsilon^2(t,y_1,y_2)\,dy_1dy_2
&\leq2\iint\varepsilon^2(t,y_1,y_2)\rho_D(y_1)\,dy_1dy_2\\
&\lesssim\mathcal N_4(t)+\mathcal N_2^{1/2}(t)
+\int_t^{\infty}\mathcal N_4(t')\,dt'\\
&\lesssim\mathcal N_5(t)+\mathcal N_2^{1/2}(t),
\end{align*}
where the last inequality follows from \eqref{boundiplusone} and the
hierarchy $\mathcal N_4\lesssim\mathcal N_5$. Letting $D\to+\infty$ and
using monotone convergence and \eqref{N6decay}, we conclude that
\[
    \|\varepsilon(t)\|_{L^2}^2
    \lesssim\mathcal N_5(t)+\mathcal N_2^{1/2}(t)
    \longrightarrow0
    \qquad\text{as }t\to+\infty.
\]
\end{proof}

Finally, using $(\varepsilon(t),Q)=0$ and mass conservation,
\begin{align*}
\iint u_0^2\,dx_1dx_2
&=\iint\bigl[Q_{b(t)}(y_1,y_2)+\varepsilon(t,y_1,y_2)\bigr]^2\,dy_1dy_2\\
&=\iint Q^2
+2b(t)\iint\chi_{b(t)}PQ
+2b(t)\iint\chi_{b(t)}P\varepsilon(t)
+\|\varepsilon(t)\|_{L^2}^2
+b^2(t)\iint\chi_{b(t)}^2P^2\\
&=\iint Q^2+O\bigl(|b(t)|^{1-\gamma/2}\bigr)
+\|\varepsilon(t)\|_{L^2}^2
\longrightarrow\iint Q^2
\end{align*}
as $t\to+\infty$. Thus $\|u_0\|_{L^2}=\|Q\|_{L^2}$. Since
$E_0=E(Q)=0$, the sharp Weinstein inequality, see \eqref{eq:varcharac},
implies that $u_0$ belongs to the solitary-wave orbit, a contradiction.
\begin{remark} 
More generally, if \(E_0=0\) for \(u_0 \in \mathcal{A}_{\alpha_0}\), then either the solution blows up or \(u_0\) belongs to
the solitary-wave orbit.
\end{remark}

\appendix
 \chapter{Weighted Inequalities}\label{appendixA} 
 \begin{lemma}[Weighted Sobolev Lemma]\label{lemmasobolev} 
Let $\theta \in W^{1,1}_{loc}(\mathbb{R}^2)$ be a positive function such that there exists $C=C_\theta>1$ with  $|\nabla \theta|\leq C\theta $. Suppose that $u \in L^2(\mathbb{R}^2)\cap H^1(\mathbb{R}^2;\theta dx_1dx_2),$ then 

$$\iint u^4 \theta dx_1dx_2  \leq 3C^2\|u\|_{L^2}^2 \iint \Big( u^2 +u_{x_1}^2+u_{x_2}^2)\theta dx_1dx_2,$$
$$\iint |u|^3 \theta dx_1dx_2 \leq \sqrt{3}C \|u\|_{L^2} \iint \Big( u^2 +u_{x_1}^2+u_{x_2}^2)\theta dx_1dx_2.$$
\end{lemma}

\begin{proof}
We first prove the estimate for smooth compactly supported functions and then pass to the general case by the usual truncation and approximation argument.
Set, by assumptions, the finite quantities
\[
\mathcal S=\|u\|_{L^2(\mathbb{R}^2)},\qquad
A_0^2=\iint u^2\theta,\qquad
A_1^2=\iint u_{x_1}^2\theta,\qquad
A_2^2=\iint u_{x_2}^2\theta.
\]
Since \(\theta>0\) and \(|\partial_{x_i}\theta|\leq C\theta\), we have
\[
\left|\partial_{x_i}(\sqrt\theta)\right|=\frac{|\partial_{x_i}\theta|}{2\sqrt\theta}
\leq \frac{C}{2}\sqrt\theta.
\]
Define
\[
F(x_1,x_2)=u^2(x_1,x_2)\sqrt{\theta(x_1,x_2)}.
\]
Then
\[
\|F\|_{L^1}
=\iint |u|\,|u|\sqrt\theta
\leq \|u\|_{L^2}\left(\iint u^2\theta\right)^{1/2}
=\mathcal SA_0.
\]
Moreover, for $i=1,2$
\[
\partial_{x_i}F
=2uu_{x_i}\sqrt\theta+u^2\partial_{x_i}(\sqrt\theta),
\]
so that
\[
\|\partial_{x_i}F\|_{L^1}
\leq
2\iint |u|\,|u_{x_i}|\sqrt\theta
+\frac C2\iint u^2\sqrt\theta
\leq
2\mathcal SA_i+\frac{C}{2}\mathcal SA_0.
\]
We use the elementary anisotropic Sobolev inequality
\[
\|F\|_{L^2(\mathbb{R}^2)}^2
\leq
\|\partial_{x_1}F\|_{L^1(\mathbb{R}^2)}
\|\partial_{x_2}F\|_{L^1(\mathbb{R}^2)}.
\]

Therefore,
\[
\iint u^4\theta
=\|F\|_{L^2}^2
\leq
\mathcal S^2\left(2A_1+\frac C2A_0\right)\left(2A_2+\frac C2A_0\right)\leq 3C^2\mathcal S^2(A_0^2+A_1^2+A_2^2).
\]

For the cubic estimate, H\"older's inequality gives
\[
\iint |u|^3\theta
\leq
\left(\iint u^2\theta\right)^{1/2}
\left(\iint u^4\theta\right)^{1/2}.
\]
Using the quartic estimate, we obtain
\[
\iint |u|^3\theta
\leq
A_0\sqrt{3}C\,\mathcal S(A_0^2+A_1^2+A_2^2)^{1/2}
\leq
\sqrt{3}C\,\mathcal S(A_0^2+A_1^2+A_2^2),
\]
which is the desired cubic estimate.

For general \(u\), one applies the previous argument to standard smooth compactly supported approximations of \(u\), or equivalently first to \(\chi_R u\), where \(\chi_R\) is a smooth cut-off, and then lets \(R\to\infty\).
\end{proof}

\begin{lemma}\label{lemmasobolev2}
There exists $C_B>0$, depending on $B$, such that for all
 $u$ for which the right-hand side is finite,
\begin{align}\label{eq:main-paper-notation}
&\iint_{\mathbb{R}^2} u^2 |\nabla u|^2 \psi_B'
  +\iint_{\mathbb{R}^2} u^4 \psi_B'
\notag\\
&\qquad\leq
C_B\norm{u}_{L^2(\mathbb{R}^2)}^2
\left[
\iint_{\mathbb{R}^2} (|D^2u|^2+u^2)\psi_B'
+
\iint_{\mathbb{R}^2} (|\nabla u|^2+u^2)\phi_{i,B}'
\right].
\end{align}
Moreover,
\begin{align}\label{eq:L6-paper-notation}
\iint_{\mathbb{R}^2} |u|^6 \psi_B'
\leq
C_B\norm{u}_{L^2(\mathbb{R}^2)}^4
\left[
\iint_{\mathbb{R}^2} (|D^2u|^2+u^2)\psi_B'
+
\iint_{\mathbb{R}^2} (|\nabla u|^2+u^2)\phi_{i,B}'
\right].
\end{align}
\end{lemma}
\begin{proof}
    We prove the estimate first for $u\in C_c^\infty(\R^2)$. The general case
follows by density in the weighted norm appearing on the right-hand side.
From \eqref{psiglaeser} there exists $D_B>0$ such that
\begin{equation}\label{eq:product-paper-notation}
        |\psi_B''|^2
        \leq D_B\psi_B'\phi_{i,B}'.
\end{equation}
Moreover,  $\lim_{y_1\uparrow -\frac{B}{2}}\frac{|\psi_B''|^2}{\psi_B'}=0.$ 
 Set
\begin{align*}
    \mathcal S &= \norm{u}_{L^2(\R^2)}, & \mathcal I &= \iint_{\R^2}u^2|\nabla u|^2\psi_B', \\
    \mathcal A &= \iint_{\R^2}(|D^2u|^2+u^2)\psi_B', & \mathcal E &= \iint_{\R^2}(|\nabla u|^2+u^2)\phi_{i,B}',
\end{align*}
and
\[
    \mathcal{I}_6 = \iint_{\R^2}|u|^6\psi_B'.
\]
We first show that
\begin{equation}\label{eq:L-pre-paper-notation}
        \mathcal I_6\leq C \mathcal S^2 \mathcal I+C \mathcal S^4 \mathcal E.
\end{equation}

Let
\[
        \mathcal I_3=u^3\sqrt{\psi_B'}.
\]
All identities involving $[\psi_B']^{-1/2}$ are used only on the open set
$\{\psi_B'>0\}=\Omega_B$. The endpoint $y_1=-B/2$ has measure zero, and
$\mathcal I_3$ has zero trace there because $\psi_B'(-B/2)=0$. Thus no boundary term
is created by extending $\mathcal I_3$ by zero to the region $y_1\geq -B/2$.

On $\Omega_B$,
\[
        \partial_{y_2}\mathcal I_3
        =3u^2u_{y_2}\sqrt{\psi_B'},
\]
while
\[
        \partial_{y_1}\mathcal I_3
        =3u^2u_{y_1}\sqrt{\psi_B'}
        +\frac12u^3\frac{\psi_B''}{\sqrt{\psi_B'}}.
\]
The second formula is never evaluated at points where $\psi_B'=0$.
Using \eqref{eq:product-paper-notation},
\begin{align*}
        \iint |u|^3
        \frac{|\psi_B''|}{\sqrt{\psi_B'}}
        1_{\{\psi_B'>0\}}
        &\leq
        C_B\iint_{\Omega_B}|u|^3\sqrt{\phi_{i,B}'}\\
        &\leq
        C_B\mathcal S
        \left(\iint_{\Omega_B}u^4\phi_{i,B}'\right)^{1/2}\\
        &\leq C_B \mathcal S^2\mathcal E^{1/2},
\end{align*}
where in the last line we used Lemma \ref{SobolevLemma} with weight $\phi_{i,B}'.$

Also, for $j=1,2$,
\[
        \iint |u|^2|u_{y_j}|\sqrt{\psi_B'}
        \leq
        \mathcal S\left(\iint u^2u_{y_j}^2\psi_B'\right)^{1/2}.
\]
Therefore
\[
        \norm{\nabla \mathcal I_3}_{L^1(\R^2)}
        \leq C_B\left(\mathcal S\mathcal I^{1/2}+\mathcal S^2\mathcal E^{1/2}\right).
\]
The Sobolev inequality $\norm{\mathcal I_3}_{L^2(\R^2)}\leq C\norm{\nabla \mathcal I_3}_{L^1(\R^2)}$
gives
\[
        \mathcal I_6=\norm{\mathcal I_3}_{L^2}^2
        \leq C_B\left(\mathcal S^2\mathcal I+\mathcal S^4\mathcal E\right),
\]
which proves \eqref{eq:L-pre-paper-notation}.

For $j=1,2$, integration by parts gives
\[
        \iint u^2u_{y_j}^2\psi_B'
        =-\frac13\iint u^3u_{y_jy_j}\psi_B'
        -\frac13\delta_{1j}\iint u^3u_{y_1}\psi_B''.
\]
There is no boundary contribution at $y_1=-B/2$ because
$\psi_B'(-B/2)=0$. Summing over $j=1,2$ and using Cauchy--Schwarz,
\begin{align*}
        \mathcal I
        &\leq C_B \mathcal{I}_6^{1/2}\mathcal A^{1/2}
        + C_B\left(\iint u^2u_{y_1}^2\psi_B'\right)^{1/2}
        \left(
        \iint u^4
        \frac{|\psi_B''|^2}{\psi_B'}
        1_{\{\psi_B'>0\}}
        \right)^{1/2}\\
        &\leq C_B \mathcal{I}_6^{1/2}\mathcal A^{1/2}
        + C_B \mathcal I^{1/2}
        \left(\iint_{\Omega_B}u^4\phi_{i,B}'\right)^{1/2}.
\end{align*}
Using Lemma \ref{SobolevLemma} with weight $\phi_{i,B}'$ and \eqref{eq:L-pre-paper-notation},
\[
        \mathcal I
        \leq
        C_B\left(\mathcal S\mathcal I^{1/2}+\mathcal S^2\mathcal E^{1/2}\right)\mathcal A^{1/2}
        +C\mathcal S \mathcal I^{1/2}\mathcal E^{1/2}.
\]
Young's inequality yields
\[
       \mathcal  I\leq \frac12 \mathcal I+C_B \mathcal S^2\mathcal A+C_B \mathcal S^2\mathcal E.
\]
Hence
\begin{equation}\label{eq:I-paper-notation}
        \mathcal I\leq C_B \mathcal S^2(\mathcal A+\mathcal E).
\end{equation}
Combining \eqref{eq:I-paper-notation}, \eqref{psiandphi} with Lemma \ref{SobolevLemma} with weight $\phi_{i,B}'$ proves
\eqref{eq:main-paper-notation}. Finally, substituting
\eqref{eq:I-paper-notation} into \eqref{eq:L-pre-paper-notation} gives
\[
        \mathcal I_6\leq C_B \mathcal S^4(\mathcal A+\mathcal E),
\]
which proves \eqref{eq:L6-paper-notation}.
\end{proof}

\chapter{Coercivity Criteria and Virial Estimates}\label{AppendixB}   

This appendix proves the linearized coercivity estimate of Lemma~ \ref{lem:CoercivityMain} and the virial estimate of Lemma~ \ref{Virial2}. We  begin with a generalization of Weinstein's coercivity criterion \cite{Weinstein}.

\section{Generalized Weinstein Criterion}

\begin{proposition}[Generalized Weinstein Criterion]
\label{Generalized Weinstein Lemma}
Let \(T\) be a self-adjoint operator on a real Hilbert space
\(\mathcal H\), and assume that the negative spectral subspace
\[
    \mathcal H_-:=\mathbf 1_{(-\infty,0)}(T)\mathcal H
\]
has finite dimension \(n\). Set
\[
    \mathfrak G_T:=\operatorname{sgn}(T),
    \qquad
    \mathcal Q(T):=\mathcal D(|T|^{1/2}),
\]
and define the quadratic form associated with \(T\) by
\[
    \mathfrak t[u,v]
    :=
    \bigl(
        \mathfrak G_T|T|^{1/2}u,\,
        |T|^{1/2}v
    \bigr),
    \qquad
    u,v\in\mathcal Q(T).
\]
We write \(\mathfrak t[u]:=\mathfrak t[u,u]\).

Let \(m\geq n\), and let
\(f_1,\ldots,f_m\in\operatorname{Ran}(T)\). For each \(i\), choose
\(F_i\in\mathcal D(T)\) such that \(TF_i=f_i\), and define
\[
    M_{ij}:=(F_i,f_j),
    \qquad 1\leq i,j\leq m,
\]
and
\[
    \Phi:\mathbb R^m\to\mathcal H,
    \qquad
    \Phi\xi:=\sum_{i=1}^m\xi_i f_i.
\]

Assume that there exists an \(n\)-dimensional subspace
\(\mathcal Z\subset\mathbb R^m\) such that
\[
    \xi^TM\xi\leq0
    \qquad\text{for every }\xi\in \mathcal Z,
\]
and such that \(\Phi|_{\mathcal Z}\) is injective. Then
\[
    \mathfrak t[u]\geq0
\]
for every \(u\in\mathcal Q(T)\) satisfying
\[
    (u,f_i)=0,
    \qquad i=1,\ldots,m.
\]

If, moreover, \(M\) is negative definite on \(\mathcal Z\), then there exists
\(\kappa>0\) such that
\[
    \mathfrak t[u]
    \geq
    \kappa\bigl\||T|^{1/2}u\bigr\|_{\mathcal H}^2
\]
for every such \(u\).

If, in this strictly negative case, \(0\in\rho(T)\), then there exists \(\delta>0\)
such that
\[
    \mathfrak t[u]
    \geq
    \delta\|u\|_{\mathcal H}^2
\]
for every such \(u\).
\end{proposition}

\begin{proof}
The case \(n=0\) is immediate, so assume that \(n\geq1\).

First, \(M\) is independent of the choices of the preimages \(F_i\).
Indeed, if \(F_i'\) is another preimage of \(f_i\), then
\(F_i'-F_i\in\ker T\), while
\(\operatorname{Ran}(T)\perp\ker T\); hence
\[
    (F_i'-F_i,f_j)=0.
\]
Moreover,
\[
    M_{ij}=(F_i,TF_j)=(TF_i,F_j)=M_{ji},
\]
so \(M\) is symmetric.

Choose a basis \(c^{(1)},\ldots,c^{(n)}\) of \(\mathcal Z\), and let
\(C\in\mathbb R^{m\times n}\) be the matrix with these vectors as
columns. For \(k=1,\ldots,n\), define
\[
    G_k:=\sum_{i=1}^{m}C_{ik}F_i,
    \qquad
    g_k:=\sum_{i=1}^{m}C_{ik}f_i,
    \qquad
    TG_k=g_k.
\]
The vectors \(g_1,\ldots,g_n\) are linearly independent. Indeed, if
\(\sum_{k=1}^n\xi_k g_k=0\), then
\[
    \Phi(C\xi)=0.
\]
Since \(C\xi\in \mathcal Z\) and \(\Phi|_{\mathcal Z}\) is injective, \(C\xi=0\). Since the
columns of \(C\) are linearly independent, it follows that \(\xi=0\).

Set \(N:=C^TMC\). Since \(\operatorname{Ran}C=\mathcal Z\), we have \(N\leq0\),
and
\[
    N_{kl}
    =\sum_{i,j=1}^{m}C_{ik}C_{jl}(F_i,f_j)
    =(G_k,g_l),
    \qquad 1\leq k,l\leq n.
\]
If \(M\) is negative definite on \(\mathcal Z\), then \(N\) is negative
definite.

Set
\[
    \mathcal H_0:=\ker T,
    \qquad
    \mathcal H_+:=\mathbf 1_{(0,\infty)}(T)\mathcal H.
\]
Then
\[
    \mathcal H
    =
    \mathcal H_-\oplus\mathcal H_0\oplus\mathcal H_+.
\]
Recall that \(\mathfrak G_T=\operatorname{sgn}(T)\). Thus \(T=\mathfrak G_T|T|\), where
\(\mathfrak G_T=-I\) on \(\mathcal H_-\), \(\mathfrak G_T=0\) on \(\mathcal H_0\), and
\(\mathfrak G_T=I\) on \(\mathcal H_+\).

For \(u\in\mathcal Q(T)=\mathcal D(|T|^{1/2})\), set
\[
    z:=|T|^{1/2}u,
\]
and, for \(k=1,\ldots,n\), set
\[
    h_k:=\mathfrak G_T|T|^{1/2}G_k.
\]
Since
\[
    \operatorname{Ran}(|T|^{1/2})
    \subset(\ker T)^\perp,
\]
both \(z\) and \(h_k\) belong to
\[
    (\ker T)^\perp=\mathcal H_-\oplus\mathcal H_+.
\]

By the definition of the quadratic form associated with \(T\),
\[
    \mathfrak t[u]=(\mathfrak G_Tz,z).
\]
Moreover, since \(G_k\in\mathcal D(T)\) and \(TG_k=g_k\),
\[
\begin{aligned}
    (u,g_k)
    &=(u,TG_k)\\
    &=
    \bigl(
        |T|^{1/2}u,\,
        \mathfrak G_T|T|^{1/2}G_k
    \bigr)\\
    &=(z,h_k).
\end{aligned}
\]
Finally,
\[
\begin{aligned}
    (\mathfrak G_Th_k,h_l)
    &=
    \bigl(
        |T|^{1/2}G_k,\,
        \mathfrak G_T|T|^{1/2}G_l
    \bigr)\\
    &=(G_k,TG_l)
    =N_{kl}.
\end{aligned}
\]

We claim that \(h_1,\ldots,h_n\) are linearly independent. Suppose
that \(\sum_{k=1}^n\xi_kh_k=0\), and set
\[
    G_\xi:=\sum_{k=1}^n\xi_kG_k.
\]
Then \(\mathfrak G_T|T|^{1/2}G_\xi=0\). Since
\(|T|^{1/2}G_\xi\in(\ker T)^\perp\) and \(\mathfrak G_T\) is invertible on
\((\ker T)^\perp\), it follows that
\[
    |T|^{1/2}G_\xi=0.
\]
Thus \(G_\xi\in\ker T\), and hence
\[
    \sum_{k=1}^n\xi_k g_k=TG_\xi=0.
\]
The linear independence of \(g_1,\ldots,g_n\) gives \(\xi=0\).

Write
\[
    h_k=a_k+b_k,
    \qquad
    a_k:=P_-h_k\in\mathcal H_-,
    \qquad
    b_k:=P_+h_k\in\mathcal H_+,
\]
and define \(A:\mathbb R^n\to\mathcal H_-\) and
\(B:\mathbb R^n\to\mathcal H_+\) by
\[
    A\xi:=\sum_{k=1}^n\xi_k a_k,
    \qquad
    B\xi:=\sum_{k=1}^n\xi_k b_k.
\]
For every \(\xi\in\mathbb R^n\),
\[
    \xi^TN\xi
    =\left(
        \mathfrak G_T\sum_{k=1}^n\xi_kh_k,
        \sum_{l=1}^n\xi_lh_l
      \right)
    =-\|A\xi\|^2+\|B\xi\|^2\leq0.
\]
Consequently,
\[
    \|B\xi\|\leq\|A\xi\|.
\]
If \(A\xi=0\), then \(B\xi=0\), so
\(\sum_{k=1}^n\xi_kh_k=0\). The linear independence of the \(h_k\)
then gives \(\xi=0\). Thus \(A\) is injective. Since
\[
    \dim\mathbb R^n=\dim\mathcal H_-=n,
\]
the map \(A\) is an isomorphism from \(\mathbb R^n\) onto
\(\mathcal H_-\).

Define
\[
    K:=BA^{-1}:\mathcal H_-\to\mathcal H_+.
\]
For \(x=A\xi\), the preceding inequality gives
\[
    \|Kx\|=\|B\xi\|\leq\|A\xi\|=\|x\|,
\]
and therefore \(\|K\|\leq1\). If \(N\) is negative definite, then the
inequality is strict for every \(x\neq0\). Since \(\mathcal H_-\) is
finite-dimensional, compactness of its unit sphere yields
\[
    \theta:=\|K\|<1.
\]

Now let \(u\in\mathcal Q(T)\) satisfy
\((u,f_i)=0\) for \(i=1,\ldots,m\). Then \((u,g_k)=0\) for every
\(k\). Writing \(z=z_-+z_+\), where \(z_\pm:=P_\pm z\), we obtain
\[
    0=(z,h_k)=(z_-,a_k)+(z_+,b_k),
    \qquad k=1,\ldots,n.
\]
Equivalently,
\[
    A^*z_-+B^*z_+=0.
\]
Since \(A^*\) is invertible,
\[
    z_-=-(A^*)^{-1}B^*z_+=-K^*z_+.
\]
It follows that
\[
    \mathfrak t[u]=(\mathfrak G_Tz,z)
    =\|z_+\|^2-\|z_-\|^2
    =\|z_+\|^2-\|K^*z_+\|^2
    \geq0,
\]
because \(\|K^*\|=\|K\|\leq1\).

Suppose now that \(M\) is negative definite on \(\mathcal Z\). Then \(N\) is negative definite
and \(\theta=\|K\|<1\). Hence
\[
    \mathfrak t[u]\geq(1-\theta^2)\|z_+\|^2,
    \qquad
    \|z\|^2
    =\|z_-\|^2+\|z_+\|^2
    \leq(1+\theta^2)\|z_+\|^2.
\]
Therefore, with
\[
    \kappa:=\frac{1-\theta^2}{1+\theta^2}>0,
\]
we obtain
\[
    \mathfrak t[u]
    \geq
    \kappa\|z\|^2
    =
    \kappa\bigl\||T|^{1/2}u\bigr\|^2,
\]

Finally, if \(0\in\rho(T)\), there exists \(\eta>0\) such that
\(|T|\geq\eta I\). Consequently,
\[
    \mathfrak t[u]\geq\kappa\eta\|u\|_{\mathcal H}^2,
\]
and the conclusion follows with \(\delta:=\kappa\eta\).
\end{proof}

\begin{remark}\label{coercivityremark}
Suppose that
\[
    \operatorname{ind}_{-}(M)=n,
\]
where \(\operatorname{ind}_{-}(M)\) denotes the number of negative
eigenvalues of \(M\), counted with multiplicity. Let \(\mathcal Z\) be the
negative spectral subspace of \(M\). Then \(\dim \mathcal Z=n\) and \(M\) is
negative definite on \(\mathcal Z\). Moreover, \(\Phi|_{\mathcal Z}\) is automatically
injective. Indeed, if \(\xi\in \mathcal Z\) and \(\Phi\xi=0\), then, setting
\(F_\xi:=\sum_{i=1}^m\xi_iF_i\), we have
\[
    \xi^TM\xi=(F_\xi,\Phi\xi)=0.
\]
The negative definiteness of \(M\) on \(\mathcal Z\) therefore implies
\(\xi=0\).
\end{remark}

\begin{remark}
The geometric hypothesis on \(\mathcal Z\) has an equivalent formulation in
terms of the inertia of \(M\). First,
\[
    \ker\Phi\subset\ker M,
\]
because, for \(\xi\in\ker\Phi\),
\[
    (M\xi)_i=(F_i,\Phi\xi)=0,
    \qquad i=1,\ldots,m.
\]
Consequently, the existence of an \(n\)-dimensional subspace
\(\mathcal Z\subset\mathbb R^m\) such that
\[
    \xi^TM\xi\leq0\quad\text{for every }\xi\in \mathcal Z,
    \qquad
    \mathcal Z\cap\ker\Phi=\{0\},
\]
is equivalent to
\[
    \operatorname{ind}_{-}(M)
    +\dim\ker M-\dim\ker\Phi
    \geq n.
\]

Indeed, let
\[
    p:=\operatorname{ind}_{-}(M),\qquad
    z:=\dim\ker M,\qquad
    r:=\dim\ker\Phi.
\]
If such a subspace \(\mathcal Z\) exists, then
\(\mathcal Z\oplus\ker\Phi\) is \(M\)-non-positive, because
\(\ker\Phi\subset\ker M\). Since the maximal dimension of an
\(M\)-non-positive subspace is \(p+z\), we obtain
\[
    n+r\leq p+z.
\]
Conversely, suppose that \(p+z-r\geq n\). Let \(V_-\) be the negative
spectral subspace of \(M\), and choose a subspace \(Z\) such that
\[
    \ker M=\ker\Phi\oplus Z.
\]
Then \(V_-\oplus Z\) is \(M\)-non-positive, intersects
\(\ker\Phi\) trivially, and has dimension \(p+z-r\). Hence it contains
an \(n\)-dimensional subspace \(\mathcal Z\) with the required properties.
\end{remark}
\begin{remark}
Let
\[
    T=-\kappa_1\partial_{x_1}^2-\kappa_2\partial_{x_2}^2+V
\]
 on \(L^2(\mathbb R^2)\), where \(\kappa_1,\kappa_2>0\) are constants and
\(V\in L^\infty(\mathbb R^2;\mathbb R)\). The constant-coefficient
operator
\[
    T_0:=-\kappa_1\partial_{x_1}^2-\kappa_2\partial_{x_2}^2
\]
is self-adjoint on \(L^2(\mathbb R^2)\) with
\[
    \mathcal D(T_0)=H^2(\mathbb R^2).
\]
Since multiplication by \(V\) is a bounded self-adjoint operator on
\(L^2(\mathbb R^2)\), the bounded perturbation theorem gives
\[
    \mathcal D(T)=H^2(\mathbb R^2).
\]

The closed quadratic form associated with \(T\) is
\[
    \mathfrak l[u,v]
    :=
    \kappa_1\int_{\mathbb R^2}u_{x_1}v_{x_1}\,dx
    +\kappa_2\int_{\mathbb R^2}u_{x_2}v_{x_2}\,dx
    +\int_{\mathbb R^2}Vuv\,dx,
    \qquad
    \mathcal D(\mathfrak l)=H^1(\mathbb R^2).
\]
Indeed, if \(\tau_0>\|V\|_{L^\infty}+1\), then
\[
\begin{aligned}
    \mathfrak l[u]+\tau_0\|u\|_{L^2}^2
    &=
    \kappa_1\|u_{x_1}\|_{L^2}^2
    +\kappa_2\|u_{x_2}\|_{L^2}^2
    +\int_{\mathbb R^2}(V+\tau_0)u^2\,dx \\
    &\asymp \|u\|_{H^1}^2.
\end{aligned}
\]
 The representation theorem for closed semibounded forms therefore
yields
\[
    \mathcal D\bigl((T+\tau_0 I)^{1/2}\bigr)=H^1(\mathbb R^2).
\]
Since \(T\) is bounded from below, the spectral weights
\(\lambda+\tau_0\) and \(1+|\lambda|\) are comparable on \(\sigma(T)\).
Consequently,
\[
    \mathcal D(|T|^{1/2})
    =
    \mathcal D\bigl((T+\tau_0 I)^{1/2}\bigr)
    =
    H^1(\mathbb R^2).
\]

If \(\mathfrak G_T=\operatorname{sgn}(T)\), the spectral representation of the
quadratic form gives
\[
\begin{aligned}
    \mathfrak t[u]
    &=
    \bigl(\mathfrak G_T|T|^{1/2}u,|T|^{1/2}u\bigr) \\
    &=
    \kappa_1\int_{\mathbb R^2}|u_{x_1}|^2\,dx
    +\kappa_2\int_{\mathbb R^2}|u_{x_2}|^2\,dx
    +\int_{\mathbb R^2}Vu^2\,dx,
    \qquad u\in H^1(\mathbb R^2).
\end{aligned}
\]
For \(u\in H^2(\mathbb R^2)\), this also equals the genuine
\(L^2\)-inner product \((Tu,u)\). For a general
\(u\in H^1(\mathbb R^2)\), it is understood only in the
quadratic-form sense.
\end{remark}

\begin{remark}
For \(m=n=1\), the injectivity of \(\Phi\) is equivalent to
\(f_1\neq0\), while the matrix condition reduces to
\[
    (F_1,f_1)\leq0,
    \qquad TF_1=f_1.
\]
Thus, if the negative spectral subspace of \(T\) has dimension one,
the proposition gives
\[
    \mathfrak t[u]\geq0
    \qquad\text{whenever }(u,f_1)=0, \qquad u\in \mathcal Q(T).
\]
In particular, for
\[
    T=-\Delta+1-3Q^2,\qquad f_1=Q,
\]
using
\[
    \dim\mathcal H_-(T)=1,\qquad
    -\frac12T(\Lambda Q)=Q,\qquad
    (\Lambda Q,Q)=0,
\]
and \(Q\neq0\), one recovers Lemma~E.1 in \cite{Weinstein}.
\end{remark}

    While we give the explicit linear combination of orthogonal conditions for coercivity, sometimes identifying such a linear combination explicitly can be difficult. Motivated by this, we state the following lemma which shows the existence of such a linear combination without writing it explicitly.

\begin{lemma}\label{Sylvester}
    Let $\mathcal{H}$ be a real Hilbert space. Let $T:\mathcal{D}(T)\rightarrow \mathcal{H}$ be a self-adjoint, linear operator with $0\in \rho(T)$. Consider $f_1, f_2, \ldots, f_m \in \mathcal{H}$, with $m\geq 1,$ to be $m$ linearly independent elements in $\mathcal{H}$ and denote $V=\text{span}\{f_1, \ldots, f_m\}.$ We write $S$ to be the $m\times m$ matrix with its elements $S_{i,j}=(T^{-1}f_i,f_j)$ for $1\leq i,j\leq m$. If the matrix $S$ has at least one negative eigenvalue, then there exists $g\in V$, such that 
    $$(T^{-1}g,g)<0.$$
\end{lemma}
\begin{proof}
    Let $\Gamma:\mathbb{R}^m\rightarrow V$ defined by $\Gamma(z_1, \ldots,z_m)=\sum_{i=1}^{m}z_if_i\in V.$ Since $\text{dim}V=m,$ $\Gamma$ is bijective. If $w=\sum_{i=1}^{m}z_if_i,$ $(T^{-1}w,w)=\vec{z}^TS\vec{z},$ where $\vec{z}^T=(z_1, \ldots,z_m)\in \mathbb{R}^m.$  
    The matrix $S$ is symmetric since $T^{-1}$ is self-adjoint. Since $S$ has at least one negative eigenvalue, from Sylvester's law of inertia, the maximum dimension on which the quadratic form $Q: \mathbb{R}^m \rightarrow \mathbb{R}$, $Q(\vec{z})=\vec{z}^TS \vec{z}$ is negative definite is at least one. Therefore, there exists $\vec{c} \in \mathbb{R}^m$, such that $Q(\vec{c})<0$.

    Taking such $\vec{c}=(c_1,\cdots,c_m)^T$, we obtain that $$(T^{-1}g,g)<0$$
    with $g=\sum_{i=1}^m c_i f_i.$
\end{proof}

From now on, let \(\kappa_1,\kappa_2>0\) and
\(V\in L^\infty(\mathbb R^2;\mathbb R)\). For

\[
    T=-\kappa_1\partial_{x_1}^2-\kappa_2\partial_{x_2}^2+V,
\]
and \(u,v\in H^1(\mathbb R^2)\), we use the quadratic-form notation
\[
    (Tu,v)
    :=
    \kappa_1\int_{\mathbb R^2}u_{x_1}v_{x_1}
    +\kappa_2\int_{\mathbb R^2}u_{x_2}v_{x_2}
    +\int_{\mathbb R^2}Vuv.
\]
In particular, \((Tu,u)\) denotes the quadratic form of \(T\) and does
not require \(u\in\mathcal D(T)=H^2(\mathbb R^2)\).

\section{Coercivity of the Linearized Operator}\label{appendixB2}

\begin{lemma}
Define the subspace $\mathcal X \subset H^1(\mathbb{R}^2)$ by
\[
\mathcal X  = \left\{ u \in H^1(\mathbb{R}^2) : (u, Q) = (u, \varphi(x_1)Q_{x_2}) = (u, \varphi(x_1)\Lambda Q) = (u, \varphi(x_1)Q_{x_1}) = 0 \right\}.
\]
Let $L\coloneqq -\Delta+1-3Q^2$ and define its associated quadratic form 
\[
    (Lu,u)
    :=
    \int_{\mathbb R^2}
    \left(
        |\nabla u|^2
        +(1-3Q^2)u^2
    \right)\,dx,
    \qquad u\in H^1(\mathbb R^2).
\]
\begin{itemize}
    \item[a)] \textbf{Coercivity on $\mathcal X $:} There exists $\delta > 0$ such that for all $u \in \mathcal X $,
    \[
    (Lu, u) \geq \delta \|u\|_{H^1}^2.
    \]
    \item[b)] \textbf{Global Coercivity:} There exists $\tilde{\delta} > 0$ such that for any $u \in H^1(\mathbb{R}^2)$,
    \begin{equation*}
    (Lu, u) \geq \tilde{\delta}\|u\|_{H^1}^2 - \frac{1}{\tilde{\delta}} \left( (u,Q)^2 + (u, \varphi(x_1)Q_{x_2})^2 + (u, \varphi(x_1)\Lambda Q)^2 + (u, \varphi(x_1)Q_{x_1})^2 \right).
    \end{equation*}
\end{itemize}
\end{lemma}

\begin{proof}
\textit{a) }
Recall that the negative spectral subspace of \(L\) has dimension one and that
\[
    \ker L=\operatorname{span}\{Q_{x_1},Q_{x_2}\}.
\]
Set
\[
    A:=\bigl(\varphi(x_1)\Lambda Q,Q_{x_1}\bigr),
    \qquad
    B:=\bigl(\varphi(x_1)Q_{x_1},Q_{x_1}\bigr),
\]
and define
\[
    \Xi
    :=
    B\varphi(x_1)\Lambda Q
    -
    A\varphi(x_1)Q_{x_1}.
\]
By construction,
\[
    (\Xi,Q_{x_1})=0,
\]
while \((\Xi,Q_{x_2})=0\) by parity. Since \(0\) is an
isolated eigenvalue of \(L\), \(\operatorname{Ran}(L)\) is closed. Hence
\[
    \Xi\in(\ker L)^\perp=\operatorname{Ran}(L).
\]
Choose \(R_\Xi\in\mathcal D(L)\) such that \(LR_\Xi=\Xi\).

We apply Proposition~\ref{Generalized Weinstein Lemma} with
\[
    f_1:=Q,
    \qquad
    f_2:=\Xi.
\]
Since \(L(\Lambda Q)=-2Q\), we may choose
\[
    F_1:=-\frac12\Lambda Q,
    \qquad
    F_2:=R_\Xi.
\]
The corresponding constraint matrix is
\[
    \mathfrak M
    :=
    \begin{pmatrix}
        (F_1,Q)&(F_1,\Xi)\\
        (F_2,Q)&(F_2,\Xi)
    \end{pmatrix}.
\]
Using \((Q,\Lambda Q)=0\) and the self-adjointness of \(L\), we obtain 
\[
    \mathfrak M
    =
    \begin{pmatrix}
        0&-\frac12(\Lambda Q,\Xi)\\[1mm]
        -\frac12(\Lambda Q,\Xi)&(R_\Xi,\Xi)
    \end{pmatrix}.
\]
Moreover,
\[
\begin{aligned}
    (\Lambda Q,\Xi)
    &=
    \bigl(\varphi Q_{x_1},Q_{x_1}\bigr)
      \bigl(\varphi\Lambda Q,\Lambda Q\bigr)
    -
    \bigl(\varphi\Lambda Q,Q_{x_1}\bigr)
      \bigl(\varphi Q_{x_1},\Lambda Q\bigr)=\det M^*,
\end{aligned}
\]
where $M^*$ is defined in \eqref{M*}. Recall that $\det M^*\neq 0.$ 
Thus
\[
    \det\mathfrak M
    =
    -\frac14(\det M^*)^2<0.
\]
As \(\mathfrak M\) is symmetric, it follows that
\(
    \operatorname{ind}_{-}(\mathfrak M)=1,
\)
which equals the dimension of the negative spectral subspace of \(L\).
By Proposition~\ref{Generalized Weinstein Lemma} and
Remark~\ref{coercivityremark}, there exists \(\kappa>0\) such that
\[
    (Lu,u)\geq\kappa\bigl\||L|^{1/2}u\bigr\|_{L^2}^2
\]
whenever \((u,Q)=(u,\Xi)=0\). In particular, this estimate holds for
every \(u\in \mathcal X \).

For \(u\in \mathcal X \), write
\[
    u=v+aQ_{x_1}+bQ_{x_2},
    \qquad v\perp\ker L.
\]
By parity and the defining orthogonality conditions,
\[
    a
    =
    -\frac{(v,\varphi Q_{x_1})}
           {(Q_{x_1},\varphi Q_{x_1})},
    \qquad
    b
    =
    -\frac{(v,\varphi Q_{x_2})}
           {(Q_{x_2},\varphi Q_{x_2})}.
\]
Since \(\varphi>0\), the denominators are strictly positive, and hence
\[
    |a|+|b|\lesssim\|v\|_{L^2}.
\]
Since \(0\) is isolated in \(\sigma(L)\), there exists \(\sigma_0>0\) such
that
\[
    \bigl\||L|^{1/2}v\bigr\|_{L^2}^2\geq \sigma_0\|v\|_{L^2}^2,
    \qquad v\perp\ker L.
\]
Moreover,
\[
\begin{aligned}
    \|v\|_{H^1}^2
    &=(Lv,v)+3\int Q^2v^2\leq\bigl\||L|^{1/2}v\bigr\|_{L^2}^2
       +3\|Q\|_{L^\infty}^2\|v\|_{L^2}^2\lesssim\bigl\||L|^{1/2}v\bigr\|_{L^2}^2.
\end{aligned}
\]
Since \(|L|^{1/2}\) vanishes on \(\ker L\),
\[
    \bigl\||L|^{1/2}u\bigr\|_{L^2}^2=\bigl\||L|^{1/2}v\bigr\|_{L^2}^2.
\]
Consequently,
\[
    \|u\|_{H^1}^2
    \lesssim
    \|v\|_{H^1}^2+a^2+b^2
    \lesssim
    \bigl\||L|^{1/2}u\bigr\|_{L^2}^2.
\]
It follows that
\[
    (Lu,u)
    \geq\kappa\bigl\||L|^{1/2}u\bigr\|_{L^2}^2
    \geq\delta\|u\|_{H^1}^2.
\]

\textit{b)}
Let $u \in H^1(\mathbb{R}^2)$. We construct a vector $v \in \mathcal X $ by subtracting appropriate components from $u$. Define
\[
v = u + a Q + b \Lambda Q + c Q_{x_1} + d Q_{x_2},
\]
where the scalar coefficients $a, b, c, d$ are chosen such that $v$ satisfies the four orthogonality conditions defining $\mathcal X$.
The condition $(v, Q) = 0$ implies $a = -(u, Q) / \|Q\|_{L^2}^2$. Similarly, $(v, \varphi(x_1)Q_{x_2}) = 0$ determines $d$. The coefficients $b$ and $c$ are determined by the non-singular $2 \times 2$ system involving $M^*$ arising from the conditions $(v, \varphi(x_1)\Lambda Q) = 0$ and $(v, \varphi(x_1)Q_{x_1}) = 0$.
Crucially, the vector $\vec{p} = (a, b, c, d)$ depends linearly on the projections of $u$ onto the finite set of functions $\{Q, \varphi(x_1)\Lambda Q, \varphi(x_1)Q_{x_1}, \varphi(x_1)Q_{x_2}\}$. Thus, there exists a constant $C > 0$ such that
\begin{equation} \label{eq:coeff_bound}
    |a|^2 + |b|^2 + |c|^2 + |d|^2 \leq C \left( (u,Q)^2 + (u, \varphi(x_1)\Lambda Q)^2 + (u, \varphi(x_1)Q_{x_1})^2 + (u, \varphi(x_1)Q_{x_2})^2 \right).
\end{equation}
Let $w = aQ + b\Lambda Q + cQ_{x_1} + dQ_{x_2}$, so that $u = v - w$. Since $v \in \mathcal X $, part a) implies
\begin{equation} \label{eq:v_coercivity}
    (Lv, v) \geq \delta \|v\|_{H^1}^2.
\end{equation}
We expand the quadratic form for $u$:
\[
(Lu, u) = (L(v-w), v-w) = (Lv, v) - 2(Lv, w) + (Lw, w).
\]
Since $w$ lies in a finite-dimensional subspace spanned by smooth functions, we have estimates $\|w\|_{H^1}^2 \lesssim |\vec{p}|^2$ and $|(Lw, w)| \lesssim \|w\|_{H^1}^2$. Using the Cauchy-Schwarz inequality for the cross term:
\[
|2(Lv, w)| \leq 2\|Lv\|_{H^{-1}} \|w\|_{H^1} \lesssim \|v\|_{H^1} \|w\|_{H^1} \leq \frac{\delta}{2} \|v\|_{H^1}^2 + C_{\delta} \|w\|_{H^1}^2.
\]
Substituting these into the expansion and using \eqref{eq:v_coercivity}:
\[
(Lu, u) \geq \delta \|v\|_{H^1}^2 - \frac{\delta}{2} \|v\|_{H^1}^2 - C' \|w\|_{H^1}^2 = \frac{\delta}{2} \|v\|_{H^1}^2 - C' \|w\|_{H^1}^2.
\]
Finally, using the triangle inequality $\|u\|_{H^1} \leq \|v\|_{H^1} + \|w\|_{H^1}$, we have $\|u\|_{H^1}^2 \leq 2\|v\|_{H^1}^2 + 2\|w\|_{H^1}^2$, which implies $\|v\|_{H^1}^2 \geq \frac{1}{2}\|u\|_{H^1}^2 - \|w\|_{H^1}^2$.
Therefore,
\[
(Lu, u) \geq \frac{\delta}{4} \|u\|_{H^1}^2 - C'' \|w\|_{H^1}^2.
\]
Combining this with \eqref{eq:coeff_bound} and renaming constants yields the desired estimate.
\end{proof}

\section{Coercivity of the Virial Operator}

We now proceed with the proof of Lemma \ref{Virial2}, using \(y_1,y_2\)
  for the spatial variables to preserve the notation of the aforementioned lemma. First, we prove that the Virial operator is coercive. 

\begin{lemma}\label{virialcoercivity}
Recall that $\alpha_1=1.01.$ There exists a constant $\mu>0$ such that, for all $v \in H^1(\mathbb{R}^2),$
\begin{align*}
&\iint_{\mathbb R^2}
\Bigg[
3v_{y_1}^2+v_{y_2}^2
+\Bigg(
 1-\frac{1}{4\alpha_1^2}-3Q^2
 +6\alpha_1
  \left(1+e^{-y_1/\alpha_1}\right)QQ_{y_1}
 \Bigg)v^2
\Bigg]\,dy_1dy_2\\
&\qquad\geq
\mu\|v\|_{H^1}^2
-\frac1\mu
\Bigg[
 \left(v,\frac{Q}{\sqrt{\varphi_{y_1}}}\right)^2
 +\left(
   v,\frac{\varphi Q_{y_1}}{\sqrt{\varphi_{y_1}}}
  \right)^2
 +\left(
   v,\frac{\varphi\Lambda Q}{\sqrt{\varphi_{y_1}}}
  \right)^2
 +\left(
   v,\frac{\varphi Q_{y_2}}{\sqrt{\varphi_{y_1}}}
  \right)^2
\Bigg].
\end{align*}
\end{lemma}

\begin{proof}
    We first establish this result for functions $v \in H^1(\mathbb{R}^2)$ satisfying the orthogonality conditions:
    $$
    \left(v,\frac{1}{\sqrt{ \varphi'}} Q\right)=\left(v,\frac{\varphi}{\sqrt{ \varphi'}}Q_{y_1}\right)=\left(v,\frac{\varphi}{\sqrt{ \varphi'}}\Lambda Q\right)=\left(v,\frac{\varphi}{\sqrt{ \varphi'}}Q_{y_2}\right)=0.
    $$

    Define the weight function $\tilde{\varphi}(y_1)=\frac{\varphi(y_1)}{2\sqrt{\alpha_1 \varphi'(y_1)}}=\cosh\left(\frac{y_1}{2\alpha_1}\right)$ and the functions
    \[
    f_1=e^{-\frac{y_1}{2\alpha_1}}Q, \quad f_2=\tilde{\varphi} Q_{y_1}, \quad f_3=\tilde{\varphi} \Lambda Q, \quad \text{and} \quad f_o=\tilde{\varphi} Q_{y_2}.
    \]

    Theorem ~\ref{cap:verified-main} shows that the operator
     
$$
    \mathcal{L}=-3\partial_{y_1}^2-\partial_{y_2}^2+\left(1-\frac{1}{4\alpha_1^2}\right)-3Q^2+6\alpha_1(1+e^{-\frac{y_1}{\alpha_1}})QQ_{y_1}
    $$
 possesses exactly two negative eigenvalues, counted with multiplicity, and $0\in \rho(\mathcal{L})$. Consequently, $\operatorname{Ran}\mathcal L=L^2(\R^2).$

    Suppose there exist constants $c_1,c_2,c_3 \in \mathbb{R}$ such that $g=c_1f_1+c_2f_2+c_3f_3=0.$ Taking the inner product with
\(\sqrt{\varphi'}\varphi^{-1}Q\), and using
\[
    (Q,Q_{y_1})=(Q,\Lambda Q)=0,
\]
we obtain
\[
    0=c_1(\varphi^{-1}Q,Q).
\]
Hence \(c_1=0\). Dividing the remaining
identity by the strictly positive function \(\widetilde\varphi\), we
find
\[
    c_2Q_{y_1}
    +
    c_3\Lambda Q
    =0.
\]
Since \(Q_{y_1}\) is odd and \(\Lambda Q\) is even in \(y_1\), it
follows that \(c_2=c_3=0\).

    Consider the $3\times 3$ matrix $A$ with entries $a_{ij}=(\mathcal{L}^{-1}f_i,f_j)$ for $1\leq i,j\leq 3.$ Theorem \ref{cap:verified-main} yields that $\det A<0.$ Since the odd-dimensional matrix $A$ is real and symmetric with negative determinant, it has at least one negative eigenvalue . Therefore, Lemma \ref{Sylvester} ensures the existence of $f_e \in \text{span}\{f_1,f_2,f_3\}$ such that $(\mathcal{L}^{-1}f_e,f_e)<0.$ 

    Since $f_e$ is even in $y_2$ and $f_o$ is odd in $y_2$, and the operator $\mathcal{L}$ preserves parity in $y_2$, we obtain $(\mathcal{L}^{-1}f_e,f_o)=0.$ Theorem \ref{cap:verified-main} verifies that $(\mathcal{L}^{-1}f_o,f_o)<0.$ Thus, the matrix
    \begin{align}
         M= \begin{bmatrix}
        (\mathcal{L}^{-1}f_e,f_e) & (\mathcal{L}^{-1}f_e,f_o) \\
        (\mathcal{L}^{-1}f_e,f_o) & (\mathcal{L}^{-1}f_o,f_o)
        \end{bmatrix}
    \end{align}
    is a diagonal matrix with negative real numbers on the diagonal; hence, it is negative definite. Applying Proposition ~\ref{Generalized Weinstein Lemma}  and Remark ~\ref{coercivityremark} to $\mathcal{L}$ and $v \in H^1(\mathbb{R}^2)$ satisfying $(v,f_e)=(v,f_o)=0$, we deduce that there exists $\mu>0$ such that $(\mathcal{L}v,v)\geq \mu \|v\|_{L^2}^2.$ Therefore, there exists $\mu_1>0$ such that 
    \begin{equation*}
        (\mathcal{L}u,u)\geq \mu_1 \|u\|_{H^1}^2 \quad \text{if} \quad (u,f_1)=(u,f_2)=(u,f_3)=(u,f_o)=0.
    \end{equation*}

    Since \(f_1,f_2,f_3,f_o\) are linearly independent, choose
\(g_1,g_2,g_3,g_o\in\operatorname{span}\{f_1,f_2,f_3,f_o\}\) such that
\((g_j,f_k)=\delta_{jk}\). For \(u\in H^1\), set
\[
    w=\sum_{j\in\{1,2,3,o\}}(u,f_j)g_j,
    \qquad u_\perp=u-w.
\]
Then \((u_\perp,f_j)=0\) for every \(j\), and hence, by constrained
coercivity and continuity of the quadratic form,
\[
\begin{aligned}
(\mathcal Lu,u)
&\geq \mu_1\|u_\perp\|_{H^1}^2
   -C\|u_\perp\|_{H^1}\|w\|_{H^1}
   -C\|w\|_{H^1}^2  \\
&\geq \frac{\mu_1}{4}\|u\|_{H^1}^2
   -C\sum_{j\in\{1,2,3,o\}}(u,f_j)^2.
\end{aligned}
\]
After renaming the constant, this is the desired global coercivity
estimate.
\end{proof}

\begin{remark}[Parity of the negative directions]
\label{rem:parity}
  We prove analytically that the operator $\mathcal{L}$ has at least two negative eigenvalues. Since the potential is even in $y_2$, $\mathcal{L}$ commutes with the reflection
$\mathcal{P}u(y_1,y_2)=u(y_1,-y_2)$ and
\[
L^2(\R^2)=\mathcal{H}_{\mathrm{even}}\oplus\mathcal{H}_{\mathrm{odd}},
\]
where
\[
\mathcal{H}_{\mathrm{even}}
=\{u:u(y_1,-y_2)=u(y_1,y_2)\},\qquad
\mathcal{H}_{\mathrm{odd}}
=\{u:u(y_1,-y_2)=-u(y_1,y_2)\},
\]
is an orthogonal reducing decomposition.

We first show analytically that each parity sector contains a negative direction. Set
\[
    w(y_1):=\varphi'(y_1)
    =\frac1{\alpha_1}e^{y_1/\alpha_1},
    \qquad
    \rho_\varphi(y_1):=\sqrt{w(y_1)}
    =\frac1{\sqrt{\alpha_1}}e^{y_1/(2\alpha_1)}.
\]
For a sufficiently decaying real-valued function $f$, put
$V=1-3Q^2$ and $w=\varphi'=\rho_\varphi^2$. Integration by parts gives
\begin{equation}
\label{eq:virial-conjugation}
\begin{aligned}
-2\bigl((Lf)_{y_1},\varphi f\bigr)
=\int_{\R^2}\Bigl(3w f_{y_1}^2+w f_{y_2}^2
+\bigl(wV-w''-\varphi V_{y_1}\bigr)f^2\Bigr).
\end{aligned}
\end{equation}
Since
\[
\frac{\rho_\varphi'}{\rho_\varphi}=\frac{1}{2\alpha_1},\qquad
\frac{w''}{w}=\frac{1}{\alpha_1^2},\qquad
\frac{\varphi}{w}=\alpha_1\bigl(1+e^{-y_1/\alpha_1}\bigr),
\qquad V_{y_1}=-6QQ_{y_1},
\]
the right-hand side of \eqref{eq:virial-conjugation} is precisely
$(\mathcal{L}(\rho_\varphi f),\rho_\varphi f)$. Hence
\begin{equation*}
(\mathcal{L}(\rho_\varphi f),\rho_\varphi f)
=-2\bigl((Lf)_{y_1},\varphi f\bigr).
\end{equation*}
Because $LQ=-2Q^3$, we obtain
\[
\begin{aligned}
(\mathcal{L}(\rho_\varphi Q),\rho_\varphi Q)
&=12\int_{\R^2}\varphi Q^3Q_{y_1}
 =3\int_{\R^2}\varphi (Q^4)_{y_1}
 =-3\int_{\R^2}\rho_\varphi^2Q^4<0.
\end{aligned}
\]
Since $\rho_\varphi Q\in\mathcal{H}_{\mathrm{even}}$, it follows that
\[
\operatorname{ind}_{-}\!\left(\mathcal{L}\big|_{\mathcal{H}_{\mathrm{even}}}\right)\geq1.
\]

Next, differentiating the equation for $Q$ yields $LQ_{y_2}=0$, and therefore, for
$z_o=\rho_\varphi Q_{y_2}\in\mathcal{H}_{\mathrm{odd}}$,
\[
(\mathcal{L}z_o,z_o)=0.
\]
Nevertheless, $\mathcal{L}z_o\neq0$. Indeed,
\begin{equation}
\label{eq:Lv-formula}
\rho_\varphi^{-1}\mathcal{L}(\rho_\varphi Q_{y_2})
=-2Q_{y_1y_1y_2}-\frac{3}{\alpha_1}Q_{y_1y_2}
-\frac{1}{\alpha_1^2}Q_{y_2}
+6\alpha_1\bigl(1+e^{-y_1/\alpha_1}\bigr)QQ_{y_1}Q_{y_2}.
\end{equation}
If the left-hand side vanished identically, writing $Q(y_1,y_2)=q(r)$ and restricting
\eqref{eq:Lv-formula} to $y_1=0$, $y_2=r>0$, would give
\[
2\left(\frac{q'}{r}\right)'+\frac{1}{\alpha_1^2}q'=0.
\]
Thus $q'(r)/r=Ce^{-r^2/(4\alpha_1^2)}$, and the decay of $q$ would force
$q(r)=Ae^{-r^2/(4\alpha_1^2)}$, which is incompatible with
$-q''-r^{-1}q'+q-q^3=0$. Hence $G_o:=\mathcal{L}z_o\neq0$. Since $G_o$ is odd in
$y_2$, for all sufficiently small $\tau>0$,
\[
\begin{aligned}
(\mathcal{L}(z_o-\tau G_o),z_o-\tau G_o)
&=-2\tau\|G_o\|_{L^2}^2
+\tau^2(\mathcal{L}G_o,G_o)<0.
\end{aligned}
\]
Consequently,
\[
\operatorname{ind}_{-}\!\left(\mathcal{L}\big|_{\mathcal{H}_{\mathrm{odd}}}\right)\geq1,
\qquad
\operatorname{ind}_{-}(\mathcal{L})
=\operatorname{ind}_{-}\!\left(\mathcal{L}\big|_{\mathcal{H}_{\mathrm{even}}}\right)
+\operatorname{ind}_{-}\!\left(\mathcal{L}\big|_{\mathcal{H}_{\mathrm{odd}}}\right)
\geq2.
\]
The exponential decay of $Q$ and $\alpha_1>1/2$ also imply
\[
\sigma_{\mathrm{ess}}(\mathcal{L})
=\left[1-\frac{1}{4\alpha_1^2},\infty\right]\subset(0,\infty),
\]
so these negative directions produce discrete negative eigenvalues.

Finally, the lowest eigenvalue $\lambda_1(\mathcal L)$ is simple and has a strictly positive eigenfunction
$e_1$. Since $\mathcal{L}$ commutes with $\mathcal{P}$, simplicity and positivity imply
$\mathcal{P}e_1=e_1$, hence $e_1\in\mathcal{H}_{\mathrm{even}}$. If, in addition, one knows that
$\operatorname{ind}_{-}(\mathcal{L})=2$, with multiplicities counted, then the preceding estimates force
\[
\operatorname{ind}_{-}\!\left(\mathcal{L}\big|_{\mathcal{H}_{\mathrm{even}}}\right)
=\operatorname{ind}_{-}\!\left(\mathcal{L}\big|_{\mathcal{H}_{\mathrm{odd}}}\right)=1.
\]
Therefore the second negative eigenvalue satisfies $\lambda_1(\mathcal L)<\lambda_2(\mathcal L)<0$, it is simple, and
its eigenfunction $e_2$ is odd in $y_2$:
\[
e_2(y_1,-y_2)=-e_2(y_1,y_2).
\]
\end{remark}

\begin{remark}\label{numericaleigenvalues}
The choice $\alpha_1=1.01$ is supported by the interval certificates in
Appendix~\ref{appendixE}. For comparison, an exploratory computation
with $\alpha_1=1.1$ gives the following numerical approximation to the
inverse constraint matrix:
\[
A\approx
\begin{bmatrix}
3.7423 & 0.3435 & -5.9056 \\
0.3435 & 2.7157 & 4.0117 \\
-5.9056 & 4.0117 & 17.9427
\end{bmatrix}.
\]
The corresponding approximate eigenvalues are
$\lambda^A_1\approx0.1254$, $\lambda^A_2\approx3.4533$, and
$\lambda^A_3\approx20.8220$. These numerical approximations suggest
positivity of the even constraint matrix at this parameter. This would suggest that $-\mathcal{L}^{-1}$ is negative definite on the finite-dimensional subspace $\text{span}\{f_1,f_2,f_3\}.$ Consequently, there would be no linear combination $f\in \text{span}\{f_1,f_2,f_3\}$ such that $(\mathcal{L}^{-1}f,f)<0.$
 \end{remark}

Now, we can finally prove the Virial Lemma \ref{Virial2}.

\begin{lemma}\label{Virial}
Fix $i\geq1$. There exist constants $\mu>0$, $C>0$, and
$B_0:=B_0(\mu,i)>0$ such that, for every $B\geq B_0$ and every
$v\in H^1(\mathbb{R}^2)$ satisfying
\[
(v,Q)=(v,\varphi(y_1)Q_{y_1})
=(v,\varphi(y_1)Q_{y_2})
=(v,\varphi(y_1)\Lambda Q)=0,
\]
the following estimate holds:
\begin{equation*}
\begin{aligned}
&\iint_{\{|y_1|<\frac B2\}}
\Big[
(3v_{y_1}^2+v_{y_2}^2)\varphi'
+v^2(\varphi'-\varphi''')
-3Q^2v^2\varphi'
+6QQ_{y_1}v^2\varphi
\Big] \\
&\qquad\geq
\mu\iint_{\{|y_1|<\frac B2\}}
(v_{y_1}^2+v_{y_2}^2+v^2)\varphi'
-Ce^{\frac{B}{4\alpha_1}}\iint v^2\phi_{i,B}'.
\end{aligned}
\end{equation*}
\end{lemma}

\begin{proof}
If $\iint v^2\phi_{i,B}'=+\infty$, the conclusion is immediate. We
therefore assume that this integral is finite. Let
$\zeta\in C^\infty(\mathbb{R})$ satisfy $0\leq\zeta\leq1$ and
\[
\zeta(s)=
\begin{cases}
1, & |s|<\frac14,\\
0, & |s|>\frac12.
\end{cases}
\]
Set
\[
\zeta_B(y_1)=\zeta\left(\frac{y_1}{B}\right),
\qquad
\tilde v(y_1,y_2)=v(y_1,y_2)\zeta_B(y_1).
\]
Since $\tilde v$ is supported in $\{|y_1|<B/2\}$, we have
$\tilde v\sqrt{\varphi'}\in H^1(\mathbb{R}^2)$.

\medskip
\noindent\textbf{Step 1: Coercivity estimate.}
For $\tilde v\in C_c^\infty(\mathbb{R}^2)$, integration by parts gives
\begin{align*}
\iint \Big[
(3\tilde v_{y_1}^2+\tilde v_{y_2}^2)\varphi'
&+\tilde v^2(\varphi'-\varphi''')
-3Q^2\tilde v^2\varphi'
+6QQ_{y_1}\tilde v^2\varphi
\Big] \\
&=
\bigl(
\mathcal{L}(\tilde v\sqrt{\varphi'}),
\tilde v\sqrt{\varphi'}
\bigr),
\end{align*}
where $\mathcal{L}$ is the operator defined in
Lemma~\ref{virialcoercivity}. Since $\tilde v$ is supported in a bounded
$y_1$-region, both sides are continuous quadratic forms on
$H^1(\mathbb{R}^2)$; the identity therefore extends to every
$\tilde v\in H^1(\mathbb{R}^2)$ by density.

Let $\mu_0>0$ be the constant in Lemma~\ref{virialcoercivity}. Applying
that lemma to $\tilde v\sqrt{\varphi'}$ yields
\begin{equation}\label{eq:coercivity_raw}
\begin{aligned}
\bigl(
\mathcal{L}(\tilde v\sqrt{\varphi'}),
\tilde v\sqrt{\varphi'}
\bigr)
\geq{}&
\mu_0\|\tilde v\sqrt{\varphi'}\|_{H^1}^2 \\
&-\frac1{\mu_0}
\Big[
(\tilde v,Q)^2
+(\tilde v,\varphi Q_{y_1})^2
+(\tilde v,\varphi\Lambda Q)^2
+(\tilde v,\varphi Q_{y_2})^2
\Big].
\end{aligned}
\end{equation}
Since
\[
\frac{\varphi''}{\varphi'}=\frac1{\alpha_1},
\qquad
\frac{\varphi'''}{\varphi'}=\frac1{\alpha_1^2},
\]
another integration by parts gives the exact identity
\[
\|\tilde v\sqrt{\varphi'}\|_{H^1}^2
=
\iint
\left[
\tilde v_{y_1}^2+\tilde v_{y_2}^2
+\left(1-\frac1{4\alpha_1^2}\right)\tilde v^2
\right]\varphi'.
\]
Define
\[
\tilde\mu
=
\min\left\{
\mu_0\left(1-\frac1{4\alpha_1^2}\right),
\frac12\left(1-\frac1{\alpha_1^2}\right)
\right\}.
\]
Using
\[
\varphi'-\varphi'''
=
\left(1-\frac1{\alpha_1^2}\right)\varphi',
\qquad
\varphi
=
\alpha_1\left(1+e^{-y_1/\alpha_1}\right)\varphi',
\]
we deduce from \eqref{eq:coercivity_raw} that
\begin{equation}\label{eq:main_lower_bound}
\begin{aligned}
&(3-\tilde\mu)\iint\tilde v_{y_1}^2\varphi'
+(1-\tilde\mu)\iint\tilde v_{y_2}^2\varphi'
+\left(1-\frac1{\alpha_1^2}-\tilde\mu\right)
 \iint\tilde v^2\varphi' \\
&\quad
-3\iint Q^2\tilde v^2\varphi'
+6\iint\alpha_1\left(1+e^{-y_1/\alpha_1}\right)
QQ_{y_1}\tilde v^2\varphi' \\
&\geq
-\frac1{\tilde\mu}
\Big[
(\tilde v,Q)^2
+(\tilde v,\varphi Q_{y_1})^2
+(\tilde v,\varphi Q_{y_2})^2
+(\tilde v,\varphi\Lambda Q)^2
\Big].
\end{aligned}
\end{equation}

\medskip
\noindent\textbf{Step 2: Localization errors.}
The identity
\[
\tilde v_{y_1}=v_{y_1}\zeta_B+v\zeta_B'
\]
and an integration by parts in the cross term give
\[
\begin{aligned}
\iint\tilde v_{y_1}^2\varphi'
&=
\iint\zeta_B^2v_{y_1}^2\varphi'
-\iint
\left(
\zeta_B\zeta_B''
+\frac1{\alpha_1}\zeta_B\zeta_B'
\right)v^2\varphi' \\
&\leq
\iint_{\{|y_1|<\frac B2\}}v_{y_1}^2\varphi'
+\frac{C(\zeta)}{B}
\iint_{\{\frac B4<|y_1|<\frac B2\}}v^2\varphi'.
\end{aligned}
\]
Moreover,
\[
\iint\tilde v_{y_2}^2\varphi'
\leq
\iint_{\{|y_1|<\frac B2\}}v_{y_2}^2\varphi',
\qquad
\iint\tilde v^2\varphi'
\leq
\iint_{\{|y_1|<\frac B2\}}v^2\varphi'.
\]

On the annulus $\{B/4<|y_1|<B/2\}$, the decay of $Q$ and $Q_{y_1}$
implies
\[
Q^2\lesssim e^{-\frac{B}{2\alpha_1}},
\qquad
\alpha_1\left(1+e^{-y_1/\alpha_1}\right)|QQ_{y_1}|
\lesssim e^{-\frac{B}{4\alpha_1}}.
\]
Consequently,
\begin{gather*}
\left|
\iint Q^2\tilde v^2\varphi'
-
\iint_{\{|y_1|<\frac B2\}}Q^2v^2\varphi'
\right|
\lesssim
e^{-\frac{B}{2\alpha_1}}
\iint_{\{\frac B4<|y_1|<\frac B2\}}v^2\varphi',
\\
\begin{split}
\bigg| \iint \alpha_1\left(1+e^{-y_1/\alpha_1}\right) QQ_{y_1}\tilde v^2\varphi'
&- \iint_{\{|y_1|<\frac B2\}} \alpha_1\left(1+e^{-y_1/\alpha_1}\right) QQ_{y_1}v^2\varphi' \bigg| \\
&\lesssim
e^{-\frac{B}{4\alpha_1}}
\iint_{\{\frac B4<|y_1|<\frac B2\}}v^2\varphi'.
\end{split}
\end{gather*}
\medskip
\noindent\textbf{Step 3: Orthogonality errors.}
Since $v$ satisfies the exact orthogonality conditions,
\[
(\tilde v,h)=-(v(1-\zeta_B),h)
\]
for
\[
h\in\{Q,\varphi Q_{y_1},\varphi Q_{y_2},\varphi\Lambda Q\}.
\]
Using $\varphi=1+\alpha_1\varphi'$ and Cauchy--Schwarz, we obtain
\[
\begin{aligned}
|(\tilde v,\varphi Q_{y_1})|^2
&\lesssim
\left(\iint v^2\phi_{i,B}'\right)
\left(
\iint_{\{|y_1|>\frac B4\}}
\frac{(\varphi')^2Q_{y_1}^2+Q_{y_1}^2}
     {\phi_{i,B}'}
\right).
\end{aligned}
\]
For $y_1<B/2$, the definition of $\phi_{i,B}$ gives
\[
\phi_{i,B}'\gtrsim
e^{-\frac{B}{2\alpha_1}}\varphi'.
\]
Using the decay of $Q_{y_1}$ from Lemma ~\ref{Qdecay}, we have
\[
\begin{aligned}
\iint_{\{y_1<-\frac B4\}}
\frac{(\varphi')^2Q_{y_1}^2}{\phi_{i,B}'}
&\lesssim e^{-\frac{B}{4\alpha_1}},
&
\iint_{\{y_1<-\frac B4\}}
\frac{Q_{y_1}^2}{\phi_{i,B}'}
&\lesssim e^{\frac{B}{4\alpha_1}},\\
\iint_{\{\frac B4<y_1<\frac B2\}}
\frac{(\varphi')^2Q_{y_1}^2}{\phi_{i,B}'}
&\lesssim e^{\frac{B}{4\alpha_1}},
&
\iint_{\{\frac B4<y_1<\frac B2\}}
\frac{Q_{y_1}^2}{\phi_{i,B}'}
&\lesssim e^{-\frac{B}{4\alpha_1}}.
\end{aligned}
\]
For $y_1>B/2$, the lower bound $\phi_{i,B}'\geq c_i>0,$ the relation
$\alpha_2=\alpha_1-\frac1{200}$ from \eqref{alphas} and
\eqref{pointwiseboundortho} yield
\[
\iint_{\{y_1>\frac B2\}}
\frac{(\varphi')^2Q_{y_1}^2}{\phi_{i,B}'}
\lesssim
e^{-\frac{B}{200\alpha_1\alpha_2}},
\qquad
\iint_{\{y_1>\frac B2\}}
\frac{Q_{y_1}^2}{\phi_{i,B}'}
\lesssim e^{-\frac{B}{\alpha_1}}.
\]
It follows that
\[
|(\tilde v,\varphi Q_{y_1})|^2
\lesssim
e^{\frac{B}{4\alpha_1}}
\iint v^2\phi_{i,B}'.
\]
The same argument, using the decay of $Q$ and its derivatives on the
left and the pointwise bounds \eqref{pointwiseboundortho} on the right,
gives
\[
|(\tilde v,Q)|^2
+|(\tilde v,\varphi Q_{y_2})|^2
+|(\tilde v,\varphi\Lambda Q)|^2
\lesssim
e^{\frac{B}{4\alpha_1}}
\iint v^2\phi_{i,B}'.
\]

\medskip
\noindent\textbf{Step 4: Conclusion.}
Choose $B_0$ sufficiently large that, for every $B\geq B_0$,
\[
\frac{C(\zeta)}{B}
+Ce^{-\frac{B}{4\alpha_1}}
\leq\frac{\tilde\mu}{2}.
\]
The localization estimates from Step 2, together with
$1-\alpha_1^{-2}-\tilde\mu>0$, then imply
\begin{align*}
&\iint_{\{|y_1|<\frac B2\}}
\Big[
(3-\tilde\mu)v_{y_1}^2
+(1-\tilde\mu)v_{y_2}^2
+\left(1-\frac1{\alpha_1^2}-\frac{\tilde\mu}{2}\right)v^2\\
&\hspace{39mm}
-3Q^2v^2
+6\alpha_1\left(1+e^{-y_1/\alpha_1}\right)QQ_{y_1}v^2
\Big]\varphi'\\
&\qquad\geq
\iint
\Big[
(3-\tilde\mu)\tilde v_{y_1}^2
+(1-\tilde\mu)\tilde v_{y_2}^2
+\left(1-\frac1{\alpha_1^2}-\tilde\mu\right)\tilde v^2\\
&\hspace{39mm}
-3Q^2\tilde v^2
+6\alpha_1\left(1+e^{-y_1/\alpha_1}\right)
QQ_{y_1}\tilde v^2
\Big]\varphi'.
\end{align*}
Combining this inequality with \eqref{eq:main_lower_bound} and the
projection estimates from Step 3 yields
\begin{align*}
&\iint_{\{|y_1|<\frac B2\}}
\Big[
(3-\tilde\mu)v_{y_1}^2
+(1-\tilde\mu)v_{y_2}^2
+\left(1-\frac1{\alpha_1^2}-\frac{\tilde\mu}{2}\right)v^2\\
&\hspace{39mm}
-3Q^2v^2
+6\alpha_1\left(1+e^{-y_1/\alpha_1}\right)QQ_{y_1}v^2
\Big]\varphi'\\
&\qquad\geq
-Ce^{\frac{B}{4\alpha_1}}\iint v^2\phi_{i,B}'.
\end{align*}
Finally, using
\[
\varphi'-\varphi'''
=
\left(1-\frac1{\alpha_1^2}\right)\varphi',
\qquad
\varphi
=
\alpha_1\left(1+e^{-y_1/\alpha_1}\right)\varphi',
\]
and restoring the positive terms that were subtracted on the left,
we obtain
\begin{align*}
&\iint_{\{|y_1|<\frac B2\}}
\Big[
(3v_{y_1}^2+v_{y_2}^2)\varphi'
+v^2(\varphi'-\varphi''')
-3Q^2v^2\varphi'
+6QQ_{y_1}v^2\varphi
\Big]\\
&\qquad\geq
\frac{\tilde\mu}{2}
\iint_{\{|y_1|<\frac B2\}}
(v_{y_1}^2+v_{y_2}^2+v^2)\varphi'
-Ce^{\frac{B}{4\alpha_1}}\iint v^2\phi_{i,B}'.
\end{align*}
Renaming $\tilde\mu/2$ as $\mu$ completes the proof.
\end{proof}

\chapter{Variational Compactness Near the Soliton Orbit}\label{appendixC}

 We prove the variational compactness result used to identify the blow-up behavior for non-positive energy. We restate it for convenience.

\begin{theorem} \label{orbitalstability2}
There exists $\alpha_{\text{orb}}>0$ such that the following holds. For all $0<\alpha'\leq \alpha_{\text{orb}},$ there exists $\delta=\delta(\alpha')>0,$ with $\delta(\alpha')\rightarrow 0$ as $\alpha'\rightarrow 0,$ such that for all $u \in H^1(\mathbb{R}^2), u \not\equiv 0,$ if
$$\|u\|_{L^2}^2-\|Q\|_{L^2}^2=\alpha(u)\leq \alpha' \quad \text{and} \quad E(u)\leq \alpha'\iint  |\nabla u|^2,$$
then there exists $\tilde{y}_1, \tilde{y}_2$ $\in \mathbb{R}$ and $\epsilon_0\in \{-1,1\}$ such that
$$\|Q-\epsilon_0\lambda_0u(\lambda_0 y_1 + \tilde{y}_1,\lambda_0 y_2 + \tilde{y}_2)\|_{H^1}\leq \delta(\alpha'),$$
with $$\lambda_0=\frac{\|\nabla Q\|_{L^2}}{\|\nabla u\|_{L^2}}.$$
\end{theorem}

In order to prove this result, we will go through several steps that appear in the works of Martel and Merle \cite{Merle01}, \cite{MartelMerle02}.

\textit{Claim 1.} The variational characterization of $Q$: for $v \in H^1(\mathbb{R}^2),$ if
\begin{equation}\label{eq:varcharac}
 0<\iint v^2\leq \iint  Q^2 \mbox{ and }E(v)\leq 0,   
\end{equation}
then there exists $\lambda_1>0, (z_1, z_2)\in \mathbb{R}^2, \epsilon_0\in \{\pm 1\}$ such that $v(y_1, y_2)=\epsilon_0\lambda_1Q(\lambda_1(y_1-z_1),\lambda_1(y_2-z_2)).$

This follows from the sharp Gagliardo--Nirenberg inequality
\cite{Weinstein1983}. The minimum of
\[
J(u)=\frac{\|\nabla u\|_{L^2}^2\|u\|_{L^2}^2}{\|u\|_{L^4}^4}
\]
is $C_{GN}^{-1}=\|Q\|_{L^2}^2/2$, and its real optimizers are
$aQ(b(\cdot-\mathbf z))$, with $a\ne0$, $b>0$, and
$\mathbf z\in\mathbb R^2$. Indeed, the hypotheses
\eqref{eq:varcharac} and the sharp inequality give
\[
0\ge E(v)\ge\frac12\|\nabla v\|_2^2
\left(1-\frac{\|v\|_2^2}{\|Q\|_2^2}\right)\ge0.
\]
Thus $\|v\|_2=\|Q\|_2$ and $v$ is an optimizer. Its mass
normalization forces $|a|=b$, which gives the stated form.

As a corollary of this claim, we get that for any $v \in H^1(\mathbb{R}^2),$ the condition
$$E(v)=0, \quad \iint v^2=\iint Q^2, \quad \iint |\nabla v|^2=\iint |\nabla Q|^2$$
is equivalent to
\begin{equation}\label{variational}
v(y_1, y_2)=\epsilon_0Q(y_1-z_1, y_2-z_2).
\end{equation}

\begin{lemma} \label{lemmastability} Let $(v_n)_{n\geq 1}$ be a sequence in $H^1(\mathbb{R}^2)$ such that
$$\iint v_{n}^{2}\rightarrow \iint Q^2 \mbox{ as }n\rightarrow \infty, \quad \iint |\nabla v_n|^2=\iint |\nabla Q|^2 \quad \text{and} \quad E(v_n)\to0.$$
Then there exists a sequence $\epsilon_n\in \{\pm 1\}$ and a sequence $(\mathbf{z}_n)_n$ in $\mathbb{R}^2$ such that
$$\epsilon_nv_n(\cdot+\mathbf{z}_n)\rightarrow Q \mbox{ in }H^1(\mathbb{R}^2).$$
\end{lemma}
\begin{proof}
We follow the compactness argument in the proof of
\cite[Lemma~1]{Merle01}.
Since $\|v_n\|_{L^4}^4\to2\|\nabla Q\|_{L^2}^2>0$,
Lions' vanishing lemma \cite[Lemma~I.1]{Lions1984II}
and local compactness give, after extraction and translation,
\[
 \widetilde v_n:=v_n(\cdot+\mathbf z_n)
 \rightharpoonup V\ne0\quad\text{in }H^1,
 \qquad \widetilde v_n\to V\quad\text{almost everywhere}.
\]
Set $w_n=\widetilde v_n-V$.
Weak convergence gives
\[
 0<\|V\|_{L^2}^2\le\|Q\|_{L^2}^2,
 \qquad
 \|w_n\|_{L^2}^2\to\|Q\|_{L^2}^2-\|V\|_{L^2}^2.
\]
Expanding the energy, we obtain
\[
 E(v_n)=E(V)+E(w_n)+\mathcal R_n,
\]
where
\[
 \mathcal R_n
 =\int_{\mathbb R^2}\nabla V\cdot\nabla w_n
 -\int_{\mathbb R^2}
 \left(V^3w_n+\frac32V^2w_n^2+Vw_n^3\right)
 \longrightarrow0.
\]
Indeed, the gradient term vanishes by weak convergence.
The nonlinear terms vanish by local compactness in $L^4$,
followed by H\"older's inequality on the complement of a
large ball, using $V\in L^4$ and the uniform $L^4$ bound
on $w_n$.
By the sharp Gagliardo--Nirenberg inequality
\cite{Weinstein1983}, for sufficiently large $n$,
\[
 E(w_n)\ge
 \frac12\left(1-\frac{\|w_n\|_{L^2}^2}{\|Q\|_{L^2}^2}\right)
 \|\nabla w_n\|_{L^2}^2
 \ge\frac{\|V\|_{L^2}^2}{4\|Q\|_{L^2}^2}
 \|\nabla w_n\|_{L^2}^2.
\]
Thus $E(V)\le0$. The same inequality applied to $V\ne0$
forces $\|V\|_{L^2}^2=\|Q\|_{L^2}^2$ and $E(V)=0$.
Consequently, $E(w_n)\to0$ and $w_n\to0$ in $H^1$.
The characterization \eqref{variational} identifies $V$
with $Q$ up to sign and translation.

Applying this argument to every subsequence shows that
\[
 d_n:=\inf_{\substack{\epsilon\in\{\pm1\}\\
                     \mathbf z\in\mathbb R^2}}
 \|\epsilon v_n(\cdot+\mathbf z)-Q\|_{H^1}\longrightarrow0.
\]
By the definition of $d_n$, for each $n$ we can choose
$\epsilon_n\in\{\pm1\}$ and $\mathbf z_n\in\mathbb R^2$
such that
\[
 \|\epsilon_n v_n(\cdot+\mathbf z_n)-Q\|_{H^1}
 <d_n+\frac1n.
\]
Since $d_n\to0$, these choices give the claimed strong
$H^1$ convergence.
\end{proof}

\begin{proof}[Proof of Theorem \ref{orbitalstability2}]
We proceed by contradiction. Suppose the statement is false. This means there exists $\delta_0 > 0$ and a sequence $\alpha_n \rightarrow 0$ (representing the bound $\alpha'$), such that for every $n$, there exists $u_n \in H^1(\mathbb{R}^2)$ satisfying the proximity conditions:
\begin{equation*}
\|u_n\|_{L^2}^2 - \|Q\|_{L^2}^2 \leq \alpha_n \quad \text{and} \quad E(u_n) \leq \alpha_n \iint |\nabla u_n|^2,
\end{equation*}
but which remains bounded away from the soliton orbit:
\begin{equation*}
\inf_{\lambda=\frac{\|\nabla Q\|_{L^2}}{\|\nabla u_n\|_{L^2}}, \mathbf{y} \in \mathbb{R}^2, \epsilon \in \{\pm 1\}} \|Q - \epsilon \lambda u_n(\lambda \cdot + \mathbf{y})\|_{H^1} \geq \delta_0.
\end{equation*}

We set $\lambda_n=\frac{\|\nabla Q\|_{L^2}}{\|\nabla u_n\|_{L^2}}$ and $v_n(\mathbf{y})=\lambda_n u_n(\lambda_n \mathbf{y}).$ We see that $\iint v_{n}^{2}=\iint u_{n}^{2}$ and $\iint |\nabla v_n|^2=\iint |\nabla Q|^2.$
By the sharp Gagliardo-Nirenberg inequality we get
\begin{align*}
\alpha_n\geq \frac{E(u_n)}{\iint |\nabla u_n|^2} &= \frac{1}{2} - \frac{1}{4} \frac{\|u_n\|_{L^4}^4}{\|\nabla u_n\|_{L^2}^2} \geq \frac{1}{2} - \frac{C_{GN}}{4} \|u_n\|_{L^2}^2 \\
&= \frac{1}{2}\Big(1-\frac{\iint u_{n}^{2}}{\iint Q^2}\Big)\geq -\frac{1}{2\|Q\|_{L^2}^2}\alpha_n.
\end{align*}
Therefore $(1-2\alpha_n)\|Q\|_{L^2}^2\leq \|u_n\|_{L^2}^2\leq \|Q\|_{L^2}^2+\alpha_n,$ 
and it follows that
$$\lim_{n\rightarrow \infty}\iint u_{n}^{2}=\lim_{n\rightarrow \infty}\iint v_{n}^{2}=\iint Q^2 \mbox{ and }\lim_{n\rightarrow \infty}\frac{E(u_n)}{\iint |\nabla u_n|^2}=0.$$

By direct calculations, we have:
\begin{align*}
E(v_n) = \frac{E(u_n)}{\iint |\nabla u_n|^2}\iint |\nabla Q|^2 \to 0.
\end{align*}
Therefore the sequence $(v_n)_{n\geq 1}$ satisfies the conditions of Lemma \ref{lemmastability}. Hence,
$\epsilon_n v_n(\cdot+\mathbf{z}_n)\rightarrow Q \mbox{ in } H^1 \mbox{ as }n\rightarrow \infty,$ which yields a contradiction.
\end{proof}

\chapter{An Analytic Lower Bound for the Rate Constant}\label{appendixD}
\begin{proposition}\label{c>1}
    Let $Q$ be the positive radial exponentially decaying solution of
$-\Delta Q+Q-Q^3=0$ on $\R^2$, and put
\[
 g_Q(x_2)=x_2\int_\R Q_{x_2}(x_1,x_2)\,d x_1.
\]
With the Fourier convention $\wh f(\xi)=\int_\R e^{-ix\xi}f(x)\,d x$, we have
\[
 c:=\frac{\displaystyle\int_\R \frac{2}{1+\xi^2}\,|\widehat{g_Q}(\xi)|^2\,d\xi}
 {\displaystyle\int_\R |\widehat{g_Q}(\xi)|^2\,d\xi}>1.
\]
\end{proposition}
\begin{proof}
    We write $z$ for the variable $x_2$ and define
\begin{equation*}
 h_Q(z):=\int_\R Q(x_1,z)\,d x_1,
 \qquad
 m_Q(z):=\int_\R Q(x_1,z)^3\,d x_1.
\end{equation*}
The smoothness and exponential decay of $Q$ imply that $h_Q$ and $m_Q$ are smooth and
exponentially decaying.  Integrating the equation for $Q$ in $x_1$ yields
\begin{equation}\label{eq:h-ode}
     0
 =\int_\R\bigl(-Q_{x_1x_1}-Q_{zz}+Q-Q^3\bigr)(x_1,z)\,d x_1 =-h_Q''(z)+h_Q(z)-m_Q(z). 
\end{equation}
Set
\begin{equation*}
 k(z):=e^{-|z|},\qquad \mathcal Cf:=k*f.
\end{equation*}
Since $(1-\partial_z^2)(k/2)=\delta_0$ in the sense of distributions, the unique decaying
solution of \eqref{eq:h-ode} is
\begin{equation*}
 h_Q=\frac12 k*m_Q.
\end{equation*}
Define
\begin{equation*}
 w_Q(z):=-g_Q(z)=-zh'_Q(z).
\end{equation*}
The function $h_Q$ is real and even, so $w_Q$ is real and even.  Hence $\wh {w_Q}$ is real and
$\bigl(\wh {w_Q}(\xi)\bigr)^2=|\wh {w_Q}(\xi)|^2$; moreover,
$\bigl(\wh {w_Q}(\xi)\bigr)^2=\bigl(\widehat{g_Q}(\xi)\bigr)^2$.  Our Fourier convention gives
\begin{equation*}
 \wh k(\xi)=\int_\R e^{-|z|}e^{-iz\xi}\,d z=\frac{2}{1+\xi^2}.
\end{equation*}
Consequently, by Plancherel,
\begin{equation}\label{eq:plancherel-reduction}
\begin{aligned}
&\int_\R
  \left(\frac{2}{1+\xi^2}-1\right)
  \widehat{g_Q}(\xi)^2\,d\xi \\
&\quad =
  \int_\R
  \bigl(\wh{\mathcal Cw_Q}(\xi)-\wh{w_Q}(\xi)\bigr)
  \wh{w_Q}(\xi)\,d\xi \\
&\quad =
  2\pi\int_\R
  w_Q(z)\bigl((\mathcal Cw_Q)(z)-w_Q(z)\bigr)\,dz.
\end{aligned}
\end{equation}
Thus it remains to show that the last integral is strictly positive. 

Write the radial soliton as $Q(x_1,z)=q(r),$ with $r=(x_1^2+z^2)^{\frac{1}{2}}.$ From Theorem 2 of \cite{GidasNiNirenberg}, we have that $q'(r)<0.$ We have 
$$m'_Q(z)=3z\int_{\R}q^2(r)\frac{q'(r)}{r}dx_1.$$ Since $m_Q(z)=M_Q(|z|)$ is even and strictly decreasing in $|z|,$ for every $\tau\in (0,M_Q(0)),$ there is a unique $a(\tau)>0$ such that $\{z:m_Q(z)>\tau\}=(-a(\tau),a(\tau)).$ Therefore, by a layer-cake decomposition, 
$$m_Q(z)=\int_{0}^{M_Q(0)}\textbf{1}_{\{m_Q(z)>\tau\}}d\tau=\int_{0}^{M_Q(0)}\textbf{1}_{\{|z|<a(\tau)\}}d\tau.$$
Denote $m_a=\textbf{1}_{\{|z|<a\}},$ $h_a=\frac{1}{2}k*m_a$ and $w_a=-zh_a'(z),$ with 
$$h_Q(z)=\int_{0}^{M_Q(0)}h_{a(\tau)}(z)d\tau, \quad w_Q(z)=\int_{0}^{M_Q(0)}w_{a(\tau)}(z)d\tau.$$
The first identity follows from Tonelli's theorem. Since
$|k'|\le k$ almost everywhere, we have $|h_a'|\le h_a$.
For a fixed $z_0$, the bound
$h_a(z)\le e^{|z-z_0|}h_a(z_0)$ supplies an integrable majorant
on a neighborhood of $z_0$. Differentiation under the layer integral
is therefore justified by dominated convergence, proving the second
identity.

\begin{lemma}
For $a>0$, the functions $h_a$ and $w_a$ are even.  For $z\ge0$,
\begin{equation}\label{eq:ha-formula}
 h_a(z)=
 \begin{cases}
 1-e^{-a}\cosh z, & 0\le z\le a,\\[2mm]
 \sinh(a)e^{-z}, & z\ge a,
 \end{cases}
\end{equation}
and
\begin{equation}\label{eq:wa-formula}
 w_a(z)=
 \begin{cases}
 e^{-a}z\sinh z, & 0\le z\le a,\\[2mm]
 z\sinh(a)e^{-z}, & z\ge a.
 \end{cases}
\end{equation}
In particular, $w_a(z)>0$ for every $z>0$.
\end{lemma}

\begin{proof}
    Since $h_a$ is the decaying distributional solution of
\begin{equation*}
 -h_a''+h_a=m_a,
\end{equation*}
the Green kernel formula gives immediately 
$h_a(z)=\frac{1}{2}\int_{-a}^{a}e^{-|z-\tilde{z}|}d\tilde{z}.$ 
For $0\leq z\leq a,$ 
 $h_a(z)=\frac{1}{2}\big(\int_{-a}^{z}e^{-(z-\tilde{z})}\,d\tilde{z}+\int_{z}^{a}e^{-(\tilde{z}-z)}d\tilde{z}\big)=1-e^{-a}\cosh z.$
 For $z\geq a,$
$$h_a(z)=\frac{1}{2}\int_{-a}^{a}e^{-(z-\tilde{z})}d\tilde{z}=e^{-z}\sinh a.$$
Differentiating it and using
$w_a=-zh_a'$ gives \eqref{eq:wa-formula}.  Both expressions in \eqref{eq:wa-formula} are
strictly positive when $z>0.$
\end{proof}

 The explicit formulas give $w_a\ge0$, and integration by parts yields
\[
\|w_a\|_{L^1}
=-\int_{\mathbb R}z h_a'(z)\,dz
=\int_{\mathbb R}h_a(z)\,dz=2a.
\]
Since $m_Q\in L^1$, the layer-cake formula also gives
\[
\int_0^{M_Q(0)}2a(\tau)\,d\tau=\|m_Q\|_{L^1}<\infty.
\]
Thus the representation of $w_Q$ converges in $L^1$, and Tonelli's
theorem applies to its convolution with the positive kernel $k$.
In particular,
\[
(\mathcal Cw_Q)(z)
=\int_0^{M_Q(0)}(\mathcal Cw_{a(\tau)})(z)\,d\tau.
\]
We now compute $\mathcal Cw_a$.

\begin{lemma}
For $a>0$ and $z\ge0$,
\begin{equation}\label{eq:Twa-formula}
 (\mathcal Cw_a)(z)=
 \begin{cases}
 \displaystyle
 \frac{e^{-a}}2\Bigl((a^2+a-z^2)\cosh z+z\sinh z\Bigr),
 &0\le z\le a,\\[3mm]
 \displaystyle
 \frac{e^{-z}}2\Bigl((z^2+z-a^2)\sinh a+a\cosh a\Bigr),
 &z\ge a.
 \end{cases}
\end{equation}
Consequently,
\begin{equation}\label{eq:Da-formula}
 \mathcal D_a(z):=(\mathcal Cw_a)(z)-w_a(z)=
 \begin{cases}
 \displaystyle
 \frac{e^{-a}}2\Bigl((a^2+a-z^2)\cosh z-z\sinh z\Bigr),
 &0\le z\le a,\\[3mm]
 \displaystyle
 \frac{e^{-z}}2\Bigl((z^2-z-a^2)\sinh a+a\cosh a\Bigr),
 &z\ge a.
 \end{cases}
\end{equation}
\end{lemma}

\begin{proof}
We first derive a convenient formula for convolution with $k$.  If $v$ is even and
integrable against $k(z-\cdot)$, then for $z\ge0$,
\begin{align}
 (k*v)(z)
 &=\int_0^\infty\bigl(e^{-|z-\tilde{z}|}+e^{-|z+\tilde{z}|}\bigr)v(\tilde{z})\,d \tilde{z} \notag\\
 &=\int_0^z\bigl(e^{-z+\tilde{z}}+e^{-z-\tilde{z}}\bigr)v(\tilde{z})\,d \tilde{z}
   +\int_z^\infty\bigl(e^{-\tilde{z}+z}+e^{-z-\tilde{z}}\bigr)v(\tilde{z})\,d \tilde{z} \notag\\
 &=2e^{-z}\int_0^z\cosh \tilde{z}\,v(\tilde{z})\,d \tilde{z}
   +2\cosh z\int_z^\infty e^{-\tilde{z}}v(\tilde{z})\,d \tilde{z}.             \label{eq:even-convolution}
\end{align}
We shall use the following elementary antiderivatives:
\begin{align}
 \mathcal K_1(z)&:=\int_0^z \tilde{z}\sinh \tilde{z}\cosh \tilde{z}\,d \tilde{z}
      =\frac{z\cosh(2z)}4-\frac{\sinh(2z)}8,               \notag\\
 \mathcal K_2(z)&:=\frac{z^2}{4}+\left(\frac z4+\frac18\right)e^{-2z},
 &\mathcal K'_2(z)&=ze^{-z}\sinh z,                                  \notag\\
 \mathcal K_3(z)&:=\frac{z^2}{4}-\left(\frac z4+\frac18\right)e^{-2z},
 &\mathcal K'_3(z)&=ze^{-z}\cosh z,                                  \notag\\
 \mathcal K_4(z)&:=\int_z^\infty \tilde{z}e^{-2\tilde{z}}\,d \tilde{z}
      =\left(\frac z2+\frac14\right)e^{-2z}.              \notag
\end{align}

Suppose first that $0\le z\le a$.  In the first integral of
\eqref{eq:even-convolution}, formula \eqref{eq:wa-formula} uses its interior branch.  In the
second integral, we split at $\tilde{z}=a$.  This gives
\begin{align*}
 (\mathcal Cw_a)(z)
 &=2e^{-z}\int_0^z \cosh \tilde{z}\,e^{-a}\tilde{z}\sinh \tilde{z}\,d \tilde{z} \notag\\
 &\quad+2\cosh z\int_z^a e^{-\tilde{z}}e^{-a}\tilde{z}\sinh \tilde{z}\,d \tilde{z}
      +2\cosh z\int_a^\infty e^{-\tilde{z}}\tilde{z}\sinh(a)e^{-\tilde{z}}\,d \tilde{z} \notag\\
 &=2e^{-a-z}\mathcal K_1(z)
   +2e^{-a}\cosh z\bigl(\mathcal K_2(a)-\mathcal K_2(z)\bigr)
   +2\sinh(a)\cosh z\,\mathcal K_4(a). \\                            
 &=\frac{e^{-a}}2
 \Bigl((a^2+a-z^2)\cosh z+z\sinh z\Bigr) 
\end{align*}

Now suppose that $z\ge a$.  We split the first integral of
\eqref{eq:even-convolution} at $\tilde{z}=a$; the second integral uses only the exterior branch of
\eqref{eq:wa-formula}.  Thus
\begin{align*}
 (\mathcal Cw_a)(z)
 &=2e^{-z}\int_0^a\cosh \tilde{z}\,e^{-a}\tilde{z}\sinh \tilde{z}\,d \tilde{z}
   +2e^{-z}\int_a^z\cosh \tilde{z}\,\tilde{z}\sinh(a)e^{-\tilde{z}}\,d \tilde{z} \notag\\
 &\quad+2\cosh z\int_z^\infty e^{-\tilde{z}}\tilde{z}\sinh(a)e^{-\tilde{z}}\,d \tilde{z} \notag\\
 &=2e^{-z}\left[e^{-a}\mathcal K_1(a)+\sinh(a)\bigl(\mathcal K_3(z)-\mathcal K_3(a)\bigr)\right]
   +2\sinh(a)\cosh z\,\mathcal K_4(z). \notag\\      
 &=\frac{e^{-z}}2
 \Bigl((z^2+z-a^2)\sinh a+a\cosh a\Bigr).
\end{align*}
 Finally, \eqref{eq:Da-formula} follows by subtracting the corresponding branch of
\eqref{eq:wa-formula}.
\end{proof}

In order to have positivity of the expression in \eqref{eq:plancherel-reduction}, it suffices to prove for every $a>0$ and every $z\in\R$,
\[
 (\mathcal Cw_a)(z)-w_a(z)>0.
\]

Since $w_a$ and $\mathcal Cw_a$ are even,  it suffices to take $z\ge0$.  If $0\le z\le a$, then
by \eqref{eq:Da-formula},
\begin{align*}
 2e^a\mathcal D_a(z)
 &=(a^2+a-z^2)\cosh z-z\sinh z \notag\\
 &=\cosh z\Bigl((a^2-z^2)+(a-z\tanh z)\Bigr).             
\end{align*}
Here $a^2-z^2\ge0$, while
\[
 a-z\tanh z\ge z-z\tanh z>0
\]
when $z>0$, and for $z=0$ it equals $a>0$.  Thus $\mathcal D_a(z)>0$ on $[0,a]$.

If $z\ge a$, then \eqref{eq:Da-formula} gives
\begin{equation*}
 2e^z\mathcal D_a(z)=\sinh a\,\mathcal P_a(z),
 \qquad
 \mathcal P_a(z):=z^2-z-a^2+a\coth a.
\end{equation*}
Since $\sinh a>0$, it remains to prove $\mathcal P_a(z)>0$.  If $a\ge\tfrac12$, then
$\mathcal P_a'(z)=2z-1\ge0$ for $z\ge a$, and hence
\[
 \mathcal P_a(z)\ge\mathcal P_a(a)=a(\coth a-1)>0.
\]
If $0<a<\tfrac12$, the minimum of $\mathcal P_a$ on $[a,\infty)$ occurs at $z=\tfrac12$.
Thus $a\coth a=a/\tanh a>1$, and
\[
 \mathcal P_a(z)\ge\mathcal P_a\!\left(\frac12\right)
 =a\coth a-a^2-\frac14
 >\frac34-a^2>0.
\]
Therefore $\mathcal D_a(z)>0$ also for $z\ge a$.

For every $z\ne0$, the layer representation gives
\[
w_Q(z)=\int_0^{M_Q(0)}w_{a(\tau)}(z)\,d\tau>0.
\]
The justified convolution interchange above gives,
for every $z\in\mathbb R$,

\[
(\mathcal Cw_Q)(z)-w_Q(z)
=\int_0^{M_Q(0)}\mathcal D_{a(\tau)}(z)\,d\tau>0.
\]
The final pairing is finite: $w_Q\in L^2$ by its exponential decay,
and convolution by $k\in L^1$ is bounded on $L^2$. Since its integrand
is positive for every $z\ne0$, we conclude
\[
\int_{\mathbb R}w_Q(z)\bigl((\mathcal Cw_Q)(z)-w_Q(z)\bigr)\,dz>0.
\]
\end{proof}

\chapter{Interval-Arithmetic Verification} \label{appendixE}

This appendix verifies, by interval arithmetic, the numerical assertions
used in the virial argument of Appendix~\ref{AppendixB} and in the
choice of the polynomial weight of \(\mathcal A_{\alpha_0}^{(21)}\).
The verification has two branches. The first
encloses the potential, establishes the spectral index and the spectral
gap of the virial operator, and bounds the determinant of the even
inverse Gram matrix and the odd inverse pairing required by the virial
argument; the second encloses the constant $c$ governing the blow-up
rate. Each finite approximation is accompanied by a bound for its entire
remainder. The calculations below use C++ with CAPD
 \cite{CAPDDynSys2021} for the ground state and scalar interval
computations, and verified binary64 matrix products for the spectral
tests. The analytic implications of the resulting interval bounds are
proved below.

This appendix is organized as follows. Section~\ref{cap:q-section}
describes the validation of the ground state and the bounds it
supplies. Section~\ref{cap:p-approximation} transfers these bounds to
an explicit approximation of the potential, and
Section~\ref{cap:s-section} derives from it the negative index and the
spectral gap of $\mathcal L$. Section~\ref{cap:m-section} encloses the
inverse pairings, and Section~\ref{cap:c-section} encloses the constant
$c$. The main conclusions are collected in
Theorem~\ref{cap:verified-main} below.

\textbf{Notation.} We keep the notation of the virial argument of
Appendix~\ref{AppendixB}. Thus
\[
 \alpha_1=\frac{101}{100},\qquad
 \varphi(y_1)=1+e^{y_1/\alpha_1},\qquad
 \widetilde\varphi(y_1)=\cosh\!\left(\frac{y_1}{2\alpha_1}\right),
 \qquad \Lambda Q=Q+y\cdot\nabla Q.
\]
The operator to be verified is
\begin{equation}\label{cap:operator}
 \mathcal L=-3\partial_{y_1}^2-\partial_{y_2}^2
                 +\kappa_{\alpha_1}+W,
 \qquad
 \kappa_{\alpha_1}=1-\frac{1}{4\alpha_1^2}
                  =\frac{7701}{10201},
\end{equation}
where
\begin{equation}\label{cap:potential}
 W=-3Q^2+6\alpha_1(1+e^{-y_1/\alpha_1})QQ_{y_1}.
\end{equation}
We write $r=|y|$, and we also write $Q(r)$ for the radial profile. All
functions and Hilbert spaces in the argument are real. For a
self-adjoint operator or a real symmetric matrix, $\operatorname{ind}_-$
denotes the dimension of its negative spectral subspace. Unless a domain
is indicated, two-dimensional inner products and norms are taken on
$\R^2$. For $u\in H^1(\R^2)$, the expression $(\mathcal Lu,u)$ denotes
the quadratic form
\begin{equation}\label{cap:form}
 \mathcal{B}(u,u)
 =3\|u_{y_1}\|_{L^2}^2+\|u_{y_2}\|_{L^2}^2
          +\int_{\R^2}(\kappa_{\alpha_1}+W)u^2,
\end{equation}
and $\mathcal B(u,v)$ denotes the associated symmetric bilinear form,
$\mathcal B(u,v)=3(u_{y_1},v_{y_1})+(u_{y_2},v_{y_2})
+((\kappa_{\alpha_1}+W)u,v)$;
the identity $\mathcal B(u,v)=(\mathcal Lu,v)$ holds for
$u\in H^2(\R^2)$ and $v\in H^1(\R^2)$. The four constraint functions
used in Appendix~\ref{AppendixB} are
\begin{equation}\label{cap:constraints}
 \begin{aligned}
 f_1&=e^{-y_1/(2\alpha_1)}Q,&
 f_2&=\widetilde\varphi Q_{y_1},\\
 f_3&=\widetilde\varphi\Lambda Q,&
 f_o&=\widetilde\varphi Q_{y_2}.
 \end{aligned}
\end{equation}
The first three functions are even in $y_2$, and the last one is odd.
Once invertibility has been established, we write
\[
 A_{ij}=(\mathcal L^{-1}f_i,f_j),\quad 1\le i,j\le3,
 \qquad m_{oo}=(\mathcal L^{-1}f_o,f_o),
\]
and we enclose the $3\times3$ matrix $A$ and the scalar $m_{oo}$.

The regularity and decay of the ground state are established in
Lemma~\ref{Qdecay}. In particular, the potential in
\eqref{cap:potential} is bounded and tends to zero exponentially, since
$2-\alpha_1^{-1}>0$. We use the following operator properties, as in
Appendix~\ref{AppendixB} and Remark~\ref{rem:parity}; for the standard
bounded-perturbation and essential-spectrum results, see
\cite[Chapter XIII]{ReedSimon}. The quantitative exterior bound needed
by the computation is proved in Section~\ref{cap:p-approximation}.

\begin{lemma}\label{cap:analytic-operator}
The operator $\mathcal L$, with domain $H^2(\R^2)$, is self-adjoint
on $L^2(\R^2)$. Its form domain is $H^1(\R^2)$ and
\[
 \sigma_{\mathrm{ess}}(\mathcal L)
      =[\kappa_{\alpha_1},\infty).
\]
Moreover, $f_1,f_2,f_3,f_o\in H^1(\R^2)$ and reflection in $y_2$
commutes with $\mathcal L$.
\end{lemma}

An interval $[a]$ always denotes a closed enclosure of the exact
quantity $a$. For a scalar interval $I$, we write $I^-$ and $I^+$ for
its endpoints and $\operatorname{rad}I=(I^+-I^-)/2$. An interval matrix
encloses every entry of one exact matrix, and the symmetry of that exact
matrix is retained in all matrix tests. The polynomial coefficients,
mesh points, and preconditioners defining a numerical approximation are
specified as exact rational numbers or enclosed together with their
rounding errors. Finite integrals are enclosed cell by cell: if an
interval $[F]_K$ contains the range of an integrand on a cell $K$, then
$|K|[F]_K$ encloses its integral, and polynomial integration and Taylor
remainders give sharper versions of the same inclusion. This elementary
enclosure principle, together with the explicit infinite-tail bounds
below, gives a finite meaning to every certificate quantity.

\begin{theorem}[Verified conclusions]\label{cap:verified-main}
The executed  ground-state construction and  interval certificates below
 prove the following conclusions for the operator \eqref{cap:operator}.
\begin{itemize}
\item[(a)] The operator $\mathcal L$ has exactly two negative
eigenvalues, both simple, $0\in\rho(\mathcal L)$, and

\[
 \operatorname{ind}_-(\mathcal L)=2,\qquad
 \operatorname{dist}(0,\sigma(\mathcal L))
       \ge0.62076895409267.
\]

 \item[(b)] The matrix $A$ is enclosed by the interval matrix in
Section~\ref{cap:m-section}, and

\[
\begin{aligned}
 \det A&\in[-7.946966743614993375,-1.232927442846774158],\\
 m_{oo}&\in[-0.669373,-0.624397].
\end{aligned}
\]
\item[(c)] The blow-up rate constant satisfies

\[
 c\in[1.66126332085379,1.66170874135113]
       \subset\left(\frac{21}{13},\frac53\right).
\]
\end{itemize}
\end{theorem}
\begin{proof}
The CAPD construction verifies Proposition~\ref{cap:q-certificate}.
The  spectral certificate of Section~\ref{cap:s-section} then verifies the
hypotheses of Proposition~\ref{cap:spectral-gap}, which gives (a); the
simplicity of the two negative eigenvalues is proved after that
proposition. The compact spline residual certificate of
Section~\ref{cap:m-section}, through \eqref{cap:m-quadratic-error},
encloses $A$ and $m_{oo}$ and gives (b). Finally, the quotient
certificate of Section~\ref{cap:c-section} proves both bounds in (c).
\end{proof}

\section{Verification of the Ground State}\label{cap:q-section}

We first describe the enclosures of the radial profile $Q$; the idea
comes from \cite{DF2024}. The construction joins a solution regular at
the origin to a decaying solution at a fixed radius $R$. Both solutions
and their parameter derivatives are enclosed, and the latter give the
Jacobian of the matching equations. All interval operations below,
including the evaluations of elementary functions, are understood with
outward rounding.

\subsection*{The Regular Solution at the Origin}

For a real parameter $a$, we consider
\begin{equation}\label{cap:q-origin-ode}
 Q_0''+r^{-1}Q_0'-Q_0+Q_0^3=0,
 \qquad Q_0(0;a)=a,\qquad Q_0'(0;a)=0.
\end{equation}
Its derivative $P=\partial_aQ_0$ with respect to the parameter satisfies
\begin{equation}\label{cap:q-origin-sensitivity}
 P''+r^{-1}P'-P+3Q_0^2P=0,
 \qquad P(0;a)=1,\qquad P'(0;a)=0.
\end{equation}
We write
\begin{equation}\label{cap:q-series}
 Q_0(r;a)=\sum_{n\geq0}A_n(a)r^{2n},\qquad
 P(r;a)=\sum_{n\geq0}B_n(a)r^{2n}.
\end{equation}
The coefficients are determined by $A_0=a$, $B_0=1$, and the
recurrences
\begin{equation}\label{cap:q-series-recurrence}
 A_{n+1}=\frac{A_n-\sum_{i+j+k=n}A_iA_jA_k}{4(n+1)^2},\qquad
 B_{n+1}=\frac{B_n-3\sum_{i+j+k=n}A_iA_jB_k}{4(n+1)^2}.
\end{equation}
We now give uniform remainder bounds for these series. Fix an interval
$[a]$, let $a_*\geq\sup_{a\in[a]}|a|$, and set
\[
 b=\frac{1+a_*^2}{4},\qquad b_P=5b.
\]
Then
\begin{equation}\label{cap:q-coefficient-bounds}
 |A_n(a)|\leq a_*b^n,\qquad |B_n(a)|\leq b_P^n
 \quad(a\in[a],\ n\geq0).
\end{equation}
Indeed, if these estimates hold up to order $n$, the cubic sum is at
most $a_*^3b^n\binom{n+2}{2}\leq a_*^3b^n(n+1)^2$, while the
convolution in the sensitivity equation is bounded by
\[
 a_*^2b_P^n\sum_{k=0}^n(k+1)5^{-k}
 \leq\frac{25}{16}a_*^2b_P^n.
\]
Consequently,
\[
 |B_{n+1}|\leq
 b_P^n\frac{1+75a_*^2/16}{4(n+1)^2}\leq b_P^{n+1},
\]
where the last inequality follows from
$1+75a_*^2/16\leq5(1+a_*^2)$; the bound for $A_{n+1}$ follows in the
same way from its recurrence.

Choose $0<r_0\leq\tilde r_0$ with $b_P\tilde r_0^2<1$; the larger
radius allows the same remainder bound on the origin interval and on an
overlap with the continuation. For $N\geq0$, $d>0$, and
$0\leq t<d^{-1/2}$, set
\[
 G_N(t;d)=\frac{(dt^2)^{N+1}}{1-dt^2}.
\]
If $S_{Q,N}=\sum_{n=0}^NA_nr^{2n}$ and
$S_{P,N}=\sum_{n=0}^NB_nr^{2n}$, then, for every required derivative
order $j$,
\begin{equation}\label{cap:q-origin-remainders}
 \begin{split}
 \sup_{0\leq r\leq\tilde r_0}|\partial_r^j(Q_0-S_{Q,N})(r)|
 &\leq a_*\partial_t^jG_N(\tilde{r}_0;b),\\
 \sup_{0\leq r\leq\tilde r_0}|\partial_r^j(P-S_{P,N})(r)|
 &\leq\partial_t^jG_N(\tilde{r}_0;b_P).
 \end{split}
\end{equation}
These bounds follow by differentiating the convergent majorant series
with nonnegative coefficients; in particular, the right-hand sides are
explicit rational functions of $\tilde{r}_0$, $b$, and $b_P$. Uniform
convergence of the series and of their derivatives verifies
\eqref{cap:q-origin-ode} and \eqref{cap:q-origin-sensitivity}.
Uniqueness at the origin also follows from the integral equation
\[
 Q_0(r)=a+\int_0^r\frac1t\int_0^t s\bigl(Q_0(s)-Q_0(s)^3\bigr)\,ds\,dt:
\]
setting $B_*=a_*+1$ and choosing $0<r_*\leq r_0$ such that
\[
 \frac{r_*^2(B_*+B_*^3)}4<1,\qquad
 \frac{r_*^2(1+3B_*^2)}4<1,
\]
the right-hand side maps the ball $\|u-a\|_{C([0,r_*])}\leq1$ into
itself and is a contraction there, since the integral of the constant
function $1$ equals $r^2/4$. These inequalities hold uniformly for
$a\in[a]$ and prove uniqueness near the origin.

Each $A_n$ is a polynomial in $a$. Differentiating its recurrence and
grouping the three equal cubic sums gives the recurrence for $B_n$, so
that $A_n'=B_n$ by induction. The bound for $B_n$ makes the
parameter-derivative series uniformly convergent on
$[a]\times[0,\tilde r_0]$, and since the strict radius inequality
persists after a small enlargement of $[a]$, termwise differentiation
proves $P=\partial_aQ_0$, including at the endpoints of the parameter
interval. For $r>0$, local existence and uniqueness follow from the
Picard--Lindel\"of theorem \cite[Theorem~2.2]{Teschl2012}, and the
continuation to $R$ is justified by the estimates below.

\subsection*{Validated Continuation on a Finite Interval}

The origin series supplies an enclosure at $r_0>0$ for each $a\in[a]$.
Away from the origin, we propagate the four quantities
$Y=(Q_0,Q_0',P,P')$ through  $Y'=F(r,Y)$, where
\begin{equation}\label{cap:q-four-system}
 F(r,(q,v,p,t))=
 \bigl(v,\ q-q^3-v/r,\ t,\ (1-3q^2)p-t/r\bigr).
\end{equation}
The parameter $a$ enters through the initial state. Smooth dependence
on the initial data \cite[Theorem~2.11]{Teschl2012}, together with the
origin construction  and uniqueness for the differentiated equation,
shows  that the last two components remain $\partial_aQ_0$ and
$\partial_aQ_0'$. The  letter  $t$ in this display is a state coordinate;
 the independent variable is $r$.  In particular, the Jacobian needed
for matching requires the flow of the whole initial interval, including
its  parameter dependence.

 We use the validated Taylor integrator of CAPD::DynSys
\cite{CAPDDynSys2021}. The following  form of its remainder inclusion
explains both existence on a complete step and the interval polynomial
exported to the subsequent computations. Since $r\ge r_0>0$, the vector
field is smooth on every bounded set used below. Define its normalized
derivative functions recursively by

\[
 F^{[0]}(r,Y)=Y,\qquad
 F^{[j+1]}=
 \frac{\partial_rF^{[j]}+D_YF^{[j]}F}{j+1}.
\]

Along an exact solution, $F^{[j]}(r,Y(r))=Y^{(j)}(r)/j!$.
Automatic differentiation evaluates these functions and their interval
extensions without differentiating an unknown numerical remainder.

 \begin{lemma}[Validated Taylor step] \label{cap:q-tube}
Let $I_k=[s_k,s_{k+1}]\subset[r_0,R]$ , let
$h_k=s_{k+1}-s_k>0$,  and let  $[Y_k]$ enclose the incoming states.
Choose an order $n\ge1$ and coefficient boxes $[C_{j,k}]$ such that

\[
 F^{[j]}(s_k,[Y_k])\subset[C_{j,k}],\qquad 0\le j\le n.
\]
Let $[\Pi_k]$ contain all vectors
$\sum_{j=0}^n c_j\tau^j$ with $c_j\in[C_{j,k}]$ and
$0\le\tau\le h_k$. Choose a box $[E_k]$ with
$0\in\operatorname{int}[E_k]$, and set
$[Z_k]=[\Pi_k]+[E_k]$. If

\begin{equation}\label{cap:q-step-condition}
 [0,h_k]^{n+1}[R_k]\subset\operatorname{int}[E_k],
 \qquad
 F^{[n+1]}(I_k,[Z_k])\subset[R_k],
\end{equation}

then each incoming state has a unique solution throughout $I_k$, and

\begin{equation}\label{cap:q-step-enclosure}
 Y(s_k+\tau)\in
 \sum_{j=0}^n[C_{j,k}]\tau^j+[R_k]\tau^{n+1},
 \qquad 0\le\tau\le h_k.
\end{equation}
\end{lemma}
\begin{proof}
Local existence and uniqueness follow from the Picard--Lindel\"of
theorem \cite[Theorem~2.2]{Teschl2012}. The initial state belongs to
$\operatorname{int}[Z_k]$. While its solution remains in $[Z_k]$,
Taylor's formula with a componentwise Lagrange remainder gives
\eqref{cap:q-step-enclosure}. The polynomial part belongs to
$[\Pi_k]$ and, by  \eqref{cap:q-step-condition},  the  remainder belongs
to $\operatorname{int}[E_k]$. Thus the solution belongs to
$\operatorname{int}[Z_k]$, also at any alleged first exit time.
No first exit is possible. Boundedness in $[Z_k]$ and $s_k>0$ permit
continuation to $s_{k+1}$ \cite[Corollary~2.15]{Teschl2012}.
The same argument applies to every initial state in $[Y_k]$.
 \end{proof}

CAPD encloses the Taylor polynomial, predicts a remainder box, and
recomputes the order-$n+1$ coefficient on the resulting state box.
Its high-order enclosure routine accepts a step only after the strict
remainder inclusion in \eqref{cap:q-step-condition}. If the current
box does not pass, it is enlarged and the inclusion is tested again;
an unresolved step stops the calculation. The endpoint set is propagated
using CAPD's change-of-coordinates representation to control repeated
overestimation of correlations. This representation does not replace
the complete-step inclusion. These operations are performed with interval
coefficients and outward arithmetic.

In the executed calculation, $r_0=1/512$, $R=15$, and $n=20$.
Every time endpoint and time difference is checked to be an exact
dyadic point. The step is chosen from dyadic values between $1/4096$
and $1/8$,  and  the continuation is checked to reach exactly $R$.
The four-dimensional continuation encloses $P(R;a)$ and $P'(R;a)$
for every $a$ in the matching search interval. A separate
two-dimensional continuation encloses the residual at the proposed
matching center. After the strict matching inclusion is established,
the two-dimensional flow is recomputed from the new parameter interval
and supplies the profile used in the remaining certificates.

For each of its $184$ cells, the saved profile contains the interval
coefficients $[C_{j,k}]$ for the first two components and the
order-$21$ remainder coefficient $[R_k]$. Formula
\eqref{cap:q-step-enclosure}, evaluated by Horner's rule, encloses
$Q_0$ and $Q_0'$ on any subinterval of that cell. The second derivative
is evaluated from $Q_0''=Q_0-Q_0^3-Q_0'/r$, and further derivatives
are obtained from the differential equation. No derivative of a
saved remainder is taken. Eight interval subcells covering each full
cell verify $Q_0>0$ and $Q_0'<0$; these are range tests, not tests at
sampled points. The origin series and the exterior fixed point cover
the remaining radii.

\subsection*{Modified Bessel Functions}

We use the standard modified Bessel functions $K_\nu$ and $I_\nu$.
Their differential equation, derivative identities, and Wronskian
\cite[10.25.1, 10.29.3, 10.28.2]{DLMF} give
\begin{equation}\label{cap:bessel-identities}
 K_0'=-K_1,\quad I_0'=I_1,\quad
 K_0''=K_0+\frac{K_1}{r},\quad
 I_0''=I_0-\frac{I_1}{r},\quad
 K_0I_1+K_1I_0=\frac1r;
\end{equation}
in particular, second radial derivatives are evaluated through these
identities. For the decaying functions we use the large-argument
expansions with explicit remainders. Define $a_0^{K}=b_0^{K}=1$ and
\[
 a_{n+1}^{K}=\frac{(2n+1)^2}{8(n+1)}a_n^{K},\qquad
 b_{n+1}^{K}=\frac{4-(2n+1)^2}{8(n+1)}b_n^{K}.
\]
Then, for $r>0$ and $N\geq1$,
\begin{equation}\label{cap:bessel-K-remainders}
 \begin{aligned}
 K_0(r)&=\sqrt{\frac\pi{2r}}e^{-r}
 \left(\sum_{n=0}^{N-1}(-1)^na_n^{K}r^{-n}+\rho_{0,N}(r)\right),\\
 |\rho_{0,N}(r)|&\leq a_N^{K}r^{-N},\\
 K_1(r)&=\sqrt{\frac\pi{2r}}e^{-r}
 \left(\sum_{n=0}^{N-1}b_n^{K}r^{-n}+\rho_{1,N}(r)\right),\\
 |\rho_{1,N}(r)|&\leq |b_N^{K}|r^{-N}.
 \end{aligned}
\end{equation}
Indeed, $(-1)^na_n^{K}$ and $b_n^{K}$ are the coefficients of Hankel's
expansions of $K_0$ and $K_1$ \cite[10.40.2]{DLMF}, and for real
positive argument the remainder after $N\geq1$ terms is bounded in
absolute value by the first neglected term and has the same sign
\cite[\S10.40(ii)]{DLMF}. The sign information gives, for every $r>0$,
\begin{equation}\label{cap:bessel-K-first-order}
 \left(1-\frac1{8r}\right)\sqrt{\frac\pi{2r}}e^{-r}
 \leq K_0(r)\leq\sqrt{\frac\pi{2r}}e^{-r},
 \qquad
 K_1(r)\leq
 \left(1+\frac3{8r}\right)\sqrt{\frac\pi{2r}}e^{-r},
\end{equation}
since the first neglected terms are $-1/(8r)$ for $K_0$ and $3/(8r)$
for $K_1$.

The following uniform bounds suffice on the entire exterior interval.
Set
\begin{equation}\label{cap:bessel-uniform-constants}
 k_0=\sqrt{\pi/2},\qquad
 k_1=k_0\left(1+\frac3{8R}\right),\qquad
 i_0=\sqrt{\pi/8}.
\end{equation}
Then, for $r\geq R>0$,
\begin{equation}\label{cap:bessel-uniform-bounds}
 0<K_0(r)\leq k_0\frac{e^{-r}}{\sqrt r},\quad
 0<K_1(r)\leq k_1\frac{e^{-r}}{\sqrt r},\quad
 0<I_1(r)\leq I_0(r)\leq i_0\frac{e^r}{\sqrt r}.
\end{equation}
The bounds for $K_0$ and $K_1$ follow from
\eqref{cap:bessel-K-first-order}, and their positivity from the
integral representation \cite[10.32.9]{DLMF}. For $I_0$, we use
\cite[10.32.1]{DLMF}, namely
\[
 I_0(r)=\frac1\pi\int_0^\pi e^{r\cos\theta}\,d\theta,
 \qquad \cos\theta\leq1-\frac{2\theta^2}{\pi^2}
 \quad(0\leq\theta\leq\pi),
\]
and extend the resulting Gaussian integral to infinity; the cosine
inequality follows from $\sin(\theta/2)\geq\theta/\pi$. Finally, the
integral representation \cite[10.32.3]{DLMF} gives $|I_1|\leq I_0$,
and the positive series \cite[10.25.2]{DLMF} gives $I_1>0$.

\subsection*{The Decaying Solution and Its Parameter Derivative}

Let $\Gamma=[\gamma_-,\gamma_+]\subset(0,\infty)$; in what follows,
$\gamma$ denotes the coefficient at the matching radius $R$. Define
\begin{equation}\label{cap:q-tail-space}
 \mathcal X_R=\left\{u\in C([R,\infty)):
 \|u\|_{\mathcal X_R}:=\sup_{r\geq R}e^r\sqrt r\,|u(r)|<\infty\right\}.
\end{equation}
Multiplication by $e^r\sqrt r$ identifies $\mathcal X_R$ isometrically
with the Banach space of bounded continuous functions on
$[R,\infty)$, so that $\mathcal X_R$ is complete. For
$u\in\mathcal X_R$, set
\begin{equation}\label{cap:q-tail-map}
 \begin{split}
 T_\gamma u(r)={}&\gamma K_0(r)
 +K_0(r)\int_R^r I_0(s)s u(s)^3\,ds\\
 &+I_0(r)\int_r^\infty K_0(s)s u(s)^3\,ds.
 \end{split}
\end{equation}
Let
\begin{equation}\label{cap:q-contraction-constants}
 d_R=\frac{3k_0i_0}{4R}e^{-2R},\qquad
 d_{R,1}=\frac{i_0(2k_1+k_0)}{4R}e^{-2R},\qquad
 \theta_R=3d_Rz^2.
\end{equation}
The two finite inequalities
\begin{equation}\label{cap:q-tail-pass}
 \gamma_+k_0+d_Rz^3<z,\qquad \theta_R<1
\end{equation}
give a unique fixed point $Q_\infty(\,\cdot\,;\gamma)$ in
$\{u:\|u\|_{\mathcal X_R}\leq z\}$ for every $\gamma\in\Gamma$.
Indeed, for $\|u\|_{\mathcal X_R}\leq z$,
\[
 I_0(s)s|u(s)|^3\leq i_0z^3s^{-1}e^{-2s},\qquad
 K_0(s)s|u(s)|^3\leq k_0z^3s^{-1}e^{-4s},
\]
and since $\int_t^\infty s^{-1}e^{-as}ds\leq e^{-at}/(at)$, the two
nonlinear terms in \eqref{cap:q-tail-map} have norms at most
$k_0i_0z^3e^{-2R}/(2R)$ and
$k_0i_0z^3e^{-2R}/(4R)$. The same estimates and
\[
 |u(s)^3-v(s)^3|
 \leq3z^2\frac{e^{-2s}}s|u(s)-v(s)|
\]
give the Lipschitz constant $\theta_R$ on the ball. Banach's
fixed-point theorem \cite[Theorem~2.1]{Teschl2012} proves the
assertion, together with
\begin{equation}\label{cap:q-tail-errors}
 \|Q_\infty-\gamma K_0\|_{\mathcal X_R}\leq d_Rz^3.
\end{equation}
For brevity, we write
\[
 A(r)=\int_R^r I_0(s)sQ_\infty(s)^3\,ds,\qquad
 B(r)=\int_r^\infty K_0(s)sQ_\infty(s)^3\,ds.
\]
Differentiating \eqref{cap:q-tail-map} and using
\eqref{cap:bessel-identities}, we obtain
\begin{equation}\label{cap:q-tail-derivatives}
 \begin{split}
 Q_\infty'&=-K_1(\gamma+A)+I_1B,\\
 Q_\infty''&=K_0''(\gamma+A)+I_0''B-Q_\infty^3.
 \end{split}
\end{equation}
The first differentiation cancels the two endpoint terms, while in the
second their sum equals
$-r(K_1I_0+I_1K_0)Q_\infty^3=-Q_\infty^3$.
These formulas verify the ground-state equation on $[R,\infty)$, and
they also show that
\begin{equation}\label{cap:q-tail-derivative-error}
 \|Q_\infty'+\gamma K_1\|_{\mathcal X_R}\leq d_{R,1}z^3.
\end{equation}
Thus valid tail constants for the first three orders are
\begin{equation}\label{cap:q-tail-constants}
 C_{Q,0}=z,\qquad C_{Q,1}=\gamma_+k_1+d_{R,1}z^3,\qquad
 C_{Q,2}=z+R^{-1}C_{Q,1}+R^{-1}e^{-2R}z^3,
\end{equation}
in the sense that $|Q_\infty^{(j)}(r)|\leq C_{Q,j}e^{-r}/\sqrt r$.
The matching coefficient and the asymptotic coefficient are related by
\begin{equation}\label{cap:q-asymptotic-amplitude}
 \lim_{r\to\infty}\frac{Q_\infty(r;\gamma)}{K_0(r)}
 =\gamma+\int_R^\infty I_0(s)sQ_\infty(s;\gamma)^3\,ds.
\end{equation}
Indeed, the integral converges by the bounds above, while the lower
bound for $K_0$ in \eqref{cap:bessel-K-first-order} shows that the
remaining quotient $I_0(r)B(r)/K_0(r)$ tends to zero.

We next justify the differentiation with respect to the parameter in
the weighted space. Write $\mathcal N(u)=T_\gamma u-\gamma K_0$, which
is independent of $\gamma$. The same integral bounds show that
$\mathcal N$ is a continuously differentiable cubic map on
$\mathcal X_R$, with
\[
\begin{aligned}
 \|D\mathcal N(u)\|&\leq3d_R\|u\|_{\mathcal X_R}^2,\\
 \|\mathcal N(u+v)-\mathcal N(u)-D\mathcal N(u)v\|_{\mathcal X_R}
 &\leq d_R\|v\|_{\mathcal X_R}^2
                 (3\|u\|_{\mathcal X_R}+\|v\|_{\mathcal X_R}).
\end{aligned}
\]
Indeed, the derivative replaces $u^3$ in both integrals by $3u^2v$,
and expanding the cubic proves the displayed remainder bound and the
continuity of the derivative. At $u=Q_\infty(\cdot;\gamma)$ we have
$\|D\mathcal N(u)\|\leq\theta_R<1$, so that $I-D\mathcal N(u)$ is
invertible, with inverse given by the Neumann series and of norm at
most $(1-\theta_R)^{-1}$. The implicit function theorem, applied to
the $C^1$ map $(\gamma,u)\mapsto u-\gamma K_0-\mathcal N(u)$ on
$\R\times\mathcal X_R$, therefore gives a $C^1$ branch of fixed points
through $Q_\infty(\cdot;\gamma)$. Since
$\|Q_\infty\|_{\mathcal X_R}\leq\gamma_+k_0+d_Rz^3<z$ by
\eqref{cap:q-tail-pass}, this branch stays in the ball for nearby
parameters and coincides there with $Q_\infty$ by uniqueness. Thus
$Q_\infty$ is $C^1$ as an $\mathcal X_R$-valued function of $\gamma$,
with $\partial_\gamma Q_\infty=(I-D\mathcal N(Q_\infty))^{-1}K_0$;
since the differentiated radial integral operator satisfies the same
bounds with $d_R$ replaced by $d_{R,1}$, the same argument applies to
$Q_\infty'$.

Writing $Z=\partial_\gamma Q_\infty$, we obtain
\begin{equation}\label{cap:q-tail-sensitivity}
 \begin{split}
 Z(r)={}&K_0(r)
 +3K_0(r)\int_R^r I_0(s)sQ_\infty(s)^2Z(s)\,ds\\
 &+3I_0(r)\int_r^\infty K_0(s)sQ_\infty(s)^2Z(s)\,ds,\\
 Z'(r)={}&-K_1(r)
 -3K_1(r)\int_R^r I_0(s)sQ_\infty(s)^2Z(s)\,ds\\
 &+3I_1(r)\int_r^\infty K_0(s)sQ_\infty(s)^2Z(s)\,ds.
 \end{split}
\end{equation}
Equivalently, the first equation is solved by the convergent Neumann
series of a linear operator of norm at most $\theta_R$; in particular,
\begin{equation}\label{cap:q-tail-sensitivity-bounds}
 \begin{split}
 C_Z:=\frac{k_0}{1-\theta_R},\qquad
 \|Z\|_{\mathcal X_R}&\leq C_Z,\\
 \|Z-K_0\|_{\mathcal X_R}&\leq3d_Rz^2C_Z,\qquad
 \|Z'+K_1\|_{\mathcal X_R}\leq3d_{R,1}z^2C_Z.
 \end{split}
\end{equation}
Finally, positivity and monotonicity hold on the whole tail under the
additional finite condition
\begin{equation}\label{cap:q-tail-positive}
 ze^{-R}/\sqrt R<1.
\end{equation}
Indeed, $T_\gamma$ preserves nonnegative functions, so that its
iteration gives $Q_\infty\geq\gamma K_0>0$, while
\eqref{cap:q-tail-positive} gives $Q_\infty<1$ there. Since the
derivative bound gives $rQ_\infty'(r)\to0$, integration of the
differential equation shows that
\[
 rQ_\infty'(r)
 =-\int_r^\infty sQ_\infty(s)(1-Q_\infty(s)^2)\,ds<0.
\]

\subsection*{Matching and Identification of the Profile}

We apply the interval Newton method to the two matching equations;
see \cite[Chapter~8]{MooreKearfottCloud2009}. Define the matching map on
$X=[a]\times\Gamma$ by
\begin{equation}\label{cap:q-matching-map}
 \mathcal F(a,\gamma)=
 \begin{pmatrix}Q_0(R;a)-Q_\infty(R;\gamma)\\
 Q_0'(R;a)-Q_\infty'(R;\gamma)\end{pmatrix},\qquad
 D\mathcal F=
 \begin{pmatrix}P(R;a)&-Z(R;\gamma)\\
 P'(R;a)&-Z'(R;\gamma)\end{pmatrix}.
\end{equation}
Let $X$ be a box  with positive radii, fix an explicitly represented
point $x_0\in\operatorname{int}X$, and let
$J$ be an interval matrix enclosing $D\mathcal F(X)$; we require that
every real matrix in $J$ be nonsingular. An interval linear solve gives
an enclosure $K$ satisfying
\begin{equation}\label{cap:q-krawczyk}
 K\supset
 \{x_0-A^{-1}y:A\in J,\ y\in[\mathcal F(x_0)]\}.
\end{equation}
The interval Newton inclusion
\begin{equation}\label{cap:q-krawczyk-pass}
 K\subset\operatorname{int}X
\end{equation}
then proves that $\mathcal F$ has a unique zero in $X$, and that this
zero lies in $K$. Indeed, the convexity of the interval matrix implies
\[
 \overline J(x):=\int_0^1D\mathcal F(x_0+s(x-x_0))\,ds\in J.
\]
The continuous map
$x\mapsto x_0-\overline J(x)^{-1}\mathcal F(x_0)$ sends $X$ into
$K\subset X$; Brouwer's theorem gives a fixed point, and the integral
mean-value identity shows that it is a zero. If two points of $X$ were
zeros, the mean derivative on the segment joining them would belong to
$J$ and annihilate their difference, so that nonsingularity proves
uniqueness.

It remains to identify the matched solution. We require positive lower
bounds for $Q_0$ and negative upper bounds for $Q_0'$ on every
continuation cell in $[r_0,R]$. At the origin, a convenient sufficient
condition is
\begin{equation}\label{cap:q-origin-positive}
 a_- -\frac{a_*br_0^2}{1-br_0^2}>1,
 \qquad a_-:=\inf[a].
\end{equation}
The series estimate then gives $Q_0>1$ on $[0,r_0]$, and
$(rQ_0')'=rQ_0(1-Q_0^2)<0$ proves $Q_0'<0$ on $(0,r_0]$. Together with
\eqref{cap:q-tail-positive}, these checks give a positive, strictly
decreasing solution on $(0,\infty)$. The two matching equations join
the solutions together with their first derivatives, and the
differential equation gives the agreement of all higher derivatives at
$R$. Regularity at zero and the tail bounds place the resulting radial
function in $H^1(\R^2)$, and the uniqueness of the positive radial
solution of $-\Delta Q+Q-Q^3=0$ \cite{Kwong} identifies it with the
ground state $Q$.

\begin{proposition}\label{cap:q-certificate}
Suppose that the following conditions hold:
\begin{itemize}
\item[(i)] the series condition $b_P\tilde r_0^2<1$ and the origin
condition \eqref{cap:q-origin-positive};
\item[(ii)] the validated Taylor continuation and coefficient enclosures
\eqref{cap:q-step-condition}--\eqref{cap:q-step-enclosure}
 on every cell of $[r_0,R]$, together with the positivity and
monotonicity bounds for $Q_0$ and $Q_0'$ stated above;
\item[(iii)] the tail conditions \eqref{cap:q-tail-pass} and
\eqref{cap:q-tail-positive};
\item[(iv)] the nonsingularity of every matrix in $J$ and the matching
inclusion \eqref{cap:q-krawczyk-pass}.
\end{itemize}
Then the unique matching pair in $X$ represents the positive radial
ground state $Q$. Its values and radial derivatives are enclosed on
$[0,R]$ by the origin series, the validated continuation, and the
differential equation. On $[R,\infty)$, the first three orders satisfy
\eqref{cap:q-tail-constants}, and further derivatives at positive radii
are enclosed by differentiating the differential equation.
\end{proposition}
\begin{proof}
The origin, continuation, and tail arguments construct both solutions
with the stated enclosures, and the interval Newton inclusion gives
their unique matching pair. The sign conditions make the joined
solution positive, while its regularity and decay place it in
$H^1(\R^2)$; the uniqueness theorem \cite{Kwong} then identifies it
with $Q$. The derivative enclosures follow from the series and from the
differential identities already proved.
\end{proof}

\subsection*{Ground-State Construction and Downstream Evaluators}

The program \path{groundstate/validate_groundstate.cpp} carries out
an independent validation using CAPD \cite{CAPDDynSys2021}. The origin
values and their parameter derivatives are initialized, with
$\tilde r_0=r_0$, by the recurrences
\eqref{cap:q-series-recurrence} through order $20$, with the remainders
\eqref{cap:q-origin-remainders}. The continuation starts at
$r_0=1/512$ and ends at $R=15$; its Taylor order is $20$, with an
interval remainder coefficient of order $21$. The four-dimensional
system is integrated over the search  interval  for $a$ to enclose the
matching Jacobian , and the  two-dimensional system is integrated at the
parameter center to enclose the matching residual. After the interval
Newton test, the two-dimensional flow is recomputed for the enclosed
matching values of $a$ and exported on $184$ radial cells.

The exact search box $X=[a]\times\Gamma$, polynomial coefficients,
and all interval endpoints are stored as hexadecimal binary numbers in
the certificate. The following decimals are outward  enclosures of the
recorded search intervals:

\begin{equation}\label{E:Q validation data}
 \begin{split}
 [a]&\subset[2.206200864650349,2.206200864651151],\\
 \Gamma&\subset[2.700000000000000,2.900000000000001].
 \end{split}
\end{equation}

The interval Jacobian satisfies
\begin{equation}\label{E:Q Jac}
\begin{aligned}
 J_{11}&\in[-112452.500021809,-112452.500021327],\\
 J_{12}&\in[-9.819536482400306\cdot10^{-8},-9.819536482380195\cdot10^{-8}],\\
 J_{21}&\in[-108636.937664357,-108636.937663879],\\
 J_{22}&\in[1.014172936975801\cdot10^{-7},1.014172936977937\cdot10^{-7}].
\end{aligned}
\end{equation}
Its determinant is contained in
$[-0.022072271949077,-0.022072271948935]$, so every matrix in the
Jacobian enclosure is nonsingular. Direct interval inversion gives an
interval Newton image whose outward decimal enclosure is

\begin{equation}\label{E:K}
\begin{aligned}
 K\subset{}&[2.206200864650683,2.206200864650809]\\
 &\quad\times[2.737760223369866,2.876254800126656]
 \subset\operatorname{int}X.
\end{aligned}
\end{equation}
The last inclusion is verified against the exact stored endpoints of
$X$. The matching residual, the four Jacobian intervals,  and both
strict inclusion margins are recorded in
\path{groundstate/groundstate_certificate.txt}. %Thus the matching
%conclusion uses the  executed origin and continuation bounds, rather
%than an independently supplied numerical profile.

The tail fixed-point ball has radius $z=4$. On the search interval,
its self-mapping margin is greater than $0.365389$, and its contraction
constant satisfies $\theta_R<1.764\cdot10^{-13}$. The matching
calculation retains $16$ terms in the Bessel expansions, together with
the first omitted term as remainder. At $R$, the integral from $R$ to
$r$ in \eqref{cap:q-tail-map} vanishes. The remaining integral gives
\[
 |Q_\infty(R)-\gamma K_0(R)|,
 \ |Q_\infty'(R)+\gamma K_1(R)|
 \leq\frac{k_0i_0z^3}{4R^{3/2}}e^{-3R}.
\]
The analogous errors for $Z(R)$ and $Z'(R)$ are bounded by replacing
$z^3$ by $3z^2C_Z$. These are the tail errors included in the matching
residual and Jacobian.

The origin positivity bound in \eqref{cap:q-origin-positive} is
greater than $2.2061885>1$. Subdivision of every continuation cell
gives the uniform bounds $Q>2.61122\cdot10^{-7}$ and
$-Q_r>2.70797\cdot10^{-7}$ throughout $[r_0,R]$; the origin and tail
arguments prove the corresponding signs on the remaining intervals.
The matched solution therefore is the positive radial ground state,
by Proposition~\ref{cap:q-certificate}. For all subsequent evaluations,
we restrict  the  parameter intervals  to the coordinate projections of  the
interval Newton image. In  particular, the  exterior constants use the
upper endpoint of the new interval for $\gamma$. For its exterior derivatives,
\eqref{cap:q-tail-constants} gives the conservative constants
\[
 C_{Q,0}=4,\qquad C_{Q,1}<3.694973.
\]
The  common evaluator \path{common/profile_eval.hpp} uses the exported
Taylor coefficients and their order-$21$ remainder on each finite
cell. A query crossing a cell boundary is split at that boundary, and
the resulting interval values are combined by their interval hull.
The basic origin evaluator uses $S_{Q,10}$ on $[0,r_0]$; the potential
and quotient evaluators  use $A_0,\ldots,A_{40}$ for $r\leq1/2$, and
the spline derivative evaluator uses $A_0,\ldots,A_{80}$  there. These
orders for $Q$ require only  $br^2<1$, which holds at $r=1/2$;
the stronger condition involving $b_P$ is used only for the initial
sensitivity series near $r_0$. Each evaluation includes its differentiated
geometric remainder. These  orders are distinct from the  order-$20$ initialization and Taylor flow.
For $r\geq R$, the common evaluator retains ten Bessel terms and
bounds both the omitted linear terms and the nonlinear correction in
\eqref{cap:q-tail-errors}--\eqref{cap:q-tail-derivative-error}.
Higher derivatives away from zero follow from the exact differential
equation. In particular, no derivative of an unspecified Taylor
remainder enters the potential, residual, or quotient calculation.

\section{Approximation of the Potential}
\label{cap:p-approximation}

We now transfer the enclosures of the exact radial profile to a
bilinear approximation of $W$. The interval evaluations enclose the
true potential and its derivatives, including the uncertainty in $Q$.
Let $0<d<L$, let $\Omega=(-L,L)^2$, and define
\begin{equation}\label{cap:p-cutoff}
 \begin{gathered}
 \chi_{L,d}(y_1,y_2)=
 \theta\left(\frac{|y_1|-L+d}{d}\right)
 \theta\left(\frac{|y_2|-L+d}{d}\right),\\
 \theta(t)=
 \begin{cases}
 1,&t\leq0,\\
 1-10t^3+15t^4-6t^5,&0<t<1,\\
 0,&t\geq1.
 \end{cases}
 \end{gathered}
\end{equation}
The values and the derivatives up to order two agree at both joins of
$\theta$, and the factors are constant near $y_i=0$. Thus
$\chi_{L,d}\in C^2(\mathbb R^2)$, $\chi_{L,d}$ equals one on
$[-L+d,L-d]^2$, and it vanishes outside $\overline\Omega$. Moreover,
\begin{equation}\label{cap:p-cutoff-derivatives}
 0\leq\chi_{L,d}\leq1,\qquad
 |\partial_{y_i}\chi_{L,d}|\leq\frac{15}{8d},\qquad
 |\partial_{y_i}^2\chi_{L,d}|\leq\frac{10\sqrt3}{3d^2},\qquad
 |\partial_{y_1}\partial_{y_2}\chi_{L,d}|\leq\frac{225}{64d^2}.
\end{equation}
Indeed, $\theta'=-30t^2(1-t)^2$ on $(0,1)$, and maximizing
$|\theta''|=60t(1-t)|1-2t|$ gives $10\sqrt3/3$.

Set $H=\chi_{L,d}W$. On a rectangular mesh of $\overline\Omega$, we
choose at each node an exact rational number $w_{ij}$ from the
validated enclosure of $H(y_{1,i},y_{2,j})$, with zero boundary
values, and we let $W_h$ be their bilinear interpolant, extended by
zero outside $\Omega$. More explicitly, on
$K=[y_{1,i},y_{1,i+1}]\times[y_{2,j},y_{2,j+1}]$, put
\[
 s=\frac{y_1-y_{1,i}}{h_{1,K}},\qquad
 t=\frac{y_2-y_{2,j}}{h_{2,K}}.
\]
On $K$, the bilinear interpolant is
\[
\begin{aligned}
 W_h(y_1,y_2)={}&(1-s)(1-t)w_{ij}+s(1-t)w_{i+1,j}\\
 &+(1-s)t w_{i,j+1}+st w_{i+1,j+1}.
\end{aligned}
\]
The four coefficients are nonnegative and sum to one on $K$, and
adjacent cells have the same interpolated trace on their common edge.
Thus $W_h$ is a bounded continuous function, and hence
it defines a bounded multiplication operator. The boundary conditions
introduced below apply to the functions on which the differential
operator acts.

\begin{proposition}\label{cap:potential-error}
Let $h_{1,K},h_{2,K}$ be the side lengths of a mesh cell $K$.
Suppose that the validated profile gives bounds
\[
\begin{aligned}
 \varepsilon_{\mathrm{cut},K}&\geq\|(1-\chi_{L,d})W\|_{L^\infty(K)},\\
 \varepsilon_{\mathrm{interp},K}&\geq
 \frac{h_{1,K}^2}{8}\|\partial_{y_1}^2H\|_{L^\infty(K)}
 +\frac{h_{2,K}^2}{8}\|\partial_{y_2}^2H\|_{L^\infty(K)},\\
 \varepsilon_{\mathrm{node},K}&\geq
 \max_{v\text{ vertex of }K}\ \sup_{z\in[H(v)]}|z-w_v|.
\end{aligned}
\]
Then
\begin{equation}\label{cap:p-direct-error}
 \|W-W_h\|_{L^\infty(\Omega)}\leq\varepsilon_\Omega,
 \qquad
 \varepsilon_\Omega=
 \max_K\bigl(\varepsilon_{\mathrm{cut},K}
 +\varepsilon_{\mathrm{interp},K}+\varepsilon_{\mathrm{node},K}\bigr).
\end{equation}
Suppose also that $L\geq R>0$ and that the tail certificate gives
\[
 |Q(r)|\leq C_0\frac{e^{-r}}{\sqrt r},\qquad
 |Q_r(r)|\leq C_1\frac{e^{-r}}{\sqrt r}\qquad(r\geq R).
\]
With $\beta_W=2-\alpha_1^{-1}>0$ and
$E=\mathbb R^2\setminus\overline\Omega$, we have
\begin{equation}\label{cap:p-exterior-error}
 \|W\|_{L^\infty(E)}\leq\varepsilon_E,
 \qquad
 \varepsilon_E=
 \frac{(3C_0^2+6\alpha_1C_0C_1)e^{-2L}
       +6\alpha_1C_0C_1e^{-\beta_W L}}{L}.
\end{equation}
\end{proposition}
\begin{proof}
For $f\in C^2([a,b])$, let $If$ denote its linear interpolant at $a$
and $b$. At an interior point $t$, Rolle's theorem applied twice to
$f-If-c(\,\cdot-a)(\,\cdot-b)$, with $c$ chosen so that this function
vanishes at $t$, gives
\[
 f(t)-If(t)=\frac12 f''(\xi)(t-a)(t-b)
 \quad\text{for some }\xi\in(a,b).
\]
Since $|(t-a)(t-b)|\leq(b-a)^2/4$, the uniform error is at most
$(b-a)^2\|f''\|_{L^\infty}/8$. For a $C^2$ function on a rectangle, we
write the bilinear interpolant as $I_hf=I_1I_2f$, where $I_i$
interpolates in $y_i$. Each $I_i$ is a convex combination of endpoint
values, so that it does not increase the uniform norm, and the
decomposition $f-I_hf=(f-I_1f)+I_1(f-I_2f)$ therefore gives
\[
 \|f-I_hf\|_{L^\infty(K)}
 \leq\frac{h_{1,K}^2}{8}\|\partial_{y_1}^2f\|_{L^\infty(K)}
     +\frac{h_{2,K}^2}{8}\|\partial_{y_2}^2f\|_{L^\infty(K)}.
\]
The four bilinear basis functions are nonnegative and sum to one;
replacing the exact nodal values of $H$ by $w_v$ thus contributes at
most $\varepsilon_{\mathrm{node},K}$ on $K$. Applying these estimates
to
\[
 W-W_h=(1-\chi_{L,d})W+(H-I_hH)+(I_hH-W_h)
\]
and taking the maximum of the cellwise sums, we obtain
\eqref{cap:p-direct-error}.

On $E$, we have $r\geq L$,
$|Q_{y_1}|=|Q_r y_1/r|\leq|Q_r|$, and
$e^{-y_1/\alpha_1}\leq e^{r/\alpha_1}$. Consequently,
\[
 |W|\leq\frac{3C_0^2e^{-2r}
       +6\alpha_1C_0C_1(e^{-2r}+e^{-\beta_W r})}{r}.
\]
Since both $e^{-2r}/r$ and $e^{-\beta_W r}/r$ decrease for $r>0$, this
proves \eqref{cap:p-exterior-error}.
\end{proof}

Every upper bound used in this certificate is rounded upward. In
particular, the nodal error includes the distance to both endpoints of
$[H(v)]$, even when the chosen exact rational value $w_v$ is not its
exact midpoint. Subdivision sharpens the derivative and cutoff bounds,
since it pairs the growing exponential factor with local profile
bounds before the maxima are taken.

The Cartesian derivatives needed for the interpolation can be evaluated
explicitly. Set $U=Q^2$, $b=e^{-y_1/\alpha_1}$, and $a=1+b$. Then
\begin{align}
 W&=-3U+3\alpha_1a\partial_{y_1}U,\notag\\
 \partial_{y_1}W&=-3\partial_{y_1}U+3\alpha_1a\partial_{y_1}^2U-3b\partial_{y_1}U,\notag\\
 \partial_{y_2}W&=-3\partial_{y_2}U+3\alpha_1a\partial_{y_1}\partial_{y_2}U,\notag\\
 \partial_{y_1}^2W&=-(3+6b)\partial_{y_1}^2U
       +3\alpha_1a\partial_{y_1}^3U+\frac{3b}{\alpha_1}\partial_{y_1}U,
 \label{cap:p-cartesian}\\
 \partial_{y_2}^2W&=-3\partial_{y_2}^2U+3\alpha_1a\partial_{y_1}\partial_{y_2}^2U,\notag\\
 \partial_{y_1}\partial_{y_2}W&=-(3+3b)\partial_{y_1}\partial_{y_2}U
       +3\alpha_1a\partial_{y_1}^2\partial_{y_2}U,\notag
\end{align}
and the product rule gives
\[
 \partial_{y_i}^2H=\chi_{L,d}\partial_{y_i}^2W
 +2(\partial_{y_i}\chi_{L,d})(\partial_{y_i}W)
 +(\partial_{y_i}^2\chi_{L,d})W.
\]
At the origin, the regular series avoids the division by $r$: writing
$U(y)=p(y_1^2+y_2^2)$ locally, where $p$ is the analytic function
obtained by squaring the series for $Q$, the chain rule gives
\begin{align*}
 \partial_{y_i}U&=2y_ip',\\
 \partial_{y_i}\partial_{y_j}U&=2\delta_{ij}p'+4y_iy_jp'',\\
 \partial_{y_i}\partial_{y_j}\partial_{y_k}U&=4(\delta_{ij}y_k+\delta_{ik}y_j+\delta_{jk}y_i)p''
                  +8y_iy_jy_kp'''.
\end{align*}
The order-$40$ origin evaluator encloses the required derivatives in
$s=r^2$ by the series and its differentiated geometric majorant. Away
from zero, the enclosed $Q$, $Q_r$ and the identities
\[
 Q_{rr}=Q-Q^3-\frac{Q_r}{r},\qquad
 Q_{rrr}=(1-3Q^2)Q_r-\frac{Q_{rr}}r+\frac{Q_r}{r^2}
\]
give all the radial derivatives needed here, and the conversion to
Cartesian coordinates then uses the chain rule; for example,
\[
 \partial_{y_i}\partial_{y_j}U
 =U_{rr}\frac{y_iy_j}{r^2}
 +\frac{U_r}{r}\left(\delta_{ij}-\frac{y_iy_j}{r^2}\right).
\]
Thus the nonradial exponential in $W$ is retained in every derivative
bound, and the profile uncertainty remains in each interval evaluation.

\subsection*{The Recorded Potential Certificate}

The recorded CAPD potential calculation uses \eqref{cap:p-direct-error}
with the square and cutoff parameters $L=16$ and $d=2$, and with a mesh
of $1409$ nodes per axis, of spacing $1/128$ on $[-4,4]$, $1/32$ for
$4\leq|y_i|\leq8$, and $1/8$ for $8\leq|y_i|\leq16$. The profile tail
enters through $R=15$, $C_0=4$, and  $C_1<3.694973$. The
resulting outward-rounded upper bounds are
\[
\begin{aligned}
 \|W_h\|_{L^\infty}&\leq36.609366495687,\\
 \max_K\varepsilon_{\mathrm{node},K}&\leq8.226155064000\cdot 10^{-8},\\
 \max_K\varepsilon_{\mathrm{interp},K}&\leq0.004231045898735,\\
 \max_K\varepsilon_{\mathrm{cut},K}&\leq1.121028250570\cdot 10^{-6},
\end{aligned}
\]
 and, in \eqref{cap:p-direct-error} and \eqref{cap:p-exterior-error},
\[
 \varepsilon_\Omega\leq0.004231045907326,\qquad
 \varepsilon_E\leq5.376665828\cdot 10^{-7}.
\]
The specified nodal values are retained as exact dyadic numbers.
The exact nodal values and the interval output are recorded in
\begin{center}\small
\path{spectrum/data/potential128_nodes.txt}\\
\path{potential/results/potential128_certificate.txt}.
\end{center}
Since the separate maxima
of the interpolation, nodal, and cutoff errors need not occur on the
same cell, the bound $\varepsilon_\Omega$ uses the maximum of their
cellwise sums, as in \eqref{cap:p-direct-error}.

\section{Spectral Bounds}
\label{cap:s-section}

The spectral information on $\mathcal L$ is obtained in two steps. We
first bound eigenvalues of the operator with the explicit potential
$W_h$ on the square $\Omega$, with Dirichlet and with Neumann boundary
conditions, by finite-dimensional tests; we then transfer these bounds
to $\mathcal L$ on the whole space by Dirichlet--Neumann bracketing,
using the potential errors of Section~\ref{cap:p-approximation}.

\subsection*{Bounded-Domain Problems}

On $L^2(\Omega)$, let $T_h^D$ and $T_h^N$ be the self-adjoint
operators associated with the form
\begin{equation}\label{cap:s-form}
 a_h(u,u)=3\|\partial_{y_1}u\|_{L^2(\Omega)}^2+\|\partial_{y_2}u\|_{L^2(\Omega)}^2
               +\int_\Omega(\kappa_{\alpha_1}+W_h)|u|^2,
\end{equation}
on $H_0^1(\Omega)$ and on $H^1(\Omega)$, respectively. The form is
closed and bounded from below, since $W_h$ is bounded, and after the
addition of a sufficiently large constant its norm is equivalent to
the $H^1$ norm; the compactness of the embedding of either form domain
into $L^2(\Omega)$ therefore gives compact resolvents. On this
rectangle, odd and even reflections in the sides reduce the Dirichlet
and the Neumann problems to periodic Fourier expansions, so that the
sine and cosine products introduced below form orthonormal bases of
$L^2(\Omega)$, and the operator domains are
$H^2(\Omega)\cap H_0^1(\Omega)$ and
$\{u\in H^2(\Omega):(3\partial_{y_1}u,\partial_{y_2}u)\cdot n=0
\text{ on each boundary edge}\}$, where $n$ is the outward unit normal;
indeed, the summability of the squared kinetic eigenvalues times the
squared coefficients is precisely $H^2$ regularity with the indicated
boundary conditions, and bounded multiplication by
$\kappa_{\alpha_1}+W_h$ preserves these domains. We denote the
eigenvalues, in increasing order and counted with multiplicity, by
$\mu_k^D$ and $\mu_k^N$. The two boundary conditions play different
roles: the Dirichlet problem supplies upper bounds, while the Neumann
problem controls the number of eigenvalues below a prescribed level.

The form version of the min--max principle
\cite[Theorem I.3 and Corollary I.2]{Krejcirik2010} gives
\begin{equation}\label{cap:s-minmax}
 \mu_k^D=\inf_{\substack{S\subset H_0^1(\Omega)\\\dim S=k}}
       \ \sup_{0\ne u\in S}\frac{a_h(u,u)}{\|u\|_{L^2(\Omega)}^2},
\end{equation}
and replacing $H_0^1(\Omega)$ by $H^1(\Omega)$ gives the Neumann
formula. Let $\psi_1,\ldots,\psi_m\in H_0^1(\Omega)$ be linearly
independent finite linear combinations, with rational coefficients, of
the orthonormal Dirichlet products $s_m(y_1)s_n(y_2)$, where
$s_j(t)=L^{-1/2}\sin(j\pi(t+L)/(2L))$ for $j\geq1$, and define
\begin{equation}\label{cap:s-ritz}
 \mathbf K_{ij}=a_h(\psi_i,\psi_j),\qquad
 \mathbf M_{ij}=(\psi_i,\psi_j)_{L^2(\Omega)}.
\end{equation}
If $d_k$ is a real number and $\mathbf K-d_k\mathbf M$ has at least $k$
negative eigenvalues, then
\begin{equation}\label{cap:s-dirichlet-upper}
 \mu_k^D<d_k.
\end{equation}
Indeed, $\xi^T\mathbf M\xi=\|\sum_i\xi_i\psi_i\|_{L^2(\Omega)}^2$, so that
$\mathbf M>0$ and the coefficient map is injective. A $k$-dimensional
negative subspace of $\mathbf K-d_k\mathbf M$ therefore gives a
$k$-dimensional space of trial functions on which the Rayleigh quotient
is strictly below $d_k$, and its maximum on the compact unit sphere of
that space is strictly below $d_k$; \eqref{cap:s-minmax} then proves
\eqref{cap:s-dirichlet-upper}. We use two trial functions for $k=2$.
Sine orthogonality gives the mass and the kinetic terms, while the
identity $\sin A\sin B=(\cos(A-B)-\cos(A+B))/2$ reduces the potential
terms to polynomial--cosine integrals on the mesh cells; the finite
moment formulas below enclose these integrals, including every Taylor
remainder and arithmetic rounding.

For the lower bound, we use the Schur complement of a finite cosine
space; this is the elementary form of the Feshbach--Schur method, see
\cite{DSS2021,DSSmap2021}. The argument below includes the eigenvalue
count, so that the index of the eigenvalue being bounded is part of the
certificate. Let
\[
 c_0(t)=(2L)^{-1/2},\qquad
 c_j(t)=L^{-1/2}\cos\frac{j\pi(t+L)}{2L}\quad(j\geq1).
\]
The functions $\phi_{mn}(y)=c_m(y_1)c_n(y_2)$ form an orthonormal
Neumann basis, and for $H_0=-3\partial_{y_1}^2-\partial_{y_2}^2$ their
eigenvalues are
\[
 d_{mn}=\frac{\pi^2}{4L^2}(3m^2+n^2).
\]
Fix integers $N_1,N_2\geq0$, and let $P$ be the orthogonal projection
onto the modes with $0\leq m\leq N_1$ and $0\leq n\leq N_2$. Write
$P^\perp=I-P$, let $M=(N_1+1)(N_2+1)$, and assume that $M\geq3$. The
modes omitted from $P$ are those with $m>N_1$ or $n>N_2$, and their
least kinetic eigenvalue is
\begin{equation}\label{cap:s-omitted}
 d_* =\frac{\pi^2}{4L^2}
       \min\{3(N_1+1)^2,(N_2+1)^2\}.
\end{equation}
We index the retained modes by $i,j\in\{1,\ldots,M\}$ and define the
real symmetric matrices
\begin{equation}\label{cap:s-fourier-matrices}
 \begin{aligned}
 \mathbf V_{ij}&=\int_\Omega W_h\phi_i\phi_j,
 &\mathbf C_{ij}&=\int_\Omega W_h^2\phi_i\phi_j,\\
 \mathbf A_{ij}&=(d_i+\kappa_{\alpha_1})\delta_{ij}+\mathbf V_{ij},
 &\mathbf G&=\mathbf C-\mathbf V^2.
 \end{aligned}
\end{equation}
For every vector $z$, with $p=\sum_i z_i\phi_i$, orthonormality gives
$z^T\mathbf Cz=\|W_hp\|_{L^2(\Omega)}^2$ and
$z^T\mathbf V^2z=\|P(W_hp)\|_{L^2(\Omega)}^2$, and Pythagoras' theorem therefore
yields
\begin{equation}\label{cap:s-residual-gram}
 z^T\mathbf Gz=\|P^\perp(W_hp)\|_{L^2(\Omega)}^2.
\end{equation}
In particular, $\mathbf G$ is positive semidefinite. We choose a
certified upper bound $V_*\geq\|W_h\|_{L^\infty}$ and a rational number
$\tau$ satisfying
\begin{equation}\label{cap:s-tail-threshold}
 \tau\leq d_*+\kappa_{\alpha_1}-V_*.
\end{equation}
The maximum of the absolute nodal values is a valid choice for $V_*$,
since bilinear interpolation is a convex combination of those values.
Thus $\tau$ bounds from below the operator restricted to the
complementary Fourier space, as proved below. The word ``omitted''
refers only to the exclusion from the finite projection: these modes
are controlled by $\tau$ and by the Schur correction, rather than
discarded.

\begin{proposition}\label{cap:s-schur}
Let $t<\tau$, and put
\begin{equation}\label{cap:s-schur-matrix}
 \mathbf F(t)=\mathbf A-tI-\frac{\mathbf G}{\tau-t}.
\end{equation}
If $\mathbf F(t)$ has at most $k-1$ negative eigenvalues, where
$1\leq k\leq M$, then
\begin{equation}\label{cap:s-neumann-lower}
 \mu_k^N\geq t.
\end{equation}
\end{proposition}
\begin{proof}
The decomposition $L^2(\Omega)=\operatorname{Ran}P\oplus
\operatorname{Ran}P^\perp$ gives
\[
 T_h^N=\begin{pmatrix}\mathbf A&B^*\\B&D\end{pmatrix},
 \qquad B=P^\perp W_hP,
 \qquad D=P^\perp T_h^NP^\perp\big|_{\operatorname{Ran}P^\perp},
\]
where the operator $D$ is defined by restricting the form to
$H^1(\Omega)\cap\operatorname{Ran}P^\perp$. Since $P$ commutes with
$H_0$, \eqref{cap:s-omitted} and \eqref{cap:s-tail-threshold} imply
$D\geq\tau I$, so that $D-t$ has a positive bounded inverse. Functional
calculus gives $\|(D-t)^{-1/2}\|\leq(\tau-t)^{-1/2}$ and
$\|(D-t)^{1/2}(D-t)^{-1}v\|\leq(\tau-t)^{-1/2}\|v\|$; hence this
inverse maps $L^2$ into the form domain of $D$. Completing the square,
for $p\in\operatorname{Ran}P$ and $q$ in that form domain, we obtain
\begin{align*}
 a_h(p+q,p+q)-t\|p+q\|_{L^2(\Omega)}^2
 &=\bigl((\mathbf A-tI-B^*(D-t)^{-1}B)p,p\bigr)\\
 &\quad+\bigl\|(D-t)^{1/2}
             (q+(D-t)^{-1}Bp)\bigr\|_{L^2(\Omega)}^2.
\end{align*}
The change of variables $(p,q)\mapsto(p,q+(D-t)^{-1}Bp)$ is invertible
on the form domain, and therefore the negative index of $T_h^N-t$
equals that of $\mathbf S(t)=\mathbf A-tI-B^*(D-t)^{-1}B$. Since
$0<(D-t)^{-1}\leq(\tau-t)^{-1}I$ and $B^*B=\mathbf G$ by
\eqref{cap:s-residual-gram}, we have $\mathbf S(t)\geq\mathbf F(t)$.
Consequently, every subspace on which $\mathbf S(t)$ is negative
definite is a negative subspace for $\mathbf F(t)$, and hence
\[
 \#\{j:\mu_j^N<t\}
 =\operatorname{ind}_-\mathbf S(t)
 \leq\operatorname{ind}_-\mathbf F(t)\leq k-1,
\]
which is precisely \eqref{cap:s-neumann-lower}.
\end{proof}

The finite matrices in \eqref{cap:s-fourier-matrices} are enclosed by
integrating the same exact $W_h$ as in the potential comparison.
Products of two cosines reduce to sums of cosines, and the cell
restrictions of $W_h$ and $W_h^2$ have degree at most one and two,
respectively, in each variable. The interval Taylor moments and the
normwise inertia test below give the corresponding finite certificate.

\subsection*{The Whole-Space Spectral Gap}

The estimates just proved concern operators on a bounded square; their
connection with $\mathcal L$ follows from the comparison of form
domains in Dirichlet--Neumann bracketing, see
\cite[Section II.2.1]{Krejcirik2010}. We state the complete implication
needed for the inverse estimates.

\begin{proposition}\label{cap:spectral-gap}
Suppose that the potential certificate gives
$\|W-W_h\|_{L^\infty(\Omega)}\leq\varepsilon_\Omega$ and
$\|W\|_{L^\infty(E)}\leq\varepsilon_E$.
Let $\overline\mu_2^D$ and $\underline\mu_3^N$ be certified bounds
for the second Dirichlet and the third Neumann eigenvalues of the
operators in \eqref{cap:s-form}. Assume that
\begin{equation}\label{cap:s-gap-conditions}
 \overline\mu_2^D+\varepsilon_\Omega<0,\qquad
 \underline\mu_3^N-\varepsilon_\Omega>0,\qquad
 \kappa_{\alpha_1}-\varepsilon_E>0.
\end{equation}
Then $\mathcal L$ has exactly two negative eigenvalues, counted with
multiplicity, $0\in\rho(\mathcal L)$, and
\begin{equation}\label{cap:s-gap-value}
 \operatorname{dist}(0,\sigma(\mathcal L))\geq\delta_0,
 \quad
 \delta_0=\min\left\{
 -\overline\mu_2^D-\varepsilon_\Omega,
 \underline\mu_3^N-\varepsilon_\Omega,
 \kappa_{\alpha_1}-\varepsilon_E\right\}>0.
\end{equation}
Consequently,
\begin{equation}\label{cap:s-inverse-bound}
 \|\mathcal L^{-1}\|_{L^2\to L^2}\leq\delta_0^{-1}.
\end{equation}
\end{proposition}
\begin{proof}
Let $\mathcal L_\Omega^D$, $\mathcal L_\Omega^N$, and
$\mathcal L_E^N$ denote the interior Dirichlet, the interior Neumann,
and the exterior Neumann realizations with the exact potential $W$.
Writing $\mathfrak l_\Omega$ for the exact interior form, on either
interior form domain we have
\[
 \bigl|\mathfrak l_\Omega[u]-a_h(u,u)\bigr|
 =\left|\int_\Omega(W-W_h)|u|^2\right|
 \leq\varepsilon_\Omega\|u\|_{L^2(\Omega)}^2.
\]
Taking the supremum on each trial space and then the infimum in
\eqref{cap:s-minmax} proves that
\[
 |\lambda_k(\mathcal L_\Omega^{D,N})-\mu_k^{D,N}|
 \leq\varepsilon_\Omega.
\]
Since zero extension embeds $H_0^1(\Omega)$
into $H^1(\mathbb R^2)$, the second whole-space min--max value obeys
\[
 \lambda_2(\mathcal L)
 \leq\lambda_2(\mathcal L_\Omega^D)
 \leq\overline\mu_2^D+\varepsilon_\Omega<0.
\]
As the essential spectrum starts at the positive number
$\kappa_{\alpha_1}$, this produces two negative eigenvalues and places
both at or below $\overline\mu_2^D+\varepsilon_\Omega$.

Conversely, the restriction of a whole-space $H^1$ function gives an
element of $H^1(\Omega)\oplus H^1(E)$, a space which also admits
independently chosen traces on the common boundary. Its quadratic form
is the same integral on this larger domain, so that
\[
 \mathcal L_\Omega^N\oplus\mathcal L_E^N\leq\mathcal L
\]
in the sense of forms: the enlarged form domain admits every
whole-space trial space, and its min--max infimum can only decrease.
The exterior Neumann operator satisfies
$\mathcal L_E^N\geq\kappa_{\alpha_1}-\varepsilon_E$, and the third
eigenvalue of $\mathcal L_\Omega^N$ is at least
$\underline\mu_3^N-\varepsilon_\Omega$. Set
\[
 \eta_+=\min\{\underline\mu_3^N-\varepsilon_\Omega,
                   \kappa_{\alpha_1}-\varepsilon_E\}>0.
\]
The decoupled operator has at most two spectral directions strictly
below $\eta_+$, and the min--max principle therefore gives
$\lambda_3(\mathcal L)\geq\eta_+$, where a min--max value is allowed to
equal the bottom of the essential spectrum. Consequently, there are
exactly two negative eigenvalues, and the remaining spectrum is
contained in $[\eta_+,\infty)$. Combining the negative and the positive
separations proves \eqref{cap:s-gap-value}. Finally, functional
calculus gives
\[
 \|\mathcal L^{-1}\|=\sup_{\lambda\in\sigma(\mathcal L)}|\lambda|^{-1}
 \leq\delta_0^{-1},
\]
which proves \eqref{cap:s-inverse-bound}.
\end{proof}

Recall that Remark~\ref{rem:parity} supplies a negative direction in
each parity sector with respect to $y_2$. The lowest eigenvalue is
simple and has a positive eigenfunction: under the unitary change of
variables $(Uf)(x)=3^{1/4}f(\sqrt3x_1,x_2)$, the operator
$\mathcal L-\kappa_{\alpha_1}$ becomes $-\Delta+W(\sqrt3x_1,x_2)$, to
which \cite[Theorem II.7]{Krejcirik2010} applies, since its lowest
eigenvalue lies below its essential spectrum and the transformed
potential is real, bounded, and tends to zero.
Proposition~\ref{cap:spectral-gap} then leaves precisely one negative
direction in each parity sector. Reflection preserves the positive
normalized lowest eigenfunction, and simplicity makes it even; the
other negative eigenfunction is odd, and the two negative eigenvalues
are distinct and simple.

\subsection*{The Recorded Spectral Certificate}

The spectral source specifies the following inputs: the square $L=16$
and the modes $(N_1,N_2)=(48,84)$, so that $M=4165$; the spectral
thresholds $t=80/128=5/8$ and $d_2=-5/2$; the kinetic level
$d_*=7203\pi^2/1024$; the block bound
$\tau=[d_*+\kappa_{\alpha_1}-V_*]^->33$ with
$V_* =\max_{i,j}|w_{ij}|$; and the moment-series order $J=48$ used in
the interval implementation below. The endpoint in the definition of
$\tau$ is taken after outward interval evaluation, so that $\tau$ is an
exact dyadic lower bound, and the construction requires $\tau>t$ before
the division by $\tau-t$. For the selected inputs, outward scalar
evaluation establishes $\tau>33>t$, independently of the matrix tests;
the C++ replay checks this strict inequality before the division. The
actual value of $\tau$ is written to the spectral transcript, together
with the potential digest and the matrix error bounds.

The finite matrices are large, and the certificate represents them by
floating-point centers with normwise error radii. We first fix the
rounding model and the propagation rules, then describe the enclosure
of the Fourier matrices, and finally the inertia tests and the recorded
bounds.

\medskip
\noindent\textbf{Floating-point model and propagation rules.}
A matrix enclosure is represented by a pair $(A_0,e_A)$ satisfying
$\|A-A_0\|_F\leq e_A$, where every entry of $A_0$ is a binary64
number interpreted as its exact dyadic value. Here $\|\cdot\|_F$
denotes the Frobenius norm and $\|\cdot\|_{\mathrm{op}}$ the Euclidean
operator norm of a matrix. All scalar error bounds in what follows are
evaluated with outward interval arithmetic.

We first justify the arithmetic bounds used in the finite assembly.
Set $u_{\rm fp}=2^{-53}$, $\lambda_{\rm fp}=2^{-1022}$, and
$\gamma_j=ju_{\rm fp}/(1-ju_{\rm fp})$, with $ju_{\rm fp}<1/4$.
The ordinary floating-point operations use rounding to nearest, gradual
underflow, and no overflow, so that each rounded operation, including a
fused multiply-add when one is used, has the representation
\[
 \operatorname{fl}(z)=z(1+\epsilon)+\zeta,
 \qquad |\epsilon|\leq u_{\rm fp},\qquad
 |\zeta|\leq u_{\rm fp}\lambda_{\rm fp}.
\]
When the exact result lies in the normal range, the additive term may
be taken equal to zero; in the subnormal range, the absolute rounding
error is at most half the least positive subnormal number, namely
$u_{\rm fp}\lambda_{\rm fp}$. Expanding a classical dot product into
its rounded operations gives one factor $1+\vartheta$ on each exact
product, with $|\vartheta|\leq\gamma_j$ if its path contains at most
$j$ roundings; see \cite[Chapter 3]{Higham2002}. Each additive error
is amplified by at most $(1+u_{\rm fp})^j\leq(1-ju_{\rm fp})^{-1}$.
For a dot product of length $k$, we may take $j=2k+2$, and there are
at most $j$ error terms, whose sum in absolute value is at most
$j u_{\rm fp}\lambda_{\rm fp}/(1-ju_{\rm fp})<j\lambda_{\rm fp}$.
Thus the following deliberately larger matrix bound is valid:
\begin{equation}\label{cap:s-product-rounding}
 \|\operatorname{fl}(A_0B_0)-A_0B_0\|_F
 \leq\gamma_{2k+2}\|A_0\|_F\|B_0\|_F
       +(4k+4)mn\lambda_{\rm fp},
\end{equation}
where $A_0$ is $m\times k$ and $B_0$ is $k\times n$. Indeed, the
componentwise relative-error bound is $\gamma_j|A_0||B_0|$, whose
Frobenius norm is at most $\gamma_j\|A_0\|_F\|B_0\|_F$, while the
absolute-error contribution is at most
$j\sqrt{mn}\lambda_{\rm fp}\leq2jmn\lambda_{\rm fp}$. Reordering the
additions or using fused multiply-adds preserves these bounds for
classical dot products, whereas reduced-precision products and
algorithms with a different multiplication scheme are excluded.

For a stored array with $N$ entries, let $s_{\rm fl}$ be the computed
nonnegative sum of squares, and set $j=2N+2$. The same expansion has
at most $N$ squarings and $N-1$ nontrivial additions, and since the
exact squared terms are nonnegative, it gives
$s_{\rm fl}\geq(1-\gamma_j)\|A_0\|_F^2-j\lambda_{\rm fp}$.
Consequently,
\begin{equation}\label{cap:s-norm-rounding}
 \|A_0\|_F\leq
 \left(\frac{s_{\rm fl}+j\lambda_{\rm fp}}{1-\gamma_j}\right)^{1/2}.
\end{equation}
All norms of stored centers below are bounded by this formula. The
finiteness of the computed entries and of the norm reductions, as well
as $ju_{\rm fp}<1/4$, are checked, and the arithmetic assumptions and
the BLAS implementation used are part of the recorded environment.

We now record the error-propagation rules used throughout. Write
$n_A\geq\|A_0\|_F$ and $n_B\geq\|B_0\|_F$, and let the output center
be the indicated floating-point operation. Valid output radii are
\begin{align*}
 e_{AB}&=\gamma_{2k+2}n_An_B+(4k+4)mn\lambda_{\rm fp}
          +e_An_B+e_Bn_A+e_Ae_B,\\
 e_{A\pm B}&=e_A+e_B+\gamma_2(n_A+n_B)+N\lambda_{\rm fp},\\
 e_{sA}&=|s_0|e_A+r_s(n_A+e_A)
          +\gamma_2|s_0|n_A+N\lambda_{\rm fp},
 \qquad |s-s_0|\leq r_s,
\end{align*}
where in the last two lines $N$ is the number of output entries. For
the first line, we expand
$AB-A_0B_0=(A-A_0)B_0+A_0(B-B_0)+(A-A_0)(B-B_0)$
and use the submultiplicativity of the Frobenius norm together with
\eqref{cap:s-product-rounding}; the other two lines follow from the
same scalar rounding model, the triangle inequality, and
$sA-s_0A_0=s_0(A-A_0)+(s-s_0)A$. For an entrywise product of two exact
dyadic arrays, the analogous radius is
$\gamma_2n_An_B+N\lambda_{\rm fp}$, since
$\|A_0\mathbin{\circ}B_0\|_F\leq n_An_B$. An entrywise interval array
is converted into a norm enclosure by choosing dyadic centers $a_{ij}$
and outward radii $r_{ij}\geq\sup_{x\in[A_{ij}]}|x-a_{ij}|$, and then
using $e_A\geq(\sum_{ij}r_{ij}^2)^{1/2}$.

\medskip
\noindent\textbf{Enclosure of the Fourier matrices.}
For the cell integration, we write $x=a+hs$, $0\leq s\leq1$, and
$\sigma=\omega h$. The required moments are
\begin{equation}\label{cap:s-taylor-moments}
 \int_0^1s^je^{i\sigma s}\,ds
 =\sum_{\ell=0}^{J}\frac{(i\sigma)^\ell}{\ell!(\ell+j+1)}+r_{j,J},
 \qquad |r_{j,J}|\leq e^{|\sigma|}
                  \frac{|\sigma|^{J+1}}{(J+1)!},\quad 0\leq j\leq2.
\end{equation}
To prove this, we integrate the Taylor polynomial termwise; the
absolute remainder is bounded uniformly on $[0,1]$ by
$\sum_{\ell\geq J+1}|\sigma|^\ell/\ell!
\leq e^{|\sigma|}|\sigma|^{J+1}/(J+1)!$, and multiplication by $s^j$
and integration do not increase this bound. Multiplying by
$he^{i\omega(a+L)}$ and taking real parts gives the cosine moments.
Every phase, coefficient, sum, and remainder is enclosed by scalar
interval arithmetic.

The linear Bernstein functions $1-s$ and $s$ represent $W_h$ exactly
on the cells, and the quadratic functions $(1-s)^2$, $2s(1-s)$, $s^2$
represent $W_h^2$. If a cell has nodal values $a,b,c,d$ at
$(0,0),(1,0),(0,1),(1,1)$, the latter tensor coefficients, with rows
indexed by the first coordinate, are
\[
 \begin{pmatrix}
 a^2&ac&c^2\\ ab&(ad+bc)/2&cd\\ b^2&bd&d^2
 \end{pmatrix}.
\]
This identity follows by squaring the bilinear polynomial; in
particular, the mixed coefficient retains both products $ad$ and $bc$.
If $H_x$ and $H_y$ are the integrated linear hat matrices, the full
moment array for $W_h$ is $H_x(w_{ij})H_y^T$, and for $W_h^2$ the full
moment array is obtained by summing the nine products of the
corresponding integrated quadratic Bernstein matrices and their
cell-coefficient arrays. The preceding product, addition, scaling, and
entrywise-product bounds enclose these operations. Reflection symmetry,
when verified exactly in the nodal array, allows the odd moments in the
second coordinate to be set to zero, which cannot increase the error.

For completeness, the following bound controls the last product-to-sum
assembly. Let $(T_0,e_T)$ enclose either full moment array, write
$s_i$ for the exact product of the one-dimensional basis normalization
constants, and choose dyadic $s_i^0$ with
\[
 |s_i-s_i^0|\leq r_s,\qquad
 s_{\max}\geq\max_i|s_i|,\qquad
 s_0^{\max}=\max_i|s_i^0|,\qquad m_T=\max_{p,q}|(T_0)_{pq}|.
\]
Each assembled entry has the form
$\tfrac14s_is_j(T_{pq}+T_{p'q}+T_{pq'}+T_{p'q'})$ for cosines, the two
middle signs being negative for sines; repeated indices are retained,
so that this formula includes the constant cosine modes. For an
$M\times M$ output, the implemented radius is
\begin{equation}\label{cap:s-assembly-error}
 \begin{split}
 e_{\rm ass}={}&M s_{\max}^2 e_T
 +M\bigl(2s_0^{\max}r_s+r_s^2\bigr)m_T\\
 &+\gamma_{10}M(s_0^{\max})^2m_T
       +20M^2\lambda_{\rm fp}.
 \end{split}
\end{equation}
Indeed, each moment error is at most $e_T$, and the factor $1/4$
cancels the four terms; the normalization error is bounded by
$|s_is_j-s_i^0s_j^0|\leq2s_0^{\max}r_s+r_s^2$; and for the central
arithmetic there are three additions and three multiplications, so
that the relative-error expansion is bounded by
$\gamma_{10}(s_0^{\max})^2m_T$. In this calculation, $L=16$ and
$|s_i^0|\leq1$, hence none of these multiplications amplifies an
additive underflow error by a factor exceeding one, apart from the
rounding factors, and their combined contribution per entry is less
than $20\lambda_{\rm fp}$. Passing from a uniform entry bound to the
Frobenius norm multiplies by $M$, and the larger
$20M^2\lambda_{\rm fp}$ used in \eqref{cap:s-assembly-error} is
therefore valid. This proves the radii $e_V$ and $e_C$ for the
assembled matrices.

The same rules enclose $\mathbf G=\mathbf C-\mathbf V^2$ and
$\mathbf V-\mathbf G/(\tau-t)$. The scalar bound $\tau>33>t$ verifies
the hypothesis of Proposition~\ref{cap:s-schur} and makes the interval
reciprocal positive. To add the kinetic diagonal, each new diagonal
entry is evaluated as an interval and recentered; if the resulting
rounding radii are $r_i$, it is valid to add $(\sum_i r_i^2)^{1/2}$ to
the existing radius. Finally, the arithmetic operation
$(F_0+F_0^T)/2$, enclosed with the addition and scaling rules, gives an
exactly symmetric center for the exact symmetric matrix $\mathbf F(t)$.
Thus its radius includes the integration, normalization,
kinetic-diagonal, and symmetrization errors.

\medskip
\noindent\textbf{The inertia tests.}
We use the following normwise congruence test. Let
$\|H-H_0\|_F\le e_H$ with $H$ and $H_0$ symmetric, and let $R$ and
$D=\operatorname{diag}(d_i)$ be exact dyadic matrices proposed by an
ordinary eigensolver. We certify
\[
 q\ge\|R^TR-I\|_{\mathrm{op}},\quad q<1,\qquad
 b\ge\|R^TH_0R-D\|_{\mathrm{op}}+(1+q)e_H.
\]
Here $q$ is the \emph{congruence invertibility defect}: since
$(1-q)I\leq R^TR\leq(1+q)I$, the inequality $q<1$ proves the
invertibility of $R$ and $\|R\|_{\mathrm{op}}^2\leq1+q$, and in particular
\[
 \|R^T(H-H_0)R\|_{\mathrm{op}}\leq(1+q)e_H.
\]
The \emph{congruence perturbation radius} $b$ therefore bounds
$\|R^THR-D\|_{\mathrm{op}}$, including the uncertainty in $H$. We define the
\emph{smallest signed margin} by
\[
 s_*:=\min_i(|d_i|-b).
\]
If $s_*>0$, then $H$ has exactly as many negative eigenvalues as $D$.
Indeed, $R^THR$ remains invertible along the straight line joining it
to $D$, since the smallest singular value of $D$ exceeds the
perturbation norm; the inertia is constant on this path, and it is
preserved by the invertible congruence. Both displayed operator norms
are bounded from above by certified Frobenius norms, and the factor
$1+q$ retains the operator-norm estimate for $R$. Note that the margin
$s_*$ concerns the congruent matrix and certifies its inertia; it is
not itself a lower bound for an eigenvalue of the differential
operator. In each test, $q$ and $b$ are bounded upward and $s_*$
downward.

For the Dirichlet bound, let $\widehat H_D$ be the finite sine matrix
of $T_h^D-d_2$ on the $48\cdot84=4032$ retained modes. The two columns
of an exact dyadic matrix $U$ define the trial functions; interval
multiplication proves $q_U\geq\|U^TU-I\|_{\mathrm{op}}$ with $q_U<1$, and hence
their mass matrix $\mathbf M=U^TU$ is positive definite. If
$\|\widehat H_D-H_{D,0}\|_F\leq e_D$, then
\[
 \mathbf K-d_2\mathbf M=U^T\widehat H_DU,\qquad
 \|U^T(\widehat H_D-H_{D,0})U\|_F\leq(1+q_U)e_D,
\]
because $\|AXB\|_F\leq\|A\|_{\mathrm{op}}\|X\|_F\|B\|_{\mathrm{op}}$. The arithmetic error of
the two products with $H_{D,0}$ and of the final symmetrization is
added according to the rules above, and the normwise inertia test is
then applied to this $2\times2$ enclosure. An ordinary eigensolver proposes  $U$, $R$, and $D$; these are stored
as exact dyadic witnesses, and only the subsequent interval inequalities
enter the proof. The Neumann test acts
on $H=\mathbf F(5/8)$, of size $4165\times4165$, and the Dirichlet
test acts on $H=\mathbf K+(5/2)\mathbf M$, of size $2\times2$. Their
successful inertia tests give the explicit one-sided bounds
\[
 \mu_1^D\leq\mu_2^D<-\frac52,\qquad \mu_3^N\geq\frac58,
\]
the first because two Ritz directions are negative at $d_2=-5/2$, and
the second because $\mathbf F(5/8)$ has two negative directions. The
same certified upper threshold suffices for both Dirichlet eigenvalues,
and no separate sharper bound for the first eigenvalue is used. These
are eigenvalues of the bounded-domain operators with the potential
$W_h$; including the potential errors in the whole-space comparison
gives
\[
\begin{aligned}
 \delta_0&\ge \min\biggl\{
 \frac52-0.004231045907326,\ \frac58-0.004231045907326,\\
 &\hspace{35mm}\frac{7701}{10201}-5.376665828\cdot 10^{-7}\biggr\}\\
 &=0.620768954092674>0.62.
\end{aligned}
\]
 \medskip
\noindent\textbf{Recorded bounds.}
The C++ calculation reconstructs the Fourier matrices from the exact
potential nodes and encloses every arithmetic operation as described
above. For the Neumann Schur test, the proposed diagonal has  two
 negative and $4163$ positive directions, and the certified bounds are
\[
 q<3.865263\cdot10^{-9},\qquad
 b<8.373108\cdot10^{-5},\qquad s_*>0.04395517893669.
\]
For the  Dirichlet trial test,  there are two  negative and no positive
directions,  with
\[
 q<1.960305\cdot10^{-15},\qquad
 b<9.966066\cdot10^{-6},\qquad s_*>0.31005658808225.
\]
The trial-space independence defect satisfies
$q_U<1.792935\cdot10^{-12}$. The matrix radii satisfy
\[
 e_V<1.028800\cdot10^{-5},\qquad
 e_C<5.007202\cdot10^{-4},\qquad
 e_H<7.014940\cdot10^{-5}.
\]
The threshold also satisfies $\tau-t>32.94513029376356$; its
positivity is checked before the reciprocal is formed. Together with
the fresh profile and potential comparison, these finite tests give
\[
 \operatorname{dist}(0,\sigma(\mathcal L))
 \ge0.62076895409267>0.62,\qquad
 \|\mathcal L^{-1}\|_{L^2\to L^2}<1.610906.
\]
The matrices $\mathbf V$ and $\mathbf C$ each have $4165^2$ entries.
Their enclosures consist of the constructed dyadic centers and  the
Frobenius radii $e_V,e_C$, and the scalar margins above do not replace
those matrix data. The source \path{spectrum/spectrum_replay.cpp}
constructs both matrices and the two congruence tests. Three
binary witness files store the exact congruence matrices and diagonal
proposals, and the two sine trial vectors. Their dimensions are
$4165\times4166$, $2\times3$, and $4032\times2$, respectively, the
first two appending the diagonal proposal to the congruence matrix.
They serve only as explicitly specified finite approximations; their
validation is given by the present interval tests. The transcript
\path{spectrum/results/cpp_spectral_certificate.txt} records all
bounds, and  its comparison input identifies the exact potential and
the freshly validated profile by SHA-256  digests.

\section{The Constraint Matrix and the Virial Form}
\label{cap:m-section}

The spectral gap supplies the inverse bound needed for the constraint
matrix. For the functions in \eqref{cap:constraints}, we set
\[
 m_{ab}=(\mathcal L^{-1}f_a,f_b),\qquad a,b\in\{1,2,3,o\},
 \qquad A=(m_{ij})_{1\le i,j\le3}.
\]
Self-adjointness gives $m_{ab}=m_{ba}$. If $\mathcal R$ denotes the
reflection in $y_2$, then $\mathcal R\mathcal L=\mathcal L\mathcal R$
and the invertibility of $\mathcal L$ imply
$\mathcal R\mathcal L^{-1}=\mathcal L^{-1}\mathcal R$; thus
$\mathcal L^{-1}$ preserves parity, and $m_{io}=m_{oi}=0$ for
$i=1,2,3$. We enclose $A$ and $m_{oo}$ by means of compact approximate
inverses and of their true whole-space residuals.

\subsection*{Tail Bounds for the Constraints}

Let $\beta=1-(2\alpha_1)^{-1}>0$, and suppose that the ground-state
certificate supplies, for $r\ge R_Q>0$,
\[
 |Q(r)|\le C_0\frac{e^{-r}}{\sqrt r},\qquad
 |Q'(r)|\le C_1\frac{e^{-r}}{\sqrt r}.
\]
For $z=4$ and $R_Q=15$, the bounds of Section~\ref{cap:q-section}
allow the choices $C_0=4$ and  $C_1<3.694973$. Since both
$e^{-y_1/(2\alpha_1)}$ and $\widetilde\varphi(y_1)$ are at most
$e^{r/(2\alpha_1)}$, we have
\begin{equation}\label{cap:m-tail-majorants}
\begin{aligned}
 |f_1|&\le C_0r^{-1/2}e^{-\beta r},&
 |f_2|+|f_o|&\le 2C_1r^{-1/2}e^{-\beta r},\\
 |f_3|&\le(C_0+C_1r)r^{-1/2}e^{-\beta r}.
\end{aligned}
\end{equation}
Setting
\[
 (a_1,b_1)=(C_0,0),\quad
 (a_2,b_2)=(a_o,b_o)=(C_1,0),\quad
 (a_3,b_3)=(C_0,C_1),
\]
we thus have $|f_a|\le(a_a+b_ar)r^{-1/2}e^{-\beta r}$. If
$L\ge R_Q>0$, polar integration gives
\begin{equation}\label{cap:m-tail-integral}
 \|f_a\|_{L^2(r\ge L)}^2
 \le 2\pi\bigl(a_a^2J_0(L)+2a_ab_aJ_1(L)+b_a^2J_2(L)\bigr),
\end{equation}
where
\[
 J_0(L)=\frac{e^{-2\beta L}}{2\beta},\qquad
 J_k(L)=\frac{L^ke^{-2\beta L}}{2\beta}
       +\frac{k}{2\beta}J_{k-1}(L),\quad k=1,2;
\]
these identities follow by integration by parts.

\subsection*{Compact Cubic Splines and the Inverse Residual}

Let $h=2L/N$ and $x_j=-L+jh$, and define the cardinal cubic B-spline
\[
 b(t)=\begin{cases}
 (4-6t^2+3|t|^3)/6,&|t|\le1,\\
 (2-|t|)^3/6,&1\le|t|\le2,\\
 0,&|t|\ge2.
 \end{cases}
\]
For explicitly stored binary64 coefficients $c_{ij,a}$, interpreted
as exact real numbers, we set
\begin{equation}\label{cap:m-spline-definition}
 u_{h,a}(y_1,y_2)=\sum_{i,j=1}^{N-1}c_{ij,a}
 b\!\left(\frac{y_1-x_i}{h}\right)
 b\!\left(\frac{y_2-x_j}{h}\right).
\end{equation}
The spline $b$ belongs to $C^2(\mathbb R)$ and vanishes outside
$[-2,2]$; consequently, \eqref{cap:m-spline-definition} is globally
$C^2$, compactly supported in $[-L-h,L+h]^2$, and belongs to
$H^2(\mathbb R^2)$. On every knot rectangle, it is an exact tensor
cubic polynomial, and all polynomial derivatives used below are taken
within that rectangle; fourth derivatives across knot lines are not
assumed. The coefficients are proposed by a floating-point
sine-transform collocation solve: one requires the approximate equation
$\mathcal L u_{h,a}=f_a$ at the interior mesh nodes and solves the
resulting finite linear system. A discrete sine transform diagonalizes
the constant-coefficient tensor-grid part of this proposal system;
multiplication by the potential is evaluated at the nodes and couples
the sine modes. The resulting coefficients define the compact spline
in \eqref{cap:m-spline-definition}. Only its enclosed whole-space
residual below validates the inverse pairing; convergence of the
proposal solve alone is not a certificate.

For each $a\in\{1,2,3,o\}$, we define
\[
\begin{aligned}
 r_a&=f_a-\mathcal L u_{h,a},\qquad \|r_a\|_{L^2}\le\rho_a,\\
 B_{ab}&=(u_{h,a},f_b)+(f_a,u_{h,b})-\mathcal B(u_{h,a},u_{h,b}).
\end{aligned}
\]
The bilinear form $\mathcal B$ is the one defined by \eqref{cap:form};
since $u_{h,a}\in H^2$, it equals $(\mathcal L u_{h,a},u_{h,b})$ here.

\begin{proposition}[Quadratic residual bound]\label{cap:m-residual}
Suppose $\operatorname{dist}(0,\sigma(\mathcal L))\ge\delta>0$.
Then, for every $a,b\in\{1,2,3,o\}$,
\begin{equation}\label{cap:m-quadratic-identity}
 m_{ab}=B_{ab}+(\mathcal L^{-1}r_a,r_b).
\end{equation}
Consequently, any interval enclosure $[B_{ab}]$ gives
\begin{equation}\label{cap:m-quadratic-error}
 m_{ab}\in[B_{ab}]+[-\rho_a\rho_b/\delta,\rho_a\rho_b/\delta].
\end{equation}
\end{proposition}
\begin{proof}
Since $f_a=\mathcal L u_{h,a}+r_a$, self-adjointness gives
\[
\begin{aligned}
 m_{ab}
 &=(u_{h,a}+\mathcal L^{-1}r_a,\mathcal L u_{h,b}+r_b)\\
 &=(\mathcal L u_{h,a},u_{h,b})+(u_{h,a},r_b)
       +(r_a,u_{h,b})+(\mathcal L^{-1}r_a,r_b)\\
 &=B_{ab}+(\mathcal L^{-1}r_a,r_b).
\end{aligned}
\]
The spectral inverse bound and the Cauchy--Schwarz inequality then
imply $|(\mathcal L^{-1}r_a,r_b)|\le\rho_a\rho_b/\delta$.
\end{proof}

Both $r_a$ and $B_{ab}$ are evaluated directly with the enclosed true
$Q$, $W$, and $f_a$, so that all the profile uncertainty remains in the
integrands.

\subsection*{Certified Cell Quadrature}

We subdivide each spline knot rectangle into rectangles $K$, with
center $z_K$, half-widths $d_1,d_2$, and area $|K|=4d_1d_2$. For a
$C^2$ integrand $G$, the iterated midpoint theorem gives
\begin{equation}\label{cap:m-midpoint-cell}
 \int_KG\in |K|\left([G(z_K)]
       +\frac{d_1^2}{6}[\partial_{y_1}^2G(K)]
       +\frac{d_2^2}{6}[\partial_{y_2}^2G(K)]\right).
\end{equation}
Indeed, in one dimension, Taylor's formula with integral remainder
gives an average error $d^2G''(\xi)/6$ on an interval of half-width
$d$, and applying it successively in the two coordinates proves the
enclosure; the remainder kernels are nonnegative, so that derivative
ranges suffice, and no mixed derivative is needed in this formula. For
the residual, we set
\[
 e_{a,K}=\frac{d_1^2}{2}\sup_K|\partial_{y_1}^2r_a|
 +d_1d_2\sup_K|\partial_{y_1}\partial_{y_2}r_a|
 +\frac{d_2^2}{2}\sup_K|\partial_{y_2}^2r_a|.
\]
Taylor's formula writes $r_a$ as its affine polynomial at $z_K$ plus a
remainder of absolute value at most $e_{a,K}$, and the triangle
inequality in $L^2(K)$ together with the vanishing odd moments of the
centered rectangle yields
\begin{equation}\label{cap:m-residual-cell-taylor}
 \|r_a\|_{L^2(K)}^2\le |K|\left(
 \sqrt{|r_a(z_K)|^2+
 \frac{d_1^2|\partial_{y_1}r_a(z_K)|^2+
       d_2^2|\partial_{y_2}r_a(z_K)|^2}{3}}
 +e_{a,K}\right)^2.
\end{equation}
Every quantity on the right-hand side is replaced by an interval upper
bound. The Hessian of $r_a$ uses spline derivatives up to order four
within the knot rectangle and derivatives of $W$ and $f_a$ up to order
two; the latter are obtained from the radial ODE and the enclosed $Q$,
$Q_r$, so that no unknown Taylor-model remainder is differentiated. We
then integrate over all knot rectangles in $[-L-h,L+h]^2$. Outside that
square, $u_{h,a}=0$ and $r_a=f_a$, so that \eqref{cap:m-tail-integral}
with $L+h$ supplies the remaining residual and constraint norm bounds,
while the central integrals $B_{ab}$ vanish outside the spline
support. Symmetric entry enclosures are intersected, and the exact
mixed even--odd entries are zero.

\subsection*{The Recorded Constraint Certificate}

The compact-spline computation uses the following inputs. The spline
grid is given by $L=20$, $N=512$, and $h=2L/N=5/64$, so that the
support is contained in $[-20-5/64,20+5/64]^2$; the coefficients are
four $(N-1)\times(N-1)$ arrays in \texttt{coefficients.f64},
interpreted as exact dyadic numbers; the profile is the validated input
of Section~\ref{cap:q-section}, with $R_Q=15$, $z=4$, and the tail
constants above; the quadrature uses adaptive subdivision at the square
radii $3,6,10$, with respective factors $8,4,2,1$ per coordinate; and
the inverse bound is $\delta=31/50$, justified by
Proposition~\ref{cap:spectral-gap}. The recorded coefficient file has
SHA-256
\begin{center}\small\ttfamily
4c2b8a95dd992e7d8778e88c88a7a592d4ac78261733246949225eccea01cf19.
\end{center}
The recorded integration uses $1\,022\,644$ rectangles, obtained by
subdividing each knot cell into $8^2$, $4^2$, $2^2$, or one rectangle,
according as the maximum absolute coordinate of its center is below
$3$, between $3$ and $6$, between $6$ and $10$, or at least $10$. The
executed residual integration required approximately  $437$ seconds in
the recorded environment. The outward-rounded upper bounds for the
whole-space residuals and for the constraint norms are
\[
\begin{aligned}
 \rho_1&\le0.110293,& \rho_2&\le0.111000,&
 \rho_3&\le0.220809,& \rho_o&\le0.117280,\\
 \|f_1\|_{L^2}&\le4.048337,& \|f_2\|_{L^2}&\le2.908460,&
 \|f_3\|_{L^2}&\le3.853534,& \|f_o\|_{L^2}&\le2.609398.
\end{aligned}
\]
All the profile uncertainty, the finite quadrature remainders, and the
exterior tails are included. The finite integration uses directly the
true interval potential and constraints, so that this certificate needs
no further approximation of the constraints or of the potential. Using
$\delta=31/50$ in \eqref{cap:m-quadratic-error} gives the following
conservatively widened enclosure of the even block:
\begingroup\small\setlength{\arraycolsep}{4pt}
\[
 [A]=\begin{pmatrix}
 [2.830253,2.869958]&[0.560263,0.600181]&[-4.367974,-4.288499]\\
 [0.560263,0.600181]&[1.825078,1.865248]&[2.556290,2.636233]\\
 [-4.367974,-4.288499]&[2.556290,2.636233]&[12.562538,12.721890]
 \end{pmatrix}.
\]
\endgroup
For every real symmetric matrix in this enclosure,
\[
 \det A\in[-7.946966743614993375,-1.232927442846774158]<0.
\]
The odd pairing is enclosed separately by
\[
 m_{oo}\in[-0.669373,-0.624397]<0.
\]
All displayed decimal endpoints are exact rationals. To prove the
determinant enclosure, we write
$A=\left(\begin{smallmatrix}a&b&c\\b&d&e\\c&e&f\end{smallmatrix}\right)$,
so that
\[
 \det A=adf+2bce-ae^2-dc^2-fb^2.
\]
Throughout the displayed box, $b,d,e>0>c$, and
$df-e^2$, $af-c^2$, and $ad-b^2$ have positive lower bounds;
consequently, the determinant increases in $a,c,d,f$ and decreases in
$b,e$, and evaluating the two corresponding corners with exact rational
arithmetic gives precisely the stated endpoints. Moreover,
\[
 A_{11}>0,\qquad
 A_{11}A_{22}-A_{12}^2
       \in[4.805215251973,5.039288790415].
\]
The leading $2\times2$ block is therefore positive definite, and its
Schur complement in $A$, which equals $\det A$ divided by this positive
principal minor, is negative. Completing the square then proves that
every symmetric member of $[A]$ has one negative and two positive
eigenvalues.

The executed source \path{constraints/spline_verify.cpp} integrates
 all $1\,022\,644$ rectangles  using the validated $Q$ profile earlier computed from CAPD. Its 
 integral matrix, residuals, and inverse matrix enclosure are
recorded, together with a separate hexadecimal endpoint file, in
\path{constraints/results/constraint_certificate.json}.

The displayed matrix widens these computed
intervals outward. An exact rational calculation verifies the determinant
monotonicity and corner values on this displayed box; a separate  C++
replay also checks the matrix formula and the signed pivots
from the stored interval endpoints. The inverse norm assumption in that
residual calculation is discharged by the spectral certificate of
Section~\ref{cap:s-section}.

These are precisely the numerical inputs of Lemma~\ref{virialcoercivity}:
there, $\det A<0$ and the symmetry of $A$ give, through
Lemma~\ref{Sylvester}, an even combination $f_e$ of $f_1,f_2,f_3$ with
$(\mathcal L^{-1}f_e,f_e)<0$, parity gives
$(\mathcal L^{-1}f_e,f_o)=0$, and $m_{oo}<0$ completes a negative
definite $2\times2$ inverse Gram matrix; since
$\operatorname{ind}_-(\mathcal L)=2$ and $0\in\rho(\mathcal L)$,
Proposition~\ref{Generalized Weinstein Lemma} then yields the
coercivity of the virial form on the four-constraint subspace. We do
not repeat that argument here.

\section{The Constant Governing the Blow-Up Rate}
\label{cap:c-section}

We return to the variables $(x_1,x_2)$ and to the function $g_Q$
defined in \eqref{definitionofg}. In this section, $A$ denotes the
one-dimensional differential operator below, and all inner products
involving $g_Q$ are taken in $L^2(\mathbb R)$. Writing $z=x_2$, we have
\begin{equation}\label{cap:c-definition}
 g_Q(z)=z\int_{\mathbb R}Q_{x_2}(x_1,z)\,dx_1,
 \qquad
 c=\frac{2(A^{-1}g_Q,g_Q)}{(g_Q,g_Q)},\qquad
 A=-\frac{d^2}{dz^2}+1.
\end{equation}
The Fourier transform sends $A$ to the multiplication by $1+\xi^2$;
thus $A$ is self-adjoint on $H^2(\mathbb R)$ and
$\|A^{-1}\|_{L^2\to L^2}=1$. The function $k(z)=e^{-|z|}/2$ solves
$Ak=0$ away from zero and has derivative jump $-1$ there, so that $Ak$
is the unit point mass at zero and its translates give the Green kernel
$e^{-|z-t|}/2$. Plancherel's identity identifies this quotient with
\eqref{eq:definitionofc}, the Fourier normalization cancelling. For
later use, we check the elementary analytic properties directly.
Radiality gives
\begin{equation}\label{cap:c-radial}
 g_Q(z)=2z^2\int_0^\infty
 \frac{Q'(\sqrt{z^2+t^2})}{\sqrt{z^2+t^2}}\,dt,
\end{equation}
where the quotient $Q'(r)/r$ is defined at zero by its limit $Q''(0)$.
This formula proves that $g_Q$ is even and, since $Q'(r)<0$ for $r>0$,
that $g_Q(z)<0$ for $z\ne0$; in particular, $g_Q$ is not identically
zero. Smoothness near the origin and exponential decay at infinity give
$g_Q\in L^2(\mathbb R)$: for instance, for large $|z|$, the bound
$|Q'(r)|\le C e^{-r}$ and $r\ge(|z|+|x_1|)/\sqrt2$ give
$|g_Q(z)|\le2\sqrt2C|z|e^{-|z|/\sqrt2}$. Integration by parts in $x_1$
also gives the exact identity
\begin{equation}\label{cap:c-lambda-identity}
 g_Q(z)=\int_{\mathbb R}\Lambda Q(x_1,z)\,dx_1,
\end{equation}
because $\int(Q+x_1Q_{x_1})\,dx_1=[x_1Q]_{-\infty}^{\infty}=0$. The
multiplier $(1+\xi^2)^{-1}$ is positive and strictly below one away
from $\xi=0$, and since a nonzero $L^2$ Fourier transform cannot be
supported at a single point, \eqref{cap:c-definition} implies $0<c<2$.
Proposition~\ref{c>1} in Appendix~\ref{appendixD} establishes the
stronger analytic inequality $c>1$. The remaining numerical task is to
enclose this quotient with enough precision for the choice of the
polynomial weight.

\subsection*{Radial Quadrature and Tails}

The  CAPD calculation uses the soliton $Q$ profile validated in
Section~\ref{cap:q-section}. With
$R_Q=15$ and fixed-point radius $4$, the tail construction gives
$|Q(r)|\le C_0r^{-1/2}e^{-r}$ with $C_0=4$, and
\[
 |Q'(r)|\le C_1r^{-1/2}e^{-r},\qquad C_1<3.694973,
 \qquad r\ge15.
\]
Set $F(s)=Q'(\sqrt{s})/\sqrt{s}$, with its continuous value at $s=0$,
so that
\[
 g_Q(z)=2z^2\int_0^\infty F(z^2+t^2)\,dt.
\]
The numerical inputs are the profile tail radius $R_Q=15$, the tail
derivative bound $C_1<3.694973$, the support radius $L_g=16$
and the radial cutoff $R_c=24$, the radial and transverse spacings
$h_r=h_z=2^{-8}$, and the origin series of order $N=40$ on the region
$0\le r\le1/2$. The radial nodes are $r_i=i h_r$ and $s_i=r_i^2$. At
each radial node, a binary64 midpoint $f_i$ of the validated enclosure
of $F(s_i)$ is fixed as an exact real number, and we let
$\ell_i(s)=a_i+b_i s$ interpolate $f_i,f_{i+1}$ on $[s_i,s_{i+1}]$. If
$\nu_i$ bounds the two nodal errors and $B_i$ bounds
$\sup_{[s_i,s_{i+1}]}|F''|$, then
\[
 |F(s)-\ell_i(s)|\le e_i:=\nu_i+
       \frac{(s_{i+1}-s_i)^2}{8}B_i;
\]
this is the usual interpolation remainder applied to the exact
function, followed by the linear interpolation of the nodal errors.
For $r\le1/2$, the calculation uses the origin recurrence with $N=40$:
\[
 F(s)=2\sum_{n\ge1}nA_ns^{n-1},\qquad |A_n|\le a_*b^n,
 \quad b=(1+a_*^2)/4.
\]
All the needed derivatives of this series have explicit geometric
remainder bounds. Specifically, for $d=0,1,2$, we set $n_0=N+1$ and
\[
 \rho_d=b s_*\left(\frac{n_0+1}{n_0-d}\right)^{d+1}<1.
\]
The first omitted majorant term for the $d$th derivative in $s$,
divided by $1-\rho_d$, bounds the entire remainder on $0\le s\le s_*$;
this ratio bound is conservative, since the exact ratio of consecutive
coefficient factors is $(n+1)/(n-d)$, and here $s_*\le1/4$. For
$r\ge1/2$, we use $s=r^2$ and $Q_r=rF(s)$, so that $Q_{rr}=F+2sF'$ and
$\partial_s[Q(\sqrt{s})]=F/2$. Substitution into the radial equation,
followed by differentiation in $s$, gives
\[
 F'(s)=\frac{Q-Q^3-2F}{2s},\qquad
 F''(s)=\frac{(1-3Q^2)F-8F'(s)}{4s}.
\]
These identities enclose the derivatives from $Q$ and $Q_r$ without
differentiating a Taylor-model remainder, and they also give explicit
exterior bounds. Set
\[
\begin{aligned}
 C_{F'}&=\frac{C_0+C_0^3e^{-2R_c}/R_c+2C_1/R_c}{2},\\
 C_{F''}&=\frac{(1+3C_0^2e^{-2R_c}/R_c)C_1+8C_{F'}/R_c}{4}.
\end{aligned}
\]
For $r\ge R_c$, the profile bounds and the two identities imply
\begin{equation}\label{cap:c-F-tail}
 |F(r^2)|\le C_1\frac{e^{-r}}{r^{3/2}},\quad
 |F'(r^2)|\le C_{F'}\frac{e^{-r}}{r^{5/2}},\quad
 |F''(r^2)|\le C_{F''}\frac{e^{-r}}{r^{7/2}}.
\end{equation}
For instance, one divides the three terms in $Q-Q^3-2F$ by $2r^2$ and
bounds the remaining decreasing factors by their values at $R_c$; the
estimate for $F''$ follows in the same way.

At a fixed $z$ node, the radial shell corresponds to a $t$ interval
$[u,v]$, obtained by intersecting $[s_i,s_{i+1}]$ with $[z^2,R_c^2]$
and taking $t=\sqrt{s-z^2}$. Its central integral is exactly
\[
 (v-u)\left[a_i+b_i\left(z^2+\frac{v^2+uv+u^2}{3}\right)\right],
\]
and its error is at most $e_i(v-u)$. Every square root, product, sum,
and division in these formulas is evaluated by CAPD interval
arithmetic. If $T=\sqrt{R_c^2-z^2}$, the omitted radial tail is
bounded by
\[
 \int_T^\infty |F(z^2+t^2)|\,dt
 \le\frac{C_1}{R_c^{3/2}}(R_c-T+1)e^{-R_c}.
\]
Indeed, $r\ge R_c$ throughout that integral and $r\ge t$, and one
splits the integral at $t=R_c$. At the nodes $z_j=jh_z$,
$0\le j\le L_g/h_z$, we choose the binary64 midpoint of the enclosure
of $-g_Q(z_j)$. These values define a piecewise linear function $p_h$
on $[0,L_g]$, which we extend evenly to $[-L_g,L_g]$ and by zero
outside, and we set $g_h=-p_h$. The origin series and
\eqref{cap:c-F-tail} provide integrable majorants for the first two
$z$ derivatives, uniformly on compact $z$ intervals, so that
differentiation under the integral sign is justified. Differentiating
$2z^2F(z^2+t^2)$ twice gives
\[
 g_Q''(z)=4\int_0^\infty F\,dt+
 20z^2\int_0^\infty F'\,dt+
 8z^4\int_0^\infty F''\,dt.
\]
For a cell $[z_j,z_{j+1}]\subset[0,L_g]$, we partition the $t$ axis
using the radial shells at $z_j$. The map $z\mapsto\sqrt{z^2+t^2}$ is
$1$-Lipschitz, so that the radius increases by at most $h_z$; we thus
evaluate $F(r^2),F'(r^2),F''(r^2)$ for $r\in[r_i,r_{i+1}+h_z]$ on the
corresponding portion of the cell, and we integrate the absolute bounds
over these finite portions. For the remaining
$t\ge T_j:=\sqrt{R_c^2-z_j^2}$, the same split at $R_c$ and
\eqref{cap:c-F-tail} give the explicit contribution
\[
 (R_c-T_j+1)e^{-R_c}
 \left(\frac{4C_1}{R_c^{3/2}}
 +\frac{20z_{j+1}^2C_{F'}}{R_c^{5/2}}
 +\frac{8z_{j+1}^4C_{F''}}{R_c^{7/2}}\right).
\]
If $m_j\ge\sup|g_Q''|$ on that $z$ cell and $v_j$ bounds its two nodal
errors, the local error is at most $v_j+h_z^2m_j/8$. The remaining
transverse tail also has an explicit bound: for $z\ge L_g\ge15$,
radiality and $r\ge z$ give
\[
 |g_Q(z)|\le2C_1\sqrt z\int_0^\infty e^{-\sqrt{z^2+t^2}}\,dt
            \le2C_1\sqrt z\,(z+1)e^{-z},
\]
where, for the last inequality, we split at $t=z$ and use
$\sqrt{z^2+t^2}\ge z$ on the first interval and
$\sqrt{z^2+t^2}\ge t$ on the second. Let
$J_n=\int_{L_g}^{\infty}z^ne^{-2z}\,dz$, so that $J_0=e^{-2L_g}/2$ and
$J_n=L_g^ne^{-2L_g}/2+nJ_{n-1}/2$. The total error bound is
\[
 \varepsilon_g^2=
 2\sum_jh_z\left(v_j+\frac{h_z^2m_j}{8}\right)^2
 +8C_1^2(J_3+2J_2+J_1).
\]

\subsection*{The Quotient Certificate}

For the exact piecewise linear approximation just defined, we set
\begin{equation}\label{cap:c-central-integrals}
\begin{aligned}
 D_h&=\int_{-L_g}^{L_g}g_h(z)^2\,dz,\\
 N_h&=\frac12\int_{-L_g}^{L_g}\int_{-L_g}^{L_g}
 e^{-|z-t|}g_h(z)g_h(t)\,dt\,dz.
\end{aligned}
\end{equation}
Then $N_h=(A^{-1}g_h,g_h)$ for the whole-line inverse, and since
$g_h=-p_h$, the same integrals are obtained with $p_h$. For this
integration, we index the full symmetric mesh by $z_j=-L_g+jh_z$,
$0\le j\le2L_g/h_z$, and on a cell we write $p_h(z_j+u)=a+bu$,
$0\le u\le h_z$. The function
$S(z)=\int_{-L_g}^z e^{-(z-t)}p_h(t)\,dt$
satisfies $S'+S=p_h$, and if $S_j=S(z_j)$, its exact solution there is
\[
 S(z_j+u)=a+b(u-1)+(S_j-a+b)e^{-u}.
\]
Thus, starting from $S_0=0$,
\[
 S_{j+1}=a+b(h_z-1)+(S_j-a+b)e^{-h_z}.
\]
Integrating $p_h^2$ and $p_hS$ on the cell gives, respectively,
\[
 a^2h_z+ab h_z^2+b^2h_z^3/3
\]
and
\[
 a^2h_z+ab(h_z^2-h_z)+b^2(h_z^3/3-h_z^2/2)
 +(S_j-a+b)\{a(1-e^{-h_z})+
 b[1-(1+h_z)e^{-h_z}]\}.
\]
The symmetry of the Green kernel gives
\[
 N_h=\int p_h(z)\int_{-L_g}^z e^{-(z-t)}p_h(t)\,dt\,dz.
\]
All polynomial and exponential evaluations use outward arithmetic.
This evaluates the whole-line inverse without introducing a finite-box
boundary condition.

\begin{proposition}\label{cap:c-quotient}
Suppose that
\[
 \|g_Q-g_h\|_{L^2}\le\varepsilon_g,\qquad
 \|g_h\|_{L^2}\le G_h,
 \qquad
 D_h\in[D_h^-,D_h^+],\quad N_h\in[N_h^-,N_h^+].
\]
Set
\begin{equation}\label{cap:c-quadratic-error}
 \begin{aligned}
 E_g&=2G_h\varepsilon_g+\varepsilon_g^2,\\
 D_-&=D_h^--E_g, &D_+&=D_h^++E_g,\\
 N_-&=\max\{0,N_h^--E_g\}, &N_+&=N_h^++E_g.
 \end{aligned}
\end{equation}
If $D_->0$, then
\begin{equation}\label{cap:c-quotient-enclosure}
 c\in\left[\frac{2N_-}{D_+},\frac{2N_+}{D_-}\right].
\end{equation}
In particular, the strict finite inequality
\begin{equation}\label{cap:c-lower-test}
 2(N_h^--E_g)-\frac{21}{13}(D_h^++E_g)>0
\end{equation}
implies $c>21/13$. The additional inequality
\begin{equation}\label{cap:c-upper-test}
 \frac53(D_h^--E_g)-2(N_h^++E_g)\ge0
\end{equation}
implies $c\le5/3$.
\end{proposition}
\begin{proof}
Let $e_g=g_Q-g_h$, $D=(g_Q,g_Q)$, and
$N=(A^{-1}g_Q,g_Q)$. Expansion and the Cauchy--Schwarz inequality give
\[
 |D-D_h|\le2\|g_h\|_{L^2}\|e_g\|_{L^2}+\|e_g\|_{L^2}^2\le E_g,
\]
and self-adjointness together with $\|A^{-1}\|=1$ likewise implies
\[
\begin{aligned}
 |N-N_h|
 &=|2(A^{-1}g_h,e_g)+(A^{-1}e_g,e_g)|\\
 &\le2\|g_h\|_{L^2}\|e_g\|_{L^2}+\|e_g\|_{L^2}^2\le E_g.
\end{aligned}
\]
Since $N>0$ and $D>0$, these inequalities yield $N\in[N_-,N_+]$ and
$D\in[D_-,D_+]$, and division by the positive denominator proves
\eqref{cap:c-quotient-enclosure}. The last two claims follow by cross
multiplication with the same positive bounds.
\end{proof}

The norm bound $G_h$ is obtained by outward evaluation of
$\sqrt{D_h^+}$. The recorded CAPD output gives the following
conservative decimal enclosures, all margins being evaluated from the
stored interval endpoints before the decimal bounds are widened for
display:
\[
\begin{aligned}
 D_h&\in[28.6713738040094,\ 28.6713738040246],\\
 N_h&\in[23.8185433397371,\ 23.8185433397610],\\
 \varepsilon_g&\le0.000162843515960,\qquad
 E_g\le0.001743939127094,\\
\end{aligned}
\]
\[
 D_h^- -E_g>28.6696298648823;
\]
the lower margin in \eqref{cap:c-lower-test} is greater than
$1.3154855238212$, the upper margin in \eqref{cap:c-upper-test} is
greater than $0.1421418836944$, and
\[
 c\in[1.66126332085379,\ 1.66170874135113] \subset (\frac{21}{13}, \frac{5}{3}).
\]

We note here that the quotient calculation depends only on the independently  validated $Q$ profile.  Its source is \path{quotient/constant_verify.cpp}, and
\path{quotient/results/quotient.json} records the exact hexadecimal
interval endpoints.

The file \path{quotient/results/approximant.hex}
 specifies every binary64 node of $p_h$; its digest is included in the
package manifest.

For the exponent $K_i=\lfloor c(i+3)/(c-1)\rfloor+1$ of the polynomial
weight introduced in Chapter~\ref{Monotonicity Formulas}, the precise
role of the threshold is
\begin{equation}\label{cap:c-weight-index}
 K_5\le21\quad\Longleftrightarrow\quad c>\frac{21}{13},
 \qquad
 K_5=21\quad\Longleftrightarrow\quad
       \frac{21}{13}<c\le\frac53,
\end{equation}
where $c>1$ is understood. Indeed, $8c/(c-1)=8+8/(c-1)$ is decreasing
on $(1,\infty)$, and $K_5=21$ means that $20\le8c/(c-1)<21$. Thus
\eqref{cap:c-lower-test} makes the polynomial weight of order $21$
admissible, while \eqref{cap:c-upper-test} also identifies the integer
$K_5$ with $21$; the strict positivity of the left-hand side of
\eqref{cap:c-upper-test} is a sufficient test for its nonnegative
version.

\begin{remark}[Assembly of the certificates and numerical data]
The soliton $Q$ construction, the potential comparison, the finite
spectral tests, the compact-spline residuals, and the quotient bounds
are executed in the companion package
\texttt{Chapter\_E\_CAPD\_fresh}. The complete source code and implementation files are available in the public GitHub repository at \href{https://github.com/kk10-genius/ZK2d3p-blowup}{\textcolor{blue}{https://github.com/kk10-genius/ZK2d3p-blowup}}. The ground-state validation uses
only the origin recurrence, CAPD Taylor integration, and the Bessel
tail construction.
The common exported profile is then used in every potential, residual,
and quotient calculation.

The potential comparison transfers the finite spectral bounds to the
whole space. The resulting gap controls the quadratic inverse-residual
correction and supplies the constraint signs used in
Proposition~\ref{Generalized Weinstein Lemma}. %The  quotient calculation
%uses only the validated profile.  \DIFaddbegin \DIFadd{Thus no certificate uses }\DIFaddend the
%coercivity conclusion  \DIFaddbegin \DIFadd{as an input. The previously proposed potential
%nodes, sine }\DIFaddend trial vectors,  \DIFaddbegin \DIFadd{congruence matrices, and spline coefficients
%are retained as exact finite approximations and verified anew.
%}

The package contains the C++ sources, the pinned CAPD source archive,
the exact profile coefficients and interval remainders, all numerical
witnesses, and the executed output of each stage. A replay script
repeats the stages in their dependency order, and a manifest gives
 SHA-256 digests.  Binary interval endpoints are retained in hexadecimal
form; the decimal bounds printed in this appendix are widened
outward. %Failed matching inclusions, unresolved signs, or a failure of
%any arithmetic or coverage check cause the corresponding certificate
%to be rejected.
\end{remark}

\backmatter
\bibliographystyle{alpha} 
\bibliography{ZK3}

\end{document}